\documentclass[11pt]{article}

\usepackage{amsmath,amsthm,amsfonts,amssymb,epsfig, titlesec, epsfig, float, caption, mathrsfs}
\usepackage[hidelinks]{hyperref}
\usepackage{tikz-cd}
\usepackage{tikz}
\usepackage{tikz-3dplot}
\usepackage{xcolor}
\usepackage{verbatim}
\usepackage{accents}
\usepackage[citestyle=alphabetic,bibstyle=alphabetic,backend=bibtex,maxbibnames=10,minbibnames=5]{biblatex}
\usepackage{todonotes}
\usepackage[american]{babel}

\renewbibmacro{in:}{}

\usetikzlibrary{calc, decorations.pathreplacing,shapes.misc}
\usetikzlibrary{decorations.pathmorphing}
\usepackage[left=1in,top=1in,right=1in]{geometry}
\usepackage[capitalize]{cleveref}
\usepackage{subcaption}

\newtheorem{thm}{Theorem}[section]
\newtheorem{lem}[thm]{Lemma}
\newtheorem{prop}[thm]{Proposition}
\newtheorem{qs}[thm]{Question}
\newtheorem{cor}[thm]{Corollary}

\newtheorem{notation}[thm]{Notation}

\theoremstyle{remark} 
 \newtheorem{rem}[thm]{Remark}
 \crefname{rem}{Remark}{Remarks}
 \Crefname{rem}{Remark}{Remarks}
 \newtheorem{ex}[thm]{Example}

\theoremstyle{definition} 
\newtheorem{df}[thm]{Definition} 

\titleformat*{\section}{\normalsize \bfseries \filcenter}
\titleformat*{\subsection}{\normalsize \bfseries }
\newtheorem{mainthm}{Theorem}
\Crefname{mainthm}{Theorem}{Theorems}
\newtheorem{maincor}[mainthm]{Corollary}
\Crefname{maincor}{Corollary}{Corollaries}

\makeatletter
\def\namedlabel#1#2{\begingroup
   \def\@currentlabel{#2}\label{#1}\endgroup
}
\makeatother 
\newcommand{\eps}{\varepsilon}

\newcommand{\wh}{\widehat}
\newcommand{\wb}{\overline}
\newcommand{\bb}{\mathbb}

\newcommand{\RR}{\mathbb R}
\newcommand{\ZZ}{\mathbb Z}
\newcommand{\CC}{\mathbb C}
\newcommand{\PP}{\mathbb P}
\newcommand{\LL}{\mathbb L}
\newcommand{\II}{\mathbb I}
\newcommand{\CP}{\mathbb{CP}}
\newcommand{\del}{\nabla}
\newcommand{\into}{\hookrightarrow}
\newcommand{\tensor}{\otimes}
\newcommand{\Lc}{L}
\newcommand{\wrpalpha}{\phi_{\sF_\alpha}^{1-\eps}}
\renewcommand{\Re}{\text{Re}}

\newcommand{\CF}{CF^\bullet}
\newcommand{\HF}{HF^\bullet}
\newcommand{\wrp}[1]{\phi_{#1,\delta}^{t}} 
\newcommand{\tb}[1]{\check X_{\str(#1)/#1}}
\newcommand{\Halpha}{\widehat H^{\delta}_{\alpha}}
\newcommand{\afan}{\str(\alpha)/\alpha}
\newcommand{\core}{\mathfrak{c}}
\newcommand{\cocore}{\mathfrak{u}}
\newcommand{\stp}{\mathfrak{f}}
\newcommand{\hamK}[1]{(-\eps K_{#1})}
\newcommand{\hamH}[1]{\widehat H_{#1}}
\newcommand{\wrpK}[1]{\psi_{K_{#1}}^{t}} 
\newcommand{\thetaa}{p}
\newcommand{\sF}{{F}} \newcommand{\sP}{{P}} \newcommand{\funnyi}{\underaccent{\breve}{i}}
\newcommand{\funnyj}{\underaccent{\breve}{j}}
\DeclareMathOperator{\id}{id}

\DeclareMathOperator{\dg}{dg}
\DeclareMathOperator{\Hom}{Hom}
\DeclareMathOperator{\Log}{Log}
\DeclareMathOperator{\Conn}{Conn}

\DeclareMathOperator{\Coh}{Coh}
\DeclareMathOperator{\Pic}{Pic}
\DeclareMathOperator{\st}{star}
\DeclareMathOperator{\Fuk}{Fuk}

\newcommand{\val}{u}
\DeclareMathOperator{\Tw}{Tw}
\DeclareMathOperator{\bval}{\check{\Log}}
\DeclareMathOperator{\str}{star}

\DeclareMathOperator{\Ob}{Ob}
\DeclareMathOperator{\Supp}{Supp}

\DeclareMathOperator{\Vol}{Vol}

\DeclareMathOperator{\Bl}{Bl} 

\newcommand{\Addresses}{{\bigskip
  \footnotesize

  \noindent A.~Hanlon, \textsc{Department of Mathematics, Stony Brook University and Simons Center for Geometry and Physics}\par\nopagebreak
  \noindent \textit{E-mail address}: \texttt{andrew.hanlon@stonybrook.edu}

  \medskip

  \noindent J.~Hicks, \textsc{School of Mathematics,  University of Edinburgh}\par\nopagebreak
  \noindent \textit{E-mail address}: \texttt{jeff.hicks@ed.ac.uk}

  \medskip

}}

\begin{document}

\title{\normalsize \textbf{Aspects of functoriality in homological mirror symmetry for toric varieties}}
\author{\normalsize Andrew Hanlon and Jeff Hicks}
\date{}
\maketitle

\begin{abstract}    We study homological mirror symmetry for toric varieties, exploring the relationship between various Fukaya-Seidel categories which have been employed for constructing the mirror to a toric variety. 
In particular, we realize tropical Lagrangian sections as objects of a partially wrapped category and construct a Lagrangian correspondence mirror to the inclusion of a toric divisor. 
As a corollary, we prove that tropical sections generate the Fukaya-Seidel category, completing a Floer-theoretic proof of homological mirror symmetry for projective toric varieties.
In the course of the proof, we develop techniques for constructing Lagrangian cobordisms and Lagrangian correspondences in Liouville domains, which may be of independent interest.
 \end{abstract}

\maketitle

\section{Introduction}
      Mirror symmetry, first observed in the work of \cite{candelas1991pair}, is now a broad-reaching set of proposed equivalences between symplectic geometry on a Calabi-Yau manifold $X$ and complex geometry on a mirror space $\check X$.
In this article, we'll explore the interplay between three aspects of this theory.
The work of \cite{strominger1996mirror} predicts that mirror spaces $X$ and $\check X$ can be constructed as dual Lagrangian torus fibrations (SYZ-fibrations) over a common affine base $Q$.
The homological mirror symmetry (HMS) conjecture from \cite{kontsevich1995homological} proposes that the Fukaya category of $X$ is derived equivalent to the category of coherent sheaves on $\check X$.
Finally, \cite{hori2000mirror} states that when we replace a space $\check X$ with a compactification, we should modify the geometric model considered on the mirror space $X$ via incorporation of a ``superpotential'' $W: X\to \CC$ whose behavior at the boundary of $X$ encodes the compactifying divisors added to $\check X$.

Toric varieties have become a central object of study in mirror symmetry as they provide a set of examples to test all three of these features. 
The presence of a moment map gives them a Lagrangian torus fibration, and the additional combinatorial structure makes the derived category of coherent sheaves particularly amenable to computation.
Additionally, toric varieties naturally present themselves as a compactification of the algebraic torus, which we denote by $X= \check X= (\CC^*)^n$.
Consider the compactification of $\check X$ to a toric variety $\check X_\Sigma$ given by fan $\Sigma \subset \bb{R}^n$ with $A \subset \bb{Z}^n$ indexing the primitive generators of $\Sigma$. In \cite{hori2000mirror} it was deduced from physical considerations that the compactification is mirror to the incorporation of the Hori-Vafa superpotential (a Laurent polynomial) on $X$,
\[ W_\Sigma = \sum_{\alpha \in A} c_\alpha z^\alpha \]
where $z^\alpha = z_1^{\alpha_1}\hdots z_n^{\alpha_n}$ and $c_\alpha \in \bb{C}^*$ are coefficients encoding the K\"{a}hler structure of $X$.
The superpotential $W_\Sigma$ is a holomorphic function on $X=(\CC^*)^n$, and the pair $(X, W_\Sigma)$ is a Landau-Ginzburg model mirror to $\check X_\Sigma$. 
Since $(X, W_\Sigma)$ has the same underlying symplectic manifold as $X$, to understand the HMS conjecture for toric varieties we need to state how the superpotential modifies the geometry as computed by the Fukaya category. Additionally, it would be satisfying to understand this superpotential from an SYZ perspective.

There are now a variety of perspectives on the equivalence between the category of coherent sheaves on $\check X_\Sigma$ and a Fukaya category associated to the pair $((\bb{C}^*)^n, W_\Sigma)$.
In particular, there are several different flavors of the Fukaya category associated to the pair $((\bb{C}^*)^n, W_\Sigma)$, each useful for studying a different aspect of the mirror relation between $\check X_\Sigma$ and $((\CC^*)^n, W_\Sigma)$.
We now briefly summarize some of these approaches from the existing literature.

Following ideas of Kontsevich, \cite{seidel2001more} defined the Fukaya-Seidel category, $\mathcal{FS}(Y,W)$, of a holomorphic function $W$ on an exact K\"{a}hler manifold $Y$.  Objects of $\mathcal{FS}(Y, W)$ are supported on Lagrangians submanifolds which, outside of a compact subset, project to the real positive axis of $\CC$ under  $W$. 
Using this framework, Seidel deduced HMS for $\check X_\Sigma = \bb{P}^2$; this argument was later generalized by  \cite{auroux2006mirror} to del Pezzo surfaces (including the non-toric ones) and weighted projective planes. 

Further input on the construction of these mirror pairs came from the SYZ conjecture.
By deforming the superpotential $W_\Sigma$ to its ``tropical localization," \cite{abouzaid2009morse} produced a subcategory of ``tropical Lagrangian sections'' of $\mathcal{FS}(Y, W)$ equivalent to line bundles on $\check X_\Sigma$ for any smooth and projective toric variety. 
The first author in \cite{hanlon2019monodromy} developed a variant of Seidel's construction called the monomially admissible Fukaya-Seidel category and proved an analog of Abouzaid's result.
Short of a generation gap, this approach provides an entirely geometric (i.e., sheaf-free) proof of homological mirror symmetry for toric varieties. 
Using tropical Lagrangian sections has the advantage of following an  ``SYZ-first'' viewpoint on HMS, where relations between affine geometry of toric varieties, tropical geometry, and symplectic geometry of $(\CC^*)^n$ inform the HMS identification. 

A separate approach to HMS for toric varieties comes from \cite{fang2011categorification}, which associated to the data of the fan $\Sigma$ a Lagrangian skeleton for $X$,
\[ \LL_\Sigma = \bigcup_{\sigma \in \Sigma} \sigma^\perp \times \sigma \subset T^*T^n = (\bb{C}^*)^n,\]
and proposed that the category of constructible sheaves on $T^n$ with microsupport in $\LL_\Sigma$ is a suitable stand-in for the Fukaya category of $((\bb{C}^*)^n, W_\Sigma)$. 
In this setting, they were able to prove an equivariant version of HMS for projective toric varieties.
This result was extended to the non-equivariant setting and to arbitrary toric varieties by \cite{kuwagaki2020nonequivariant}.

Recent developments in Fukaya categories in the non-compact setting open the prospect of unifying these different categories associated with the pair $(X, W_\Sigma)$.
\cite{sylvan2019partially} introduced the partially wrapped Fukaya category, a general construction of a category associated to the data of a Landau-Ginzburg pair.
This category has been additionally studied by \cite{ganatra2017covariantly}.
In the greatest generality, the partially wrapped Fukaya category, $\mathcal{W}(Y,\mathfrak{f})$, is defined for a symplectic manifold $Y$ which is conical at infinity and a closed subset $\mathfrak{f}$ in the contact boundary of $Y$.
Objects are supported on Lagrangians that are conical at infinity and do not intersect $\mathfrak{f}$. 

There are known relations between the partially wrapped Fukaya category, the Fukaya-Seidel category, and sheaves with microlocal support on the FLTZ skeleton.
To recover the Fukaya-Seidel category, we look at the partially wrapped Fukaya category $\mathcal{W}((\bb{C}^*)^n, W_\Sigma^{-1}(-\infty))$. 
To relate the partially wrapped Fukaya category to the FLTZ skeleton, we can also consider $\mathcal{W}((\bb{C}^*)^n, \stp_\Sigma)$ where $\stp_\Sigma$ is the intersection of $\bb{L}_\Sigma$ with the boundary at infinity.
A general result of \cite{ganatra2017covariantly}, along with a result particular to the toric setting due to \cite{gammage2017mirror,zhou2018lagrangian} implies that they are equivalent when $\check X_\Sigma$ is Fano.
Inspired by the work of \cite{nadler2009constructible}, \cite{ganatra2018microlocal} shows that $\mathcal{W}((\bb{C}^*)^n, \stp_\Sigma)$ is equivalent to the sheaf theoretic category considered by \cite{fang2011categorification}, giving a very general proof of HMS when combined with the aforementioned work of Kuwagaki.
\cite{katzarkov2017partially} have also defined a variant of $\mathcal{W}((\bb{C}^*)^n, \stp_\Sigma)$ to give a proof of HMS for quasi-affine toric varieties.
An overview of the current state of the field is included in \cite[Section 1.3]{hanlon2019monodromy}.    \subsection{Statement of results}
      The results of this paper are best summarized working backwards from applications, where we investigate mirror symmetry for the functors described in \cref{tab:desiredHMS}.
\begin{table}
\centering
\begin{tabular}{c|c}
    $A$-side &  $B$-side \\\hline
    Monodromy of FLTZ skeleton & Tensor product with line bundles\\
    Lagrangian correspondences & Toric morphisms and inclusions\\
    Stop removal & Blow-down and localization
\end{tabular}
\caption{Aspects of functoriality in toric HMS.}
\label{tab:desiredHMS}
\end{table}
In order to understand these functors through homological mirror symmetry, we are required to reexamine homological mirror symmetry for toric varieties with the newly developed machinery of partially wrapped Fukaya categories.

Our first result encodes previous HMS results for toric varieties in this new language. 
\begin{mainthm}[Restatement of \cref{thm:embedding}]\label{thm:mainthm} Let $\Delta$ be a monomial division adapted to $\Sigma$ and let $\mathcal{F}^s_\Delta$ be the monomially admissible Fukaya-Seidel category of sections associated to $\Delta$. There is a fully faithful embedding
$$ \mathcal{F}^s_\Delta \hookrightarrow \mathcal{W}((\bb{C}^*)^n, \stp_\Sigma) .$$
\end{mainthm}
As a consequence, the subcategory of $\mathcal{W}((\bb{C}^*)^n, \stp_\Sigma)$ generated by tropical Lagrangian sections  is derived equivalent to the subcategory of line bundles on the mirror toric variety.
We then show that these tropical Lagrangian sections generate  $\mathcal W((\CC^*)^n, \stp_\Sigma)$.
The proof of generation relies on constructing the mirror to the inclusion of a toric divisor.
\begin{mainthm}[Restatement of \cref{thm:inclusioncorrespondence}] 
    Let $\str(\alpha)/\alpha$ be the fan of a toric divisor of $\check X_\Sigma$.
    There is a Lagrangian correspondence  $L_{\alpha 0}: ((\CC^*)^{n-1}, \stp_{\str(\alpha)/\alpha})\Rightarrow ((\CC^*)^n, \stp_\Sigma)$ which is mirror to the inclusion of a toric divisor $i_{\alpha 0}: \check X_{\str(\alpha)/\alpha}\into \check X_\Sigma$ on the level of the Grothendieck groups $K_0(\mathcal W((\CC^*)^n, \stp_\Sigma))$ and $K_0(D^b\Coh(\check X_\Sigma))$.
\end{mainthm}
Provided that the expected functoriality of Lagrangian correspondences extends to the partially wrapped setting, \cref{thm:inclusioncorrespondence} gives a functor between Fukaya categories mirror to the derived pushforward of the inclusion; see \cref{rem:functorstatement}.
This Lagrangian correspondence allows us to inductively construct the linking disks to the FLTZ skeleton via Lagrangian cobordisms and tropical Lagrangian sections, proving generation. 
\begin{maincor}[Restatement of \cref{thm:ageneration}]
    \label{cor:generation} 
    The subcategory of tropical Lagrangian sections generates $\mathcal W ((\CC^*)^n, \stp_\Sigma)$.
\end{maincor}
The generation result is mirror to generation of $D^b\Coh(\check X_\Sigma)$ by line bundles \cite[Proposition 1.3]{abouzaid2009morse}. In combination with \cref{thm:mainthm}, generation immediately gives a proof of HMS for toric varieties that is completely Floer-theoretic.
\begin{maincor} \label{cor:HMScor} Let $\check X_\Sigma$ be a smooth projective toric variety. There is a quasi-equivalence
\[ \Tw\mathcal{W}((\bb{C}^*)^n, \stp_\Sigma) \simeq D^b_{\dg}\Coh(\check X_\Sigma). \]
Under this equivalence, Lagrangian sections $L(\sF)$ associated to support functions $\sF$ are identified with line bundles $\mathcal O(\sF)$. 
\end{maincor}

\begin{rem} The embedding from \cref{thm:mainthm} can also be constructed using Abouzaid's category of tropical Lagrangian sections (see \Cref{rem:abouzaidembedding}). Combined with \cref{cor:generation} and the equivalence $\mathcal{W}((\bb{C}^*)^n, W_\Sigma^{-1}(-\infty)) \simeq \mathcal{W}((\bb{C}^*)^n, \stp_\Sigma)$ mentioned above, this implies that Abouzaid's tropical sections generate the Fukaya-Seidel category\footnote{At least, the partially wrapped version of the Fukaya-Seidel category, which differs from that considered in \cite{abouzaid2009morse} in its technical setup, but is expected to be quasi-equivalent.} when $\check X_\Sigma$ is Fano.
\end{rem}

\begin{rem} An obvious question is how the equivalence of \cref{cor:HMScor} relates to the one obtained in \cite[Corollary 6.16]{ganatra2018microlocal} by passing to the category of microlocal sheaves. It is reasonable to expect that they coincide. In order to verify this, one needs to check only that the images of the tropical Lagrangian sections corresponding to ample line bundles in the category of microlocal sheaves is the costandard constructible sheaf on the moment polytope of this line bundle \cite{fang2011categorification}. In general, tropical Lagrangian sections should then correspond to twisted polytope sheaves \cite{zhou2019twisted}.
\end{rem}

\begin{rem} Although the FLTZ skeleton has been defined and studied for toric stacks \cite{fang2014coherent}, we work only with toric varieties in this paper. This is done largely for clarity and to simplify the exposition. We expect that all of our results can be extended to toric Deligne-Mumford stacks in the sense of \cite{borisov2005orbifold}. 
\end{rem}

\Cref{cor:HMScor} provides a version of homological mirror symmetry that can leverage the additional structure and flexibility of the partially wrapped Fukaya category. We use these structures to obtain the desired mirror relationships listed in \cref{tab:desiredHMS}.
As specific applications, we show for projective toric varieties:
\begin{description}
    \item[Restatement of \cref{thm:wrappingaction}:]  The monodromy action on the monomially admissible Fukaya category considered in \cite{hanlon2019monodromy} extends to the partially wrapped Fukaya category. These autoequivalences of $\mathcal W((\CC^*)^n, \stp_\Sigma)$ are mirror to tensor-product with line bundles on $D^b\Coh(\check X_\Sigma)$.
    \item[Restatement of \cref{property:linkingdisks}:] 
    Under the mirror identification,  linking disks $L_\sigma$ are mirror to line bundles supported on the toric orbit indexed by $\sigma$.
    \item[Restatement of \cref{prop:stopremovallocalization} and \cref{cor:blowup}:] The stop removal functors associated to the appropriate components of the FLTZ skeleton on  $\mathcal W((\CC^*)^n, \stp_\Sigma)$ are mirror to toric blow-down, and localization away from toric strata. 
    \item[Restatement of \cref{cor:quasiprojective}:] We can extend \cref{cor:HMScor} to the more general setting of quasi-projective smooth toric varieties.
\end{description}

The construction of the mirror correspondence to inclusion and the inductive argument for generation use new methods for building Lagrangian surgeries, cobordisms, and correspondences in Liouville domains. We lay some groundwork for working with Lagrangian cobordisms and correspondences in Liouville domains in the appendices.
These tools are kept separate from the main results and may be of interest to those constructing Lagrangian submanifolds inside Liouville domains.
    \subsection{Outline}
      The paper is roughly organized as follows: background and notation in \cref{sec:background}; embedding of tropical sections into the partially wrapped Fukaya category in \cref{sec:embedding}; the construction of the inclusion correspondence and subsequent generation of the partially wrapped Fukaya category in \cref{sec:mirrordivisor,sec:inclusion,sec:generation}; and applications and extension in \cref{sec:applications}. Finally, \cref{app:cobordism,app:generation,app:correspondenceandcobordism,app:correspondenceadmissibility} cover constructions of Lagrangian correspondences and cobordisms in Liouville domains.

A review of notation and existing literature is included in \cref{sec:background}.
We set notation for the partially wrapped Fukaya category in \cref{subsec:partiallywrapped}, mostly following \cite{ganatra2018sectorial}.
\Cref{subsec:backgroundSYZ,subsec:toricnotation,subsec:HMSboundary,subsec:troplag} overview mirror symmetry for toric varieties. 
Our emphasis on SYZ mirror symmetry will color our notation used for toric varieties.
\Cref{subsec:toricnotation} fixes our choices for notation, which views line bundles and toric divisors as arising from the data of a support function on the SYZ base.
This description becomes useful in \cref{subsec:HMSboundary}, which recaps the construction of the monomially admissible Fukaya-Seidel category from \cite{hanlon2019monodromy}, and the FLTZ skeleton from \cite{fang2011categorification}.
The objects in the monomially admissible Fukaya-Seidel category that are of primary interest to us are tropical Lagrangian sections, which are similarly described in \cref{subsec:troplag}.

In \cref{sec:embedding}, we produce conical tropical Lagrangian sections from support functions by modifying the construction of tropical Lagrangian sections in the monomially admissible setting. Then, we show that this modification does not change their Floer theory, which results in \cref{thm:mainthm}. We also see that the conical twisting Hamiltonians used in defining the conical tropical Lagrangian sections act on $\mathcal{W}((\CC^*)^n, \stp_\Sigma)$ and that this action is mirror to tensoring by the corresponding line bundle on the mirror as in the monomially admissible setting. In \cref{subsec:cocores}, we prove that the cocores of the FLTZ skeleton are all tropical Lagrangian sections.

In \cref{sec:mirrordivisor,sec:inclusion}, we construct various Lagrangian submanifolds for $\mathcal W((\CC^*)^n, \stp_\Sigma)$ which realize mirror symmetry predictions. 
\Cref{sec:mirrordivisor} constructs a Lagrangian submanifold $L_{\alpha>0}\subset ((\CC^*)^n, \stp_\Sigma)$ which is mirror to the structure sheaf of a toric divisor $\mathcal O_{\alpha>0}\in D^b\Coh(\check X_\Sigma)$.
The construction is based on the surgery methods used to build tropical Lagrangian submanifolds from the second author's thesis \cite{hicks2020tropical}.
These Lagrangians provide an example of an advantage of working in the partially wrapped Fukaya category, as they cannot be constructed in the monomially admissible Fukaya category. 
In \cref{subsec:p1}, an explicit example is given for homological mirror symmetry for $\CP^1$ for the purpose of exposition. 

These methods are extended in \cref{sec:inclusion}, which looks at constructions of Lagrangian correspondences.
We explore two kinds of Lagrangian correspondences: correspondences mirror to toric morphisms and correspondences mirror to the inclusion of a toric divisor. 
In \cref{subsec:toricmorphisms}, we construct the Lagrangian correspondence mirror to a toric morphism. 
We additionally prove that this correspondence intertwines the twisting Hamiltonians of the domain and codomain and that it sends Lagrangian sections to Lagrangian sections.
\Cref{subsec:inclusionmirror} uses this construction and the surgery method from \cref{sec:mirrordivisor} to build a Lagrangian correspondence $L_{\alpha 0}: ((\CC^*)^{n-1}, \stp_{\str(\alpha)/\alpha})\Rightarrow ((\CC^*)^n, \stp_\Sigma)$, the mirror to inclusion of a toric divisor. 
The main result of this section, \cref{thm:inclusioncorrespondence}, is the construction and description of $L_{\alpha 0}$.
We prove that $L_{\alpha 0}$ intertwines the twisting Hamiltonians, sends linking disks to linking disks, pulls Lagrangian sections back to Lagrangian sections, and can be constructed by a Lagrangian cobordism.

We explore some immediate applications in \cref{sec:applications}. 
In \cref{sec:generation}, we use the inclusion correspondence to show that tropical Lagrangian sections generate the Fukaya category.
We give a sketch of the argument for the mirror to $\CP^2$ using piecewise linear pictures in \cref{subsec:warmupcp2}. 
These PL Lagrangians are replaced with smooth Lagrangians in \cref{subsec:generationproof}, which uses the inclusion correspondence to inductively show that tropical Lagrangian sections generate linking disks of the FLTZ skeleton.
This completes the proof of \cref{cor:generation}. 
Using the same techniques, we show that the mirror to the Beilinson exceptional collection for $\CP^n$ in generates in \cref{subsubsec:cpngeneration}.
We then leverage \cref{cor:generation} to examine the stop removal functor under homological mirror symmetry for toric varieties.
\cref{subsec:quotient} shows that homological mirror symmetry holds for quasi-projective smooth toric varieties. This is achieved by checking that localization at a stop is mirror to working in the complement of the corresponding toric stratum.
In \cref{subsec:blowup}, we show that the blow-down functor associated to toric blowup at a smooth orbit can also be realized by stop removal.

Finally, we include several appendices containing constructions for Lagrangian cobordisms and correspondences in Liouville domains.
These provide tools that may be of independent interest beyond the homological mirror symmetry results of this paper. 
\Cref{app:cobordism} extends the surgery cobordism construction of \cite{hicks2020tropical} to the setting of Liouville domains. 
\Cref{app:generation} describes a Lagrangian cobordism between cobordisms of Legendrians; we recover as an application the generation result of \cite{ganatra2018sectorial}.
\Cref{app:correspondenceandcobordism} proves an interchange relation of Lagrangian cobordisms and Lagrangian correspondences; this result is known to experts but (to our knowledge) has not yet appeared in the literature. As an application, we observe that the Lagrangian cobordism class under geometric composition of Lagrangian correspondences is independent of perturbations taken to achieve transversality.
Finally, \cref{app:correspondenceadmissibility} lays some groundwork for Lagrangian correspondences between Liouville domains, highlighting why ``totally wrapped to the stop'' is an important condition to ensure admissibility of geometric composition.    \subsection{Acknowledgements}
      We would like to thank Denis Auroux and Sheel Ganatra for useful discussions. We are also grateful to the anonymous referee whose many helpful comments improved this paper.

AH is partially supported by NSF RTG grant DMS-1547145.
JH is supported by EPSRC Grant (EP/N03189X/1, Classification, Computation, and Construction:
New Methods in Geometry).

\section{Background}
\label{sec:background}
   
In this section, we set our notation and review some of the necessary background on toric varieties and the Fukaya categories relevant to homological mirror symmetry for toric varieties.
As a general rule: 
\begin{itemize}
   \item Symplectic objects in this paper are written with unadorned characters (so $\phi: X\to X$ is a symplectomorphism).
   \item Complex geometric objects in this paper are decorated with checks (so $\check f: \check X_1\to \check X_2$ is a map of varieties).
   \item Piecewise linear / tropical  objects are marked with an underline (so $\underline i:Q_1\to Q_2$ is a linear function). A notable exception is for support functions $\sF: Q\to \RR$. 
\end{itemize} 

\subsection{Partially wrapped Fukaya categories}
\label{subsec:partiallywrapped}
The notation that we use for partially wrapped Fukaya categories follows \cite{ganatra2018sectorial}.
The symplectic objects which are the center of our study are \emph{Liouville domains.} A Liouville domain is a pair $(X^{int},\lambda)$, where
   \begin{itemize}
      \item $X^{int}$ is a $2n$-manifold with boundary $\partial X^{int}$ and,
      \item $\lambda\in \Omega^1(X^{int}, \RR)$ is a one form on $\Omega^1(X^{int}, \RR)$ so that $\omega=d\lambda$ is a symplectic form for $X^{int}$.
   \end{itemize}
To this data we can associate a \emph{Liouville vector field }$Z$ defined by the property $\iota_Z\omega= \lambda$.
We require that this vector field transversely points outward along $\partial X$.
The boundary of a Liouville domain is a contact manifold $(\partial X^{int}, \lambda)$.
Liouville domains admit completions to a non-compact symplectic manifold $X$ by attaching the symplectization $\partial X^{int}\times [0, \infty)_t$ to the contact boundary.
The symplectic form that we take on the cylindrical end is $d(e^t\lambda)$; we note that this is an exact symplectic form. 
Since we will mostly work with the completions of Liouville domains, we\footnote{This notation places a preference for the completion of a Liouville domain rather than the Liouville domain itself.} will use $X^{int}$ to refer to any Liouville domain which completes to $X$.

Outside of the appendix, the only Liouville domain which we will consider is $X=(\CC^*)^n$, where $X^{int}$ is the unit cotangent bundle of the product torus $\{(e^{i\theta_1}, \ldots, e^{i\theta_n}), \theta_i\in S^1\}$.
We summarize some general properties of Liouville domains. 
The Liouville vector field gives a flow $\phi_Z^t: X\to X$ which expands the symplectic form.
The \emph{core} of a Liouville manifold, which we denote by $\core_X$, is the set of points which do not escape under the flow of $Z$.
Given $x_0$ belonging to an $n$-dimensional strata of $\core_X$, we can look at the cocore,
\[\cocore(x_0) := \left\{x\in X\; : \; \lim_{t\to-\infty} \phi_Z^t(x)=x_0\right\}.\] 

The Liouville vector field provides a taming condition on Lagrangian submanifolds of $X$.
We say that a Lagrangian submanifold $L\subset (X, \lambda)$ is \emph{conical at infinity} if it is parallel to $Z$ outside of some compact set $V\subset X$.
We say a Lagrangian submanifold is \emph{admissible} if it is exact and conical at infinity. 

Associated to $(X, \lambda)$ is a \emph{wrapped Fukaya category}  $\mathcal W(X)$ whose objects are admissible Lagrangians with extra data.
We sketch the construction of the morphisms here, and refer the reader to \cite[Section 3]{ganatra2017covariantly} for a full definition of the category.
$\mathcal W(X)$ is constructed as the localization of a directed $A_\infty$-category.
For each admissible Lagrangian $L\subset X$, this directed $A_\infty$ category contains a sequence of Lagrangians $L^i$. 
The $L^i$ are obtained by applying successively larger positive wrapping Hamiltonians to the underlying Lagrangian $L$.
Morphisms are then defined between Lagrangians of decreasing wrapping. 
The wrapped category is obtained by localizing the directed category at the continuation elements, which are canonical morphisms $L \to L^k$.
In summary, to compute (the cohomology of) morphism complexes $\Hom_{\mathcal W(X)}(L_1, L_2)$, one must compute the limit of $\text{CF}^\bullet(L^i_1, L_2)$, where $L^i_1$ is a cofinal sequence of Lagrangian submanifolds starting with $L_1$, and related by positive Hamiltonian isotopies (see \cite[Lemma 3.37]{ganatra2017covariantly}).

There is also a version of the Fukaya category that does not completely wrap the ends of admissible Lagrangian submanifolds. 
A \emph{stop} for $X$ is a mostly Legendrian (\cite[Definition 1.6]{ganatra2018sectorial}) closed subset $\stp\subset \partial X^{int}$.
An admissible Lagrangian submanifold in $(X^{int}, \stp)$ is an admissible Lagrangian submanifold of $X^{int}$ whose boundary is disjoint from $\stp$.

The \emph{skeleton} of a stopped Liouville manifold is the union of the conicalization of the stop and the core, 
\begin{equation}\LL_\stp:= \{x\in X\; : \; \phi_Z^t(x)\in \stp \text{ for some $t\in \RR$ }\}\cup \core. \label{eq:conicalizationofstop}
\end{equation}
We can recover the stop from the skeleton by taking the intersection of $\LL_{\stp}$ with $\partial X^{int}$.
We can equivalently characterize an admissible Lagrangian submanifold of $(X, \stp)$ as an admissible Lagrangian submanifold of $X$ which avoids $\LL_{\stp}$ outside of a compact set $X^{int}$.

From the data of a stopped Liouville domain, \cite{sylvan2019partially,ganatra2018sectorial} constructs a \emph{partially wrapped Fukaya category} $\mathcal W(X, \stp)$.
We follow the conventions from \cite{ganatra2018sectorial} where $\mathcal W (X, \stp)$ is defined in the generality that we need.
The category $\mathcal W(X, \stp)$ has admissible Lagrangian submanifolds as objects.
Again morphisms $\Hom_{\mathcal W(X, \stp)}(L_1, L_2)$ are computed by taking an appropriate limit of Lagrangian intersection Floer theory over a set of positively Hamiltonian isotopic and admissible Lagrangians $L^i_1$. 

At a point $x$ in the interior of $n$-dimensional strata of $\LL_\stp$, we can take a Lagrangian disk $D^n_x$ that intersects the skeleton transversely at $x$. 
By applying the Liouville flow to the boundary of $D^n_x$, we obtain an admissible Lagrangian disk which only intersects $\LL$ at $x$.
This disk is called the \emph{linking disk of $\LL$ at $x$} and we will also denote it by $\cocore(x)$.\footnote{This is slightly different than the standard notation, which usually reserves the word linking disk to mean a disk which intersects the skeleton at a single point in the symplectization. However, we will use linking disk to mean both these disks, and cocores. We will use cocore to specifically mean a linking disk which intersects the core of $X$}
The admissible Hamiltonian isotopy class of the linking disk is only dependent on the stratum of the skeleton $\LL$ which contains $x$.

\begin{thm}[Theorem 1.1 of \cite{ganatra2018sectorial},\cite{chantraine2017geometric}]
   $\mathcal W(X, \stp)$ is generated by the linking disks $\cocore(x)$ of the skeleton, $\LL_{\stp}$. 
\end{thm}
We provide a Lagrangian cobordism argument for this generation result in \cref{app:generation}, which is based on the work of \cite{bosshard} extending Biran and Cornea's cobordism iterated exact sequence to the setting of Liouville manifolds.

\subsection{SYZ mirror symmetry for toric varieties}
\label{subsec:backgroundSYZ}
For accessibility, we provide a dictionary between notation used in algebraic geometry for toric varieties, and notation from symplectic geometry.
We denote by $\check X_\Sigma$ the toric variety of dimension $n$ over $\CC$ corresponding to a fan $\Sigma \subset N \otimes \RR$ where $N$ is an $n$-dimensional lattice.
By definition, $\check X_\Sigma$ is a (partial) compactification of the algebraic torus $\check X_0:= N \otimes \CC^* = (\CC^*)^n$.
$\Sigma$ encodes the strata of compactification in terms of convergence of 1-parameter subgroups of the algebraic torus \cite[Proposition 3.2.2]{cox2011toric}.\footnote{Here we are making a change in sign from \cite{cox2011toric} to be more compatible with our later conventions for Lagrangian sections and the FLTZ skeleton.} 
Unless explicitly stated otherwise, we will assume that $\check X_\Sigma$ is smooth, that is, every $n$-dimensional cone of $\Sigma$ is generated by a basis of $N$. We will denote primitive generators of $\Sigma$ by $\alpha$ and set $A$ to be the set of all primitive generators.
Additionally, $\langle \alpha_1, \hdots, \alpha_k \rangle$ denotes the cone generated by the primitive generators $\alpha_1, \hdots, \alpha_k$.

We first describe mirror symmetry for the algebraic torus $\check X_0$.
The algebraic torus $\check X_0$ has a natural fibration over the affine space $Q := N \otimes \bb{R}\simeq \RR^n$ with fibers that are real $n$-dimensional tori. From $Q$, we can also construct a symplectic manifold $X=T^*Q/M$ where $M$ is the dual lattice to $N$ acting fiberwise via the natural splitting $T^*Q \simeq Q \times (M \otimes \bb{R})$. The algebraic variety $\check X_0$ and the symplectic manifold $X$ are mirror dual. These dual torus fibrations serve as the foundational example of Strominger-Yau-Zaslow (SYZ) mirror symmetry \cite{strominger1996mirror}, and this example is often referred to as semi-flat mirror symmetry (see, for example, \cite{gross2013mirror}).

We now give the symplectic notation which we will use for the remainder of the paper.
It is natural to consider $Q$ with the additional structure of an integral affine manifold; this trivializes the tangent and cotangent bundles by identifying $\ZZ^n$ lattice subbundles $T_{2\pi \ZZ} Q\subset TQ$ and $T^*_{2\pi \ZZ} Q\subset T^*Q$.
The lattice structure on toric varieties can be recovered from the affine structure on $Q$, giving us the following dictionary:
\begin{align*}
   Q := N \otimes \bb{R} && N \simeq (T_{2\pi \ZZ})_q Q &&  \check X_0 = (\CC^*)^n \simeq TQ/T_{2\pi \ZZ} Q\\
   \; && M \simeq (T^*_{2\pi \ZZ})_q Q && X =T^*Q/M\simeq T^*Q/T^*_{2\pi \ZZ} Q.
\end{align*} 
In particular,
the torus fibration $\check X_0\to Q$ becomes the complex SYZ fibration 
\begin{align*}
   \bval: \check X_0= (\CC^*)^n\to& Q\\
    (z_1, \hdots, z_n)\mapsto& (\log|z_1|, \hdots, \log |z_n|)
\end{align*}
induced by the standard holomorphic volume form.
The torus fibration on $T^*Q/T^*_{2\pi \ZZ}Q$ is the symplectic SYZ fibration, i.e., moment map $\val:X\to Q$, of the symplectic form
$$ \omega = \sum_{i=1}^n du_i \wedge d\thetaa_i $$
where $u_i$ are our chosen coordinates on $Q \simeq \bb{R}^n$ and $\thetaa_i$ are the dual coordinates on the fibers.
In \cref{sec:mirrordivisor,sec:inclusion}, we will use the coordinates $(q_i, p_i)$ inherited from $T^*Q$.

We now discuss how equipping $\check X_0$ with a toric K\"{a}hler form yields a complex structure on $X$.
Let $h: Q\to \RR$ be a strictly convex function on $Q$ with nondegenerate Hessian.
This gives a potential $h\circ  \bval$ on $\check X_0$, whose associated K\"ahler form is toric.
As first observed by Hitchin \cite{hitchin1997moduli}, the Legendre transform $\nabla h$:
\begin{itemize}
    \item takes the complex SYZ fibration on $\check X_0$ to the symplectic SYZ fibration coming from the K\"{a}hler form and
    \item takes the symplectic SYZ fibration on $X$ to the complex SYZ fibration coming from a complex structure naturally arising from $h$.
\end{itemize} 
In other words, we obtain an isomorphism of $X$ with $\check X_0$ where 
$$ \Log = \nabla h \circ u.$$
From the perspective of tropical geometry, $\Log$ should be viewed as the valuation map. 

In this particular example, the cotangent bundle and affine structure are trivializable. As a result, there is another map which we will frequently use: the argument projection $\arg: T^*Q/M \to (M \otimes \RR)/M$.
In coordinates, we have $\arg(u, \thetaa) = \thetaa$ so that under the identification with $(\CC^*)^n$, $\arg$ is the usual argument map, or identifying with $T^*T^n$, projection to the base. The diagram below summarizes various maps involved in semi-flat mirror symmetry.

\setlength\mathsurround{0pt}
$$ \begin{tikzcd}
(\CC^*)^n \simeq N \otimes \CC^* \arrow{dd}[left]{\bval} \arrow[leftrightsquigarrow]{rr}{\text{SYZ mirror}} & & T^*Q/M \simeq T^*\RR^n/ T^*_{2\pi \ZZ} \RR^n \arrow{ddll}[above]{u} \arrow{r}{\arg} \arrow{d}[rotate=90, below]{\sim} & (M \otimes \RR)/M \simeq T^n\\
& & (\CC^*)^n \arrow{d}{\Log}  & \\
Q = N \otimes \RR \simeq \RR^n  \arrow{rr}[above]{\nabla h}[below]{\text{Legendre transform}} & & M \otimes \RR \simeq \RR^n &
\end{tikzcd} $$
\setlength\mathsurround{.8pt}
\begin{rem} We will often implicitly identify $T^*Q/T^*_{2\pi \ZZ} Q$ with $(\CC^*)^n$ as outlined above for brevity of notation. In some constructions of Fukaya categories mirror to toric varieties, the holomorphic structure plays an important role while in others, namely the partially wrapped category, it does not. We will emphasize this distinction when it is important. See for instance \cref{rem:monadmholomorphic,rem:notechnicalprob}.
\end{rem}

It remains to understand mirror symmetry for the toric variety $\check X_\Sigma$ from semi-flat mirror symmetry for the pair $\check X_0, X$.  
This amounts to understanding how compactifying by toric divisors modifies the mirror space. 
An expectation for symplectic Calabi-Yau manifolds $\check X$ is that compactification is mirror to the addition of a ``superpotential'' function $W: X\to \CC$ on the SYZ mirror. 
We first work in the setting where $\check X_\Sigma$ is equipped with a K\"{a}hler potential.
One route to constructing this superpotential function (studied in detail by \cite{auroux2007mirror}) comes from identifying the points of $X$ with SYZ fibers of $\check X_0$. The open Gromov-Witten invariants count the holomorphic disks in the compactification  $\check X_\Sigma$ with boundary on SYZ fibers and record the deformation from $\text{Fuk}(\check X_0)$ to the weakly unobstructed curved $A_\infty$ category $\text{Fuk}(\check X_\Sigma)$.
This amounts to equipping $X$ with a holomorphic function.
It is expected that each compactifying divisor on the space $\check X_\Sigma$ presents the possibility for additional holomorphic disks to appear on the boundary of SYZ fibers of $\check X_\Sigma$, giving a correspondence between the toric divisors  of $\check X_\Sigma$ and the terms to this superpotential. 
This expectation has been verified from the viewpoint that a projective toric variety $\check X$ is a compactification of $(\CC^*)^n$ with the induced K\"{a}hler form (see, for instance, \cite{chan2020syz}). In particular, by choosing coordinates and using the mirror K\"{a}hler form to obtain a complex structure, we can construct the mirror superpotential which to first order is the Hori-Vafa potential
\begin{align*}
   W_\Sigma: (\CC^*)^n \to \CC\\
   (z_1, \ldots, z_n)\mapsto  \sum_{\alpha \in A} c_\alpha z^\alpha. \stepcounter{equation}\tag{\theequation}\label{eq:HVpotential}  
\end{align*}
Here $z^\alpha = z_1^{\alpha_1}\hdots z_n^{\alpha_n}$ and $c_\alpha \in \bb{C}^*$ are coefficients that depend on the K\"{a}hler structure of $\check X_\Sigma$.
When $\check X_\Sigma$ is Fano, this is the correct mirror holomorphic function in the sense of SYZ mirror symmetry, but in general, the potential requires higher-order corrections. Although this construction of $W_\Sigma$ employs the symplectic structure on $\check X_\Sigma$ coming from the K\"{a}hler form, it is commonplace to also treat the pair $((\CC^*)^n, W_\Sigma)$ as a mirror when $\check X_\Sigma$ is considered as an algebraic/complex object due to the expected involutive nature of mirror symmetry. One palatable feature of this treatment is that $W_\Sigma$ is unchanged as a singular symplectic fibration when varying the $c_\alpha \in \CC^*$. 

However, it is more natural to only use the complex algebraic structure on $\check X_\Sigma$.
As in the case of semi-flat mirror symmetry, it should be possible to understand an SYZ mirror symmetry construction without the choice of a K\"{a}hler structure for $\check X_\Sigma$.
Recall that after taking a choice of coordinates, an element $u = (u_1, \hdots, u_n) \in N$ lies in a cone $\sigma \in \Sigma$ if and only if the limit
$$ \lim_{z \to \infty} (z^{u_1}, \hdots, z^{u_n}) $$
exists in $\check X_\Sigma$ and converges to a point in the toric strata corresponding to $\sigma$.
As $\check \Log(z^{u_1}, \hdots, z^{u_n}) = (\log|z|) u$, we can view this limiting point as lying at infinity along the ray spanned by $u$, determining a  circle $\arg(z) u$ in the fiber which is collapsed to a point in the boundary stratum. Thus, we obtain a boundary condition on the mirror $T^*Q/M$ to $N \otimes \bb{C}^*$: above a cone $\sigma \in \Sigma$, $u \cdot \thetaa$ should be asymptotically constant for all $u \in \sigma$ where $\thetaa =( \arg(z_1), \ldots, \arg( z_n))$. Writing $\sigma = \langle \alpha_1, \hdots, \alpha_k \rangle$ for primitive generators $\alpha_1, \hdots, \alpha_k$, this condition is equivalent to $\alpha_1 \cdot \thetaa, \hdots, \alpha_k \cdot \thetaa$ being asymptotically constant. Moreover, thinking of $p$ as parameterizing a dual torus, we see that the limiting constant should be $0$, that is, the unit in $S^1 = \RR/2\pi \ZZ$. These conditions begin to resemble the data of the Hori-Vafa potential, and can be interpreted in several ways to construct Fukaya categories and study homological mirror symmetry. Two such approaches are explored in \cref{subsec:HMSboundary}.

\subsection{Notation for toric varieties: Support functions, morphisms, and sheaves}
\label{subsec:toricnotation}
Let $\Sigma_1\subset Q_1, \Sigma_2\subset Q_2$ be two fans. A \emph{map of fans} is a linear map $\underline f: Q_1\to Q_2$ with the property that for every cone $\sigma_1\in \Sigma_1$, there exists a cone $\sigma_2\in \Sigma_2$ with $\underline f(\sigma_1)\subset \sigma_2$. The data of a map of fans is equivalent to the data of a toric morphism $\check f: \check X_{\Sigma_1}\to \check X_{\Sigma_2}$.

For a cone $\sigma \in \Sigma$, we define the star of $\sigma$ as the subfan of all cones containing $\sigma$, that is,
$$ \str(\sigma):=\{\tau\in\Sigma\; : \; \tau \geq \sigma \}. $$
We will often abuse this notation by conflating $\str(\sigma)$ with the subset of $N \otimes \bb{R}$ covered by $\str(\sigma)$ and by writing $\str(\alpha)$ in place of $\str(\langle \alpha \rangle)$.
Note that when $\sigma = 0$, the star of $\sigma$ is the whole fan; when $\sigma$ is top dimensional, then $\str(\sigma)=\sigma$. 
The standard inclusion $\str(\sigma)\into \Sigma$ is a map of fans and the corresponding map of toric varieties  $\check X_{\str(\sigma)}\into \check X_\Sigma$ identifies $\check X_{\str(\sigma)}$ as the
normal bundle to the closure of the toric orbit indexed by $\sigma$.

The toric orbit closure associated with the cone $\sigma$ is itself a toric variety, which has a description in terms of $\str(\sigma)$. 
Let $Q_\sigma:=Q/(\RR\cdot \langle\sigma\rangle) $ be the quotient of the cocharacter lattice by the real span of $\sigma$. 
The \emph{orbit closure fan} of $\sigma$ is the fan\footnote{This is slightly different than the notation in \cite{cox2011toric}, where $\str(\sigma)$ is used to denote what we call $\str(\sigma)/\sigma$.}
\[\str(\sigma)/\sigma:= \{\tau/(\RR\cdot \langle \sigma \rangle) \; : \; \sigma <\tau \}\subset  Q_\sigma.\]
With this notation, the toric varieties $\tb{\alpha}$ correspond to toric divisors of $\check X_\Sigma$.
$\tb{\sigma}$ is the closure of the toric orbit associated with $\sigma$.
For the inclusion of a toric orbit, we will write $\check i_{\sigma\tau}:\tb{\sigma}\into \tb{\tau}$.
Note that this is not a toric map.
Accordingly, $\check X_{\str(\sigma)}=N_{\check X_\Sigma}\check X_{\str(\sigma)/\sigma}$, and the map of fans $\str(\sigma)\to \str(\sigma)/\sigma$ yields the projection of toric varieties
    \[
       \check X_{\str(\sigma)}\to \check X_{\str(\sigma)/\sigma}.
    \] 
We now provide a mirror-symmetry biased set of notation for discussing line bundles and divisors on toric varieties.
\begin{df}
   A support function $\sF: Q\to \RR$ is a piecewise linear map whose domains of linearity are the cones of the fan $\Sigma$, and whose differentials (over the domains of linearity where it is defined) are contained in $T^*_{\ZZ} Q$. Equivalently in the smooth setting, we take the piecewise linear maps whose domains of linearity are the cones of $\Sigma$, and $\sF(\alpha)\in \ZZ$ for all $\alpha\in A$.
\end{df}
Let $\alpha$ be a 1-dimensional cone. We will write $\sF_\alpha$ for the support function which is $1$ on $\alpha$, and $0$ on all other 1-dimensional cones in $A$.
We will denote the set of support functions $\Supp(\Sigma)$. 
We favor support functions rather than divisors for encoding the geometry of $\check X_\Sigma$ as they rely only on the affine geometry of $Q$, and are therefore more natural objects to consider from a mirror-symmetry perspective. 
From this viewpoint, divisors, line bundles, and sheaves on $\check X_\Sigma$ are the $B$-model extensions of the underlying tropical geometry given by $Q$ and support functions $\sF: Q\to \RR$. 
To recover a toric divisor $D$ on $\check X_\Sigma$ from a support function, we take $D_{\sF}=\sum_{\alpha\in \Sigma} -(\sF(\alpha))\cdot D_\alpha$. 
To obtain a line bundle on $\check X_{\Sigma}$ from a support function, we take $\mathcal O_{\Sigma}(\sF):= \mathcal O_{\Sigma}(D_{\sF})$.\footnote{This leads to the somewhat unfortunate convention that  $D_{\sF_\alpha}=-D_\alpha$.}
Two support functions $\sF_1, \sF_2$ represent linearly equivalent divisors if and only if the difference $\sF_1-\sF_2$ is a linear function. 
A map of fans $\underline f: \Sigma_1\to \Sigma_2$ induces a pullback map on support functions $\underline f^*:\Supp(\Sigma_2)\to \Supp(\Sigma_1)$. This means that we have a well-defined pullback map for toric divisors.
This pullback intertwines with the toric morphism in the sense that 
\begin{equation*}
   \check f^*\mathcal O_{\Sigma_2}(\sF)= \mathcal O_{\Sigma_1} (\underline f^*\sF).
\end{equation*}

When $\tb{\sigma}$ is an orbit closure, the (non-toric) map of varieties $\check i_{\sigma 0}:\tb{\sigma}\into \check X_\Sigma$ does not have a well-defined pullback map on divisors.
 However, there is a well-defined pullback map on line bundles, and therefore a pullback map on linear equivalence classes of support functions. 
\begin{prop}
   Let $\sigma \in \Sigma$ be a cone.
   There is a surjective linear map
   \[\underline i_{\sigma 0}^*:(\Supp(\Sigma)/\sim) \to( \Supp(\str(\sigma)/\sigma)/\sim)\]
   where $\sim$ is the relation of linear equivalence.
\end{prop}
At times we will not want to work on the level of support functions modulo linear equivalence, but still make sense of pullback.
In that case, we will restrict ourselves to the support functions which give divisors transverse to a toric orbit closure. 
Let $\sigma\in \Sigma$ be a fan. The \emph{transverse support functions} are those which vanish on $\sigma$,
\[\Supp_\sigma(\Sigma):=\{\sF\; : \; \sF(q)=0 \text{ for all $q\in  \sigma$}\}\]
and there is a well-defined map from $\Supp_\sigma(\Sigma)\to \Supp(\str(\sigma)/\sigma)$.
Finally, we fix some notation for sheaves on toric varieties and toric orbit closures. We note that a toric orbit closure of $\check X_\Sigma$ is a toric variety and a subvariety of $\check X_\Sigma$, but not a toric subvariety of $\check X_\Sigma$.
Having fixed a fan $\Sigma$, the sheaf $\mathcal O_\sigma(\sF)$ will be a line bundle on the space $\tb{\sigma}$, where $\sF$ is either a support function on $\str(\sigma)/\sigma$, or is pulled back from a support function $\sF \in\Supp_\sigma(\Sigma)$. 
The line bundles on  $\check X_\Sigma$ are denoted  $\mathcal O_0(\sF)$ --- however, when it is unambiguous that we are referring to a line bundle on $\check X_\Sigma$, we will simply write $\mathcal O(\sF)$. 
The pullback map on transverse support functions $\underline i^*_{\sigma 0}: \Supp_\sigma(\Sigma)\to \Supp(\str(\sigma)/\sigma)$ agrees with the pullback of line bundles in the sense that 
\[i^*_{\sigma 0}(\mathcal O_0 (\sF))=\mathcal O_{\sigma}(\underline i^*_{\sigma 0}(\sF)).\] 
We denote the pushforward of a line bundle along $i_{\sigma 0}: \tb{\sigma}\into \check X_\Sigma$ by:
\[\mathcal O_{\sigma>0}(\sF):=(i_{\sigma 0})_* \mathcal O_\sigma(i^*_{\sigma 0}\sF)\in \Coh(\check X_\Sigma).\] 
We will denote by $\pi_\sigma: \check X_{\str(\sigma)}\to \tb{\sigma}$ the standard projection of the normal bundle to the base.
Similarly, when $\sigma>\tau$ we write $\mathcal O_{\sigma>\tau}$ for the pushforward of $\mathcal O_{\sigma}$ to $\check X_{\str(\tau)/\tau}$.

\subsection{Toric HMS: FLTZ skeleton and monomial admissibility} \label{subsec:HMSboundary}
We now apply the conclusions of \cref{subsec:backgroundSYZ} to construct Fukaya categories homologically mirror to the derived category of coherent sheaves on $\check X_\Sigma$. All definitions involve imposing conditions on Lagrangian objects \emph{near infinity}, that is, outside of a compact subset. There are two dual ways to constrain Lagrangians: require them to be contained in or to avoid a closed subset at infinity. The first approach results in what is usually called an infinitesimally wrapped category; the second gives a partially wrapped category. Under mirror symmetry, these are expected to correspond to the categories of compactly supported perfect complexes and coherent sheaves, respectively. If $\check X_\Sigma$ is smooth and complete, then these two categories are the same.

The traditional approach for toric varieties is to take the closed subset to be a fiber over a very large positive number, $W_\Sigma^{-1}(\infty)$. Then, the infinitesimally wrapped category is the Fukaya-Seidel category as envisioned by Seidel \cite{seidel2008fukaya}, and the partially wrapped category is the Fukaya-Seidel category as defined in \cite{sylvan2019partially, ganatra2017covariantly}. As both categories are generated by Lefschetz thimbles, it is expected that these categories are equivalent in the case that $\check X_\Sigma$ is smooth and projective. A proof of this equivalence requires an argument similar to that of \cref{thm:embedding} with some extra considerations on the choice of symplectic form used to define each category. However, these categories are not homological mirrors to $\check X_\Sigma$ when the anticanonical bundle on $\check X_\Sigma$ is not nef \cite{auroux2008mirror, ballard2015mori}.

As noted in \cref{subsec:backgroundSYZ}, a more natural boundary condition instead involves a collection of subtori determined by $\Sigma$ when considering $\check X_\Sigma$ on the B-side of mirror symmetry. Namely, near infinity\footnote{Here, we have traded the asymptotic condition justified in \cref{subsec:backgroundSYZ} for a condition achieved near infinity in order to establish compactness in the Floer theory. Alternatively, an asymptotic condition becomes a condition near infinity when working with conical Lagrangians. However, restricting to conical Lagrangians is also largely done for technical reasons.} in a cone $\sigma = \langle \alpha_1, \hdots, \alpha_k \rangle \in \Sigma$, we should have a boundary condition of the form
$$ \thetaa \cdot \alpha_1 = \hdots = \thetaa \cdot \alpha_k = 0. $$
Thinking in terms of these generators, we are immediately led to a stop of the form
$$\mathfrak{f}_\Sigma = \{ (u,\thetaa) \in \partial((\CC^*)^n)^{int} \; : \;\ \thetaa \cdot \alpha = 0 \text{ for all } \alpha \in A \text{ such that } u \in \text{int}(\st(\alpha)) \}$$
whose skeleton is the \emph{FLTZ skeleton} 
\begin{align} \begin{split} \label{eq:fltzskel}
   \mathbb L_\Sigma & = \{ (u, \thetaa) \in (\CC^*)^n\; : \;\thetaa \cdot \alpha = 0 \text{ for all } \alpha \in A \text{ such that } u \in \text{int}(\st(\alpha)) \} \cup \{ (0, \thetaa) \in (\CC^*)^n \} \\
   &= \bigcup_{\sigma \in \Sigma}  \{ (u, \thetaa) \in (\CC^*)^n : u \in \sigma \text{ and } \thetaa \cdot v = 0  \text{ for all } v \in \sigma \} \\
   &= \bigcup_{\sigma \in \Sigma} \sigma \times \sigma^\perp \subset N \times (M \otimes \RR)/M = T^*Q/T^*_{2\pi \ZZ} Q \simeq (\CC^*)^n
\end{split} \end{align}
introduced in \cite{fang2011categorification} using similar SYZ considerations. 
We will denote a stratum of the FLTZ skeleton by $\LL_{\sigma}=\sigma\times \sigma^\bot$, and similarly denote corresponding stratum of the stop by $\stp_\sigma\subset \stp_\Sigma$.
Thus, it is clear from this perspective that the appropriate partially wrapped category to consider is $\mathcal{W}( (\CC^*)^n, \mathfrak{f}_\Sigma)$.

It is then tempting to also construct a Fukaya-Seidel type category consisting of Lagrangian branes that lie in $\mathfrak{f}_\Sigma$ near infinity. However, such a category would not contain many interesting objects due to the singular nature of $\mathfrak{f}_\Sigma$. Instead, we must allow some room for Lagrangians to twist. One way to implement this is to use a monomial admissibility condition (introduced in \cite{hanlon2019monodromy}). We will need only a particular instance of monomial admissibility and work towards defining it now. 

In order to restrict Lagrangians near infinity, we introduce sets in the base $Q$ over which certain arguments must be controlled. We begin with the following definition which will also be useful for later exposition. 
\begin{df} The $\eps$-star open cover for $\Sigma$ and $\eps > 0$ is the collection of sets $\{U_{\alpha} \}_{\alpha \in A}$ where each $U_\alpha \subset \bb{R}^n$ is the open cone over the set
$$ \{ u \in \bb{R}\; : \;\|u\| = 1, u \in \st(\alpha), \text{ and } d(u, \sigma) > \eps \text{ for all } \sigma \in \Sigma \text{ such that } \alpha \not \in \sigma \} $$
where distance and length are taken with respect to the Euclidean metric.
\label{def:staropencover}
\end{df}

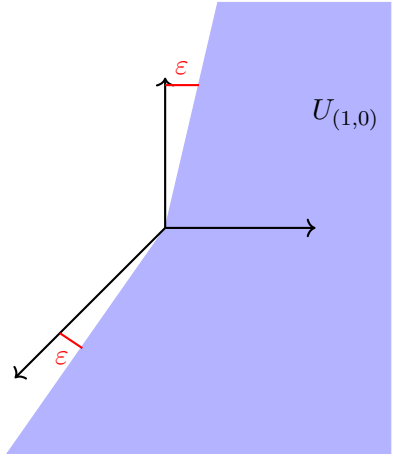
\begin{figure} 
\centering
\begin{tikzpicture}
\draw[blue, fill=blue, opacity=0.3] (0,0) -- (0.7, 3) -- (3,3) -- (3, -3) -- (-2.1, -3) -- (0,0);
\draw (2.4, 1.5) node {$U_{(1,0)}$};
\draw[->, thick] (0,0) --(2,0);
\draw[->, thick] (0,0) -- (0,2);
\draw[->, thick] (0,0) -- (-2,-2);
\draw[red, thick] (0,1.9) -- node[above]{$\eps$} (0.45, 1.9);
\draw[red, thick] (-1.4, -1.4) -- node[below, xshift=-3.5]{$\eps$} (-1.1, -1.6);
\end{tikzpicture} \caption{The open set $U_{1,0}$ in the $\eps$-star open cover for a fan corresponding to $\bb{P}^2$. Taking the closure of this set gives the cone $C_{1,0}$ in the $\eps$ combinatorial division, $\Delta_\eps$.} \label{fig:staropencover}
\end{figure}

Note that the $\eps$-star open cover for $\Sigma$ is a cover of $\bb{R}^n \setminus \{0\}$ for $\eps$ sufficiently small when $\Sigma$ is a complete fan, and we will always assume that we choose $\eps$ small enough for this to be true. 

\begin{df} The $\eps$ combinatorial division $\Delta_\eps$ for $\Sigma$
is the collection of closed sets $\{C_\alpha\}_{\alpha \in A}$ where $C_\alpha = \wb{U}_\alpha$ is the closure of the corresponding subset of the $\eps$-star open cover.
\label{def:combinatorialdiv}
\end{df}

For any fan corresponding to a smooth projective toric variety, it was shown in \cite[Corollary 2.40]{hanlon2019monodromy}  that there exists a choice of toric K\"{a}hler form on $(\CC^*)^n$ such that the $\eps$ combinatorial division is a \emph{monomial division adapted to $\Sigma$} in the sense that it satisfies the following properties:
\begin{enumerate}
    \item The closed subsets $\{ C_{\alpha} \}_{\alpha \in A}$ cover $\RR^n$;
    \item there are constants $k_\alpha \in \RR_{>0}$ such that 
    $$ \max_{\alpha \in A} | z^\alpha|^{k_\alpha} = | z^\beta|^{k_\beta} $$
    implies that $u(z) \in C_\beta$;
    \item $C_\alpha \setminus \{ 0 \} \subset \text{int}(\st(\alpha))$ for all $\alpha \in A$.
\end{enumerate}

\begin{rem} \label{rem:monadmholomorphic} The difficulty in the above statement lies only in verifying the second property, which involves the complex structure on $(\CC^*)^n$ and is used to obtain a maximum principle on holomorphic discs. Therefore, the monomially admissible Fukaya-Seidel category depends on the K\"{a}hler structure on $(\CC^*)^n$ rather than just the symplectic structure. In particular, the choice of toric K\"{a}hler form only manifests through the expression of $u$ in terms of holomorphic coordinates.
\end{rem}

After choosing such a toric K\"{a}hler form, we can then associate to $\Delta_\eps$ a monomially admissible Fukaya-Seidel category $\mathcal{F}_{\Delta_\eps}$ as detailed in \cite{hanlon2019monodromy}. Whenever discussing this category, we will always assume we have chosen such a K\"{a}hler form. Objects of $\mathcal{F}_{\Delta_\eps}$ are supported on \emph{monomially admissible Lagrangians}, that is, exact Lagrangians that are required to satisfy $ z^\alpha \in \RR_{>0} $ for $z \in L$ such that $u(z) \in C_\alpha$ near infinity. In other words, if $(u,p) \in L$ and $u \in C_\alpha$ lies outside a compact set, then $p \cdot \alpha = 0$. Roughly, morphisms are computed by applying a Hamiltonian $K$ that increases $p \cdot \alpha$ near infinity in $C_\alpha$ for all $\alpha \in A$. In practice, the category is obtained by localization at the continuation elements coming from the Hamiltonian $K$. The category of Lagrangian sections (with respect to $u$) form an important and computable subcategory as explained in the next section.

\begin{rem} More generally, \cite{hanlon2019monodromy} associates a Fukaya-Seidel category to any \emph{monomial division}, i.e., collection $\Delta$ of closed subsets satisfying the conditions above. For the purposes of this paper, we will only need to work with $\Delta_\eps$, which in the language of \cite{hanlon2019monodromy} corresponds to working with a monomial division \emph{adapted to }$\Sigma$.
\end{rem}

\begin{figure} 
\centering
\begin{tikzpicture}
\draw[blue, thick] (0,0) circle (2);
 \draw[blue, thick, fill = blue, fill opacity=0.2] (3.98,0.4) arc (0:180:1 and 0.3) --   (1.98,-0.4) arc (180:360:1 and 0.3) -- (3.98, 0.4);
\draw[blue, thick] (1.95,-0.4) arc (180:0:1 and 0.3);
\draw[xshift = -2cm, blue, thick, dashed] (3.95,0.4) arc (180:360:1 and 0.3);

\begin{scope}[rotate =90]
 \draw[blue, thick, fill = blue, fill opacity=0.2] (3.98,0.4) arc (0:180:1 and 0.3) --   (1.98,-0.4) arc (180:360:1 and 0.3) -- (3.98, 0.4);
\draw[blue, thick] (1.95,-0.4) arc (180:0:1 and 0.3);
\draw[xshift = -2cm, blue, thick, dashed] (3.95,0.4) arc (180:360:1 and 0.3);
\end{scope}

\begin{scope}[shift = {(-4.2, -4.2)}, rotate =45]
 \draw[blue, thick, fill = blue, fill opacity=0.2] (3.95,0.4) arc (0:180:1 and 0.3) --   (1.95,-0.4) arc (180:360:1 and 0.3) -- (3.95, 0.4);
\draw[blue, thick] (1.95,-0.4) arc (180:0:1 and 0.3);
\draw[xshift = -2cm, blue, thick, dashed] (3.95,0.4) arc (180:360:1 and 0.3);
\end{scope}

\draw[->, thick] (0,0) --(1,0);
\draw[->, thick] (0,0) -- (0,1);
\draw[->, thick] (0,0) -- (-1,-1);

 \begin{scope}[shift = {(8,0)}]
 \draw[blue, thick] (0,0) circle (2);
 \draw[blue, thick] (4,0) arc (0:180:1 and 0.5);
 \draw[blue, thick, dashed] (2,0) arc (180:360:1 and 0.5);
 \draw[blue, thick] (0,4) arc (90:270:0.5 and 1);
 \draw[blue, thick, dashed] (0,2) arc (-90:90:0.5 and 1);
\begin{scope}[shift = {(-4.25,-4.25)}, rotate = 45]
 \draw[blue, thick] (4,0) arc (0:180:1 and 0.5);
 \draw[blue, thick, dashed] (2,0) arc (180:360:1 and 0.5);
\end{scope}

 \draw[->, thick] (0,0) --(1,0);
\draw[->, thick] (0,0) -- (0,1);
\draw[->, thick] (0,0) -- (-1,-1);
 \end{scope}
\end{tikzpicture} \caption{The region of $\partial_\infty (\CC^*)^2$ that monomially admissible Lagrangian sections are permitted to intersect (left) and the stop $\stp_\Sigma$ of the FLTZ skeleton for $\check X_\Sigma = \bb{P}^2$.} \label{fig:monvsfltz}
\end{figure}
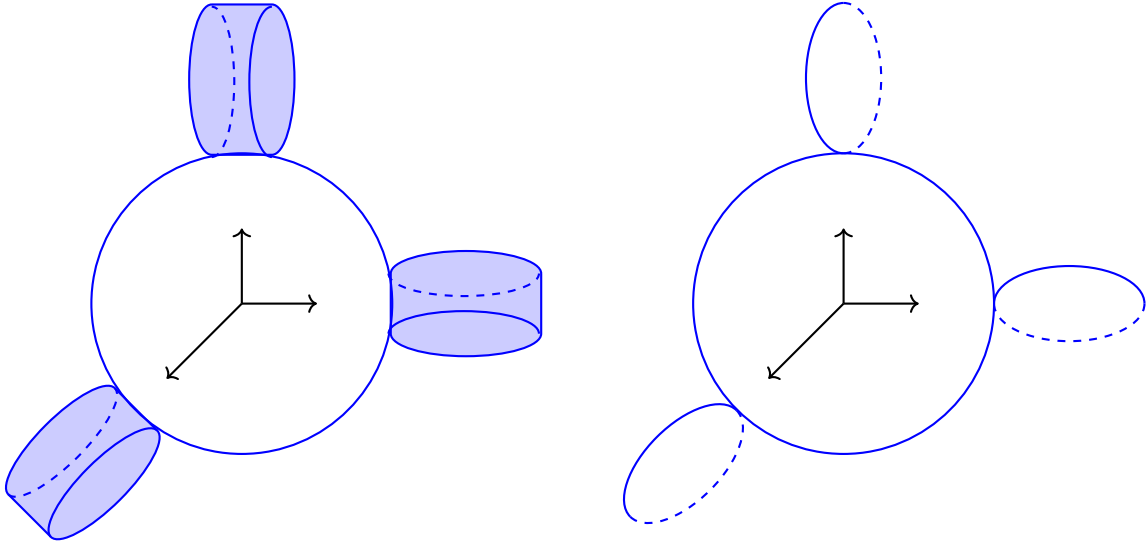

\subsection{Tropical Lagrangian sections}
\label{subsec:troplag}

It is well-understood that Lagrangian sections of an SYZ fibration should be mirror to line bundles, and this philosophy carries through to the toric setting despite the compactification. In the context of HMS for toric varieties, Lagrangian sections of $\Log$ were first studied in \cite{abouzaid2006homogeneous} where Abouzaid computes the Floer cohomology of the mirror to powers of an ample line bundle. Applying the Legendre transform and working with sections of the SYZ fibration $u$ is more amenable to constructing and computing a category of all sections as done in \cite{abouzaid2009morse}. Abouzaid studied sections lying over a region in $Q$ modeled on the moment polytope of $\check X_\Sigma$ and with boundary on a fiber of the tropical localization of $W_\Sigma$. To pass to the partially wrapped setting, it will be more convenient for us to work with the variant of Abouzaid's approach introduced in \cite{hanlon2019monodromy} and motivated above.

Namely, we consider the subcategory $\mathcal{F}^s_{\Delta_\eps} \subset \mathcal{F}_{\Delta_\eps}$ of Lagrangian sections of $\val$. Note that $\mathcal{F}^s_{\Delta_\eps}$ is independent of small $\eps$ by \cite[Proposition 3.19]{hanlon2019monodromy}. The Hamiltonian isotopy classes of monomially admissible Lagrangian sections have particularly nice representatives, which we will refer to as \emph{tropical Lagrangian sections}. Indeed, given a support function $\sF_D$ corresponding to a toric divisor $D = \sum n_\alpha D_\alpha$, we can smooth $\sF_D$ to a function $H_D$. 
By abuse of notation, we will also use $H_D$ to denote the Hamiltonian for $H\circ u : (\CC^*)^n\to \RR$,  which we call a \emph{twisting Hamiltonian} for $D$.
This Hamiltonian also satisfies 
$$ dH_D(\alpha) = -2\pi n_\alpha $$
near infinity in $C_\alpha$ for all $\alpha \in A$ by setting 
$$H_D(u) = -2 \pi \int_{\RR^n} \sF_D(u-x) \eta_\eps(x) \, dx $$
where $\eta_\eps$ is a smooth symmetric mollifier supported on the ball of radius $\eps$ as in \cite[Proposition 2.13]{hanlon2019monodromy}. From a twisting Hamiltonian $H_D$, we obtain a tropical Lagrangian section $L(D)$ as the projection of the graph of $dH_D$ to $T^*Q/T^*_{2\pi \ZZ} Q$ or, equivalently, the time-$1$ flow of the real positive locus $ \RR_{>0}^n \subset (\CC^*)^n$ under the Hamiltonian isotopy. The Floer theory of these Lagrangians can be shown to behave exactly as expected.

\begin{thm}[Corollary 3.41 of \cite{hanlon2019monodromy} following \cite{abouzaid2009morse}] There is a quasi-equivalence $$\mathcal{F}^s_{\Delta_\eps} \simeq \Pic_{\dg}(\check X_\Sigma)$$ for small $\eps$ where $\Pic_{\dg}(\check X_\Sigma)$ is a $\dg$-enhancement of the category of line bundles on $\check X_\Sigma$.
   \label{thm:OrigHMS}
\end{thm}

In fact, this construction gives slightly more as the time-$1$ flow of $H_D$ preserves monomial admissibility with respect to $\Delta_\eps$. Thus, these Hamiltonian diffeomorphisms act on $\mathcal{F}_{\Delta_\eps}$, and it can be checked that on $\mathcal{F}^s_{\Delta_\eps}$ these functors are mirror to the functors $(\cdot)\otimes \mathcal{O}(D)$ as done in \cite[Theorem 4.5]{hanlon2019monodromy}. 

\begin{rem} \label{rem:relatetoAb} In \cite{abouzaid2009morse}, tropical Lagrangian sections are defined as sections over a subset $P$ of $(\CC^*)^n$, which is roughly the polytope of $\check X_\Sigma$ for some ample line bundle, with boundary on a deformation of the Hori-Vafa superpotential. This makes Abouzaid's category of tropical Lagrangian sections\footnote{This is defined in \cite{abouzaid2009morse} as an $A_\infty$-pre-category, but using localization, one can define a quasi-equivalent $A_\infty$-category.}, $\mathscr{T}\text{Fuk}$, a subcategory of a more traditionally defined Fukaya-Seidel category. Abouzaid's tropical Lagrangian sections are closely related to those considered above. Indeed, assuming that we can work with the standard K\"{a}hler form to define $\mathcal{F}^s_{\Delta_\eps}$ (see \cref{rem:monadmholomorphic}), then there is an equivalence between $\mathscr{T}\Fuk$ and $\mathcal{F}^s_{\Delta_\eps}$ given by restricting $L(D)$ to $\Log^{-1}(P)$ as shown in \cite[Proposition 3.23]{hanlon2019monodromy}. Even without that assumption, it can be deduced from \cite[Proposition 2.36]{hanlon2019monodromy} that the restriction of $L(D)$ to $\Log^{-1}(P)$ are objects in $\mathscr{T}\Fuk$ even though they are not the smoothing of a support function in the $\Log$ coordinates. Of course, combining \cite[Theorem 1.2]{abouzaid2009morse} and \cite[Corollary 3.41]{hanlon2019monodromy} also gives a less direct equivalence between $\mathscr{T}\Fuk$ and $\mathcal{F}^s_{\Delta_\eps}$ regardless of the K\"{a}hler form used to define the latter.
\end{rem}
 
\section{Conical Lagrangian sections and twisting Hamiltonians}
   \label{sec:embedding}
   This section consists of producing the embedding in \cref{thm:mainthm}, translating the action of twisting Hamiltonians from \cite{hanlon2019monodromy} to the partially wrapped setting, and realizing cocores of $\LL_\Sigma$ as tropical Lagrangian sections. The main ingredient is the observation that tropical Lagrangian sections  (sections very close to their piecewise-linear/tropical limit) exist in both categories. Furthermore, these sections have matching Floer theory.

\subsection{Construction and embedding} 

We will now construct a full and faithful embedding of $\mathcal{F}_{\Delta_\eps}^s$ into $\mathcal{W}((\bb{C}^*)^n, \stp_\Sigma)$ for small $\eps$. Although objects of $\mathcal{F}_{\Delta_\eps}$ naturally lie in the complement of the FLTZ skeleton when perturbed by the Hamiltonian $K$, there are three difficulties in passing from monomially admissible Floer theory to partially wrapped Floer theory. Monomially admissible Lagrangians, which in general are not conical at infinity, must be modified to obtain objects of the partially wrapped category. Then, we must check that this modification does not change the Floer theory of the Lagrangians when now computed with a cylindrical almost complex structure. Finally, we need to prove that the (conical version of) the Hamiltonian $K$ gives a cofinal wrapping and thus can be used to compute morphisms in $\mathcal{W}( (\CC^*)^n, \stp_\Sigma)$.

Although it is not clear how to produce a conical Lagrangian submanifold from a monomially admissible Lagrangian submanifold in general, we can modify the construction of tropical Lagrangian sections summarized in \cref{subsec:troplag}  from support functions of toric divisors to obtain tropical Lagrangian sections in $\mathcal{W}((\bb{C}^*)^n, \stp_\Sigma)$. As the Liouville vector field is the radial vector field on the SYZ base, this amounts to smoothing the piecewise linear support function $F_D$ of a toric divisor so that the gradient is invariant under rescaling in the complement of a compact set. As observed in Remark 3.4 of \cite{hanlon2019monodromy}, this is not possible while preserving monomial admissibility. We therefore will need to allow some distortion of $\thetaa \cdot \alpha$ even close to the ray $\langle\alpha\rangle \in \Sigma$.

\begin{lem} \label{lemma:conesctions} Suppose that $F_D$ is the piecewise-linear support function of a toric divisor $D = \sum_{\alpha \in A} n_\alpha D_\alpha$, $\{U_{\alpha} \}$ is the $\eps$-star open cover for $\Sigma$, $\eps>0$ is sufficiently small, and $V \subset \bb{R}^n$ is any compact set for which the origin is an interior point. Then, there exists a smooth function $\wh{H}_D$ satisfying the following.
\begin{enumerate}
    \item $\wh{H}_D$ is positively homogeneous of degree $1$ outside of $V$. 
    \item In the complement of $V$, $\wh{H}_D = 2\pi F_D$ on any non-empty intersection of the form $U_{\alpha_1} \cap \hdots \cap U_{\alpha_n}$.\footnote{We never use this condition, but it is somewhat reassuring to make the Lagrangian sections as close to their tropical counterparts as possible.}
    \item There is a constant $C_D$, independent of $\eps$ and $V$, such that
    $$ |\nabla \wh{H}_D \cdot \alpha + 2\pi n_\alpha| \leq \eps C_D $$
    in $U_\alpha \setminus V$ for all $\alpha \in A$.
\end{enumerate}
\end{lem} 
\begin{proof} Let $\eta: \bb{R}^n \to \bb{R}$ be a smooth symmetric $(\eta(u) = \eta(-u))$ mollifier function supported on the ball of radius $1$ centered at the origin.
For any $t >0$ define a $t$-scaled mollifier $\eta_t: \bb{R}^n \to \bb{R}$ by $\eta_t (u) = t^{-n}\eta(u/t)$. 

Define $G: \bb{R}^n \setminus \{0\} \to \bb{R}$  by
\[ G(u) = 2\pi \int_{\bb{R}^n} \sF_D(u-x) \cdot \eta_{\eps r} (x) \, dx = 2\pi \int_{\bb{R}^n} \sF_D(u - rx) \cdot \eta_\eps(x) \, dx \]
where $r = \| u \|$ is the radial coordinate with respect to the Euclidean metric. It is straightforward to see that $G$ is smooth and homogeneous of degree $1$. It also follows from a direct computation that $G = F_D$ on any set of the form $U_{\alpha_1} \cap \hdots \cap U_{\alpha_n}$ using the fact that $F_D$ is a linear function on the open ball of radius $\eps r$ around any point in the intersection. 

We can then compute that
\begin{equation} \label{eq:conicalgradient} \nabla G (u) = 2 \pi \int_{ \bb{R}^n} \left[ \nabla F_D(u- rx) - \frac{\nabla F_D(u-rx) \cdot x}{r} u \right] \eta_{\eps }(x) \, dx . \end{equation}
The gradient of the support function is given by  $\nabla F_D \cdot \alpha = -n_\alpha$ on the open ball of radius $\eps r$ around any point in $U_\alpha$. 
Thus, if we set $E^\alpha(u) = (\nabla G(u) \cdot \alpha + 2\pi n_\alpha)/2\pi$, we have
$$ E^\alpha (u) = \frac{u \cdot \alpha}{r} \int_{\bb{R}^n} \left( \nabla F_D(u-rx)\cdot x \right) \eta_\eps(x) \, dx $$
for all $u\in U_\alpha$, and
\begin{align*} 
| E^\alpha (u) | &= \frac{| u \cdot \alpha| }{r} \left| \int_{\bb{R}^n} (\nabla F_D(u- rx) \cdot x )\eta_{\eps}(x) \, dx \right| \\ 
& \leq \| \alpha \|  \int_{\bb{R}^n} \| \nabla F_D(u- rx)\| \| x \| \eta_{\eps}(x) \, dx \ \\
& \leq \eps \| \alpha \| \int_{\bb{R}^n} \| \nabla F_D(u- rx)\| \eta_{\eps}(x) \, dx \\
&\leq \eps \| \alpha \| \max_{\sigma \in \Sigma^{n}} \| \nabla F_D|_{\sigma} \| 
\end{align*}
using the Cauchy-Schwarz inequality and that $\eta_\eps(x)$ is supported on the ball of radius $\eps$. This gives us the bound
$$ | E^\alpha (u)| \leq \eps C_D/2\pi$$
for all $\alpha \in A$ where the constants $C_D$ are given by
$$ C_D = 2\pi \max_{\alpha \in A} \| \alpha \| \max_{\sigma \in \Sigma^{n}} \| \nabla F_D|_{\sigma} \| .$$ 
To obtain the smooth function $\wh{H}_D:\bb{R}^n\to \RR$, we use a bump function with support contained within $V$ to smoothly interpolate between $G$ and a smooth function.
\end{proof}

The previous lemma allows us to construct Lagrangian sections in $\mathcal{W}((\bb{C}^*)^n, \stp_\Sigma)$ corresponding to divisors on $\check X_\Sigma$ by modifying the procedure outlined in \cref{subsec:troplag}. We begin by fixing a compact set $V \subset \bb{R}^n$ containing the origin as an interior point. For any $\delta > 0$ and toric divisor $D$ on $\check X_\Sigma$, we let $\wh{H}_D^\delta$ be a smooth function as guaranteed by \cref{lemma:conesctions} for our fixed compact set $V$ and $\eps = \delta/C_D$. 
We will refer to such functions as \emph{conical twisting Hamiltonians with defect} $\delta$. We will denote the time $t$ flow of $\wh{H}_D^\delta\circ u$ by $\phi_{D, \delta}^t: (\CC^*)^n\to (\CC^*)^n$. For the canonical divisor $K = \sum_{\alpha \in A} - D_\alpha$, we will simplify the notation to $\psi_\delta^t$ for its flow. 

To obtain Lagrangian sections, we set $\mathcal{L}$ to be the zero section (equivalently, $\mathcal{L}= (\bb{R}_{>0})^n \subset (\CC^*)^n$) and
\[ L^\delta(D) := \psi^{-\frac{\delta}{2\pi}}_{\delta/4} \phi_{D,\delta/4}^1 (\mathcal{L}).  \] 
In other words, $L^\delta(D)$ is the image of the graph of 
$$d\left(\wh{H}_D^{\delta/4} + \left( - \frac{\delta}{2\pi} \right)H^{\delta/4}_K \right)$$
under the projection from $T^*\bb{R}^n$ to $T^* \RR^n/ T^*_{2\pi \ZZ} \RR^n \simeq (\bb{C}^*)^n$. 
We will, at times, call $L^\delta(D)$ a \emph{conical tropical Lagrangian section with defect} $\delta$. 
We will also sometimes use the notation $L^\delta(\sF):= L^\delta(D_{\sF})$, where $\sF$ is a support function for the divisor $D$. 
We base our choice of notation on the notation used for line bundles on the toric variety $\check X_\Sigma$.
Taking $\text{Arg}$ to denote the branch of $\arg$ taking values in $[-\pi, \pi)$, note that \cref{lemma:conesctions} implies that
\begin{equation} \label{eq:anglebound} \left| \text{Arg}\left(z^\alpha|_{L^\delta(D)}\right) +\delta) \right| \leq \delta/2 < \delta \end{equation}
in $U_\alpha \setminus V$ for all $\alpha \in A$. 
Now, we observe that for every cone $\sigma \in \Sigma$ and point $u \in \sigma \setminus V$, there is an $\alpha \in \sigma$ such that $u \in U_\alpha$ for sufficiently small choice of $\delta$.
It follows that $$ L^\delta(D) \setminus V \subset (\bb{C}^*)^n \setminus \bb{L}_\Sigma $$
for all $D$ and sufficiently small $\delta$ (see \cref{eq:fltzskel}).
Thus, $L^\delta(D)$ supports an object of $\mathcal{W}((\bb{C}^*)^n, \stp_\Sigma)$, which has a canonical grading as a Lagrangian section (see \cite[Example 2.10]{seidel2000graded}) and a unique relative spin structure as a contractible manifold. We will abuse notation and also denote this object by $L^\delta(D)$. Note that as an object of $\mathcal{W}((\bb{C}^*)^n, \stp_\Sigma)$, $L^\delta(D)$ is independent of the choices ($\eta, \delta, \hdots$) made in its construction as different choices result in a Lagrangian that differs by a Hamiltonian isotopy that does not meet $\stp_\Sigma$

\begin{rem} \label{rem:conicalsecineq} Conical tropical Lagrangian sections corresponding to a divisor $D = \sum n_\alpha D_\alpha$ are characterized by the fact that they have lifts to $T^*\RR^n$ that satisfy 
$$ -2\pi (n_\alpha - 1) < \alpha \cdot p < -2\pi n_\alpha $$
in the coordinates $(u, p = p(u))$ whenever $u \in U_\alpha$ near infinity. Thus, unlike in the monomially admissible setting, there are many admissible Lagrangian sections that are not tropical or even admissibly Hamiltonian isotopic to a tropical section. See \cref{ex:blowupexample}.
\end{rem} 

\begin{figure}
   \centering
   \begin{tikzpicture}
   \draw (1,1) node {$-u_1$};
   \draw (-1.5, 0.5) node {$u_1$};
   \draw (0.5, -1) node {$-u_1 + 2u_2$};
   \draw[blue, fill=blue] (0,0) circle (0.4);
   \draw[blue, fill=blue, opacity=0.3] (0,0.3) --(2,0.3) -- (2, -0.3) -- (0, -0.3) -- (0, 0.3);
   \draw[blue, fill=blue, opacity=0.3] (0.3,0) --(0.3,2) -- (-0.3, 2) -- (-0.3, 0) -- (0.3, 0);
   \draw[blue, fill=blue, opacity=0.3] (0.21,-0.21) --(-1.79,-2.21) -- (-2.21, -1.79) -- (-0.21, 0.21) -- (0.21, -0.21);
   \draw[->, thick] (0,0) -- (1,0);
   \draw[->, thick] (0,0) -- (0,1); 
   \draw[->, thick] (0,0) -- (-1,-1);
   
   \begin{scope}[shift = {(7,0)}]
   \draw (1,1) node {$-u_1$};
   \draw (-1.5, 0.5) node {$u_1$};
   \draw (0.5, -1) node {$-u_1 + 2u_2$};
   \draw[blue, fill=blue] (0,0) circle (0.4);
   \draw[blue, fill=blue, opacity=0.3] (0,0) --(2,0.5) -- (2, -0.5) -- (0,0);
   \draw[blue, fill=blue, opacity=0.3] (0,0) --(0.5,2) -- (-0.5, 2) -- (0, 0);
   \draw[blue, fill=blue, opacity=0.3] (0,0) --(-1.65,-2.35) -- (-2.35, -1.65) -- (0, 0);
   \draw[->, thick] (0,0) -- (1,0);
   \draw[->, thick] (0,0) -- (0,1); 
   \draw[->, thick] (0,0) -- (-1,-1);
   \end{scope}
\end{tikzpicture}    \caption{The monomial admissible smoothing, $\frac{1}{2\pi}H_D$, and the conical smoothing, $\frac{1}{2\pi}\wh{H}_D^{\delta}$, of the support function $F_D$ for $D = D_{(1,0)} + D_{(-1,-1)}$ on $\bb{P}^2$ are shown on the left and right, respectively. The dark blue circle represents $u(V)$ while the shaded regions represent where the piecewise linear function is smoothed with strict and approximate control of the corresponding angles, respectively.}
   \label{fig:comparefunctions}
\end{figure}
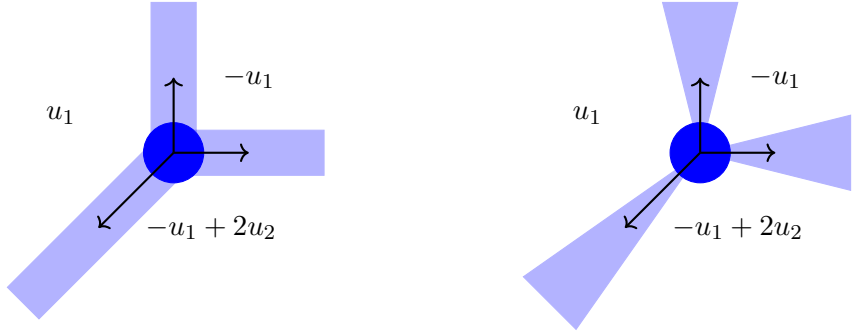

It remains to understand the Floer theory of these objects. Note first that when $\delta$ is chosen in a smooth family, we can take the functions guaranteed by \cref{lemma:conesctions} to vary smoothly by smoothly choosing the interpolation bump functions on $V$. Below, we will always assume that we have done so.

\begin{figure}
   \centering
   \begin{subfigure}{.3\linewidth}
      \centering
      \includegraphics[scale=.3]{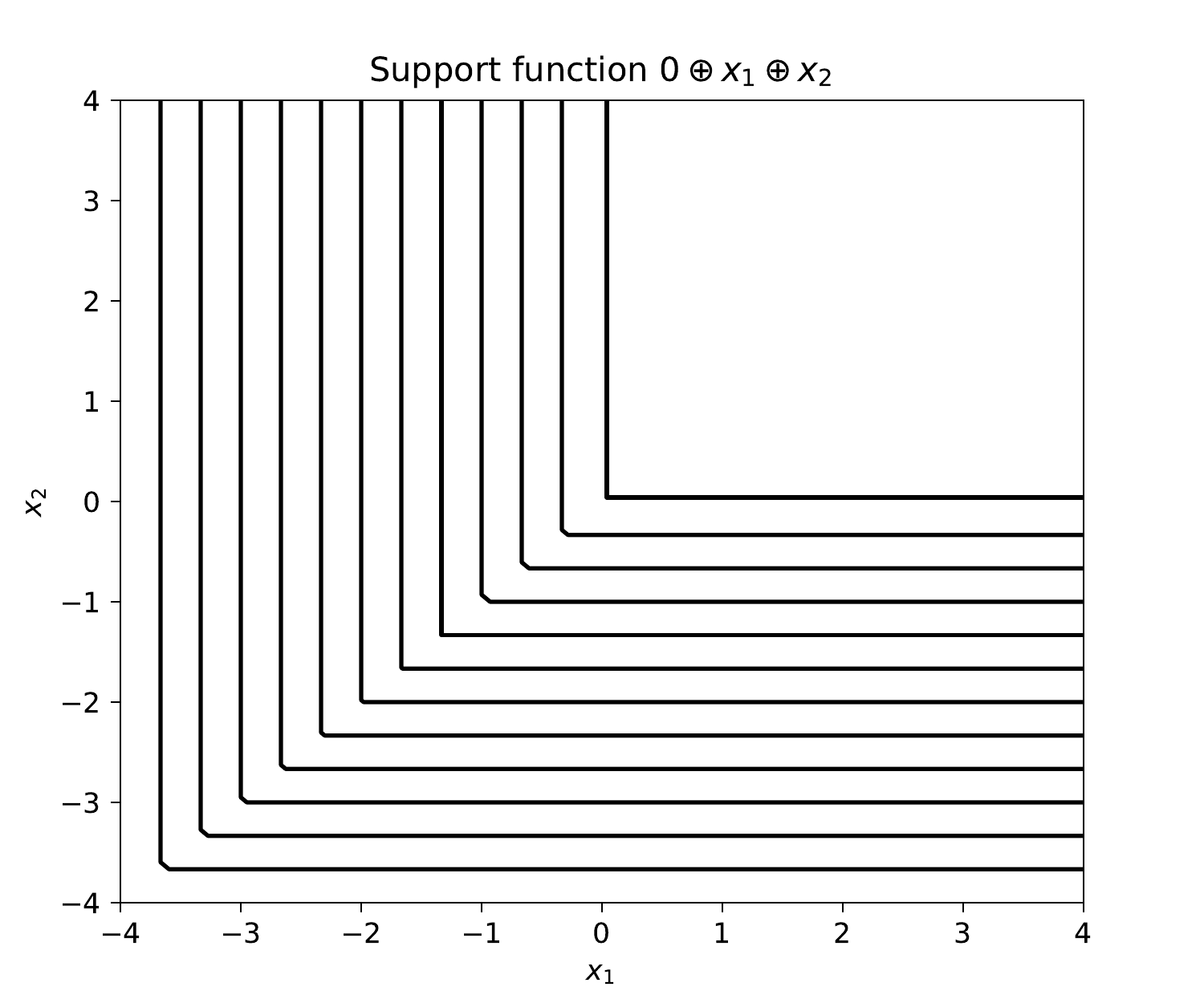}
      \caption{Contour plot of the support function $\sF_D=\max(0, x_1, x_2)$}
   \end{subfigure}
   \begin{subfigure}{.3\linewidth}
      \centering
      \includegraphics[scale=.3]{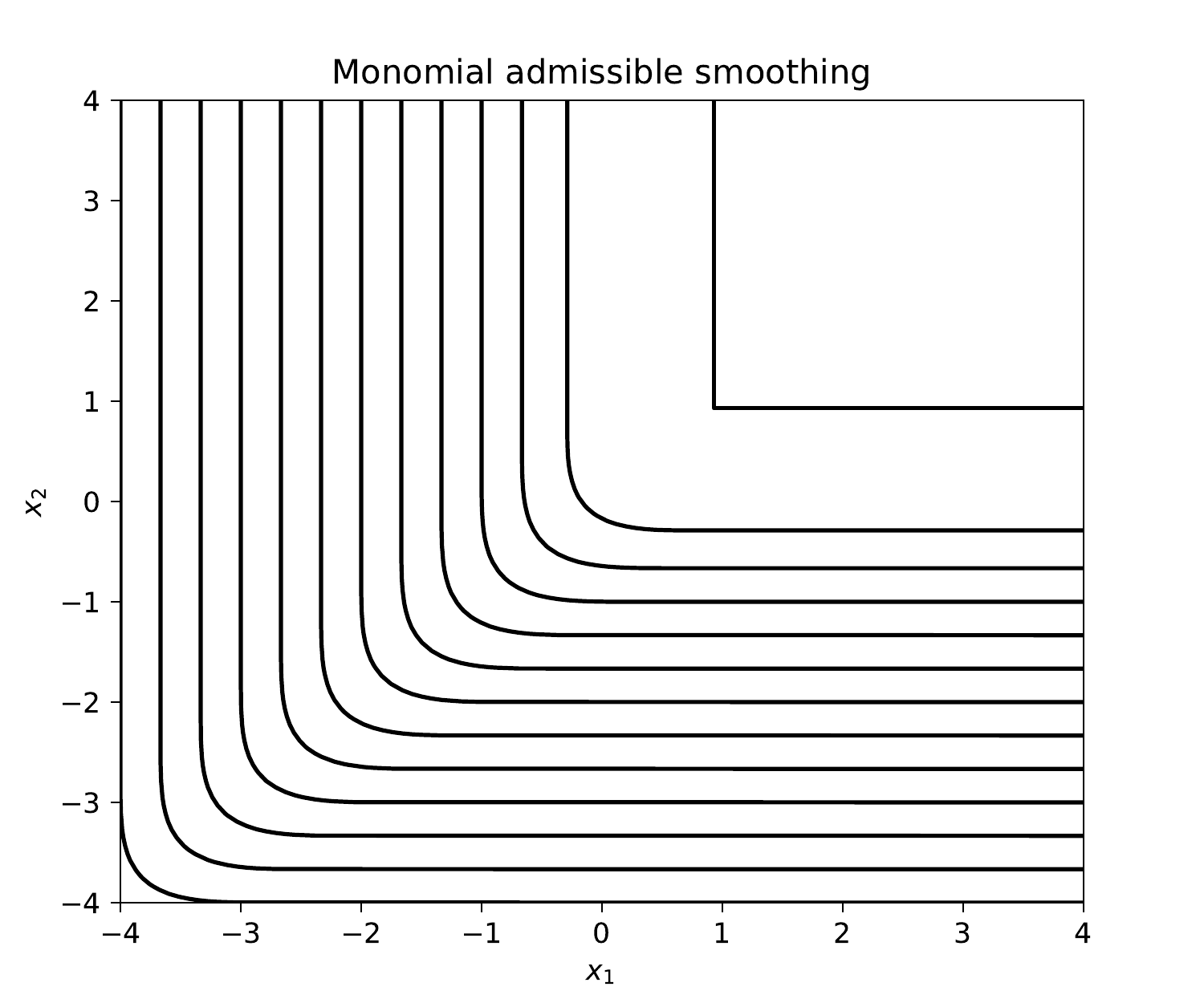}
      \caption{Contour plot of the monomial admissible smoothing $H_D$}
   \end{subfigure}
   \begin{subfigure}{.3\linewidth}
      \centering
      \includegraphics[scale=.3]{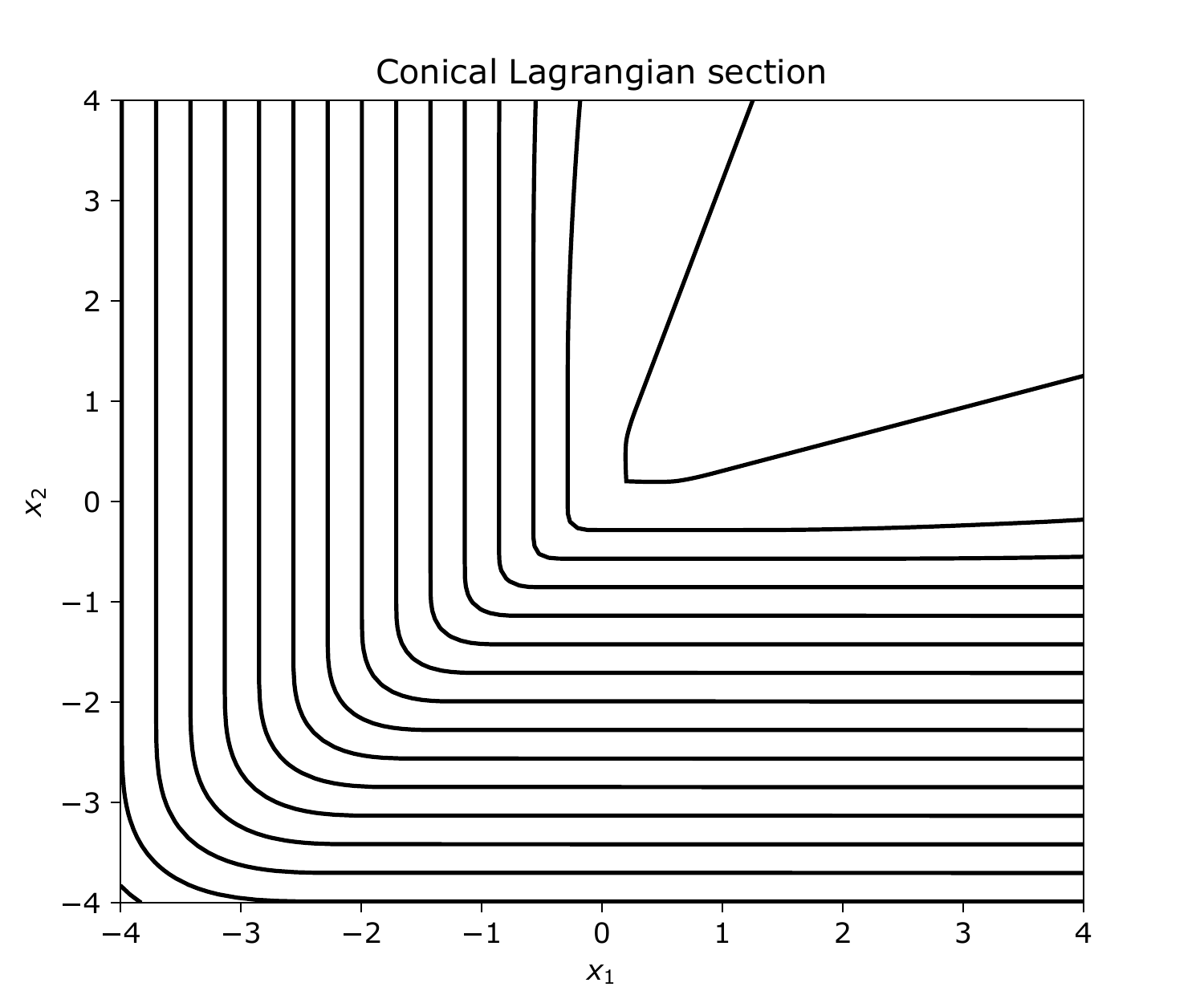}
      \caption{Contour plot of the conically admissible smoothing $\hat H_D$}
   \end{subfigure}
   \caption{Contour plots of a monomial admissible smoothing and a conical smoothing.}
\end{figure}
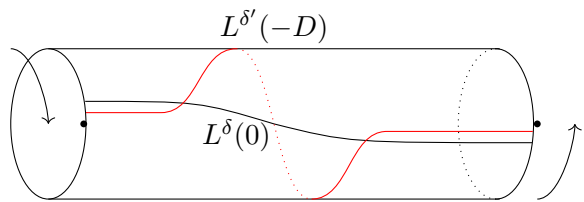
\begin{figure}
   \centering
   \begin{tikzpicture}

               \draw  (-0.5,-1.5) ellipse (0.5 and 1);
   
               \begin{scope}[shift={(3.95,-3.5)}]
                  \begin{scope}[]
   
                     \clip  (1.5,3) rectangle (2.5,1);
                     \draw  (1.5,2) ellipse (0.5 and 1);
   
                  \end{scope}
                  \begin{scope}[]
   
                     \clip  (1.5,1) rectangle (0.5,3);
                     \draw[dotted]  (1.5,2) ellipse (0.5 and 1);
                  \end{scope}
               \end{scope}
   
               \draw (-0.5,-2.5) -- (5.5,-2.5);
               \draw (-0.5,-0.5) -- (5.5,-0.5);

         \node[fill=black, circle, scale=.25] at (-0.0296,-1.5) {};
         \node[fill=black, circle, scale=.25] at (6,-1.5) {};

\draw (-0.01,-1.2) .. controls (1.5,-1.2) and (1.5,-1.2) .. (2.5,-1.5) .. controls (3.5,-1.75) and (3.5,-1.75) .. (5.93,-1.75);
\draw[red] (0,-1.35) .. controls (0.65,-1.35) and (0.55,-1.35) .. (1,-1.35) .. controls (1.5,-1.35) and (1.5,-0.5) .. (2,-0.5);
\draw[dotted, red] (2,-0.5) .. controls (2.5,-0.5) and (2.5,-2.5) .. (3,-2.5);
\draw[red] (3,-2.5) .. controls (3.5,-2.5) and (3.5,-1.6) .. (4,-1.6) .. controls (4.55,-1.6) and (5.55,-1.6) .. (5.94,-1.6);
\draw[->] (6,-2.5) .. controls (6.3,-2.5) and (6.5,-1.75) .. (6.5,-1.5);
\draw[->,yscale=-1] (-1,0.5) .. controls (-0.7,0.5) and (-0.5,1.25) .. (-0.5,1.5);
\node at (2.5,-0.2) {$L^{\delta'}(-D)$};
\node at (2,-1.65) {$L^{\delta}(0)$};
\end{tikzpicture}    \caption{Here, $\check X_\Sigma=\CP^1$, the divisor $D$ is a point, and the Lagrangians are in position to compute $\Hom_{\mathcal W(\CC^*,  \stp_\Sigma)}(L(-D), L(0))$. Note that $\delta'<\delta$. }
\end{figure}
\begin{lem} \label{lemma:cofinal} If $\delta_0$ is sufficiently small and $\delta(t)$ is a smooth monotonically decreasing function on $[0, \infty)$ satisfying $\delta(0) = \delta_0$ and 
$$ \lim_{t \to \infty} \delta(t) = 0, $$
then $\{L_t\}_{t \geq 0}$ defined by
$$ L_t = L^{\delta(t)}(D) $$
is a cofinal wrapping of $L^{\delta_0}(D)$ in $((\bb{C}^*)^n, \stp_\Sigma)$.
\end{lem}
\begin{proof} By the comments preceding the statement of the lemma, $\{ L_t \}$ is a smooth family of exact Lagrangians in $(\bb{C}^*)^n \setminus \stp_\Sigma$. We now need to check this is a positive isotopy at infinity. Recall that we take the contact boundary to be $\| u \| = R$ for large $R$ (with $V$ in the interior) and the contact form is given by 
$$ \lambda = \sum_{i=1}^n u_i dp_i .$$ 
In those coordinates, 
$$ L_t(u) = \left(u,  \frac{\partial \wh{H}_D^{\delta(t)/4}}{\partial u} - \frac{\delta(t)}{2\pi} \frac{\partial H^{\delta(t)/4}_K}{\partial u} \right)$$
and thus
$$ \lambda(\partial_t L_t)(u) = u \cdot \frac{d}{dt} \left( \nabla \wh{H}_D^{\delta(t)/4}(u) - \frac{\delta(t)}{2\pi}\nabla H^{\delta(t)/4}_K (u) \right). $$
Combining the above with \cref{eq:conicalgradient} and the fact that $F_K$ and $F_D$ are piecewise-linear, we have 
\begin{align*} \lambda(\partial_t L_t)(u) =& - \delta'(t) \int_{\bb{R}^n} F_K(u - Rx) \eta_{\eps_K(t)}(x) \, dx \\
&+ 2\pi \frac{d}{dt} \int_{ \bb{R}^n}  F_D(u- Rx)  \eta_{\eps_D(t)}(x) \, dx \\ 
&- \delta(t)  \frac{d}{dt} \int_{ \bb{R}^n} F_K(u- Rx)  \eta_{\eps_K(t)}(x) \, dx\end{align*} 
where $\eps_K(t) = \delta(t)/4 C_K$ and $\eps_D(t) = \delta(t)/4C_D$.
For the second term, we have 
\begin{align*} \frac{d}{dt} \int_{ \bb{R}^n}  F_D(u- Rx)  \eta_{\eps_D(t)}(x) \, dx &= \frac{d}{dt} \int_{ \bb{R}^n} F_D(u- R\eps_D(t)x) \eta(x) \, dx \\
& = - R\eps_D'(t) \int_{\bb{R}^n} \nabla F_D(u - R\eps_D(t)x) \cdot x \eta(x) \, dx.
\end{align*}
Applying the same calculation to the third term gives
\begin{align*} \lambda(\partial_t L_t)(u) = - \delta'(t) R \Bigg(&\int_{\bb{R}^n} F_K(u/R - x) \eta_{\eps_K(t)}(x) \, dx \\
&+ \frac{2\pi}{4 C_D}\int_{\bb{R}^n} \nabla F_D(u - R\eps_D(t)x) \cdot x \eta(x) \, dx  \\
&- \frac{ \delta(t)}{4 C_K}\int_{\bb{R}^n} \nabla F_K(u - R\eps_K(t)x) \cdot x \eta(x) \, dx \Bigg). \end{align*}
As $\delta(t)$ is monotonically decreasing, $-\delta'(t)$ is positive. The first term inside the parentheses is positive as $F_K$ is positive on $\bb{R}^n\setminus \{0 \}$. Moreover, since $\Sigma$ is complete, 
for $x \in B_{\eps_K(t)} (0)$ we may write
$$ u/R - x = \sum_{\alpha \in \sigma} c_\alpha \frac{\alpha}{\| \alpha \|} $$
where $\sigma$ is a top-dimensional cone and $c_\alpha \in \RR_{\geq 0}$ satisfy $\sum c_\alpha \geq 1- \eps_K(t) \geq  1-\eps_K(0)$.
It follows that
$$ \int_{\bb{R}^n} F_K(u/R - x) \eta_{\eps_K(t)}(x) \, dx  \geq \frac{1 - \eps_K(0)}{\max_{\alpha \in A} \| \alpha \|}. $$
For the second term, we have
$$ \left| \frac{1}{4 C_D}\int_{\bb{R}^n} \nabla F_D(u - R\eps_D(t)x) \cdot x \eta(x) \, dx \right| \leq \frac{1}{4 C_D} \max_{\sigma \in \Sigma^{n}} \| \nabla F_D|_{\sigma} \| = \frac{1}{8\pi \max_{\alpha \in A} \| \alpha \|}. $$
By the same argument, the norm of the third term has the same upper bound. Thus, $\lambda(\partial_t L_t)(u)$ is positive as long as $\delta_0$ is small enough so that $1- \eps_K(0) \geq 1/4\pi$.

It then follows from \cite[Lemma 2.1]{ganatra2018sectorial}  that the wrapping is cofinal as every point in $L_t\cap B_R(0)$ converges to a point in $\stp_\Sigma$ as shown by \cref{eq:anglebound}. 
\end{proof}

By modifying these constructions slightly and using monotonicity estimates, we can deduce that Floer cohomology of tropical Lagrangian sections is independent of the method of implementing them in a Fukaya category.

\begin{lem} \label{lemma:samehoms} Suppose that $D_1$ and $D_2$ are toric divisors on $\check X_\Sigma$, $\Delta_\eps$ is the $\eps$ combinatorial division for $\Sigma$, and $\delta > 0$ is small. If $L(D_1), L(D_2)$ are the corresponding monomially admissible tropical Lagrangian sections and $L^\delta(D_1), L^\delta(D_2)$ are the corresponding conical tropical Lagrangian sections with defect $\delta$, then
$$ H^\bullet \Hom_{\mathcal{F}_\Delta^s} (L(D_1), L(D_2))) \cong H^\bullet \Hom_{\mathcal{W}( (\bb{C}^*)^n, \stp_\Sigma)} (L^\delta(D_1), L^\delta(D_2)) $$
\end{lem}
\begin{proof} By \cite[Lemma 3.21]{ganatra2017covariantly}, \cref{lemma:cofinal}, and \cref{eq:anglebound}, we have
$$ H^\bullet \Hom_{\mathcal{W}( (\bb{C}^*)^n, \stp_\Sigma)} (L^\delta(D_1), L^\delta(D_2)) \cong HF(L^{\delta - d}(D_1), L^\delta(D_2))$$
for generically chosen $d \in (\delta/2, \delta)$ and small $\delta$.\footnote{Choosing $d>\delta/2$ guarantees that these Lagrangians do not intersect near infinity by \cref{eq:anglebound}.} Fix a compact set $V \subset \bb{R}^n$ as in the definition of the conical twisting Hamiltonians. Using a priori energy bounds from exactness and monotonicity (as in \cite[Proposition 3.19]{ganatra2017covariantly}), there is an $R>0$ such that for any almost complex structure $J$ that is conical on $\val^{-1}(B_R(0) \setminus V)$, all $J$-holomorphic discs contributing to the differential in $CF(L^{\delta - \eps}(D_1), L^\delta(D_2))$ are contained in $\val^{-1}(B_R(0))$. Moreover, this property is preserved when modifying $J$ and the Lagrangians outside of $B_R(0)$ without creating additional intersection points. We can then interpolate between the two flavors of tropical sections as follows.

Let $f_a: [0,\infty) \to \bb{R}$ be a smooth, non-decreasing concave function that satisfies $f_a(t) = t$ for $t \leq a$ and $f_a(t) = a + \eps$ for $t \geq a+\eps$ for some small $\eps >0$ as drawn in \cref{fig:interpolatingfunction}. We define an interpolating tropical Lagrangian section of radius $a$ associated to a toric divisor $D$ to be a Lagrangian section of the form $L_{a} (D) = dH_D^a$ where
$$ H_D^a(u) = 2 \pi \int_{\bb{R}^n} F_D(u-x) \eta_{\eps f_a(r)} (x) \, dx $$
outside of a compact set $V$ with small $\eps >0$. Calculating as in the proof of \cref{lemma:conesctions}, we obtain that
$$ \left| \nabla H_D^a + 2\pi n_\alpha \right| \leq \eps f_a'(r) C_D. $$
in $U_\alpha \setminus V$ for all $\alpha \in A$. Proceeding as in the case of the conical tropical Lagrangian sections we can then construct an interpolating tropical Lagrangian section of radius $a$ and defect $\delta$, $L_{a}^\delta(D)$, satisfying
\begin{equation} \label{eq:interpolateangle} \left| \text{Arg}\left( z^{\alpha}|_{L^\delta(D)(a)} \right) + \delta \right| < f_a'(r) \delta . \end{equation}
Moreover, $L_a^\delta(D) = L^\delta(D)$ in $\val^{-1}(B_a(0))$ and is equal to $\psi^{- \frac{\delta}{2\pi}} (L_a(D))$ where $L_a(D)$ is a monomially admissible Lagrangian section corresponding to $D$ with respect to $\Delta_\eps$. Finally, observe that \cref{eq:interpolateangle} and concavity of $f_a$ imply that all intersection points between $L_{a}^\delta(D_1)$ and $L_a^{\delta-d}(D_2)$ lie in  $V$ for all $a$ such that $V \subset B_a (0)$. After choosing $a > R$ and a monomially admissible almost complex structure $J$ that is conical on $\val^{-1}(B_R(0) \setminus V)$, we have a canonical identification
$$ \CF(L^{\delta-d}({D_1}), L^\delta(D_2)) = \CF(L_a^{\delta - d}(D_1), L^\delta_a(D_2)  ) = \CF(\psi^{d/2\pi}(L_{a}(D_1)), L_{a}(D_2) ) $$
as chain complexes. Finally, from Propositions 2.17 and 3.3 of \cite{hanlon2019monodromy}, we get 
$$ \HF(\psi^{d/2\pi}(L_{a}(D_1)), L_{a}(D_2) ) \cong \HF(\psi^{d/2\pi}(L(D_1)), L(D_2)) \cong H^\bullet \Hom_{\mathcal{F}_\Delta^s} (L(D_1), L(D_2)) $$
as desired.
\end{proof}

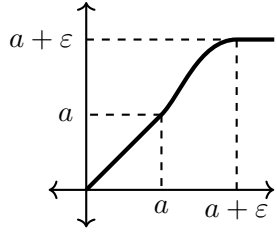
\begin{figure}
   \centering
   \begin{tikzpicture}
   \draw[<->, thick] (0, -0.5) -- (0,2.5); 
   \draw[<->, thick] (-0.5, 0) -- (2.5,0); 
   \draw[thick, dashed] (1,1) -- (1,0) node[below] {$a$};
   \draw[thick, dashed] (2,2) -- (2,0) node[below] {$a+ \eps$};
   \draw[thick, dashed] (1,1) -- (0,1) node[left] {$a$};
   \draw[thick, dashed] (2,2) -- (0,2) node[left] {$a+\eps$};
\draw[ultra thick] (0,0) .. controls (0.25,0.25) and (0.75,0.75) .. (1,1) .. controls (1.25,1.25) and (1.5,2) .. (2,2) .. controls (2.25,2) and (2.25,2) .. (2.5,2);
\end{tikzpicture}    \caption{The graph of the function $f_a$ used in the proof of \cref{lemma:samehoms}.}
   \label{fig:interpolatingfunction}
\end{figure}

We now upgrade \Cref{lemma:samehoms} to verify that we have an $A_\infty$ embedding using an argument similar to the proof of Corollary 4.36 in \cite{abouzaid2009morse}.

\begin{thm} \label{thm:embedding} The subcategory of $\mathcal{W}( (\bb{C}^*)^n, \stp_\Sigma)$ consisting of conical tropical Lagrangian sections is quasi-equivalent to $\mathcal{F}^s_{\Delta_\eps}$.
\end{thm}
\begin{proof}
We will establish the quasi-equivalence by relating the categories prior to localization. For a sequence $\{ \delta_n \}_{n \in \{ 0\} \cup \bb{N} }$ with $\delta_0$ small, $\delta_{n+1} \in (0, \delta_n/2)$, so that  $\delta_n \to 0$ as $n \to \infty$, we can construct the conical tropical Lagrangian sections $L^{\delta_n}(D)$ for all toric divisors $D$ on $\check X_\Sigma$ so that $L^{\delta_n}(D_1)$ and $L^{\delta_k}(D_2)$ intersect transversely for $k > n$. Note that all of these intersections occur in the fixed compact set $V$ by the assumption that $\delta_{n+1} \in (0, \delta_n/2)$. Clearly, the subcategory $\mathcal{A} \subset \mathcal{W}((\bb{C}^*)^n, \stp_\Sigma)$ consisting of the $L^{\delta_0}(D)$ is quasi-equivalent to the subcategory of conical tropical Lagrangian sections. Moreover, by \cref{lemma:cofinal}, the sequence $L^{\delta_n}(D)$ is cofinal for each toric divisor $D$. Thus, we can construct $\mathcal{A}$ by localizing the directed category $\mathcal{O}_\mathcal{A}$ whose objects are the pairs $(L^{\delta_n}(D), n )$ and appropriately ordered morphisms are computed with conical almost complex structures. 
By \cite[Proposition 3.19]{hanlon2019monodromy}, we can assume that $\eps = \delta_0/\min_{D}(C_D)$.\footnote{This minimum is achieved by $C_{D_\alpha}$ for some $\alpha \in A$.} As a consequence, the interpolating tropical Lagrangian sections $L^{\delta_n}_a(D)$ as defined in the proof of \cref{lemma:samehoms} are monomially admissible with $a$ large enough so that $\val(V) \subset B_a(0)$. 
Set $\mathcal{O}_\Delta$ to be the directed category whose objects are the pairs $(L^{\delta_n}_{a}(D),n)$ (with large enough $a$ in the sense of the preceding sentence) and appropriately ordered morphisms 
$$ m^k \colon CF \left(L_{a_{k-1}}^{\delta_{n_{k-1}}}(D_{k-1}), L_{a_{k}}^{\delta_{n_k}}(D_k) \right) \otimes \hdots \otimes CF\left(L_{a_0}^{\delta_{n_0}}(D_0), L_{a_1}^{\delta_{n_1}}(D_1) \right) \to CF \left( L_{a_0}^{\delta_0}(D_0), L_{a_{k}}^{\delta_{n_k}}(D_k) \right) $$
are computed with monomially admissible almost complex structures that are conical on the subset $\val^{-1}(B_{\min a_j} (0) \setminus V)$. The localization of $\mathcal{O}_\Delta$ at continuation elements is quasi-equivalent to $\mathcal{F}^s_\Delta$. Note that $\mathcal{O}_\Delta$ has a ``filtration" by subcategories $\mathcal{O}_\Delta^k$ whose objects are pairs $(L^{\delta_n}_{a}(D),n)$ with $a \geq k$. As changes in $a$ are implemented by compactly-supported Hamiltonian isotopies of each object, the inclusion $\mathscr{G}_{k \to \ell}\colon \mathcal{O}^k_\Delta \hookrightarrow \mathcal{O}^\ell_\Delta$ for $k \geq \ell$ is a quasi-equivalence with a Floer-theoretically defined inverse $\mathscr{G}_{\ell \to k}$ taking $(L^{\delta_n}_{a}(D),n)$ with $a < k$ to $(L^{\delta_n}_{k}(D),n)$ (see, for example, Section 10a of \cite{seidel2008fukaya}).

Now, we define a strictly unital functor $\mathscr{F}\colon \mathcal{O}_\Delta \to \mathcal{O}_\mathcal{A}$ as follows. On objects, we set
$$ \mathscr{F}(L^{\delta_n}_a(D) , n ) = (L^{\delta_n}(D), n)$$
for all $a$. For any finite set of objects $L^{\delta_{n_0}}(D_0), \hdots, L^{\delta_{n_k}}(D_k)$ with $n_0 > \hdots > n_k$, there is an $R > 0$ such that all pseudo-holomorphic disks computing $A_\infty$ structure maps $m^j_{\mathcal{O}_\mathcal{A}}$ with $j \leq k$ on these objects in appropriate order are contained in $\val^{-1}(B_R(0))$ by the same monotonicity argument as in the proof of \cref{lemma:samehoms}. In particular, the moduli spaces computing these structure maps are identical to those computing the structure maps $m^j_{\mathcal{O}_\Delta}$ for $j \leq k$ on the objects $L_{a_0}^{\delta_{n_0}}(D_0), \hdots, L_{a_k}^{\delta_{n_k}}(D_k)$ for any $a_0, \hdots , a_k \geq R$. Thus, for arbitrary $a_0, \hdots, a_k$, we get a well-defined $A_\infty$-functor by taking
$$ \mathscr{F}^k \colon CF \left(L_{a_{k-1}}^{\delta_{n_{k-1}}}(D_{k-1}), L_{a_{k}}^{\delta_{n_k}}(D_k) \right) \otimes \hdots \otimes CF\left(L_{a_0}^{\delta_{n_0}}(D_0), L_{a_1}^{\delta_{n_1}}(D_1) \right) \to CF \left( L^{\delta_0}(D_0) , L^{\delta_{n_k}}(D_{k}) \right) $$
to be the composition of the identification $CF \left( L_{a_0}^{\delta_0}(D_0) , L_{a_{k}}^{\delta_{n_k}}(D_{n_k}) \right) = CF \left( L^{\delta_0} ({D_0}), L^{\delta_{n_k}} ({D_{k}})\right)$ with $\mathscr{G}^k_{\min a_j \to R}$. Moreover, $\mathscr{F}$ is a quasi-equivalence as it induces an isomorphism on cohomology by definition.

To deduce a quasi-equivalence between $\mathcal{F}^s_\Delta$ and $\mathcal{A}$ from one between $\mathcal{O}_\Delta$ and $\mathcal{O}_{\mathcal{A}}$, we need to verify that the continuation elements in $H^\bullet \mathcal{O}_\Delta$ are sent to continuation elements in $H^\bullet \mathcal{O}_{\mathcal{A}}$ (see Proposition 3.39 in \cite{ganatra2017covariantly}). This follows from the fact that the underlying chain complexes are always identical as we have chosen all the interpolation parameters, ${a_i}$, to be large enough and hence identifications for small positive isotopies coincide and these identifications characterize continuation elements as in \cite[Definition 3.25]{ganatra2017covariantly}. Alternatively, observe that $HF(L^{\delta_{n+k}}(D), L^{\delta_n}(D)) \cong \bb{C}$ for all conical tropical Lagrangian sections, and thus localizing at any system of continuation elements (which must be nonzero) gives the same result. 
\end{proof}

It immediately follows from \cref{thm:OrigHMS} and \cref{thm:embedding} that the conical tropical Lagrangian sections are mirror to line bundles.

\begin{cor} \label{cor:hmsline} There is an embedding 
$$ \Pic_{\dg}(\check X_\Sigma) \hookrightarrow \mathcal{W}((\CC^*)^n, \stp_\Sigma) $$
taking a line bundle $\mathcal{O}(D)$ to the conical tropical Lagrangian section $L^\delta(D)$ for small $\delta$.
\label{cor:linebundlehms}
\end{cor}

\begin{rem} \label{rem:notechnicalprob} Recall from \cref{subsec:HMSboundary} that we need to choose a possibly nonstandard K\"{a}hler form on $(\CC^*)^n$ to ensure that $\mathcal{F}^s_{\Delta_\eps}$ is well-defined. However, $\mathcal{W}((\CC^*)^n, \stp_\Sigma)$ does not depend on the holomorphic structure, and the underlying symplectic forms are all the standard form on $T^*\RR^n/T^*_{2\pi \ZZ} \RR^n$. Thus, \cref{thm:embedding} allows us to compute the subcategory of conical tropical Lagrangian sections in $\mathcal{W}((\CC^*)^n, \stp_\Sigma)$ without any additional assumptions.
\end{rem}

\begin{rem} \label{rem:abouzaidembedding} From \cref{rem:relatetoAb}, we also have an embedding $\mathscr{T}\Fuk \hookrightarrow \mathcal{W}((\CC^*)^n, \stp_\Sigma)$ where $\mathscr{T}\Fuk$ is Abouzaid's category of tropical Lagrangian sections from \cite{abouzaid2009morse} either immediately in the case that $\mathcal{F}^s_{\Delta_\eps}$ is well-defined for the standard K\"{a}hler form or by slightly modifying the proof.
\end{rem}

\begin{rem} The construction of conical tropical Lagrangian sections from support functions works for any toric variety $\check X_\Sigma$, and the wrapping lemma, \ref{lemma:cofinal}, requires only that $\Sigma$ is complete. However, the Floer-theoretic arguments used to prove \cref{thm:OrigHMS} require smoothness and projectivity (cf. \cite[Remark 3.33]{hanlon2019monodromy}).
\end{rem}

We end this section with a technical question. In \cref{thm:ageneration}, we will show that the image of the embedding from \cref{thm:embedding} generates. However, we do not show that $\mathcal{F}^s_{\Delta_\eps}$ generates $\mathcal{F}_{\Delta_\eps}$, which would be implied by an embedding $\mathcal{F}_{\Delta_\eps} \hookrightarrow \mathcal{W}((\CC^*)^n, \stp_\Sigma)$. Such an embedding is natural to expect due to the highly constrained nature of monomially admissible Lagrangians and the expected duality between partially wrapped and infinitesimally wrapped categories.

\begin{qs} Is there a more general ``conicalization" procedure that produces a conical Lagrangian in $(\CC^*)^n \setminus \stp_\Sigma$ from a Lagrangian that is monomially admissible with respect to $\Delta_\eps$ that preserves Floer theory?
\end{qs}

\subsection{Action of conical twisting Hamiltonians}

As in the monomially admissible Fukaya-Seidel category, twisting Hamiltonians should act on $\mathcal{W}((\CC^*)^n, \stp_\Sigma)$ in addition to giving a construction of tropical Lagrangian sections. The only difficulty in constructing this action is that the flows of conical twisting Hamiltonians do not preserve the stop $\stp_\Sigma$, but this can be remedied as follows.

Let $D$ be a toric divisor on $\check X_\Sigma$ and $\wh{H}^\delta_D$ a corresponding conical twisting Hamiltonian. If $L$ is a conical Lagrangian in $ (\bb{C}^*)^n$ that satisfies\footnote{Recall that $\text{Arg}$ takes values in $[-\pi, \pi)$.}
\begin{equation} \label{argleft} \left| \text{Arg}(z^\alpha|_L) \right|  > \delta \end{equation}
in $U_\alpha \setminus V$ for all $\alpha \in A$ where $\{ U_\alpha \}$ is the $\eps$-star open cover for $\Sigma$, then the third condition of \cref{lemma:conesctions} implies that
\[ \phi^1_{D, \delta}(L) \cap \bb{L}_\Sigma = \emptyset \]
outside of $V$ where $\phi^1_{D, \delta}$ is the time-$1$ flow of $\wh{H}_D^\delta$. Thus, if we set
$$ \stp_\Sigma^\delta = \{ (u, \thetaa) \in \partial_\infty(\CC^*)^n : |2\pi k - \thetaa \cdot \alpha| \leq \delta \text{ for some } k \in \ZZ \text{ if } u \in U_\alpha \}, $$
the objects of $\mathcal{W}((\CC^*)^n, \stp_\Sigma^\delta)$ satisfy \cref{argleft}.\footnote{We could also take $\stp_\Sigma^\delta$ to be somewhat smaller because of condition 2 in \cref{lemma:conesctions}.} Thus, carefully applying standard techniques (for example, see \cite[Section 10b]{seidel2008fukaya} and \cite[Section 4.1]{hanlon2019monodromy}), we get a functor
\[ \mathscr{F}_D^\delta : \mathcal{W}((\CC^*)^n, \stp_\Sigma^\delta) \to \mathcal{W}((\bb{C}^*)^n, \stp_\Sigma) \]
induced by the flow $\phi^1_{D,\delta}$. Since $\stp_\Sigma \subset \stp_\Sigma^\delta$, there is also an inclusion functor 
$$ \mathcal{W}((\CC^*)^n, \stp_\Sigma^\delta) \overset{\sim}{\hookrightarrow} \mathcal{W}((\bb{C}^*)^n, \stp_\Sigma) $$
which is an equivalence for small $\delta$ because $\mathcal{W}((\bb{C}^*)^n, \stp_\Sigma)$ is finitely generated by objects (either linking disks or as we later prove, tropical Lagrangian sections) whose morphisms can be computed in the complement of $\stp_\Sigma^\delta$. As a result, we obtain an autoequivalence $\mathscr{F}_D$ of $\mathcal{W}((\bb{C}^*)^n, \stp_\Sigma)$ for each toric divisor $D$ on $\check X_\Sigma$. 

\begin{rem} The requirement that objects not be too close to the stop in order to define $\mathscr{F}^\delta_D$ above is an indication that we need to be careful with wrapping when working with Lagrangian correspondences. We encounter this issue again when working with toric morphisms in \cref{subsec:toricmorphisms,subsec:inclusionmirror} and explore it in some generality in \cref{app:correspondenceadmissibility}.
\end{rem}

When $D$ is an effective divisor, we can also define a natural transformation $T_D: \mathscr{F}_{-D} \to \id$ as in \cite[Section 4.2]{hanlon2019monodromy} following \cite[Section 10c]{seidel2008fukaya} and using the following lemma.\footnote{As we are working in a stopped Liouville domain rather than a Liouville sector, compactness can be obtained by a standard maximum principle argument in the radial coordinate as opposed to careful monotonicity estimates.}

\begin{lem} \label{lemma:poseff} If $D$ is an effective toric divisor, the flow $\phi^t_{-D}$ is everywhere non-negative at infinity.
\end{lem}
\begin{proof} This is a similar (but simpler) calculation to \cref{lemma:cofinal}. Using the same coordinates and notation, we have
$$ X_{\wh{H}^\delta_{-D}} = (0, \nabla \wh{H}_{-D}^\delta).$$
Thus, using \cref{eq:conicalgradient} and the fact that $F_{-D}$ is piecewise-linear, we have
\begin{align*} \lambda( \partial_t \phi^t_D(u, p)) &= \lambda\left(X_{\wh{H}^\delta_{-D}}\right) \\
&= 2\pi \int_{\RR^n} F_{-D}(u - Rx) \eta_{\eps_{-D}(t)}(x) \, dx. 
\end{align*}
However, $F_{-D}$ is non-negative because $D$ is effective.
\end{proof}

The mirrors of the corresponding functors and natural transformations were computed on $\mathcal{F}^s_{\Delta_\eps}$ in \cite[Theorems 4.5 and 4.9]{hanlon2019monodromy}. Combining those calculations with our \cref{thm:ageneration}, which appears later but does not rely on this section, we obtain the following.

\begin{thm} \label{thm:wrappingaction} Under the quasi-equivalence of \cref{cor:HMScor}, the functor $\mathcal{F}_D$ corresponds to 

$$ \_ \otimes \mathcal{O}(D): D^b_{\dg}\Coh(\check X_\Sigma) \to D^b_{\dg}\Coh(\check X_\Sigma) $$

and the natural transformation $T_D$ to

$$ \_ \otimes s_D : \_ \otimes \mathcal{O}(-D) \to \id $$

where $s_D$ is a defining section of $D$.
\end{thm} 
   \subsection{Cocores}
   \label{subsec:cocores}
   The cotangent fibers of $T^*T^n$ form another natural set of Lagrangian sections in $\mathcal{W}( (\CC^*)^n, \stp_\Sigma)$.
In fact, these are the cocores of the $\LL_0$ components of $\LL_\Sigma$.
Note that while $\LL_0$ is a smooth $T^n$, the intersections of the FLTZ components with $\LL_0$ subdivide it into a collection of smaller strata. 
More generally, define the singular locus of the stratum to be $\LL_{\Sigma,\bot}:=\bigcup_{\substack{\sigma_1, \sigma_2\in \Sigma\\ \sigma_1\neq \sigma_2}} \LL_{\sigma_1}\cap \LL_{\sigma_2}$.
We define $\Conn(\Sigma, \sigma):=\{\II_{\sigma;i}\}$ to be the set of connected components of $\LL_\sigma \setminus \LL_{\Sigma, \bot}$. 
We call the points $x\in \II_{\sigma; i}$ the internal points of the FLTZ skeleton. 
Associated to each internal point, we have a linking disk $\cocore(x)$. If $x, x'$ are in the same connected component, $\II_{\sigma; i}$, then the linking disks $\cocore(x)$ and $\cocore(x')$ are isotopic. 
For our purposes, it will be important to know that these objects in $\mathcal{W}((\CC^*)^n, \stp_\Sigma)$ are all represented by conical tropical Lagrangian sections when $\sigma = 0$.

\begin{lem} \label{lemma:cocoresaretropical} If $x \in \II_{0;i}$ for some $i$, then the cocore $\cocore(x)$ is admissibly Hamiltonian isotopic to a conical tropical Lagrangian section.
\end{lem}
\begin{proof} We lift to $T^* \RR^n$ with coordinates $(u,p)$ and write $x= (0, p_o)$ for some $p_o$ such that $\alpha \cdot p_o \not \in 2\pi \ZZ$ for all $\alpha \in A$. Then, for each $\alpha \in A$, there exists $m_\alpha \in \ZZ$ such that
$$ 2\pi(m_\alpha - 1) < \alpha \cdot p_o < 2\pi m_\alpha .$$
Let $D = \sum_{\alpha \in A} - m_\alpha D_\alpha$. By \cref{eq:anglebound}, we have
$$ 2\pi m_\alpha - 2 \delta < \nabla H_D^\delta(u) < 2\pi m_\alpha $$
for all $u \in U_\alpha$ near infinity. Thus,
$$ 2\pi(m_\alpha - 1) \leq 2\pi(m_\alpha -1) + t(1-2\delta) < (1-t)(p_o \cdot \alpha) + t\nabla \wh{H}_D^\delta(u) < 2\pi m_\alpha $$
for all $u \in U_\alpha$ near infinity, $\delta$ small, and $t \in [0,1]$ so that
$$ L_t = \left(u, tp_o + (1-t)\nabla \wh{H}_D^\delta(u) \right)$$
gives an admissible Hamiltonian isotopy from $\cocore(x)$ to the tropical Lagrangian section $L^\delta(D)$ after passing to the quotient $T^*\RR^n/T_{2\pi \ZZ}^* \RR^n$. 
\end{proof}

\begin{rem}
    One may ask about the Lagrangian ``zero'' section $L^\delta(0)$. 
    This Lagrangian is not a cocore or linking disk, as it is not contained in the interior of any of the $\II_{\sigma; k}$ strata. 
    Rather, it is something like a linking disk to all FLTZ strata simultaneously. 
    This distinction (that it is not actually a linking disk) will become important in the later discussion of generation and localization (\cref{sec:applications}). 
\end{rem}

\begin{ex} To illustrate \cref{lemma:cocoresaretropical}, we will characterize these cocores in the example of mirror symmetry for $\CP^n$.
Let $\Sigma_n$ be the fan for $\CP^n$ with generators $e_1, \hdots, e_n, -e_1 - \hdots - e_n$ where $e_i$ are standard basis vectors. Let $\LL_0=T^n$ be the core of the FLTZ skeleton.
The decomposition of $\LL_0\subset \LL_{\Sigma_n}$ into Lagrangian stratum 
    \[\LL_0\setminus \bigcup_{\alpha\in A} \LL_\alpha\]
    has $n$ connected components.
    Furthermore, the components can be indexed $\II_{0; 1}, \ldots \II_{0;n}$ so that the cocore $\cocore(x_i)$ at $x_i\in \II_{0;i}$ is Hamiltonian isotopic to a tropical section mirror to $\mathcal{O}(-n-1 +i)$.
    
    To see this, observe first that
    \[\LL_0\setminus \bigcup_{\alpha\in A} \LL_\alpha= (0, 2\pi )^n\setminus( H_{\langle 1, \ldots, 1\rangle}+2\pi \ZZ^n)\]
    where $H_{\langle 1, \ldots, 1 \rangle}$ is the plane perpendicular to the $\langle 1, \ldots , 1 \rangle$ direction. 
    The $2\pi \ZZ^n$ translates of this plane intersect the cube $n-1$ times. 
    The $n$ components $\II_{0;i}$ of $ (0, 2\pi )^n\setminus ( H_{\langle 1, \ldots, 1\rangle}+2\pi \ZZ^n)$ can be indexed in the order which they appear in the $e_1+\cdots+ e_n$ direction. As a result, the cocore for the stratum $\II_{0;i}$ can be taken at the point
    $$2\pi \left(\frac{i-1/2}{n}\right)(e_1+\cdots+ e_n) \in T^n. $$
    In the proof of \cref{lemma:cocoresaretropical}, we can then take $m_{e_j} = 1$ and $m_{-e_1 - \hdots - e_n} = 1 - i$ so that $D$ is in the class of $-n-1 +i$ times a hyperplane.
    
    We draw examples of these sections in \cref{fig:cp2section}. For simplicity, we use the notation
    $$ L^\delta(-1) = L^\delta(-D_{e_1} - D_{e_2} + D_{-e_1-e_2}) $$
    and
    $$ L^\delta(-2) = L^\delta(-D_{e_1} - D_{e_2}) $$
    for the two tropical Lagrangian sections Hamiltonian isotopic to $\cocore(x_1)$ and $\cocore(x_2)$, respectively.
    The ``zero section'' $L^\delta(0)$ has been drawn in \cref{fig:cp2section0}. This corresponds to taking the cotangent fiber above the origin in $T^2$ and pushing off the stop slightly in the negative direction. In \cref{fig:cp2section1}, notice that wrapping $L^\delta(-1)$ to $L^\delta(0)$ involves passing through several stop components at the origin.
    Finally, in \cref{fig:cp2section2} we see that there is a positive wrapping from $L^\delta(-2)$ to $L^\delta(-1)$ which passes through a single FLTZ skeleton component.
    It follows that the mapping cone of $L^\delta(-2)\to L^\delta(-1)$ is a linking disk to the FLTZ skeleton.
    \label{exam:cpn}
\end{ex}
The intuition from this example is the basis for the argument generating linking disks in \cref{sec:generation}.
\begin{figure}
    \begin{subfigure}[t]{.3\linewidth}
        \centering
        \includegraphics{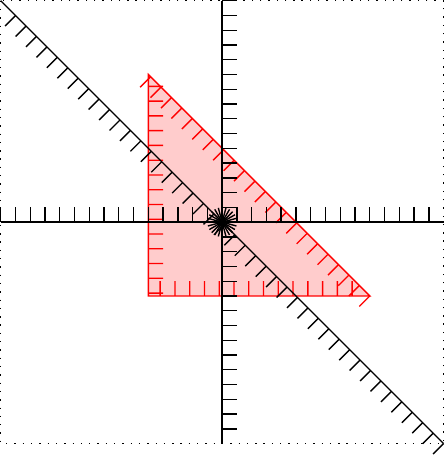}
        \caption{A piecewise linear approximation of the argument projection of the Lagrangian ``zero section'' $L^\delta(0)$.}
        \label{fig:cp2section0}
    \end{subfigure}
    \begin{subfigure}[t]{.3\linewidth}
        \centering
        \includegraphics{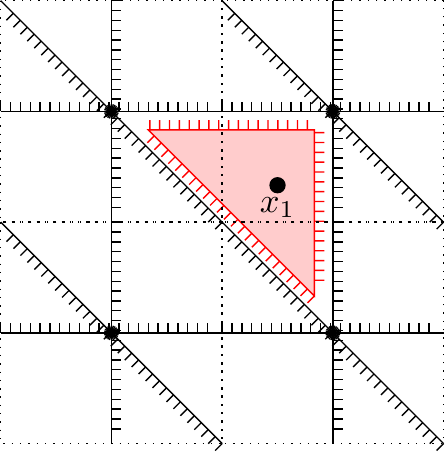}
        \caption{A piecewise linear approximation of the argument projection of $L^\delta(-1)$ to a cover of $T^2$.  }
        \label{fig:cp2section1}
    \end{subfigure}
    \begin{subfigure}[t]{.3\linewidth}
        \centering
        \includegraphics{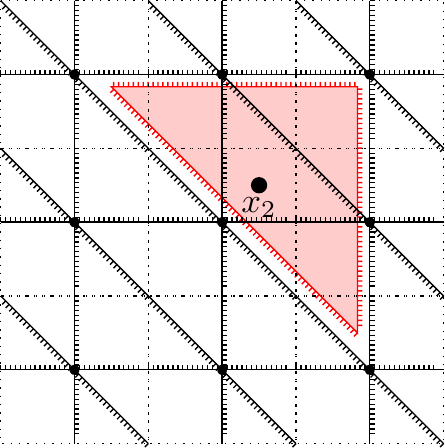}
        \caption{A piecewise linear approximation of the argument projection of $L^\delta(-2)$ to a cover of $T^2$. }
        \label{fig:cp2section2}
    \end{subfigure}
    \caption{}
    \label{fig:cp2section}
\end{figure}

We end this section with a characterization of tropical Lagrangian sections Hamiltonian isotopic to cocores.
\begin{lem}
    Let $D=\sum_{\alpha\in A} m_\alpha D_\alpha$ be a toric divisor.
    Then $L^\delta(D)$ is Hamiltonian isotopic to a cocore if and only if the set
    \[
        Q(D):=\bigcap_{\alpha\in A}\{p\;:\;2\pi(m_\alpha - 1) < \alpha \cdot p < 2\pi m_\alpha\}
        \]
    is nonempty for some $p\in M\tensor \RR$. 
    \label{lem:tropicalarecocore}
\end{lem}
\begin{proof}
    Suppose that the intersection is nonempty.
    Then pick some $p_0$ in the intersection.
    By \cref{lemma:cocoresaretropical}, we obtain a Hamiltonian isotopy between $L^\delta(D)$ and $\cocore((0,p_0))$.

    Now suppose that $L^\delta(D)$ is admissibly Hamiltonian isotopic to a $\cocore((0,p_0))$. There exists $D'=\sum_{\alpha\in A} m_\alpha D_\alpha$ so that $p_0\in Q(D')$, and $L^\delta(D)$ is admissibly Hamiltonian isotopic to $\cocore((0, p_0))$.
    By homological mirror symmetry (\cref{cor:hmsline}), $D$ and $D'$ are linearly equivalent.
    The sets $Q(D)$ and $Q(D')$ are then related by a linear transformation; in particular, $Q(D)$ is non-empty.
\end{proof} \section{Mirror to a toric divisor}
   \label{sec:mirrordivisor}
   
We now fix an $n$-dimensional toric variety $\check X_\Sigma$ and its mirror space $(X=(\CC^*)^n, \stp_\Sigma)$.
We construct a Lagrangian submanifold $L_{\alpha>0}\subset (X, \stp_\Sigma)$ which is mirror to the structure sheaf of a toric divisor $\mathcal O_{\alpha >0}$.
This is a warm-up for the next section, where we will construct a Lagrangian correspondence $\Lc_{\alpha 0}: (X_2, \stp_{\afan})\Rightarrow(X_1, \stp_\Sigma)$ mirror to the inclusion $i_{\alpha0}: \check X_{\str(\alpha)/\alpha}\to \check X_\Sigma$. 
The later construction will recover $L_{\alpha>0}$ as a special case by taking pushforward of the mirror to the structure sheaf,
\[L_{\alpha>0}=L_{\alpha 0}\circ L_\alpha(0).\] 
\subsection{The mirror Lagrangian}
\label{sub:inclusioncorrespondence}
The Lagrangian submanifold $L_{\alpha>0}$ is constructed in a similar fashion to the tropical Lagrangian submanifolds considered in \cite{hicks2020tropical}. 
Given a tropical polynomial\footnote{Tropical polynomials are piecewise affine convex, as opposed to support functions which are piecewise linear.} $\sF: Q\to \RR$, the tropical Lagrangian submanifold $L_{\sF}$ was constructed by taking two monomially admissible sections of an SYZ fibration and surgering them along their common overlap.
Since the overlap approximates the complement of the tropical locus of $\sF$,  the image $\val(L_{\sF})$ is contained in a small neighborhood of the tropical locus of $\sF$. 
The constructed tropical Lagrangian submanifold is monomially admissible and mirror to a sheaf supported on a complex hypersurface with matching tropicalization.
Our current goal is to construct the mirror to a toric divisor by the same technique.

\begin{thm}[Mirror to a toric divisor]
   Let $\Sigma$ be a smooth fan, and let
   $\alpha \in A$ generate a 1-dimensional cone.
   \begin{itemize}
   \item (\Cref{def:divisormirror}) There exists an admissible  Lagrangian disk  $L_{\alpha>0}(0)\subset (X, \stp_\Sigma)$, along with a Lagrangian cobordism:
   \[K_\alpha:(L_0(\sF_\alpha), L_0(0))\rightsquigarrow L_{\alpha>0}(0).\]
   \item (\Cref{prop:mirrordivisormatching}) Under the mirror quasi-equivalence from \cref{cor:linebundlehms}, $L_{\alpha>0}(0)$ is mirror to a toric divisor $\mathcal O_{D}(0)\subset \check X_{\Sigma}$.
   \item (\Cref{claim:valuationofdivisormirror}) In the complement of a compact set $X^{int}\subset X$
   \[(\val(L_{\alpha>0}(0)\setminus X^{int}))\subset \str(\alpha).\]
   \end{itemize}
   \label{thm:divisormirror}
\end{thm}

There are some differences between our construction of the mirror to a toric divisor and the construction of tropical Lagrangian submanifolds from \cite{hicks2020tropical}. 
A toric divisor does not tropicalize, and thus the valuation of the Lagrangian we construct will not approximate the tropicalization of $D$; instead, the Lagrangian we build is Hamiltonian displaceable from all SYZ fibers.
The Lagrangian sections considered for tropical Lagrangian submanifolds came from tropical polynomials $\sF:Q\to \RR$.
These are convex, and thus represent sections of base-point free divisors. 
As toric divisors are not necessarily ample or base-point free, the support function $\sF_\alpha$ need not be convex.
Finally, we will work in the conically admissible setting rather than the monomially admissible one; as the definition of a conically admissible Lagrangian section is a bit messier than the monomially admissible section, we will have to do a bit of additional book-keeping to obtain surgery data similar to that used in \cite{hicks2020tropical}. 

Let $\alpha\in A$ be a cone, to which we can associate the divisor $-D_\alpha$, the support function ${\sF}_\alpha$, and the  Lagrangian section $L^\delta(\sF_\alpha)= L^\delta(-D_\alpha) \subset (X, \stp_\Sigma)$.
We will always use the same $\delta$ between all Lagrangians, and therefore drop the superscript $\delta$.
We would like to apply Lagrangian surgery to the sections $L(\sF_\alpha)$ and $L(0)$. 
If the function $\sF_\alpha$ is convex, it immediately follows that the intersection of $L_0(\sF_\alpha)$ with $L_0(0)$ satisfies the conditions of \cref{def:surgerydata}.
Generally, this need not be the case; consider the example of the toric blowup of $\CP^2$ at a point, whose fan has generators $A=\{e_1, e_2, e_1+e_2, -e_1-e_2\}$.
Let $\alpha=e_1+e_2$ so that 
 the divisor $D_\alpha$ is the exceptional divisor of the blow-up of $\CP^2$, which is a (-1)-curve. The support function $\sF_\alpha$, whose contour sets are drawn in  \cref{fig:conicalsection}, fails to be convex.
However, $\sF_\alpha$ satisfies a property which can be used as a substitute for convexity, which is that when restricted along an $\alpha$-parallel ray going through a fixed point $q$, 
\[\sF_\alpha(q+t\cdot \alpha)= \max(0, (t-t_0))\] 
for some $t_0$.
Having convex slices along the $\alpha$-parallel rays also implies the existence of a surgery data  (\cref{def:surgerydata}).

The Lagrangian section $L(\sF_\alpha)$ is obtained by smoothing ${\sF}_\alpha$ with a variable radius smoothing kernel. Unfortunately, the slices of the resulting smoothing,  $\Halpha(q+t\cdot \alpha)$, will generally not be convex.
See, for example, \cref{fig:nonconvexity}  which plots the slice $\frac{d}{dt}\Halpha(q+t\cdot \alpha)$ where $\alpha$ is the exceptional divisor of the one point blow-up above.
Note that for small values of $t$ the slice $\frac{d}{dt}\Halpha(q+t\cdot \alpha)$  approximates a smoothing of the step function --- which would be increasing.  
However, at large $t$ the function becomes concave as the radius of smoothing kernel for $\Halpha$ increases to the point where it intersects the non-convexity locus of ${\sF}_\alpha$. 

Ultimately, this is not too problematic as the non-convexity of $\Halpha(q+t\cdot \alpha)$ occurs away from where $L(\sF_\alpha)$ and $L(0)$ intersect.
Without loss of generality, we can assume that in the description of the slice, we choose the point $q$ so that $\sF_\alpha(q)=0$ and $q$ is on the locus of non-linearity; equivalently, $q$ belongs to the boundary of $\str(\alpha)$. 
With this choice of $q$,  \[\sF_\alpha(q+t\cdot \alpha)=\max(0, t).\] 
\begin{figure}\centering
   \begin{subfigure}{.45\linewidth}
      \centering
      \includegraphics[scale=.45]{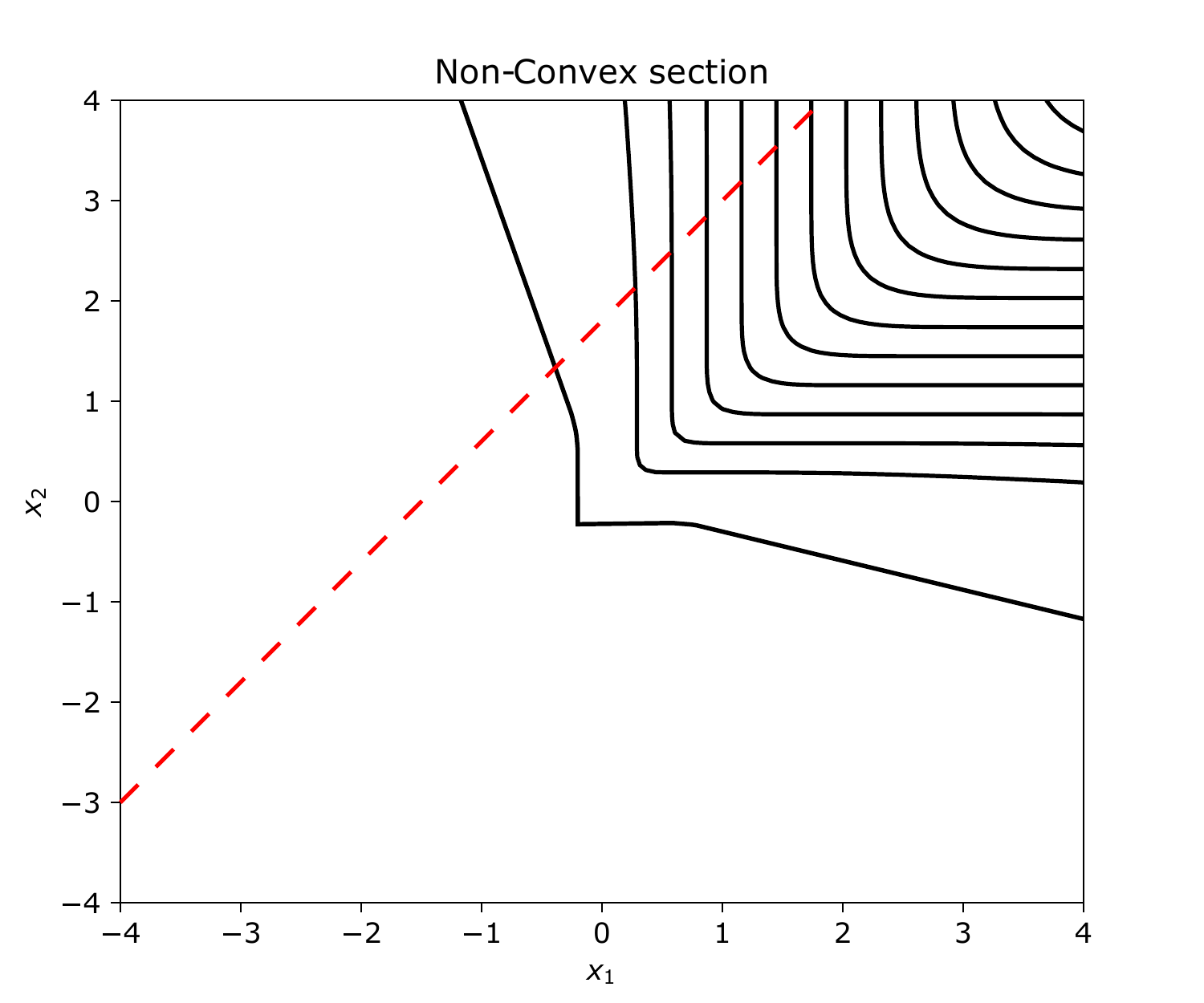}
      \caption{The smoothing of a non-convex ${\sF}_\alpha$. Here, $D_\alpha$ is the exceptional divisor of the blow-up of $\CP^2$ at a point. Note that there is a neighborhood of the ray $\alpha$ where this is non-convex. The dotted line parameterizes $q+t\cdot \alpha$, along which we slice function. }
   \end{subfigure}\;\;\;\;
   \begin{subfigure}{.45\linewidth}
      \centering
      \includegraphics[scale=.45]{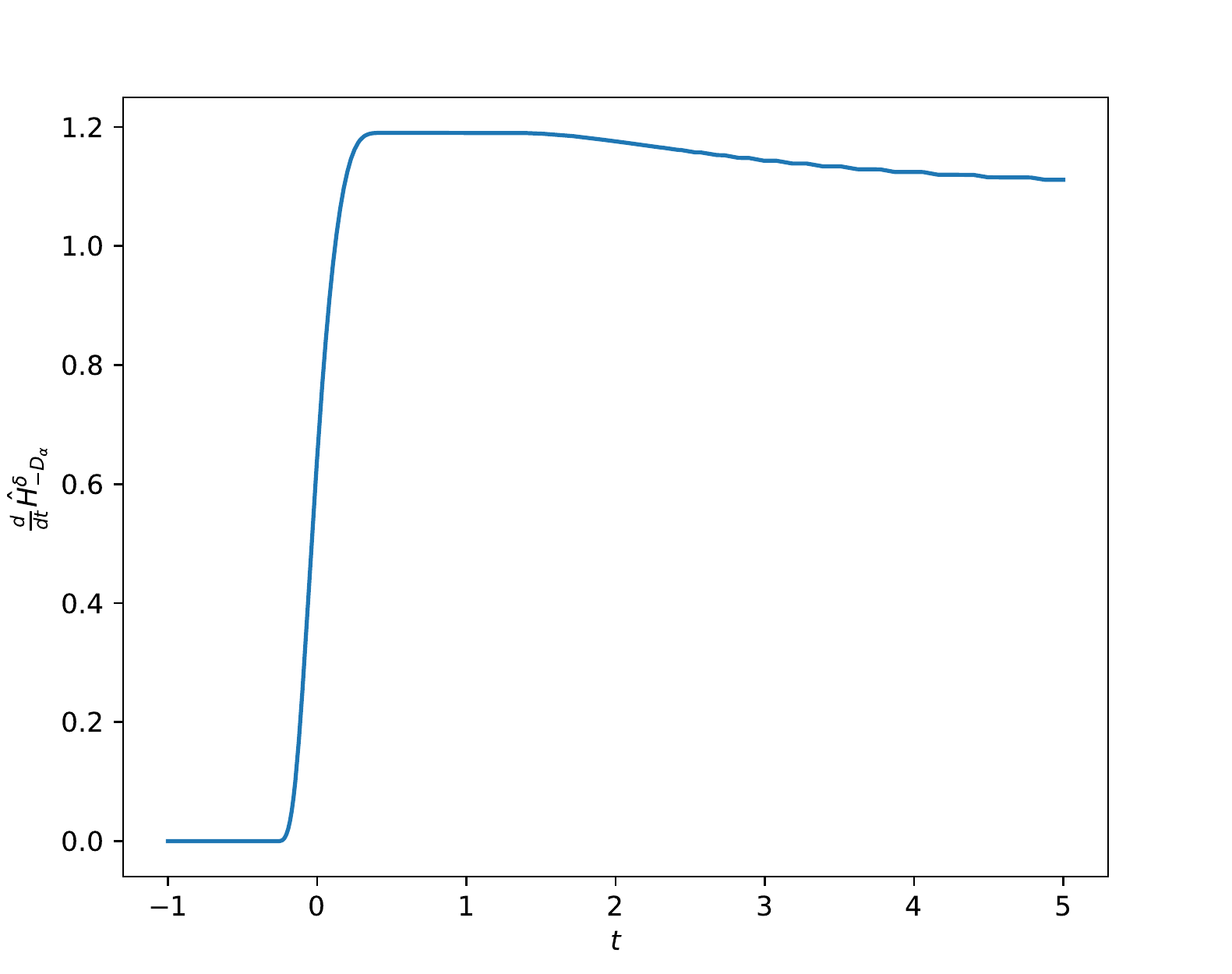}
      \caption{The graph of $\frac{d}{dt} \Halpha(q+t\cdot \alpha)$, where $\alpha=e_1+e_2$. The dotted portion of the domain represents an interval where we expect $\Halpha(q+t\cdot \alpha)$ to be convex. There is only a slight amount of non-convexity, which occurs when $q+t\cdot \alpha$ enters the $\eps |r|$ cone of $\alpha$.  }
   \end{subfigure}
   \caption{How we work with tropical sections when $\sF_\alpha$ is non-convex.}
   \label{fig:nonconvexity}
\end{figure}
\begin{lem}
   Slices of $\Halpha$ in the $\alpha$ direction are increasing, so that $\frac{d}{dt}\Halpha(q+t\cdot \alpha)\geq 0$.
   Furthermore, the set on which $d\Halpha=0$ is exactly the locus where $\Halpha=0$. 
   \label{lemma:positiveintersectionlemma}
   \label{claim:intersectioncharacterization}
\end{lem}Before we prove this lemma, we would like to state that if we additionally assume that ${\sF}_\alpha$ is convex, then $\Halpha(q+t\cdot \alpha)$ is convex. The proof in this setting is also dramatically simpler.
\begin{proof}
    We first show that $\frac{d}{dt}\Halpha(q+t\cdot \alpha)\geq 0$. Recall that $\eta_{t}(q)= \frac{1}{t^n} \eta\left(\frac{q}{t}\right)$.
   Take  $\rho: \RR_{\geq 0}\to \RR$ a smooth function which convexly interpolates between the following two domains of definition:
   \begin{align*}
      \rho(t)=1 \text{ if $t<1$} && \rho(t)=\eps t \text{ if $t\gg 1$}
   \end{align*}
   The function $\Halpha$ is defined by smoothing the function $\sF_\alpha$  with a variable radius smoothing kernel $\eta_{\rho(|q|)}$.
   Let $A(t)=\rho(|q+t\cdot \alpha|),$ which has the property that $A'(t)< |\alpha|\eps$.
   Mimicking the argument of \cref{lemma:cofinal}:
\begin{align*}
   \frac{1}{2\pi}&\frac{d}{dt}\Halpha(q+t \cdot \alpha)=\frac{d}{dt}\int_{\RR^n} \sF_\alpha(q+t \cdot \alpha-x)\cdot\eta_A(x)\, dx\\
   =& \frac{d}{dt}\int_{\RR^n} {\sF_\alpha}(q+t\alpha -A(t)x)\eta(x)\, dx\\
   =& \int_{\RR^n}  \del F_\alpha(q+t\alpha - A(t)x)\cdot (\alpha - A'(t)x)\eta(x) \, dx\\
   =& \int_{\RR^n} \left(\left(\chi_{\str(\alpha)}(q+t\alpha - A(t) x)  \right)- A'(t) \left( \del \sF_\alpha (q+t\alpha - A(t)x)\cdot x\right)\right) \cdot \eta(x) \, dx
   \intertext{ Since  $-\del \sF_\alpha (q+t\alpha - A(t)x)$ is supported on $\str(\alpha)$}
   =& \int_{\RR^n} \left(\left(\chi_{\str(\alpha)}(q+t\alpha - A(t) x)  \right) \cdot \left(1- A'(t) \left( \del \sF_\alpha (q+t\alpha - A(t)x)\cdot x\right)\right)\right) \cdot \eta(x) \, dx
\end{align*}
Here  $\chi_{\str(\alpha)}$ is the characteristic function for the star of $\alpha$.
It suffices to show that $A'(t)\del \sF_\alpha (q+t\alpha - A(t)x)$ is sufficiently small.\footnote{We use this footnote the point out that in the monomial admissible setting,  the radius of the smoothing kernel has no dependence on $q$ and the last term vanishes.}
Since $\eta(x)$ has bounded support, we are only integrating over $x\in B_1(0)$. 
Therefore, the absolute value of $\del F_\alpha(q+t\alpha - A(t)x)\cdot x\eta(x) $ is bounded by the maximal norm of $\del F_\alpha$. 
It follows that 
\begin{align*}
   \left|A'(t) \int_{\RR^n} \del \sF_\alpha (q+t\alpha - A(t)x)\cdot x \eta(x) \, dx\right|\leq & |A'(t) |\int_{B_1(0)} \left|\del \sF_\alpha (q+t\alpha - A(t)x)\cdot x \right| \, dx\\
   \leq &|A'(t)|\cdot \Vol(B_1(0)) \cdot  \max_{q\in \RR^n}  \Vert\del \sF_\alpha(q) \Vert.
\end{align*}
Since $A'(t)< \eps \Vert \alpha\Vert$, and $\eps$ can be chosen as small as desired, this quantity is as small as desired. We therefore conclude that $\frac{d}{dt}\Halpha(q+t \cdot \alpha)\geq 0$. 
Because the support of $\frac{d}{dt}\Halpha(q+t \cdot \alpha)$ matches the support of $\Halpha(q+t \cdot \alpha)$, we can also conclude from this computation that $d\Halpha=0$ if and only if $\Halpha =0$.
\end{proof}
It will also be useful to think of this in the following way:
\begin{cor}
   For $C>0$, we can parameterize the level set $\Halpha=C$ by $\funnyi:Q_{\str(\alpha)/\alpha}\to Q_{\str(\alpha)}$, which is a smooth section\footnote{Regarding the unusual choice of notation: $\funnyi:Q_{\str(\alpha)/\alpha}\to Q_{\str(\alpha)}$ is not a piecewise linear function, but it is a smooth approximation of the PL parameterization of $\partial(\str(\alpha))+C\alpha$.} of the projection $\underline \pi_{\alpha}: Q_{\str(\alpha)}\to Q_{\str(\alpha)/\alpha}$. 
   \label{cor:funnysection}
\end{cor} 
The locus where $d\Halpha=0$ is highlighted in \cref{fig:conicalsection}.
By \cref{lemma:positiveintersectionlemma}, there exists a surgery data (\cref{def:surgerydata}) for the Lagrangians $L_0(\sF_\alpha)$ and $L_0(0)$, with  surgery region $U_{\Halpha{}}$ defined to be the locus where $\Halpha{} \leq \rho c_\eps$.
Here, $\rho: X\to \RR$ is a function which increases in the symplectization coordinate and $c_\eps$ is a constant, both defined in the proof of \cref{prop:generalizedsurgeryprofile}.
\Cref{prop:generalizedsurgeryprofile} constructs for us an admissible surgery of these two sections.
\begin{df}
   The mirror to the divisor $D_\alpha$ is the  Lagrangian submanifold
   \[L_{\alpha>0}(0):=L_0(\sF_\alpha)\#_{U_{\Halpha{}}} L_0(0).\]
   which comes with an exact Lagrangian surgery cobordism 
   \[K_\alpha:(L_0(\sF_\alpha), L_0(0))\rightsquigarrow L_{\alpha>0}(0).\]
   \label{def:divisormirror}
\end{df}
We will now justify the name in \cref{prop:mirrordivisormatching}.
The surgery of $L_0(\sF_\alpha)$ and $L_0(0)$ is homotopic to gluing them along their boundary.
Since these are both disks, the Lagrangian surgery cobordism $K_\alpha$ has the topology of a disk.
It follows that both $K_\alpha$ and $L_{\alpha>0}$ are exact Lagrangian submanifolds.
In addition, \cref{prop:generalizedsurgeryprofile} guarantees admissibility of these Lagrangian submanifolds.
\begin{prop}
   The Lagrangian $L_{\alpha>0}(0)$ is mirror to $\mathcal O_{\alpha>0}(0)$. 
   \label{prop:mirrordivisormatching}
\end{prop}
\begin{proof}
By \cref{thm:cobordismgeneration}, there exists an exact triangle in the Fukaya category,
\[L_0(\sF_\alpha)\xrightarrow{a} L_0(0)\xrightarrow{b} L_{\alpha>0}(0).\]
Under \cref{cor:linebundlehms}, this is identified with an exact sequence:
\[\mathcal O_{\check X_\Sigma}(\sF_\alpha)\to \mathcal O_{\check X_\Sigma}\to \mathcal O_D(0)\]
where $\mathcal O_{\check X_\Sigma}(-D)=\mathcal O_{ \sF_\alpha}$.
Before we identify $a$, we note that the exact triangle determines $D$ up to linear equivalence. 
Furthermore, since $L_{\alpha>0}$ is disjoinable from every SYZ fiber, it must be mirror to a toric divisor. 

The morphism $a$ can be identified up to a scalar as a certain continuation element.
The Floer complex $\HF (L_0(\sF_\alpha), L_0(\sF_\alpha))$ has a single generator in degree zero, which we call $e$.
The continuation element $a$ is the image of $e$ under the continuation map  \[\Phi:\HF(L_0(\sF_\alpha), L_0(\sF_\alpha))\to \HF(L_0(\sF_\alpha), L_0(0)).\]
Note that we only identify this intersection between $L_0(\sF_\alpha)$ and $L_0(0)$ by choosing a support function $\sF_\alpha$ which determines the continuation Hamiltonian. 
The support function also identifies $a$ as the intersection corresponding to the minimum (after possibly simplifying $H^\delta_K$ to have a unique minimum) of $\wh H^{\delta}_\alpha$.
Since we are surgering at a single intersection region, we can use \cite[Proposition 4.27]{hicks2020tropical} to identify this surgery with the  mapping cone at this intersection as in \cite{fukaya2007lagrangian}. 

It remains to identify the morphism $a$ under \cref{cor:HMScor}. However, $a$ is simply the leading order term of the natural transformation in \cref{thm:wrappingaction} and thus is a defining section of $D_\alpha$. 
\end{proof}

\begin{prop}
   Outside of a compact set $X^{int}$,   $\val(L_{\alpha>0}\setminus X^{int})\subset \str(\alpha)$.
   \label{claim:valuationofdivisormirror}
\end{prop}
\begin{proof}
   $L_{\alpha>0}$ is constructed by surgering away the overlap of $L_0(\sF_\alpha)$ and $L_0(0)$.
   These are both SYZ sections, and the surgery region is where $\Halpha\leq \rho c_\eps$.
   The remaining region, where $\Halpha\geq \rho c_\eps$, lies in a small conical neighborhood  of where $\sF_\alpha>\rho c_\eps$, which is a subset of $\str(\alpha)$ (see \cref{fig:conicalrchart}).
\end{proof}
This last property --- that the valuation projection of this Lagrangian is contained within the star of $\alpha$ ---  should be compared to \cite[Proposition 3.1]{hicks2020tropical}, which shows that the valuation projection of a tropical Lagrangian submanifold lives in a small neighborhood of the defining tropical subvariety. 
One should interpret $\str(\alpha)\subset Q$ as a best attempt at tropicalizing the divisor $D_\alpha\subset X_\Sigma$.

\begin{figure}
   \centering
   \begin{subfigure}{.4\linewidth}
      \centering
      \includegraphics[scale=.3]{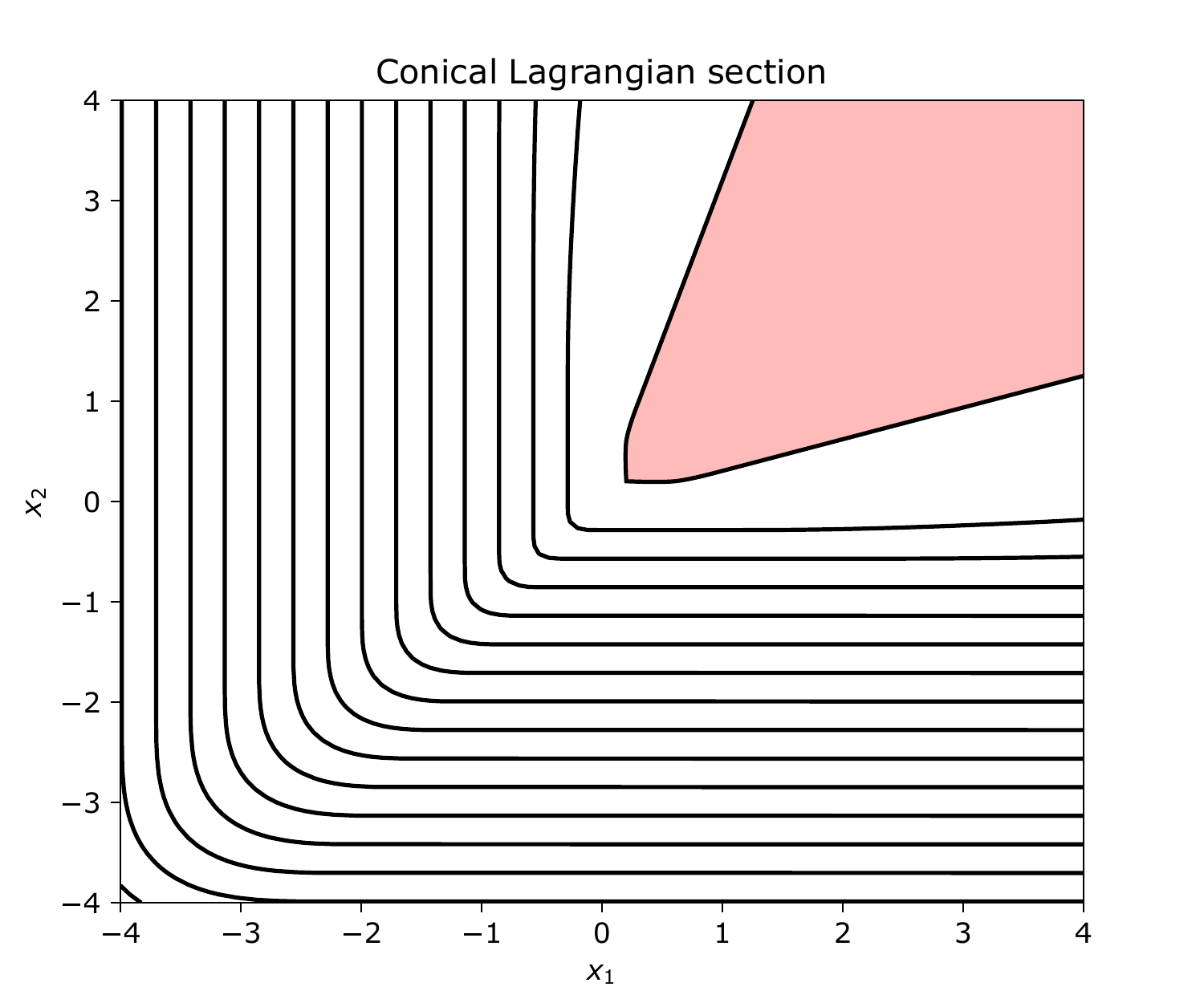}
      \caption{The conical Lagrangian section, associated to $\sF_{-e_1-e_2}$ for $\CP^2$ with region of intersection with the zero section highlighted in red. }
      \label{fig:conicalsection}
   \end{subfigure}\;\;\;\;\;\;
   \begin{subfigure}{.4\linewidth}
      \centering
      \includegraphics[scale=.3]{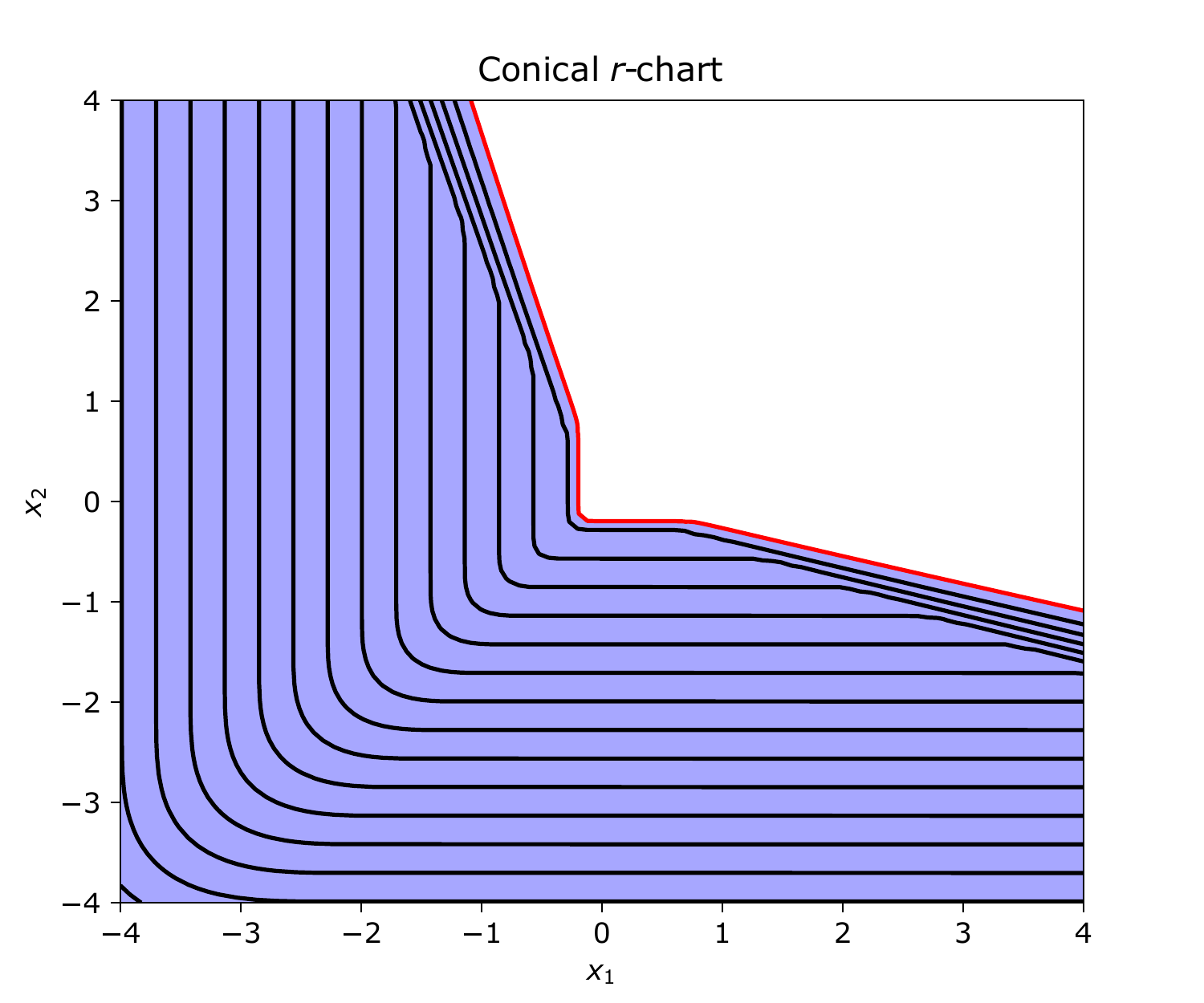}
      \caption{The valuation post-surgery. Note that $\val(L_{\alpha>0})$ is contained within $\str(\alpha)$ outside of a compact set. The image of $\funnyi: Q_1\to Q_2$ is the red line.}
      \label{fig:conicalrchart}
   \end{subfigure}
   \caption{Valuation of overlap region and resulting surgery.}
\end{figure}

\subsection[Example: Toric divisors on proj. line]{Example: Toric divisors of $\CP^1$}
\label{subsec:p1}
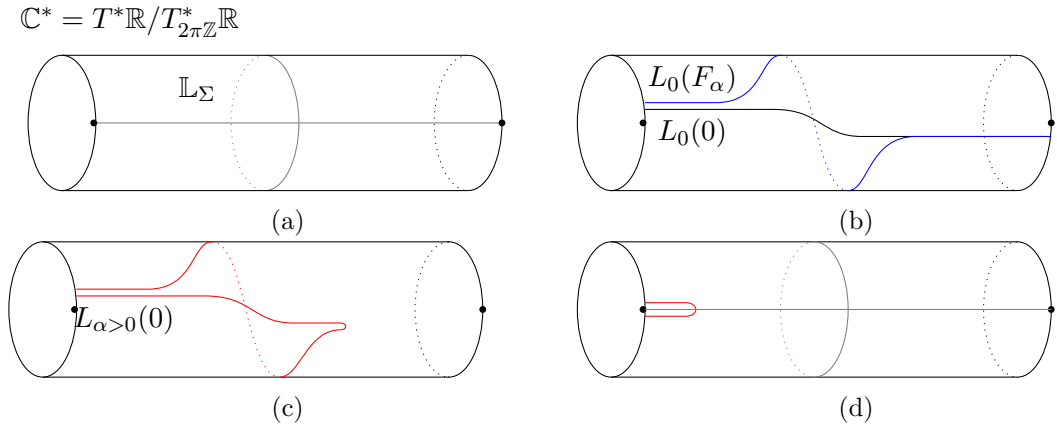
\begin{figure}
    \centering
    \begin{subfigure}[t]{.45\linewidth}
        \begin{tikzpicture}[scale=.9]
      
            \node at (0.5,0) {$\mathbb C^*=T^*\mathbb R/T^*_{2\pi \ZZ}\mathbb R$};

               \draw  (-0.5,-1.5) ellipse (0.5 and 1);
   
               \begin{scope}[shift={(4,-3.5)}]
                  \begin{scope}[]
   
                     \clip  (1.5,3) rectangle (2.5,1);
                     \draw  (1.5,2) ellipse (0.5 and 1);
   
                  \end{scope}
                  \begin{scope}[]
   
                     \clip  (1.5,1) rectangle (0.5,3);
                     \draw[dotted]  (1.5,2) ellipse (0.5 and 1);
                  \end{scope}
               \end{scope}
               \begin{scope}[shift={(1,-3.5)},gray]
                  \begin{scope}[]
   
                     \clip  (1.5,3) rectangle (2.5,1);
                     \draw  (1.5,2) ellipse (0.5 and 1);
   
                  \end{scope}
                  \begin{scope}[]
   
                     \clip  (1.5,1) rectangle (0.5,3);
                     \draw[dotted]  (1.5,2) ellipse (0.5 and 1);
                  \end{scope}
               \end{scope}
   
\draw[gray] (0,-1.5) -- (6,-1.5);
               \draw (-0.5,-2.5) -- (5.5,-2.5);
               \draw (-0.5,-0.5) -- (5.5,-0.5);

         \node[fill=black, circle, scale=.25] at (-0.0296,-1.5) {};
         \node[fill=black, circle, scale=.25] at (6,-1.5) {};

\node at (1.5,-1) {$\mathbb L_\Sigma$};
\end{tikzpicture}         \caption{ }
        \label{fig:p1skeleton}
    \end{subfigure}
    \begin{subfigure}[t]{.45\linewidth}
        \begin{tikzpicture}[scale=.9]
      
            \node[circle, fill=white] at (0.5,0) {};

               \draw  (-0.5,-1.5) ellipse (0.5 and 1);
   
               \begin{scope}[shift={(4,-3.5)}]
                  \begin{scope}[]
   
                     \clip  (1.5,3) rectangle (2.5,1);
                     \draw  (1.5,2) ellipse (0.5 and 1);
   
                  \end{scope}
                  \begin{scope}[]
   
                     \clip  (1.5,1) rectangle (0.5,3);
                     \draw[dotted]  (1.5,2) ellipse (0.5 and 1);
                  \end{scope}
               \end{scope}
   
               \draw (-0.5,-2.5) -- (5.5,-2.5);
               \draw (-0.5,-0.5) -- (5.5,-0.5);

         \node[fill=black, circle, scale=.25] at (-0.0296,-1.5) {};
         \node[fill=black, circle, scale=.25] at (6,-1.5) {};

\draw (0,-1.3) .. controls (1,-1.3) and (1.7,-1.3) .. (1.9,-1.3) .. controls (2.6,-1.3) and (2.6,-1.7) .. (3.2,-1.7) .. controls (3.6,-1.7) and (3.8,-1.7) .. (4,-1.7) .. controls (4.35,-1.7) and (5.55,-1.7) .. (5.9,-1.7);

               \begin{scope}[blue]
                  \draw (1.05,-1.2) .. controls (1.75,-1.2) and (1.7,-0.5) .. (2,-0.5);
                  \draw[dotted] (2,-0.5) .. controls (2.5,-0.5) and (2.5,-2.5) .. (3,-2.5);
                  \draw (3,-2.5) .. controls (3.3,-2.5) and (3.3,-1.7) .. (4,-1.7);
                  \draw (0,-1.2) -- (1.05,-1.2);
                  \draw (6,-1.7) -- (4,-1.7);
                  
               \end{scope}

\node at (0.7,-1.65) {$L_0(0)$};
\node at (0.7,-0.85) {$L_0(\sF_\alpha)$};
\end{tikzpicture}         \caption{ }
        \label{fig:p1sections}
    \end{subfigure}
    \begin{subfigure}[t]{.45\linewidth}
        \begin{tikzpicture}[scale=.9]

               \draw  (-0.5,-1.5) ellipse (0.5 and 1);
   
               \begin{scope}[shift={(4,-3.5)}]
                  \begin{scope}[]
   
                     \clip  (1.5,3) rectangle (2.5,1);
                     \draw  (1.5,2) ellipse (0.5 and 1);
   
                  \end{scope}
                  \begin{scope}[]
   
                     \clip  (1.5,1) rectangle (0.5,3);
                     \draw[dotted]  (1.5,2) ellipse (0.5 and 1);
                  \end{scope}
               \end{scope}
   
               \draw (-0.5,-2.5) -- (5.5,-2.5);
               \draw (-0.5,-0.5) -- (5.5,-0.5);

         \node[fill=black, circle, scale=.25] at (-0.0296,-1.5) {};
         \node[fill=black, circle, scale=.25] at (6,-1.5) {};

\draw[red] (0,-1.3) .. controls (1,-1.3) and (1.7,-1.3) .. (1.9,-1.3) .. controls (2.6,-1.3) and (2.6,-1.7) .. (3.2,-1.7) .. controls (3.6,-1.7) and (3.7,-1.7) .. (3.9,-1.7);

               \begin{scope}[red]
                  \draw (1.05,-1.2) .. controls (1.75,-1.2) and (1.7,-0.5) .. (2,-0.5);
                  \draw[dotted] (2,-0.5) .. controls (2.5,-0.5) and (2.5,-2.5) .. (3,-2.5);
                  \draw (3,-2.5) .. controls (3.3,-2.5) and (3.4,-1.8) .. (3.9,-1.8);
                  \draw (0,-1.2) -- (1.05,-1.2);
                  
               \end{scope}

\node at (0.7,-1.65) {$L_{\alpha>0}(0)$};
\draw[red] (3.9,-1.7) .. controls (4,-1.7) and (4,-1.8) .. (3.9,-1.8);
\end{tikzpicture}         \caption{}
        \label{fig:aftersurgery}
    \end{subfigure}
    \begin{subfigure}[t]{.45\linewidth}
        \begin{tikzpicture}[scale=.9]

               \draw  (-0.5,-1.5) ellipse (0.5 and 1);
   
               \begin{scope}[shift={(4,-3.5)}]
                  \begin{scope}[]
   
                     \clip  (1.5,3) rectangle (2.5,1);
                     \draw  (1.5,2) ellipse (0.5 and 1);
   
                  \end{scope}
                  \begin{scope}[]
   
                     \clip  (1.5,1) rectangle (0.5,3);
                     \draw[dotted]  (1.5,2) ellipse (0.5 and 1);
                  \end{scope}
               \end{scope}
   
               \draw (-0.5,-2.5) -- (5.5,-2.5);
               \draw (-0.5,-0.5) -- (5.5,-0.5);

         \node[fill=black, circle, scale=.25] at (-0.0296,-1.5) {};
         \node[fill=black, circle, scale=.25] at (6,-1.5) {};

\draw[red] (0,-1.6) .. controls (0.2,-1.6) and (0.4,-1.6) .. (0.6,-1.6) .. controls (0.8,-1.6) and (0.8,-1.4) .. (0.6,-1.4) .. controls (0.4,-1.4) and (0.2,-1.4) .. (0,-1.4);

               \begin{scope}[shift={(1,-3.5)},gray,thin]
                  \begin{scope}[]
   
                     \clip  (1.5,3) rectangle (2.5,1);
                     \draw  (1.5,2) ellipse (0.5 and 1);
   
                  \end{scope}
                  \begin{scope}[]
   
                     \clip  (1.5,1) rectangle (0.5,3);
                     \draw[dotted]  (1.5,2) ellipse (0.5 and 1);
                  \end{scope}
               \end{scope}
   
\draw[gray, thin] (0,-1.5) -- (6,-1.5);

\end{tikzpicture}         \caption{}
        \label{fig:afterhamiltonian}
    \end{subfigure}
    \caption{The surgery procedure for the mirror to $\CP^1$. }
\end{figure} 
The construction of the mirror to a toric divisor can be drawn out for $\check X_\Sigma=\CP^1$.
The fan for $\CP^1$  has two 1-dimensional cones.
We label the generators as $A=\{ \alpha=-e_1, \beta=e_1\}$.
The stop $\mathfrak f_\Sigma$ is a point on each of the 2 contact boundaries. The skeleton $\mathbb L_\Sigma$ has three components (indexed by the cones of the fan) and is the union of the unit circle and the real axis as shown in \cref{fig:p1skeleton}. 
With our choice of push-off, the Lagrangian zero section is the black curve drawn in \cref{fig:p1sections}, which is nearly wrapped to the stop. 

We consider the ray $\alpha$ and $L_0(\sF_\alpha)$.
This Lagrangian section intersects $\LL_0$. 
We then wrap from $L_0(0)$ to $L_0(\sF_\alpha)$ by translating the left end nearly a full rotation and leaving the right end fixed. 
The surgery of $L_0(\sF_\alpha)$ and $L_0(0)$, drawn in \cref{fig:aftersurgery}, removes the overlap on the right hand side of the figure.
As a result, the surgery is disjoint from the skeleton components on the right-hand side.
In this particular example, we can take a further Hamiltonian isotopy to identify $L_{\alpha>0}$ with a linking disk living in a neighborhood of $\LL_{\alpha}$ (\cref{fig:afterhamiltonian}). 
Generally this is not the case, and in \cref{rem:whythingsarenotlinkingdisks}, we discuss what is special about the example of $\CP^1$.
     
\section{Inclusion of a toric divisor}
   \label{sec:inclusion}
   In \cref{sec:mirrordivisor}, we showed that the Lagrangian $L_{\alpha>0}$ is mirror to $\mathcal O_{\alpha>0}$.
On the $B$-side, one can construct $\mathcal O_{\alpha>0}$ as the pushforward of a line bundle along the inclusion $\check i_{\alpha 0}:\check X_{\str(\alpha)/\alpha}\to \check X_\Sigma$. 
The goal of this section is to construct the mirror to this inclusion, thereby extending the previous discussion to line bundles supported on any toric stratum.

The mirror relation between $A$- and $B$-model morphisms is understood by representing a morphism of symplectic or complex manifolds by the submanifold in the product equal to the graph of the morphism.
Given a morphism $\check f: \check X_1\to \check X_2$, the graph is a subvariety $ \Gamma(\check f)\subset \check X_1\times \check X_2$.
Similarly, given a symplectomorphism $f: X_1\to X_2$, the graph is a Lagrangian submanifold $\Gamma(f)\subset X_1^-\times X_2$, where $X_1^{-}$ is equipped with its negative symplectic form. 
We expect that maps $f, \check f$ are mirror if the Lagrangian $\Gamma(f)$ is mirror to the sheaf  $\mathcal O_{\Gamma(\check f)}$.
More generally, the appropriate set of morphisms between symplectic manifolds are \emph{Lagrangian correspondences}. 
We refer to \cref{app:correspondenceadmissibility} for the definition of Lagrangian correspondences in the partially wrapped setting.
Lagrangian correspondences are predicted to give functors between the Fukaya category of $X_1$ and $X_2$ via geometric composition. 
The expected mirrors to Lagrangian correspondences are \emph{Fourier-Mukai kernels}, which are coherent sheaves on $\check X_1\times \check X_2$ that provide a functor between the derived category of coherent sheaves via the associated Fourier-Mukai transform.

From here on, we will fix a complete smooth projective fan $\Sigma$ of dimension $n$. 
Let $\alpha\in \Sigma$ be a 1-dimensional cone, and let $\check i_{\alpha 0}: \tb{\alpha}\to \check X_\Sigma$ be the inclusion of the toric orbit closure into $\check X_\Sigma$.
$\tb{\alpha}$ is a toric variety as well, and our goal for this section is to describe a Lagrangian correspondence $ \Lc_{\alpha 0}:((\CC^*)^{n-1}, \stp_{\str(\alpha)/\alpha})\Rightarrow( (\CC^*)^{n},\stp_\Sigma)$ which is mirror to $\mathcal O_{\Gamma(\check i_{\alpha 0})}\in D^b\Coh(\tb{\alpha}\times \check X_\Sigma)$. 
We will also describe how this correspondence acts on Lagrangian sections and linking disks. 
We write $L_\sigma(\sF)$ for a conical tropical Lagrangian section of $((\CC^*)^{n-k}, \stp_{\str(\sigma)/\sigma})$, and we will write $L_{\tau>\sigma}$ for a linking disk of the $\LL_{\tau/\sigma}$ component of the $\LL_{\str(\sigma)/\sigma}$ skeleton.
We now summarize the construction and results of this section in a theorem.
\begin{thm}
   Let $\alpha\in \Sigma$ be a 1-dimensional cone.
   There exists a toric inclusion Lagrangian correspondence (\cref{def:inclusioncorrespondence})  
   \[\Lc_{\alpha 0}: ((\CC^*)^{n-1}, \stp_{\afan})\Rightarrow ((\CC^*)^{n}, \stp_{\Sigma})
   \]
    which satisfies the following properties:
   \begin{itemize}
      \item(\Cref{property:admissibility})  $ \Lc_{\alpha 0}$ and $ \Lc^{-1}_{\alpha 0}$ are  admissible Lagrangian correspondences.
      \item(\Cref{property:topology}) If $L\subset((\CC^*)^{n-1}, \stp_{\afan})$ is an admissible Lagrangian, the geometric composition $ \Lc_{\alpha 0}\circ L$ is topologically $L\times \RR$.
      \item(\Cref{property:intertwining,property:apullback}) 
      Let $\sF\in \Supp_\alpha(\Sigma)$ be a support function transverse to $\alpha$. 
      The inclusion correspondence relates tropical Lagrangian sections 
      \[L_\alpha(\underline i^*_{\alpha0}\sF)\in \mathcal W ((\CC^*)^{n-1} ,\stp_{\afan})\]
       and  $L_0(\sF)\in\mathcal W( (\CC^*)^{n},\stp_\Sigma) $ 
       by 
       \[ \Lc_{\alpha 0}^{-1}\circ(L_0(\sF))=L_\alpha(\underline i_{\alpha 0}^*\sF).\]
      \item(\Cref{property:aexactsequence}) There exists an admissible Lagrangian cobordism (\cref{def:cobordism})
      \[(L_0(\sF+\sF_\alpha), L_0(\sF))\rightsquigarrow  \Lc_{\alpha 0}\circ L_\alpha(\underline i_{\alpha 0}^*\sF)=L_{\alpha>0}(\underline i_{\alpha 0}^* F).\]
      The topology of this cobordism is a disk.
      \item(\Cref{property:linkingdisks}) Let $\sigma>\alpha$ be a cone. Let $x\in \LL_{\underline i_{\alpha 0}^*\sigma}\subset \LL_{\afan}$ be an internal point, giving us a linking disk $\cocore(x)$. Then $ \Lc_{\alpha 0}\circ\cocore(x)$ is a linking disk for $\LL_\sigma\subset \LL_\Sigma$.
   \end{itemize}
   \label{thm:inclusioncorrespondence}
\end{thm}

\begin{rem} \label{rem:functorstatement}
   The Lagrangian correspondence $L_{\alpha 0}:((\CC^*)^{n-1}, \stp_{\str(\alpha)/\alpha})\Rightarrow(\CC^*)^n, \stp_\Sigma)$ is expected to give a functor between $\mathcal W((\CC^*)^{n-1}, \stp_{\str(\alpha)/\alpha})\rightarrow W((\CC^*)^n, \stp_\Sigma)$. 
   We now give an argument for why if such functors (as considered by \cite{wehrheim2007functoriality}) can be constructed for Lagrangian correspondences between stopped Liouville domains, then the functor arising from $L_{\alpha 0}$ is mirror to the derived pushforward $(\check i_{\alpha 0})_*: D^b\Coh(\check X_{\str(\alpha)/\alpha})\to D^b\Coh(\check X_\Sigma)$. 
   By \Cref{thm:ageneration}, the partially wrapped Fukaya categories in discussion are generated by Lagrangian sections, and geometric composition with $L_{\alpha 0}$ maps generators to generators in a way consistent with the mirror functor.
   This shows that $L_{\alpha 0}$ gives a homomorphism between the Grothendieck groups $K_0(\mathcal W((\CC^*)^{n-1}, \stp_{\str(\alpha)/\alpha})\rightarrow K_0( W((\CC^*)^n, \stp_\Sigma))$, which is mirror to the corresponding map on the mirror.
   In addition to knowing how $L_{\alpha 0}$ acts on Lagrangian sections, we know that $L_{\alpha 0}$ maps linking disks to linking disks in a way consistent with mirror symmetry.
   We now sketch why this implies that an extension of $L_{\alpha 0}$ to a functor will be mirror to $\mathcal{R}(\check i_{\alpha 0})_*$.

   Since $\mathcal W((\CC^*)^{n-1}, \stp_{\str(\alpha)\alpha})$ and $ W((\CC^*)^n, \stp_\Sigma)$ are generated by Lagrangian sections, it suffices to check that this correspondence acts appropriately on homomorphisms between Lagrangian sections. 
   In fact, we need only compute 
   $\HF(L_\alpha(0), L_\alpha(\sF))$, with $-D_\sF$ base-point free.
   In this setting, a generating set for $\CF(L_\alpha(0), L_\alpha(\sF))$  is given by morphisms $f_v$, where $v\in \Delta_\sF$ is a lattice point contained in the Newton polytope of $\sF$. 
   These morphisms are characterized by the short exact sequences
   \[L_\alpha(0)\xrightarrow{f_v} L_\alpha(\sF)\to \bigoplus_{\substack{\sigma_i>\alpha\\|\sigma|=|\alpha+1|}} L_{\sigma_i> \alpha}\]
   for some collection of Lagrangian disks  $L_{\sigma_i>\alpha}$ with $|\sigma_i|=|\alpha+1|$.
   Because $L_{\alpha 0}$ is a triangulated functor which sends mirrors to toric strata to mirror to toric strata, we can conclude that $L_{\alpha 0}$ is mirror to $\check i_{\alpha 0}$. 

   In this case, we can also conclude that the functor associated to geometric composition with $L_{\alpha 0}^{-1}$ is the mirror to derived pullback $\mathcal L i^*: D^b\Coh(\check X_\Sigma)\to D^b\Coh(\check X_{\str(\alpha)/\alpha})$ by adjunction.
   We expect that a similar strategy will yield a mirror symmetry statement for the Lagrangian correspondences defined from the data of a toric morphism in \cref{def:toricmorphismcorrespondence}.
\end{rem}

\subsection[Construction of i on the B side]{Construction of $\check i_{\alpha 0}$: On the $B$-side}
\label{subsec:bsideconstruction}

We now outline a construction of the structure sheaf of the graph of $\check i_{\alpha 0}: \tb{\alpha}\to \check X_\Sigma$ which is the Fourier-Mukai kernel giving the map $(\text{R}\check i_{\alpha 0})_*: D^b\Coh(\tb{\alpha})\to D^b\Coh(\check X_\Sigma)$. 
This exposition serves as a template for our symplectic arguments. 
The tools that we use to build   $\mathcal O_{\Gamma(\check i_{\alpha 0})}$ will be restricted to inclusion of toric subvarieties, tensor products with line bundles, mapping cones, and composition of Fourier-Mukai kernels.
Restricting ourselves to these methods gives a roundabout path to constructing $\mathcal O_{\Gamma(\check i_{\alpha 0})}$, but one which we can replicate on the $A$-side. 

Our first step is to factor the inclusion $\check i_{\alpha 0}$ as
\[
    \tb{\alpha} \into  \check X_{\str(\alpha)} \into \check X_\Sigma
\]
The map $\check i_{\str(\alpha),0}: \check X_{\str(\alpha)} \into \check X_\Sigma$ is a toric morphism. 
While the map  $\check i_{\alpha, \str(\alpha)}: \tb{\alpha} \into  \check X_{\str(\alpha)} $ is not toric, there is a toric morphism $\check \pi_{\str(\alpha),\alpha}:\check X_{\str(\alpha)}\to \tb{\alpha}$. 
The structure sheaf of the graph $\mathcal O_{\Gamma(\check \pi_{\str(\alpha), \alpha})}\in \text{Sh}(\tb{\alpha}\times \check X_{\str(\alpha)})$ is a toric subvariety.\footnote{We note that this is not a Fourier-Mukai kernel for the derived category between $\tb{\alpha}$ and $\check X_{\str(\alpha)}$, as the projection $\check \pi_{\str(\alpha),\alpha}:  \check X_{\str(\alpha)}\to \tb{\alpha}$ is not proper, so it does not give a map of coherent sheaves.
}
The line bundle $\mathcal O_{\str(\alpha)}(\sF_\alpha)$ on $\check X_{\str(\alpha)}$ pulls back to a line bundle on $\tb{\alpha}\times \check X_{\str(\alpha)}$. 
There exists an exact sequence of sheaves in $\tb{\alpha}\times \check X_{\str(\alpha)}$ 
\begin{equation}
   \left(\mathcal O_{\Gamma(\check \pi_{\str(\alpha),\alpha})}\tensor \mathcal O_{\str(\alpha)}(\sF_\alpha)\right)\to\left( \mathcal O_{\Gamma(\check \pi_{\str(\alpha),\alpha})}\right)\to\left( \mathcal O_{\Gamma(\check i_{\alpha,\str(\alpha)})}\right).
\end{equation}

Finally, we obtain the Fourier-Mukai kernel of $\check i_{\alpha 0}:\tb{\alpha}\to \check X_\Sigma$ by the composition:
\[\mathcal O_{\Gamma(\check i_{\alpha 0})}=\mathcal O_{\Gamma(\check i_{\str(\alpha),0})}\circ \mathcal O_{\Gamma(\check i_{\alpha,\str(\alpha)})}.\]
We now discuss the properties of the associated Fourier-Mukai transform.
Let $\sF\in \Supp_\alpha(\Sigma)$ be a support function transverse to  $\alpha$.
First, we observe that this transform intertwines pullbacks for line bundles and the linear pullback of support functions:
\begin{equation}
   \check i^*_{\alpha0}\mathcal O_0(\sF)= \mathcal O_{\alpha} (\underline i^*_{\alpha0}\sF).
   \tag{Mirror to \cref{property:apullback}}
\end{equation}
Additionally, there is an exact triangle 
\begin{equation}
   \mathcal O_0(\sF+\sF_\alpha)\to \mathcal O_0( \sF)\to \mathcal O_{\alpha>0}(\underline i^*_{\alpha 0} \sF).
   \tag{Mirror to  {\cref{property:aexactsequence}}}
\end{equation}
Finally, we remark that the inclusion of toric orbit is again a toric orbit, that is if $\sigma>\alpha$, then 
\begin{equation}
   (\check i_{\alpha 0})_* \mathcal O_{\sigma>\alpha}= \mathcal O_{\sigma>0}. 
\end{equation}

\subsection{Lagrangian correspondences from toric morphisms}
\label{subsec:toricmorphisms}
We first fix some notation for this section. 
$(X_1, \stp_{\Sigma_1})$ and $(X_2, \stp_{\Sigma_2})$ will always be the mirrors to the toric varieties $\check X_{\Sigma_1}, \check X_{\Sigma_2}$, and we will oftentimes refer to them by simply $X_1, X_2$ and suppress the notation of the stop.
We will frequently use the argument-valuation decomposition, identifying 
\[X_1=(\CC^*)^n=T^* Q_1/T^*_{2\pi \ZZ} Q_1 =Q_1\times \sP_1\]
 where $Q_1=\RR^n$ and $\sP_1=T^n$. 
 We give $T^*Q_1$ the standard symplectic coordinates $(q_1, p_1)$; we will also use these as coordinates on $X_1$ as well (where the $p_1$ coordinate is taken up to the period of $T^*_{2\pi \ZZ} Q_1)$. 
Recall that a Lagrangian section given by support function $\sF\in \Supp(\Sigma)$  can be parameterized by 
\[L(\sF)=(q, d(\hamH{\sF}+\hamK{}))\subset (X, \stp_\Sigma).\]
The constructions in this section do not use the parameter $\delta$, and we drop it from our notation.

Our construction will use an SYZ mirror symmetry description of  Lagrangian correspondences.
A Lagrangian torus fibration  $\val: X\to Q$  locally provides coordinates $(q, p)$ to $X$. The fiberwise sum of Lagrangian submanifolds $L', L''\subset X$ is defined (in local coordinates) by
\[L'+_QL'':= \{(q, p)\; : \; p=p'+p'', (q, p')\in L',  (q, p'')\in L''\}.\]
This operation on Lagrangian submanifolds of $X$ in the Fukaya category is conjecturally mirror to the tensor product of coherent sheaves on $\check X$ \cite{subotic2010monoidal}.
This holds for Lagrangian sections, where
\[L_0(\sF')+L_0(\sF'')=\wrpK{}(L_0(\sF'+\sF''))\sim L_0(\sF'+\sF'').\]
If the fiberwise sum is cut out transversely, then it is a Lagrangian submanifold of $X$. 
\begin{prop}
   \label{prop:syzcomp} 
Let $\bar \Lc_{01}: X_0^-\times X_1$ and $\bar \Lc_{12}: X_1^-\times X_2$ be Lagrangian correspondences between SYZ fibrations $\val:X_i\to Q_i$.
 Let $\Lc_{01}\subset X_0\times X_1$ and $\Lc_{12}\subset X_1\times X_2$ be the Lagrangians defined in local coordinates by:
\begin{align*}
   \Lc_{01}=&\{(p_0, p_1, q_0, q_1)\; : \; (q_0, q_1, -p_0, p_1)\in \bar \Lc_{01}\}\\
   \Lc_{12}=&\{(p_1, p_2, q_1, q_2)\; : \; (q_1, q_2, -p_1, p_2)\in \bar \Lc_{12}\}
\end{align*}
We can rewrite the geometric composition $\bar \Lc_{02} : X_0^-\times X_2$ (see \cref{app:correspondenceandcobordism}) using fiberwise sum:
\begin{equation}
    L_{02}:= \pi_{02}\left((X_0\times Q_1 \times X_2)\cap ((Q_0\times \Lc_{12})+_Q(L_{01}\times Q_2))\right)\subset X_2.
   \label{eq:geocomp}
\end{equation}
\[L_{02}=\{(q_0, q_1, p_0, p_1)\; : \; (q_0, q_2, -p_0, p_2)\in \bar \Lc_{12}\circ \bar \Lc_{01}\} \]
   Here, $\pi_{02}:X_0\times X_1\times X_2\to X_0\times X_2$.
\end{prop}
\begin{proof}
   We compute that
\begin{align*}
   \pi_{02}&\left((X_0\times Q_1 \times X_2)\cap ((Q_0\times \Lc_{12})+_Q(L_{01}\times Q_2))\right)\\
    =&\pi_{02}\left((X_0\times  Q_1\times X_2) \cap \left\{(  q_0, q_1, q_2,p_0, p_1, p_2)\;\middle|\; \begin{array}{c} p_0=p_0'+p_0'', p_1=p'_1+p''_1, p_2=p'_2+p''_2 \\ 
      ( q_0 ,q_1, q_2,p_0'', p_1'', p_2'')\in L_{01}\times Q_2\\
      ( q_0,q_1, q_2,p_0',p_1', p_2')\in  Q_0\times L_{12}
   \end{array}\right\}\right)\\
    =& \pi_{02} \left( \{( q_0, q_1, q_2, p_0, 0, p_2)\; : \; -p'_1=p''_1,  (q_0, q_1 ,p_0,p_1'')\in L_{01},(q_1, q_2, p_1', p_2)\in L_{12}\}\right)\\
    =&\{(q_0 , q_2,- p_0,p_2)\; : \; (q_1, q_2, -p_1, p_2)\in L_{12}, (q_0, q_1 , p_0, p_1)\in \bar L_{01}\}\\
    =& \{( q_0, q_2,p_0, p_2)\; : \; (q_0, q_1, -p_0, p_1)\in \bar \Lc_{12}\circ \bar \Lc_{01}\}
\end{align*}
as desired.
\end{proof}
This is a description of the Fourier-Mukai transform from a SYZ mirror symmetry perspective, where fiberwise sum is replacing tensor product on $\check X_1\times \check X_2$.
\begin{notation}
   Since we are always working in the setting where $X_1, X_2$ come from SYZ fibrations, we will (for \cref{sec:inclusion,sec:applications}) 
    write $L_{12}: X_1\Rightarrow X_2$ as a Lagrangian $L_{12}\subset X_1\times X_2$, with composition of Lagrangian correspondences as defined in \cref{prop:syzcomp}.
   \label{notation:correspondenceadmissibility}
\end{notation}
\begin{ex}[\cite{subotic2010monoidal}]
   Let $K_1\subset X$ be a Lagrangian submanifold, whose points can be given coordinates $(q, p)\in K_1$.
   Then there is a Lagrangian correspondence $\Lc_{+_Q K_1}: X\Rightarrow X$  parameterized by 
   \begin{align*}
      \sP\times K_1\to X\times& X\\
      (\tilde p,( q, p))\mapsto &(q, q, \tilde p+p, -\tilde p)
   \end{align*}
   which has the property that for any Lagrangian $K\subset X$, 
   \begin{align*}
      \Lc_{+_Q K_1}\circ K = K_1+_Q K
   \end{align*}
   thus recovering the operation of fiberwise sum with $K_1$.
\end{ex}
Let $\check f: \check X_{\Sigma_1}\to \check X_{\Sigma_2}$ be a toric morphism.
Then there exists a linear map associated with this toric morphism $\underline f: Q_1 \to {Q_2}$, which maps the cones of $\Sigma_1$ to subsets of cones  of $\Sigma_2$.
We use  $\underline f: Q_1\to Q_2$ to construct the Lagrangian correspondence by first constructing the graph $\Gamma(\underline f)\subset  Q_1\times {Q_2}$, taking the associated tropical Lagrangian submanifold (conormal torus), and finally wrapping slightly by the wrapping Hamiltonian $K_{12}$ for $X_1\times X_2$. 
\begin{df}
   Let $\check f: \check X_{\Sigma_1}\to \check X_{\Sigma_2}$ be a toric morphism.
   The mirror correspondence is:
      \[\Lc_{f^*}:=\psi_\delta^{-\eps} (N^*\Gamma(\underline f)/N^*_{2\pi \ZZ}\Gamma(\underline f))\subset (X_1\times X_2 ).\]
      \label{def:toricmorphismcorrespondence}
\end{df}
Consider the map $\underline f^*: (\sP_2)_{f(q_1)}\to (\sP_1)_{q_1}$, which comes from periodizing the map $\underline f^*: T^*_{f(q_1)}Q_2\to T^*_{q_1}Q_1$. 
We can explicitly parameterize $\Lc_{f^*}$ by:
\begin{align*}
   Q_1\times \sP_2\to&  Q_1\times Q_2\times \sP_1\times \sP_2\\
   (q_1, p_2)\mapsto & (q_1,\underline f(q_1), -\underline f^* p_2+d\hamK{12}|_1, p_2+d\hamK{12}|_2). 
\end{align*}
here $d\hamK{12}|_i$ are the $\sP_i$ components of $d\hamK{12}$.
\begin{lem}
   Let $\check f: \check X_{\Sigma_1}\to \check X_{\Sigma_2}$ be a toric morphism. 
   The Lagrangian $L_{f^*}$ gives a Lagrangian correspondence  
   \[L_{f^*}:((\CC^*)^{|\Sigma_2|} , \stp_{\Sigma_2})\Rightarrow ((\CC^*)^{|\Sigma_1|}, \stp_{\Sigma_1}).\]
   \label{lemma:toriccorrespondence}
\end{lem}
\begin{proof}
    We need to show that this Lagrangian is conical at infinity and disjoint from the stop.
   
   The Liouville structure on $X_1 \times X_{2}$ comes from identifying it with $T^*(\sP_1\times \sP_2)$, where $\sP_1\times \sP_2$ is the SYZ fiber over the origin $(0, 0)\in Q_1\times {Q_2}$. 
   Our Lagrangian graph can also be expressed as 
   \[\Lc_{f^*}=\psi^\eps_{K_{12}} (N^*_{\sP_1\times \sP_2} (f^*(\sP_2), \sP_2)).\]
    As all conormal bundles are conical at infinity, and the wrapping Hamiltonian is admissible, the Lagrangian  $\Lc_{f^*}$ is conical at infinity.
   Therefore, $L_{f^*}$ is conical at infinity as well.

   Similarly, we show that $\Lc_{f^*}$ is disjoint from the stop of $X_1\times X_2$. Here, we use that the stop of the product (\cref{eq:prodstop}) is the FLTZ stop of the product toric variety.
   Let $\LL_{\sigma_1\times \sigma_2}\in \LL_{\Sigma_1\times \Sigma_2}$ be a FLTZ skeleton component. 
   The points $(q_1, p_2)\in \sP_2\times Q_1$  which intersect the FLTZ skeleton component $\LL_{\sigma_1}\times \LL_{\sigma_2}$ under the parameterization of $\Lc_{f^*}$ must satisfy the constraint:
   \begin{align*}
      (-\underline f^* p_2+d\hamK{12}|_1, p_2+d\hamK{12}|_2) \cdot (\sigma_1\times \sigma_2)|_{(q_1, f(q_1))} =0.
   \end{align*}
    Let $\vec {v_q}$ be the unit vector which points in the $q$-direction, for a point $q\in \sigma_1$. 
   We test the condition on the vector $(\vec {v_q}, \underline f (\vec {v_q}))$.
   Since $f$ is a toric morphism, $(\vec {v_q}, \underline f (\vec {v_q}))$ belongs to a cone, so $d\hamK{12}(\vec {v_q}, \underline f (\vec {v_q}))<0$. 
   Therefore, $\Lc_{f^*}$ avoids the stop.
\end{proof}
\begin{rem}
   This Lagrangian correspondence does a remarkably good job of preserving the skeleton. Let $\check f: \check X_1\to \check X_2$ be a toric morphism which \emph{preserves cones}, so that $\underline f (\sigma_1)\in \Sigma_2$ for all $\sigma_1\in \Sigma_1$. Consider $\tilde \Lc_{f^*}:= N^*\Gamma(\underline f)/N^*_{2\pi \ZZ}\Gamma(\underline f)$, which should be considered as the ``monomially admissible'' version of our correspondence.\footnote{When $f$ does not preserve cones, this is not monomially admissible.}
   Let $\LL_{\sigma_1}\subset X_{1}$ be the component of the FLTZ skeleton associated to $\sigma_1\in \Sigma_1$, and $\LL_{\underline f(\sigma_1)}\subset X_{2}$ be the component of the FLTZ skeleton associated to its image. 
   \begin{align*}
      \tilde \Lc_{f^*}^{-1}\circ \LL_{\sigma_1}= &\pi_2\circ(\pi_1^{-1}( \LL_{\sigma_1})\cap L_{f^*})\\
   \intertext{Decomposing $X_1\times X_2= (Q_1\times Q_2)\times (\sP_1\times \sP_2)$,}
      =&\pi_2\circ(((\sigma_1\times Q_2))\times(\sigma_1^\bot\times \sP_{2})\cap ((Q_1, f(Q_1))\times(Q_1, f(Q_1))^\bot ))\\
      =&\left\{(q_2, p_2)\;\middle|\; \begin{array}{c}  -p_1\cdot Q_1+p_2\cdot f(Q_1)=0\\ q_2=f(q_1),q_1\in \sigma_1 , p_1\cdot Q_1=0\end{array}\right\}\\
      =& ( \underline f (\sigma_1))\times(\underline f (\sigma_1))^\bot=\LL_{\underline f(\sigma_1)}.
   \end{align*}
   This Lagrangian $\tilde \Lc_{f^*}$ is more rigid than $\Lc_{f^*}$, which only preserves a neighborhood of the stop as opposed to the stop itself. 
   The Lagrangian $\tilde \Lc_{f^*}$ can be obtained from $ \Lc_{f^*}$ by taking the wrapping Hamiltonian $-\eps K_{12}$ to zero. 
   As $\tilde \Lc_{f^*}$ has a cleaner description than $\Lc_{f^*}$, we will use it in proofs where taking the wrapping Hamiltonian $-\eps K_{12}$ to zero does not impact the argument.
\end{rem}
   We've chosen the convention that $\check f: \check X_{\Sigma_1}\to \check X_{\Sigma_2}$ gives a Lagrangian correspondence $\Lc_{f^*}: X_2 \Rightarrow X_1$ for reasons related to admissibility of the geometric composition. 
   We quickly highlight some reasons for this choice; a full discussion of the admissibility of Lagrangian correspondences is found in \cref{claim:compositionstop}.
   When $\Sigma_2$ is complete, then the inverse image functor gives a map from the derived category of coherent sheaves on $\Sigma_2$ to that of $\Sigma_1$.
   Geometric correspondence similarly gives us a map from (Hamiltonian isotopy classes of) admissible Lagrangian submanifolds of $(X_1, \stp_1)$ to those of $(X_2, \stp_2)$.  
   However, there need not be a well-defined map on Hamiltonian isotopy classes of Lagrangian submanifolds in the other direction. 
   For example, if $\Sigma_1$ is a subfan of $\Sigma_2$, then there is a morphism $((\CC^*)^n, \Sigma_2)\Rightarrow ((\CC^*)^n, \Sigma_1)$ arising from stop removal, but not necessarily a map in the other direction. 
   The stop removal morphism is mirror to localization at a toric stratum. An analogous situation occurs in the mirror, where (without a properness condition) the Fourier-Mukai transform only preserves coherence in one direction. 
   As an example, consider on the $B$-side the map $i: \CC^*\to \CC$.
   This induces a pullback map $i^*: \text{Sh}(\CC)\to \text{Sh}(\CC^*)$ on the level of coherent sheaves. However $i_*:\text{Sh}(\CC^*)\to \text{Sh}(\CC) $ only preserves quasi-coherence.
   
   However, even if $L_{f^*}^{-1}$ does not give an admissible Lagrangian correspondence from $X_1\Rightarrow X_2$, it is possible that for some $L\in X_1$, the composition $L_{f^*}^{-1}\circ L$ remains admissible. In particular, if $L$ is totally stopped (and conjecturally mirror to a compactly supported perfect complex) then $L_{f^*}^{-1}\circ L$ is admissible. 
   For example, in the mirror, there are plenty of coherent sheaves on $\CC^*$ which pushforward to coherent sheaves on $\CC$. 
   We will, on occasion, use a Lagrangian correspondence in the non-admissible direction to build Lagrangian submanifolds on $X_2$ from Lagrangian submanifolds on $X_1$, and we will point out if we need check the admissibility of the composition in these situations.

   The next two claims prove that this correspondence appropriately ``pulls back'' Lagrangian sections.
\begin{prop}
   Let $L_2(0)\subset X_2$ be the admissible Lagrangian zero section of the Lagrangian torus fibration $\val: X_2\to Q_2$.
   Then $L_{f^*}\circ L_2(0)$ is Hamiltonian isotopic to the zero section of $X_1\to Q_1$.
   \label{claim:toricpullback}
\end{prop}
In the construction of the geometric composition for admissible Lagrangian correspondences (\cref{claim:correspondencecomposition}), we're required to take an additional Hamiltonian isotopy to obtain admissibility of the composition. This is denoted by $\Phi(\pi^{-1}(L_1))$ in \cref{claim:correspondencecomposition}. 
For this section, we will ignore this additional Hamiltonian (whose incorporation does not modify the admissible Hamiltonian class of our computations).
\begin{proof}
   \begin{align*}
      L_{f^*}\circ L_2(0)=&\pi_1\left((Q_2\times X_1) \cap \left\{( q_1, q_2, p_1, p_2)\;\middle|\; \begin{array}{c} p_1=p'_1+p''_1, p_2=p'_2+p''_2 \\ (q_1, q_2, p_1', p_2')\in \Lc_{f^*}\\(q_1, q_2, p_1'', p_2'')\in Q_1\times L_2\end{array}\right\}\right)\\
      =&   \left\{( q_1, p_1)\;\middle|\; \begin{array}{c} p_1=p'_1+p''_1, 0=p'_2+p''_2 \\ p_1'=-\underline f^*(p)+d\hamK{12}|_1, p_2'=p+ d\hamK{12}|_2 ,q_2=\underline f(q_1)\\ p_2''=d\hamK{2}, p\in \sP_2, q_1\in Q_1\end{array}\right\}\\
      =&\left\{(q_1,d\hamK{12}|_1+\underline f^*(d\hamK{12}|_2+d\hamK{2})) \; \mid \;q_1\in Q_1\right\},
   \end{align*}
   which is a Lagrangian section of $X_1$.
   It remains to show that this section is admissibly Hamiltonian isotopic to $(d\hamK{1}, q_1)$. 
  Outside of a compact set, for all $\sigma$, and $q\in \sigma, \vec v\in \sigma$, $|\vec v|=1$, 
  \[-1<(d\hamK{12}|_1+\underline f^*(d\hamK{12}|_2+d\hamK{2})\cdot(\vec v)|_q<0 .\]
   Therefore, the isotopy of Lagrangian submanifolds which is parameterized by 
   \[\{(q_1, t d\hamK{1}+(1-t)(d\hamK{12}|_1+\underline f^*(d\hamK{12}|_2+d\hamK{2}))) \mid q_1\in Q_1\}\]
   avoids the stop, giving us an admissible Hamiltonian isotopy between $L_{f^*}\circ L_2(0)$ and $L_1(0)$. 
\end{proof}
For any support function $\sF: {Q_2}\to \RR$, the pullback function $\underline f^* \sF$ is a support function on $Q_1$. 
Let $\wrp{\sF}$ be the time $t$ flow of the associated conical twisting Hamiltonian $\hamH{\sF}$, with $\delta$ chosen small enough so that the Lagrangians we consider are still admissible.
\begin{prop}
   Let $\underline f: Q_1\to Q_2$ be a toric morphism  and consider a support function $\sF_2\in \Supp(\Sigma_2)$. Let $\underline{}\sF_1=\underline f^*\sF_2$ and $L_2\subset Q_2$ be an admissible Lagrangian submanifold.
   Then there is an admissible Hamiltonian isotopy
   \[L_{f^*}\circ (\wrp{\sF_2}( L_2)) \sim \wrp{\sF_1}(L_{f^*}\circ L_2).\]
   \label{claim:intertwining} 
\end{prop}
\begin{proof}
\begin{align*}
   L_{f^*}\circ (\wrp{\sF_2}( L_2))=\pi_1&\left((Q_2\times X_1) \cap \left\{( q_1, q_2, p_1, p_2)\;\middle|\; \begin{array}{c} p_1=p'_1+p''_1, p_2=p'_2+p''_2 \\ (p_1', p_2', q_1, q_2)\in \Lc_{f^*}\\(q_1, q_2, p_1'', p_2'')\in Q_1\times \wrp{\sF_2}(L_2)\end{array}\right\}\right)\\
   =& \left(  \left\{( q_1, p_1)\;\middle|\; \begin{array}{c} p_1=p'_1+p''_1, 0=p'_2+p''_2 \\ p_1'=d\hamK{12}|_1-\underline f^*(p_2)\\p_2'= d\hamK{12}|_2+p_2 ,q_2=f(q_1), (q_1, p_2)\in Q_1\times  \sP_2\\ p_2''=d\hamH{\sF_2}+\tilde p_2'', (q_2'',\tilde p_2'')\in L_2 \end{array}\right\}\right)\\
   \intertext{Letting $\tilde p'_2= p_2'+d\hamH{\sF_2},\tilde p_2=p_2-d\hamH{\sF_2}, \tilde p_1'=p_1'+\underline f^*(d\hamH{\sF_2}), \tilde p_1=p_1-\underline f^*(d\hamH{\sF_2})$} 
   =& \left(  \left\{( q_1, \tilde p_1+ f^*d\hamH{\sF_2})\;\middle|\; \begin{array}{c} \tilde p_1=\tilde p'_1+p''_1, 0=\tilde p'_2+\tilde p''_2 \\ \tilde p_1'=d\hamK{12}|_1-\underline f^*(\tilde p_2)\\ \tilde p_2'= d\hamK{12}|_2+\tilde p  ,q_2=\underline f(q_1) ,(q_2'',\tilde p_2'')\in L_2 \end{array}\right\}\right)\\
   \intertext{Letting $\tilde H_2=f^*\hamH{\sF_2}$}
   =& \wrp{\tilde H_2} (L_{f^*}\circ L_2).
\end{align*}
By a similar argument to \cref{claim:toricpullback}, this is admissibly Hamiltonian isotopic to $\wrp{\underline f^* \sF_2} (L_{f^*}\circ L_2)$. 
\end{proof} 
\begin{cor}
   If $L_{\Sigma_2}(\sF_2)\subset X_2$ is a tropical section, $L_{f^*}\circ L_{\Sigma_2}(\sF_2)\sim L_{\Sigma_1}(\underline f^*\sF_2)$.
   \label{property:intertwining}
\end{cor}

\subsection{Construction of the inclusion correspondence}
\label{subsec:inclusionmirror}
We now construct a correspondence mirror to the inclusion $\check i_{\alpha 0}:\tb{\alpha}\into \check X_\Sigma$. 
As before, we fix some notation. In this subsection, we reserve $X_1=(\CC^*)^n, X_2=(\CC^*)^{n-1}$ for symplectic tori, which can be equipped with stops to give us $(X_1, \stp_{\Sigma}), (X_1, \stp_{\str(\alpha)})$ and $(X_2, \stp_{\str(\alpha)/\alpha})$. 
These come with Lagrangian torus fibrations $X_1=Q_1\times \sP_1$, and $X_2=Q_2\times \sP_2$. 
Denote by $\underline \pi_{\alpha}: Q_1\to Q_2$ the linear map whose corresponding toric morphism is the projection $\check X_{\str(\alpha)}\to \tb{\alpha}$.
We write $\underline \pi_{\alpha}^*: \sP_2\to \sP_1$ for the dual map on the fibers.

The idea of the construction is to use \cref{subsec:toricmorphisms} to build the mirror to $\Gamma(\check \pi_\alpha)\subset \check X_{\str(\sigma)}\times X_{\str(\alpha)/\alpha}$, and then mimic \cref{subsec:bsideconstruction}. 
\begin{prop}
   Let $\alpha$ be a 1-dimensional cone of $\Sigma$.
   Consider the morphism of fans $\underline \pi_{\alpha}:\str(\alpha)\to \afan$.
   Let $\Lc_{\pi_{\alpha}^*}:(X_2,\stp_{\afan})\Rightarrow (X_1, \stp_{\str(\alpha)})$ be the associated Lagrangian correspondence built in \cref{lemma:toriccorrespondence}.
   Let $\wrpalpha: X_1\to X_1$ be the wrapping Hamiltonian for the support function $\sF_\alpha$ which is $1$ on the generator $\alpha$, and $0$ on all other 1-dimensional cones. 
   Then the intersection of the Lagrangians $\Lc_{\pi_{\alpha}^*}$ and $ \wrpalpha( \Lc_{\pi_{\alpha}^*})$ has a surgery data (as in \cref{def:surgerydata}).
   \label{claim:convexintersectionofsections}
\end{prop}
Before we prove the claim, we note that by taking the $1-\eps$ flow $\wrpalpha$, the Lagrangians $\Lc_{\pi_{\alpha}^*}$ and $ \wrpalpha( \Lc_{\pi_{\alpha}^*})$ only intersect over the region where $\frac{d}{dt}\phi^t_{\sF_\alpha}=0$.
\begin{proof}
   The Lagrangian submanifold $\Lc_{\pi_{\alpha}^*}:(X_2,\stp_{\afan})\Rightarrow (X_1, \stp_{\str(\alpha)})$  is topologically a  $\sP_2\times Q_1$,  parameterized by 
   \[L_{\pi^*_\alpha}=\{((q, \underline \pi_\alpha(q)), (-\underline \pi_{\alpha}^* p+d\hamK{12}|_1, p+d\hamK{12}|_2))\; : \; q\in Q_1, p\in \sP_2\}\subset Q_1\times Q_2\times \sP_1\times \sP_2\]
   The wrapping Hamiltonian acts on the $\sP_1$ factor, and sends this Lagrangian submanifold to 
   \[\wrpalpha (L_{\pi^*_\alpha})= \{( (q, \underline \pi_\alpha(q)),(-\underline \pi_{\alpha}^* p+d\hamK{12}|_1+d(1-\eps)\Halpha, p+d\hamK{12}|_2))\; : \; q\in Q_1, p\in \sP_2\}.\]
   Here, $\wrpalpha$ is the conical smoothing of $\sF_\alpha$, which in \cref{sec:embedding} was denoted $\widehat H^{\delta}_{-D_\alpha}$.
   The intersection between $\Lc_{\pi_{\alpha}^*}$ and $ \wrpalpha( \Lc_{\pi_{\alpha}^*})$ is determined by the intersection between the zero section and the Lagrangian section $d\Halpha$ in $T^*Q_1/T^*_{2\pi \ZZ} Q_1$.
   By \cref{lemma:positiveintersectionlemma}, the intersection of the zero section and the graph of $d\Halpha$ form a surgery data (\cref{def:surgerydata}). 
\end{proof}
This allows us to use \cref{prop:generalizedsurgeryprofile} to build the mirror to the graph of the inclusion $i_{\alpha 0}: \check X_{\afan}\to \check X_\Sigma$.
\begin{df}
   Let $\alpha\in \Sigma$ be a 1-dimensional cone.
   The local model for the toric orbit inclusion is the Lagrangian correspondence $\Lc_{\alpha 0}^{pre}:(X_2, \stp_{\afan})\Rightarrow (X_1, \stp_{\str(\alpha)})$ defined by the surgery
   \[
      \Lc_{\alpha 0}^{pre}=\wrpalpha(\Lc_{\pi_{\alpha}^*})\#_{U_{\Halpha}} \Lc_{\pi_{\alpha}^*} .
   \]
   We additionally obtain a Lagrangian surgery cobordism, 
   \[K_{\alpha 0}^{pre}:(\wrpalpha(\Lc_{\pi_{\alpha}^*}),\Lc_{\pi_{\alpha}^*}^{pre} )\rightsquigarrow \Lc_{\alpha 0}.\]
   Let $\check i_{\str(\alpha),0}: \check X_{\str(\alpha)}\to \check X_\Sigma$ be the inclusion of the normal bundle to $\check X_{\afan}$, and $\Lc_{i_{\str(\alpha),0}^*}: (X_1, \stp_{\Sigma})\Rightarrow (X_1, \stp_{\str(\alpha)})$ be the mirror correspondence.\footnote{This is a stop removal correspondence, which will be explored further in \cref{subsec:quotient}.}
   The toric orbit inclusion is the Lagrangian correspondence
   \[
      \Lc_{\alpha 0}:= ((\Lc_{\alpha 0}^{pre})^{-1} \circ \Lc_{i_{\str(\alpha),0}^*})^{-1} : (X_2, \stp_{\afan})\Rightarrow (X_1, \stp_{\Sigma}),
   \]
   which similarly comes with a surgery cobordism.
   \label{def:inclusioncorrespondence}
\end{df}
The correspondence $\Lc_{i_{\str(\alpha),0}^*}$ does not significantly modify the local model of the toric orbit inclusion. In fact, if we naively include $\Lc_{\alpha 0}^{pre}$ as a subset of $(X_1, \stp_{\Sigma})$, it is admissibly Hamiltonian isotopic to $L_{\alpha 0}$. 
\begin{prop}
   The valuation of the $X_1$ component of $\Lc_{\alpha 0}$ is contained within the star of $\alpha$:
   \[\val(\pi_1(\Lc_{\alpha 0}))\subset \str(\alpha).\]
   \label{claim:valuationofcorrespondence}
\end{prop}
\begin{proof}
   Analogous to \cref{claim:valuationofdivisormirror}.
\end{proof}
\begin{prop}
$\Lc_{\alpha 0}: (X_2, \stp_{\str(\alpha)/\alpha})\Rightarrow (X_1, \stp_\Sigma)$ is an admissible Lagrangian correspondence.
\label{property:admissibility}
\end{prop}
\begin{proof}
Since $\Lc_{\pi_{\alpha}^*},\wrpalpha( \Lc_{\pi_{\alpha}^*})$ are both admissible, the surgery $\Lc_{\alpha 0}$ is admissible (\cref{prop:generalizedsurgeryprofile}). 
\end{proof}
From here on, for the sake of notational clarity, we write $\wh H_\alpha^\delta$ for $(1-\eps)\wh H_\alpha^\delta$.
From its construction in \cref{prop:generalizedsurgeryprofile}, $\Lc_{\alpha 0}$ comes with two charts which sit inside $Q_1\times Q_2\times \sP_1\times \sP_2$ as 
   \begin{align*}
      \Lc_{r}:=&\left\{\left((q_1,\underline \pi_\alpha(q_1)),(\underline \pi_{\alpha}^* p_2+d\hamK{12}|_1+d(r_{\rho c_\eps} \circ \Halpha),p_2+d\hamK{12}|_2)\right)\; : \; q_1\in Q_1, \Halpha\geq \rho c_\eps \right\} \\
      \Lc_{s}:=&\left\{\left((q_1,\underline \pi_\alpha(q_1)),(\underline \pi_{\alpha}^* p_2+d\hamK{12}|_1+d(s_{\rho c_\eps} \circ \Halpha),p_2+d\hamK{12}|_2)\right)\; : \; q_1\in Q_1, \Halpha\geq \rho c_\eps\right\}\\
   \end{align*}
   which we will call the $r$ and $s$-charts.
   In the same way that the addition of the wrapping Hamiltonian $d\hamK{}$ made using coordinate charts for $\Lc_{f^*}$ a bit unwieldy in computations in \cref{subsec:toricmorphisms}, the presence of the wrapping Hamiltonian can clutter computations here. 
   As before, the $d\hamK{12}$ terms which are required to disjoin $\Lc_{\alpha 0}$ from the stop can be chosen to be very small, and for some computations, we can set them equal to 0 without impacting the result; equivalently, this is the same as working with the Lagrangian $\psi^\eps_{K_{12},\delta}(\Lc_{\alpha 0 })$.
   We work with this simplified Lagrangian frequently enough that we will call it $\tilde \Lc_{\alpha 0 }$; we will similarly use the charts $\tilde \Lc_r, \tilde \Lc_s$.

\begin{prop}
   Let $L\subset (X_2, \stp_{\str(\alpha)/\alpha})$ be an admissible Lagrangian submanifold. Then $\Lc_{\alpha 0} \circ L$ has the topology of $L\times \RR$.
   \label{property:topology}
\end{prop}
\begin{proof}
   This computation can be simplified by using $\tilde \Lc_{\alpha 0 }$ instead of $\Lc_{\alpha 0}$.
   We use the $r$ and $s$-charts to compute the geometric composition, as  $\tilde \Lc_r\circ L $ and  $\tilde \Lc_s \circ  L$ form charts for $\tilde \Lc_{\alpha0}\circ L$.
   By \cref{cor:funnysection}, we can choose a section  $\funnyi: Q_2\into Q_1$ which parameterizes  the $\rho c_\eps$ level set of $\Halpha$. See \cref{fig:conicalrchart}.
   We look at one of these charts:
   \begin{align*}
   \tilde \Lc_r\circ L =&\left\{( q_1,\underline \pi_{\alpha}^*(p_2)+d(r_{\rho c_\eps}\circ \eps \Halpha))\; : \;\Halpha(q_1)\geq \rho c_\eps , (\underline \pi_\alpha(q_1), p_2)\in L\right\}\\
    \intertext{We wish to replace $\underline \pi_{\alpha}(q_1)=q_2$ with $q_1\in \underline \pi_{\alpha}^{-1}(q_2)$. 
    Since the set $\Halpha(q_1)\geq \rho c_\eps$ fibers over the image of $\funnyi$ under the projection $\underline \pi_\alpha$, we can parameterize this preimage as: }
   =& \left\{( \funnyi(q_2)+ \alpha\cdot t,\underline \pi_{\alpha}^*(p_2)+d(r_{\rho c_\eps} \circ \eps \Halpha))\; : \; (q_2, p_2)\in L , t\in \RR_{>0}\right\}. \tag{$r$-parameterization} \label{eq:rparameterization}
   \end{align*}
This chart has the topology of $L\times \RR_{\geq 0}$. 
A similar argument shows that $\tilde \Lc_s \circ L$ has the topology of $L\times \RR_{\leq 0}$. These two charts are identified along their boundaries,  so that $\tilde \Lc_{\alpha 0}\circ L \cong L\times \RR$.
\end{proof}
\begin{rem}
There are a few constructions in the literature similar to  $\Lc_{\alpha 0}\circ L$ .
Particularly relevant examples come from Lagrangian cobordisms and  Lefschetz fibrations.
\begin{itemize}
   \item In Lagrangian cobordisms, there is an inclusion of Fukaya categories:
   \begin{align*}
      \Fuk(X)\to &\Fuk(X\times \CC)\\
      L\mapsto &L\times \gamma
   \end{align*}
   where $\gamma\subset \CC$ is a curve.
   Parallel transport along a $u$-shaped curve around a stop of the cobordism is used in \cite{biran2014lagrangian} to give $\Fuk(X\times \CC)$ the structure of a module over $\Fuk(X)$. 
   \item In Lefschetz fibrations, there is a ``cup functor'' from $\Fuk(W^{-1}(z))\to \Fuk(X, W)$, which includes Lagrangians by parallel transport along a curve in the base of the Lefschetz fibration. It is called the cup functor as the shape of the curve has a $\cup$ shape. This functor is also sometimes referred to as the Orlov functor and the image objects as $u$-shaped Lagrangians. The Orlov functor for stopped Liouville domains where the stop is a smooth hypersurface is well-understood \cite{sylvan2019orlov}. Our construction can be viewed as (a special case of) the Orlov functor for singular stops.
\end{itemize}
Our Lagrangian $\Lc_{\alpha 0}$ gives an interpretation to these constructions as arising from a Lagrangian correspondence.
\end{rem}
\begin{prop}
   Let $L_0\subset (X_1, \stp_\Sigma)$ be the Lagrangian zero section. Then $\Lc_{\alpha 0}^{-1}\circ L_0$ is a Lagrangian zero section of $(X_2, \stp_{\str(\alpha)/\alpha})$.
   \label{property:apullback}
\end{prop} 
\begin{proof}
We compute $ \Lc_{r}^{-1}\circ L_0(0)$ and $ \Lc_{s}^{-1}\circ L_0(0)$.
Recall that $L_0(0)$ is given by taking the zero section and pushing it off slightly by the Hamiltonian flow of $-\eps_1 K_1$. 
For any $c_1>0$, there exists a choice of $\eps_1$ small enough so that  $|d(\eps_1K_1)(\alpha)|<c_1$.
The charts $\Lc_{r}$ and $\Lc_{s}$ are determined by choices of wrapping Hamiltonian $\eps_{12}K_{12}, \Halpha$ and surgery parameter $c_\eps$.
We choose $\eps_{12}$ to be large enough and $c_\eps$ small enough so that over the region $\Halpha\geq \rho\cdot c_\eps$,
\begin{equation}
   d(-\eps_{12}K_{12}+s_{\rho c_\eps} \Halpha)(\alpha)<-c_1. \label{eq:separatingsections}
\end{equation}
For these choices, the $\Lc_{s}^{-1}$ correspondence does not intersect $L_0(0)$, as 
\begin{align*}
 \arg(\pi_1 ( \Lc_s^{-1} ))\subset \{p_1 \text{ such that } p_1(\alpha)<-c_1\} && \arg( L_0(0))\subset \{p_1 \text{ such that }  |p_1(\alpha)|<c_1\}.
\end{align*}
Therefore $ \Lc_{s}^{-1}\circ L_0(0)=\emptyset.$
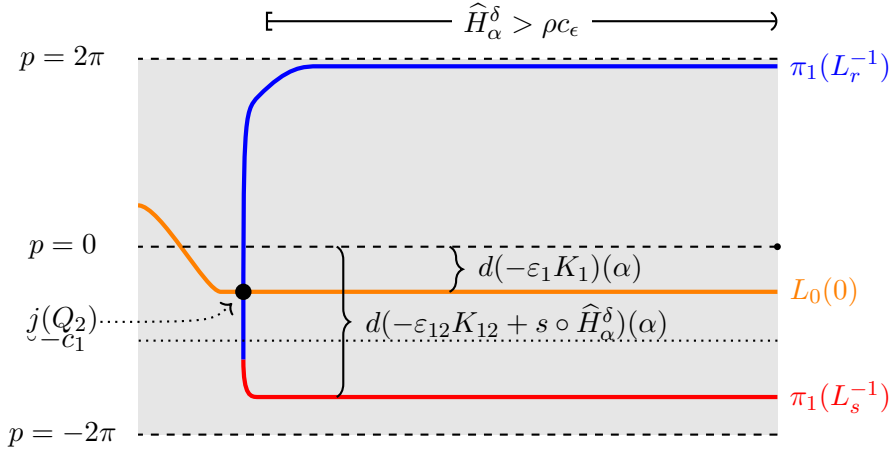
\begin{figure}
   \centering
   \begin{tikzpicture}[]
\usetikzlibrary{decorations.pathreplacing}
\fill[gray!20]  (11,-4) rectangle (2.5,1);

         \node[fill=black, circle, scale=.25] at (11,-1.5) {};

\draw[red, ultra thick] (11,-3.5) .. controls (10.85,-3.5) and (4.2,-3.5) .. (4.1,-3.5) .. controls (4,-3.5) and (3.9,-3.5) .. (3.9,-3);
\draw[blue, ultra thick] (3.9,-3) .. controls (3.9,0.2) and (3.9,0.3) .. (4.1,0.5) .. controls (4.49,0.9) and (4.7,0.9) .. (4.9,0.9) .. controls (5.1,0.9) and (10.81,0.9) .. (11,0.9);
\node[red, right] at (11,-3.5) {$\pi_1( \Lc_s^{-1})$};
\node[blue, right] at (11,0.9) {$\pi_1( \Lc_r^{-1})$};
\node[orange, right] at (11,-2.1) {$L_0(0)$};
\draw[dotted, thick] (11,-2.75) -- (2.5,-2.75);
\node at (1.5,-2.75) {$-c_1$};
\node at (1.5,-1.5) {$p=0$};
\node at (1.5,1) {$p=2\pi$};
\draw[orange, ultra thick] (11,-2.1) .. controls (10.8,-2.1) and (3.8,-2.1) .. (3.6,-2.1) .. controls (3.4,-2.1) and (2.8,-0.95) .. (2.5,-0.95);
\draw[dashed, thick] (11,-4) -- (2.5,-4);
\draw[dashed, thick] (11,-1.5) -- (2.5,-1.5);
\draw[dashed, thick] (11,1) -- (2.5,1);
\node at (1.5,-4) {$p=-2\pi$};
\draw [thick,decorate,decoration={brace, amplitude=5pt,mirror,raise=4pt},yshift=0pt]
(5,-3.5) -- (5,-1.5) node [black,midway,xshift=10pt,right] {$d(-\eps_{12}K_{12}+s\circ\widehat H^\delta_\alpha )(\alpha)$};
\draw [thick,decorate,decoration={brace, amplitude=5pt,mirror,raise=4pt},yshift=0pt]
(6.5,-2.1) -- (6.5,-1.5) node [black,midway,xshift=10pt,right] {$d(-\eps_{1}K_{1})(\alpha)$};

\draw[thick,[-)] (4.2,1.5) -- (11,1.5) node[midway, fill=white] {$\widehat H_\alpha^\delta > \rho c_\epsilon$};
\node[fill=black, circle, scale=.6] at (3.9,-2.1) {};
\node at (1.5,-2.5) {$\funnyj(Q_2)$};
\draw[thick, dotted, ->] (2,-2.5) .. controls (3.5,-2.5) and (3.5,-2.5) .. (3.75,-2.25);
\end{tikzpicture}    \caption{A $\CC^*=\{(q\alpha, p\alpha)\}/(p\sim p+2\pi )$ slice of $X_1$, exhibiting the choices of $c_\eps, \eps_{12}, \eps_1$.}
   \label{fig:intersectionofsection}
\end{figure}
It remains to compute $ \Lc_{r}^{-1}\circ L_0(0)$: 
\begin{align*}
    \Lc_{r}^{-1}\circ L_0(0)=& \pi_2\circ\left(\left\{(q_1, q_2, 0,p_2)\;\middle|\;
   \begin{array}{c} 0=p_1'+p_1'', p_2=p_2'+p_2''\\
      p_1'=-\underline \pi_{\alpha}^*(p)+d(-\eps_{12}K_{12})|_1+d(r_{\rho c_\eps} \Halpha) \\
      p_2=p+d(-\eps_{12}K_{12})|_2, p\in \sP_2\\
      q_2= f(q_1), \Halpha(q_1)\geq \rho c_\eps\\
      p_1''=d(-\eps_1K_1),p_2''=0, q_2\in Q_2
   \end{array}\right\}\right)\\
   \intertext{We examine the condition $(\underline \pi_{\alpha}^*(p)+d(-\eps_{12}K_{12})|_1+d(r_{\rho c_\eps} \Halpha) =d(-\eps_1K_1)$.
   This states that $d(-\eps_{12}K_{12})|_1+d(r_{\rho c_\eps} \Halpha)-d(-\eps_1K_1)$ is the pullback of a differential form from $\sP_2$; equivalently, that $d(-\eps_{12}K_{12})|_1+d(r_{\rho c_\eps} \Halpha)+d(\eps_1K_1)$ vanishes on $\alpha$.}
   =&\left\{(p_1, \underline \pi_\alpha(q_1))\;\middle|\;
   \begin{array}{c}  p_1=d(-\eps_{12}K_{12})|_1+d(r_{\rho c_\eps} \Halpha)+d(\eps_1K_1)\\
      p_1(\alpha)=0, p_2=\underline \pi_\alpha^*(p), \Halpha(q_1)\geq \rho c_\eps
   \end{array}\right\}
\end{align*}
This reduces us to computing the locus on which $(d(-\eps_{12}K_{12})|_1+d(r_{\rho c_\eps} \Halpha)+d(\eps_1K_1))(\alpha)=0.$

We now construct  $\funnyj: Q_2\to Q_1$, a section of $\underline \pi_{\alpha}:Q_1\to Q_2$ which parameterizes the locus $\{q_1\; : \;(d(-\eps_{12}K_{12})|_1+d(r_{\rho c_\eps} \Halpha)+d(\eps_1K_1))(\alpha)=0\}$.  
This is equivalent to showing that the function $r_{\rho c_\eps} \circ \Halpha(q+t\cdot \alpha)$ is \emph{convex} in $t$ when $d(r_{\rho c_\eps} \circ \Halpha))(\alpha)=0$.
By \cref{lemma:positiveintersectionlemma}, $\frac{d}{dt}\left(\Halpha(q+t\cdot \alpha)\right)\geq 0$ in this region.
We have that 
\begin{align*}
   \frac{d^2}{dt^2} (r_{\rho c_\eps} \circ \Halpha(q+t\cdot \alpha))=& r_{\rho c_\eps}'' \cdot  \frac{d}{dt}\Halpha(q+t\cdot \alpha) +  r_{\rho c_\eps}' \cdot  \frac{d^2}{dt^2}\Halpha(q+t\cdot \alpha)
\end{align*}
In the region near the level set, $r_{\rho c_\eps}''$ is arbitrarily large (see profile curve \cref{fig:profilecurves}), $\frac{d}{dt}\Halpha(q+t\cdot \alpha)$ is positive, and $ r_{\rho c_\eps}'$ and $\frac{d^2}{dt^2}\Halpha(q+t\cdot \alpha)$ are bounded. See \cref{fig:intersectionofsection} to see the set parameterized by $\funnyj(Q_2)$

We obtain a parameterization of the geometric composition:
\begin{equation} \Lc_{r}^{-1}\circ L_0(0)= \left\{(q_2, p)\;\middle|\;
\begin{array}{c}  p_1=d(\hamK{12}|_1+(r_{\rho c_\eps} \Halpha)-\hamK{1})\\
   q_1=\funnyj(q_2), p_2=\underline \pi_\alpha^*(p)
\end{array}\right\}
\label{eq:funnyjishere}
\end{equation}
which is a Lagrangian section admissibly Hamiltonian isotopic to the zero section.
\end{proof}
\begin{prop}
   Let $\sF\in \Supp_\alpha(\Sigma)$ be a support function transverse to the divisor $D_\alpha$. 
   There exists a Lagrangian cobordism in $(X_1, \stp_\Sigma)\times \CC$, whose topology is that of a disk and whose ends are 
   \begin{equation}
      (L_0(\sF+\sF_\alpha), L_0(\sF))\rightsquigarrow \Lc_{\alpha 0}\circ L_\alpha(\underline i^*_{\alpha0}\sF) \label{eq:brelation}
   \end{equation}
   \label{property:aexactsequence}
\end{prop}
\begin{proof}
We combine \cref{def:inclusioncorrespondence}, which describes the surgery used to build our Lagrangian correspondence, and \cref{claim:correspondenceofcobordism}, which proves that Lagrangian cobordism and geometric correspondence can be interchanged.
We first consider the setting where $\check X_{\Sigma}=\check X_{\str(\alpha)}$. 
We can construct a sequence of Lagrangian cobordisms
\begin{align*}
   \left( \wrpalpha(\Lc_{\pi_{\alpha}^*}),\Lc_{\pi_{\alpha}^*}\right)\rightsquigarrow& \Lc_{\alpha 0} \tag{\cref{def:inclusioncorrespondence}}\\
   \left( \wrpalpha(\Lc_{\pi_{\alpha}^*})\circ L_{\alpha}(\underline i^*_{\alpha0} \sF), \Lc_{\pi_{\alpha}^*} \circ  L_{\alpha}(\underline i^*_{\alpha0} \sF) \right)\rightsquigarrow& \Lc_{\alpha 0}\circ L_{\alpha}(\underline i^*_{\alpha0} \sF)\tag{\cref{claim:correspondenceofcobordism}}\\
   \intertext{By \cref{property:intertwining}, the pullbacks $\wrpalpha(\Lc_{\pi_{\alpha}^*})\circ L_{\alpha}(\underline i^*_{\alpha0} \sF)$ and $\Lc_{\pi_{\alpha}^*} \circ  L_{\alpha}(\underline i^*_{\alpha0} \sF)$ are the desired Lagrangian sections of $(X_1, \stp_{\str(\alpha)})$ giving}
   ( L_0(\sF+\sF_\alpha),L_0(\sF))\rightsquigarrow & \Lc_{\alpha 0}\circ L_\alpha(\underline i^*_{\alpha0}\sF).\tag{\cref{claim:intertwining}}
\end{align*}
To obtain the same statement when $\check X_{\Sigma}$ is a completion of $X_{\str(\alpha)}$, we na\"ively include $(X_1, \stp_{\str(\alpha)})$ into $(X_1, \stp_\Sigma)$. This normally cannot be done (as stop addition usually does not preserve admissibility) however we have already checked that all of the Lagrangians considered are disjoint from $\stp_{\Sigma}$.
\end{proof}
\begin{prop}
   Let $\sigma>\alpha$ be a cone, and let $\LL_{\underline i_{\alpha 0}^*\sigma}\subset \LL_{\str(\alpha)/\alpha}$ be the stratum of the skeleton associated to $\sigma$. Let $L\subset (X_2, \LL_{\str(\alpha)/\alpha})$ be an admissible Lagrangian intersecting $\LL_{\underline i_{\alpha 0}^*\sigma}$ transversely at $k$ points.
   Then $(\Lc_{\alpha 0}\circ L)$ intersects $ \LL_{\sigma}$ transversely at $k$ points.
\label{property:intersections}
\end{prop}
\begin{proof}
The proof is very similar to that of \cref{claim:toricpullback}.
We use \cref{eq:rparameterization}, and the analogous parameterization of the $s$-chart.
\begin{align*}
   \Lc_r\circ L=&\{(\funnyi(q_2)+ \alpha\cdot t, \underline \pi_{\alpha}^*(p_2)+d\hamK{12}|_1+ d(r_{\rho c_\eps} \circ \Halpha), (q_2, p_2)\in L ,t\in \RR_{\geq 0}\}\\
   \Lc_s\circ L=&\{(\funnyi(q_2)+ \alpha\cdot t, \underline \pi_{\alpha}^*(p_2)+d\hamK{12}|_1+ d(s_{\rho c_\eps} \circ \Halpha)), (q_2, p_2)\in L,t\in \RR_{\geq 0} \}
\end{align*}
As the condition that $\funnyi(q_2)+ \alpha\cdot t \in \sigma$ holds for all $q_2$ and $t$,  the intersection between $\Lc_r\circ L$ and $\LL_\sigma$ is given by the points $(q_2, p_2)\in L$ satisfying
\[\{(q_2, p_2) \; : \; \forall \vec v \in \sigma, (\underline \pi_\alpha^*(p_2)+d\hamK{12}|_1+d(r_{\rho c_\eps} \circ \Halpha))(\vec v)=0\}.\] 

When restricted to the image of $\funnyj$ (from \cref{eq:funnyjishere}), we can find $\tilde H_2: Q_2 \to \RR$ so that 
\[\left. d\hamK{12}|_1+d(r_{\rho c_\eps} \circ \Halpha)\right|_{\funnyj(x)}=  \underline \pi_\alpha^*  d\tilde H_2 |_x .\] 
We note that for any compact set $V\subset Q_2$, we can choose $\eps$ small enough and $r_{\rho c_\eps}$ so that $|\tilde H_2|, \|d\tilde H_2\|$ are as small as desired. Geometrically, this means that $L_{r} \circ L_{\alpha}(0)$ can be made to approximate $\LL_\alpha$ very closely.
We can describe the intersection between $\Lc_r\circ L$ and the skeleton component $\LL_\sigma$ by:
\begin{align*}
   \{(q_2, p_2)\; : \; \forall \vec v\in \sigma, & (\underline \pi_{\alpha}^*(p_2)+ \underline \pi_{\alpha}^*d \tilde H_2 )(\vec v)=0,  (q_2, p_2)\in L \}\\& = \{(q_2, p_2)\; : \;\forall w\in \underline i_{\alpha 0}^*\sigma, (p_2+d\underline \pi_\alpha^*\tilde H_2)( w)=0 , (q_2, p_2)\in L\}.
\end{align*}
This is the intersection between the $\tilde H_2$ perturbed $L$, and $\LL_{\underline i_{\alpha 0}^*\sigma}$. 
Because we can choose $\tilde H_2$ to be very small, this is in bijection with intersections of $L$ and $\LL_{\underline i_{\alpha 0}^*\sigma}$.

For the $s$-chart, we note that the bound from \cref{eq:separatingsections} implies that $\Lc_s \circ L$ is disjoint from the stop. 
\end{proof}
\begin{prop}
   \label{property:linkingdisks}
   Let $\sigma>\alpha$ be a cone, and let $(q,p)\in  \LL_{\underline i_{\alpha 0}^*\sigma}\subset \LL_{\str(\alpha)/\alpha}$ be a point away from the singular locus of $\LL_{\str(\alpha)/\alpha}$, giving us a linking disk $\cocore(x)$. 
   The Lagrangian $L_{\alpha 0 }\circ \cocore(x)$ is a  linking disk or cocore for $\LL_{\sigma}$.
\end{prop}
\begin{proof}
   By \cref{property:intersections} we know that $L_{\alpha 0}\circ \cocore(x)$ intersects $\LL_{\sigma}$ transversely at a single point.
   It remains to show that $L_{\alpha 0}\circ \cocore(x)$ is disjoint from all other skeleton components $\LL_\tau$.
   This can be achieved by computing the valuation and argument  of $L_{\alpha 0}\circ \cocore(x)$.
   We break into two cases: $\sigma<\tau $ or $\sigma \not \leq \tau$. Denote the $Q_2\times \sP_2$ coordinates of $x$ by $(q,p)$
   \begin{itemize}
      \item If $\sigma<\tau$, then $\alpha<\tau$. See \cref{fig:pushforwardlinkingdisk}. Because $p$ was chosen outside of the singular locus, $p\in \underline i^*_{\alpha 0}(\sigma^\bot) \setminus \underline i^*_{\alpha 0}(\tau^\bot).$ Then $\underline \pi_\alpha^*(p)\not\in \tau^\bot$, which we will show is enough to separate $L_{\alpha 0}\circ \cocore(x)$ from $\LL_\tau$.
      We will compute  $\arg(\Lc_r \circ \cocore(x))$ to show that these two sets are disjoint.
      \[\arg(\Lc_r \circ \cocore(x))= \{\underline \pi_{\alpha}^*(p_2)+d\hamK{12}|_1+ d(r_{\rho c_\eps} \circ \Halpha), (q_2, p_2)\in \cocore(x)\}.\]
      Since $\cocore(x)$ is a cocore of $\LL_{\underline i_{\alpha 0}^*\sigma}$, the image of $\arg(\cocore(x))\subset \sP_2$ is contained within a neighborhood of the point $p$, which is an interior point of  $\LL_{\underline i_{\alpha 0}^*\sigma}$:
      \[\{p_2\; : \; (q_2, p_2)\in \cocore(x)\}\subset B_\eps(p).\]
      It follows that the image     $\underline \pi_{\alpha}^*(\arg(\cocore(x)))$  is contained within a neighborhood of $\underline \pi_{\alpha}^*(p)$: 
      \[\{\underline \pi_{\alpha}^*(p_2)\;:\; (q_2, p_2)\in \cocore(x)\}\subset B_\eps (\pi_{\alpha}^* (p))).\]
     Since $d\hamK{12}|_1$ is small, $|d(r_{\rho c_\eps} \circ \Halpha)(\alpha)|<2\pi$, and $d(r_{\rho c_\eps} \circ \Halpha)(v)$ is approximately $\alpha\cdot v$, we obtain that $\arg(\Lc_r \circ \cocore(x))$ is contained within a small neighborhood $(\underline \pi^*_\alpha(p)+t\alpha)$. This is disjoint from $\tau^\bot$, so $\Lc_{0\alpha}\circ\cocore(x)$ is disjoint from $\LL_\tau$.
      \item If $\sigma \not \leq \tau$, then observe that 
      \[\val(\tilde \Lc_r  \cocore(x)) =   \funnyi(q_2)+ \alpha\cdot \RR_{>0}\]
      is contained within the interior of $\str(\sigma)$. Since $\sigma \not \leq \tau$, $\val(\LL_\tau)=\tau$ does not intersect the interior of $\str(\sigma)$.
   \end{itemize}
   A similar argument holds for $\Lc_s$.
\end{proof}
\begin{figure}
   \centering
   \begin{tikzpicture}
\usetikzlibrary{calc, decorations.pathreplacing,shapes.misc}
\usetikzlibrary{decorations.pathmorphing}

\tikzstyle{fuzz}=[
    postaction={draw, decorate, decoration={border, amplitude=0.15cm,angle=90 ,segment length=.15cm}},
]

   \begin{scope}[shift={(0.5,1.5)}]

\fill[fill=red!20]  (-2.5,1) rectangle (-0.5,0.5);
\fill[fill=red!20]  (-3,1) rectangle (-3.5,0.5);
\draw[fuzz] (-3.5,0) -- (-0.5,0);
\draw[dotted]  (-3.5,1.5) rectangle (-0.5,-1.5);
\draw[fuzz, orange] (-2,1.5) -- (-2,-1.5);
\draw[fuzz] (-0.5,-1.5) -- (-3.5,1.5);

\begin{scope}[shift={(-2,0)}]
\foreach \i in {6,...,12}
{
        \pgfmathtruncatemacro{\y}{15* \i };
        \draw (0,0)-- (\y:.2) ;
}

\node at (0,0) {};
\end{scope}

\begin{scope}[shift={(-2,0)}]
\foreach \i in {12,...,24}
{
        \pgfmathtruncatemacro{\y}{15* \i };
        \draw (0,0)-- (\y:.2) ;
}

\node at (0,0) {};
\end{scope}

\begin{scope}[shift={(-2,0)}]
\foreach \i in {0,...,6}
{
        \pgfmathtruncatemacro{\y}{15* \i };
        \draw[blue] (0,0)-- (\y:.2) ;
}

\node at (0,0) {};
\end{scope}
\end{scope}

\draw[gray!20,ultra thick] (5,0) -- (5,3.5);

\node at (5,1.5) {};
\draw (5,1.35) -- (5,1.65) ;
\node[circle, fill, scale=.5] at (5,1.5) {};
\draw[red] (5,2.65) -- (5,1.85);
\draw[ultra thick, red!20] (5,2.5) -- (5,2);
\draw[red, fuzz] (-2,2.5) -- (-2,2);
\draw[red, fuzz] (-2,2.5) -- (0,2.5) (-3,2.5) -- (-2.5,2.5) -- (-2.5,2) -- (-3,2) (0,2) -- (-2,2);

\fill[fill=gray!20]  (0,7) rectangle (-3,4);
\fill[blue!20]  (0,7) rectangle (-1.5,5.5);
\draw[->] (5,5.5) -- (5,7);
\draw[->] (5,5.5) -- (5,4);
\draw[->] (4.5,1.5) -- (1,1.5) node[midway, fill=white]{$\underline \pi^*_\alpha$};
\draw[->] (1,5.5) -- (4.5,5.5) node[midway, fill=white]{$\underline \pi_\alpha$};
\draw[->] (-1.5,5.5) -- (-1.5,7);
\draw[ultra thick, orange,->] (-1.5,5.5) -- (0,5.5) node[right] {$\sigma$};
\draw[->] (-1.5,5.5) -- (-3,4);
\node at (-1,6) {$\tau$};
\node at (-2.5,6.5) {$Q_1$};
\node at (4.5,6.5) {$Q_2$};
\node at (-2.5,0.5) {$F_1$};
\node at (4.5,0.5) {$F_2$};
\node[right] at (5,2.25) {$\mathfrak u(x)$};

\node[right] at (-1.575,2.25) {\small$L_{0\alpha}\circ\mathfrak u(x)$};
\node[circle, fill, scale=.5] at (-1.5,5.5) {};
\node[circle, fill, scale=.5] at (5,5.5) {};
\end{tikzpicture}    \caption{The pushforward of a linking disk is a linking disk.}
   \label{fig:pushforwardlinkingdisk}
\end{figure}
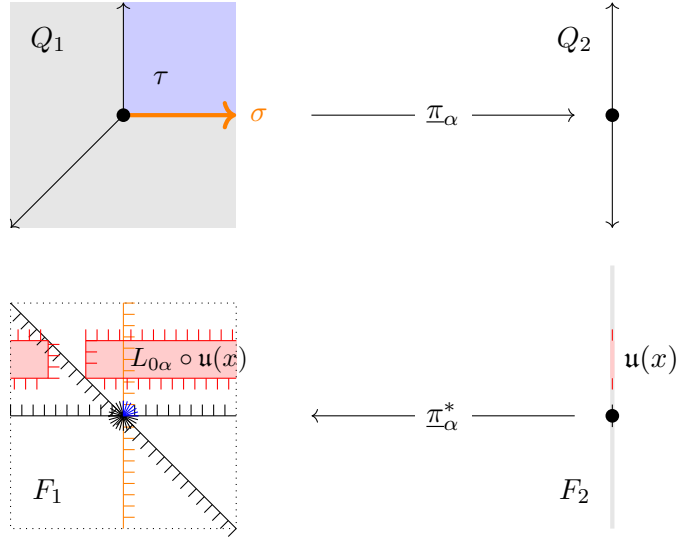
\begin{rem}
   The picture from \cref{fig:afterhamiltonian} can lead to the misconception that the Lagrangian $L_{\alpha>0}: =L_{\alpha 0} \circ L_\alpha(0)$ is the linking disk to $\LL_\alpha$. 
   This is usually \emph{not} a linking disk, as the Lagrangian $L_{\alpha}(0)$ is not a cocore or linking disk of the skeleton of $\LL_{\str(\alpha)/\alpha}$ as the origin is contained in the singular locus. After applying a Hamiltonian perturbation, $L_{\alpha}(0)$ is seen to hit several skeleton components.
   Furthermore, it is inaccurate to say that $\LL_\alpha$ has \emph{a} linking disk, as the linking disks of $\LL_\alpha$ are in general not admissibly Hamiltonian isotopic.
   
   In the particular case where $\check X_\Sigma=\CP^1$, it is a coincidence that $L_\alpha(0)\subset X_{\str(\alpha)/\alpha}=\{pt\}$ is a point, so it trivially is a linking disk. Therefore its pushforward as drawn in \cref{fig:afterhamiltonian} is the  linking disk.
   \label{rem:whythingsarenotlinkingdisks}
\end{rem}
\begin{cor}
   Let $\alpha\in \Sigma$ be a one-dimensional cone, and $\sigma\in \str(\alpha)$ a cone. Let $x\in \LL_\sigma$ be a point of the FLTZ skeleton that does not lie in the singular locus.
   Then there exists a linking disk which is admissibly Hamiltonian isotopic to $\cocore(x),$ and is contained in the image of $L_{\alpha 0}$.  
   \label{cor:linkingdisk}
\end{cor}
\begin{proof}
   Let $\gamma: [0,1]\to \LL_\sigma$ be a path. 
   The Lagrangian linking disks  $\cocore(\gamma(t))$ are Lagrangian isotopic if the path $\gamma$ is disjoint from $\LL_{\tau}$, with $\tau>\sigma$ and $|\tau|=|\sigma|+1$.
   Since $\sigma$ is contained in the star of $\alpha$, $\tau>\sigma\geq\alpha$. 
   Therefore, the line $(x_0+\alpha \RR)\cap \LL_\sigma$ is disjoint from all of the $\LL_{\tau}$.
   Since $(x_0+\alpha \RR)\cap \LL_\sigma$ intersects every section of $\underline \pi_\alpha: Q_1\to Q_2$, it follows that there always exists a path from $x$ to the image of $\funnyj: Q_2\to Q_1$ which is disjoint from the singular locus of the FLTZ skeleton.
   We may therefore assume that after taking an admissible Hamiltonian isotopy, the cocore $\cocore(x)$ is centered at a point $x=(q_1, p_1)$, where $q_1$ is contained within the image of $\funnyj:  Q_2\to Q_1$.
   Let $q_2$ be the point in $Q_2$ so that $\funnyj (q_2)=q_1$.
   Since $x\in \LL_\sigma$, $p_1\cdot \alpha=0$ and there exists $p_2$ so that $\underline \pi_\alpha^*p_2=p_1$.
   The point $x'=(q_2, p_2)\in X_2$ lies in the FLTZ component $\LL_{\underline i_{\alpha 0}^*\sigma}$.
   Construct a linking disk $\cocore_{\LL_{\str(\alpha)/\alpha}}(x')$ for this point.
   The pushforward $L_{\alpha 0} \circ (\cocore_{\LL_{\str(\alpha)/\alpha}}(x')$ is a linking disk for $\LL_{\sigma}$, which by design intersects $\LL_{\sigma}$ at $x$.
\end{proof}

\section{Applications} \label{sec:applications}
We now look at applications of the constructions from \cref{sec:embedding,sec:inclusion}. The main idea is to combine \cref{cor:hmsline}, which proves homological mirror symmetry for the subcategory generated by tropical Lagrangian sections, with \cref{thm:inclusioncorrespondence}, which constructs Lagrangian correspondences acting on those tropical Lagrangian sections. 

In \cref{sec:generation}, we leverage those two theorems to show that tropical Lagrangian sections generate the partially wrapped Fukaya category. This proves \cref{cor:HMScor}, identifying the partially wrapped Fukaya category of $((\CC^*)^n, \stp_\Sigma)$ with the derived category of coherent sheaves on smooth projective $\check X_\Sigma$.
We look at the specific example of $\CP^n$ in \cref{subsubsec:cpngeneration}, recovering the mirror to the Beilinson exceptional collection.
We then examine how the stop removal functor on $\stp_\Sigma$ behaves under this mirror identification.
In \cref{subsec:quotient,subsec:blowup}, we use our matching of linking disks to line bundles supported on toric strata to identify the mirror operations to localization away from a toric orbit of $\check X_\Sigma$ and toric blow-up of $\check X_\Sigma$.
The former allows us to deduce homological mirror symmetry for smooth quasi-projective toric varieties. \subsection{Generation by tropical sections}
\label{sec:generation}
   In this section, we show that the partially wrapped Fukaya category, $\mathcal W ((\CC^*)^n, \stp_\Sigma)$ is generated by the tropical sections $L_0(\sF)$.
\begin{df}
   Let $\mathcal C$ be a triangulated category. We say that objects $A_1, \ldots, A_k\in \mathcal C$ generate $A_0\in \mathcal C$ if there exist objects $Z_1, \ldots Z_k\in \mathcal C$, and exact triangles 
   \[A_i \to Z_{i-1}\to Z_i \;\; \forall 2\leq i \leq k\]
   with $Z_1=A_1$ and $Z_k=A_0$.
   If every object of a subcategory $\mathcal C'$ can be generated by objects $A_1, \ldots, A_k$, we write $\mathcal C'= \langle A_1, \ldots, A_k\rangle.$
\end{df}

\begin{thm}
   Tropical sections generate $\mathcal W((\CC^*)^n, \stp_\Sigma)$.
   \label{thm:ageneration}
\end{thm}

The idea of the proof is as follows. We will use an inductive argument with \cref{thm:inclusioncorrespondence} to show that there are Lagrangian cobordisms with left ends Lagrangian sections $L(\sF)$ and right end a given linking disk $\cocore(x_i)$ with $x_i\in \LL_\sigma$ and $\sigma>\alpha$.
By \cref{thm:cobordismgeneration}, this means that Lagrangian sections generate the linking disks of $((\CC^*)^n, \stp_\Sigma)$. We then use  \cite[Theorem 1.1]{ganatra2018sectorial} to obtain generation of $\mathcal W((\CC^*)^n, \stp_\Sigma)$.
We obtain homological mirror symmetry by identifying the category generated by Lagrangian sections with the category of line bundles on the mirror, and \cite[Proposition 1.3]{abouzaid2009morse} which states that line bundles generate $D^b\Coh(\check X_\Sigma)$. A summary of the steps is given in  \cref{fig:outlineofgeneration}. 

\begin{figure}
   \centering
   \begin{tikzpicture}

\node at (2,0.5) {$D^b\text{Coh}(\check X_\Sigma)$};
\node at (2,2.5) {$\langle  (i_{\sigma,0})_*\mathcal O_{\sigma>0}\rangle$};
\node at (2,4.5) {$\langle \mathcal O(\sF)\rangle$};
\node at (-3.5,2.5) {$\langle \cocore(x) \rangle$};
\node at (-3.5,0.5) {$\mathcal{W}((\mathbb C^*)^n, \stp_\Sigma$)};
\node at (-3.5,4.5) {$\langle L(\sF)\rangle$};
\draw[<->] (-2,4.5) -- (0.5,4.5) node [midway, fill=white] {\cite{abouzaid2009morse}};
\draw[->] (2,4) -- (2,3) node [midway, fill=white] {\Cref{claim:bgeneration}};
\draw [->] (2,2) -- (2,1)node [midway, fill=white] {\cite[Prop. 1.3]{abouzaid2009morse}};
\draw[->]  (-3.5,4) -- (-3.5,3)node [midway, fill=white] {\Cref{cor:linkingdisk}, \cite{biran2014lagrangian}};
\draw[->]  (-3.5,2) -- (-3.5,1)node [midway, fill=white] {\cite[Thm. 1.1]{ganatra2018sectorial}};
\draw[<->,dashed] (-2,0.5) -- (0.5,0.5);
\end{tikzpicture}    \caption{Theorems for generation of $D^b\Coh(\check X_\Sigma)$ and its mirror.}
   \label{fig:outlineofgeneration}
\end{figure}

\Cref{claim:bgeneration} is a standard fact for toric varieties, whose proof is in the same spirit as the proof of \cref{thm:ageneration}.
\begin{prop}[B-generation]
   Suppose that the following statement holds for all smooth toric varieties $\check X_\Sigma$ with $\dim(\Sigma)<n$: For every cone $\sigma\in\Sigma$ and line bundle $\mathcal O_\sigma (\sF)$ on $\check X_{\str(\sigma)/\sigma}$, there are line bundles $\mathcal O_0(\sF_1), \ldots, \mathcal O_0(\sF_{2^{|\Sigma|-|\sigma|}})$ generating $ \mathcal O_{\sigma>0}(\sF)$ in $D^b\Coh(\check X_\Sigma)$.

   Then the same statement holds for all smooth toric varieties $\check X_\Sigma$ with $\dim(\Sigma)=n$.
   \label{claim:bgeneration}
\end{prop}

\begin{proof}
   Let $\sigma\in\Sigma$ be a cone, $|\Sigma|=n, |\sigma|=k$ and $\sF\in \Supp_\sigma(\Sigma)$ be a support function whose support is transverse to $\sigma$.
   Let $\alpha>\sigma$ be a 1-dimensional cone.
   Recall that  $\mathcal O_{\sigma>0}(\underline i^*_{\sigma 0}\sF):=(i_{\alpha0}  )_*(\mathcal O_{\sigma>\alpha}(\underline i_{\sigma 0}^*\sF)).$
   By the inductive hypothesis, there exists support functions $\tilde \sF_1, \ldots, \tilde \sF_{2^{n-1-k}}\in \Supp(\str(\alpha)/\alpha)$ with $\sF(\alpha)\neq 0$ only if $\alpha'\in \underline i^*_{\sigma 0} \sigma$ , so that $\{\mathcal O_{\alpha}(\sF_i)\}$ generate $\mathcal O_{\sigma>\alpha}(\sF)\in D^b\Coh(X_{\str(\alpha)/\alpha})$. 
   Take  $ \sF_1, \ldots,  \sF_{2^{n-1-k}}\in \Supp_\alpha(\Sigma)$ so that $ (\underline i_{\alpha0})^*( \sF_i)=\tilde \sF_i$.
   The exact sequence 
   \[
      \mathcal O_0( \sF_i+\sF_\alpha)\to \mathcal O_0( \sF_i)\to (i_{\alpha 0})_*(O_{\alpha}(\tilde \sF_i))
   \]
   shows that the line bundles with support functions $ \{\sF_i, \sF_i+\sF_\alpha\}_{i=1}^{2^{n-1-k}}$ generate $\mathcal O_{\alpha>0}(\tilde \sF_i)$.
   As triangles in the derived category are preserved by $(i_{\alpha0})_*$, we get that $\{\mathcal O_0(\sF_i), \mathcal O_{0}(\sF_i+\sF_\alpha)\}_{i=1}^{2^{n-1-k}}$ generate  $\mathcal O_{\sigma>0}(\underline i^*_{\sigma 0}\sF)$ as desired. 
\end{proof}
Another way to see that  $\mathcal O_{\sigma}(\sF)$ can be generated from line bundles is to build them entirely inside $D^b\Coh(\check X_\Sigma)$ by  iterating the exact sequence:
\begin{equation}
   \mathcal O_{\sigma}(\sF+\sF_\alpha)\to \mathcal O_{\sigma}(\sF)\to \mathcal O_{\tau}(\underline i^*_{\tau\sigma} \sF)
   \label{eq:orbitrelation}
\end{equation}
where $\sF\in \Supp_{\tau}(\str(\sigma)/\sigma)$, and $\langle \tau,\alpha\rangle= \sigma$ with $\alpha \not\in \tau$.
Understanding the mirror to this exact sequence is the basis for the (motivational and expository) \cref{subsec:warmupcp2}, where we build the linking disks in the mirror of $\CP^2$ using piecewise linear representations of Lagrangian sections and linking disks. 
Later, in \cref{subsec:generationproof}, we prove \cref{cor:linkingdisk}, which is the $A$-model analogue to \cref{claim:bgeneration}. 
\subsubsection[Warm up: generating mirror to proj. plane]{Warm up: Generation of the mirror to $\CP^2$.} 
\label{subsec:warmupcp2}
We now explain \cref{fig:generationexample}, which demonstrates how tropical Lagrangian sections generate linking disks in the mirror to $\CP^2$. 
The basic principle is to use Lagrangian surgery to inductively remove intersection points with the FLTZ skeleton. 
We consider the fan whose one-dimensional cones are $e_1, e_2$ and $-e_1-e_2$. 
The Lagrangian $L(-2)$ intersects\footnote{As drawn, $L(-2)$ is not a cocore. However, it is Hamiltonian isotopic to a cocore.} the FLTZ skeleton components $\LL_0$ and $\LL_{-e_1-e_2}$, while the Lagrangian $L(-1)$ is a cocore for $\LL_0$.
The surgery of these two Lagrangian submanifolds yields a Lagrangian which only intersects $\LL_{-e_1-e_2}$; a candidate linking disk. 
\begin{figure}
   \centering
   \begin{subfigure}{\linewidth}
      \centering
   \includegraphics{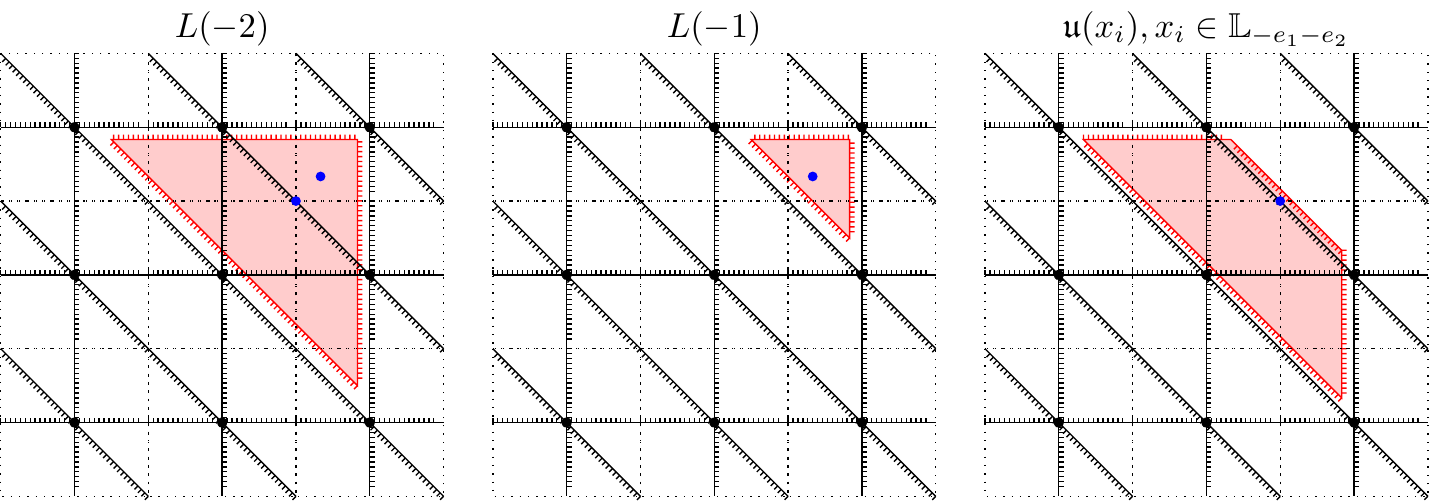}
   \caption{}
   \end{subfigure}
   \begin{subfigure}{\linewidth}
   \centering
   \includegraphics{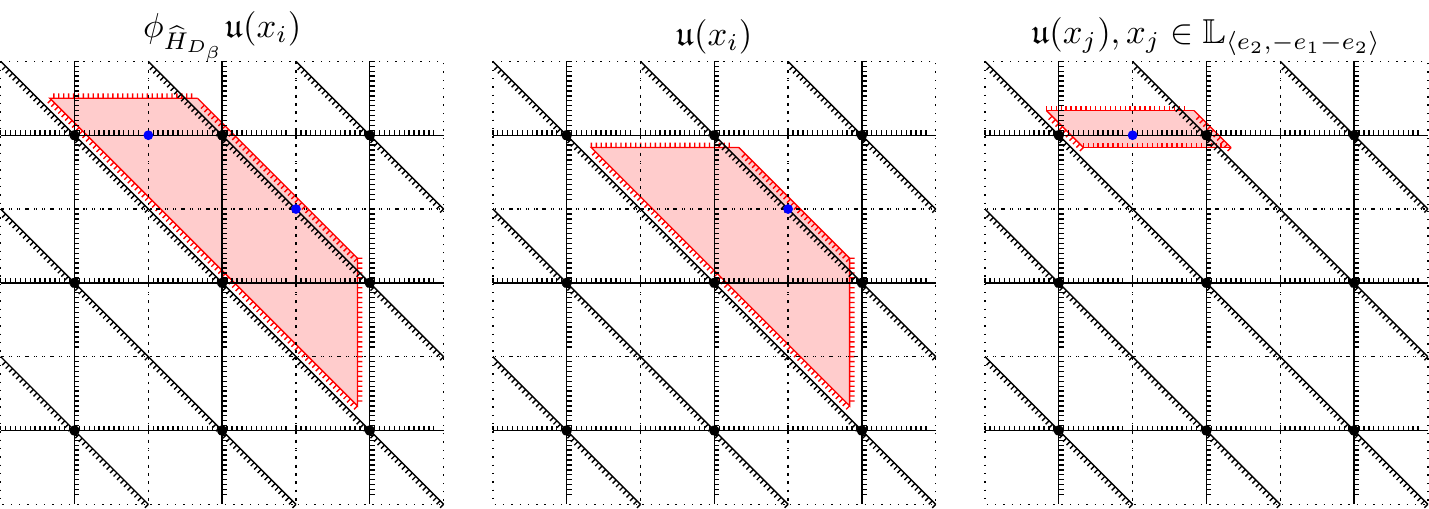}
   \caption{}
   \label{fig:generationexample2}
   \end{subfigure}
   \caption{Generating a linking disk for the mirror of $\CP^2$ from twists of linking disks. Here $x_i\in \LL_\alpha$, $\beta\neq \alpha \in A$ and $x_j$ is in $\LL_{\langle \alpha, \beta\rangle}$.}
   \label{fig:generationexample}
\end{figure}
This process will generate all of the linking disks of the $\LL_\alpha$ for $\alpha \in A$. However, to generate the linking disks to strata indexed by larger cones, one has to repeat this process with twists of $\cocore(x_i)$, where $x_i\in \LL_\alpha$. 
For example \cref{fig:generationexample2}, shows how the linking disk to $\LL_{\langle\alpha_2, -\alpha_1-\alpha_2\rangle}$ is generated, where $\langle\alpha_2, -\alpha_1-\alpha_2\rangle$ is a maximal cone. 
These pictures, while instructional, quickly become unwieldy; furthermore, checking that these Lagrangian disks all intersect appropriately to generate surgery data is difficult. 
For those reasons, we instead use the machinery developed in \cref{sec:inclusion} to build the Lagrangian disks. 
\subsubsection{Generation}
\label{subsec:generationproof}
 
\begin{prop}[A-generation]
   Let $\sigma\in \Sigma$ be a cone with $|\sigma|=k$.
   Let $x\in \LL_\sigma$ be a internal point of the FLTZ skeleton.
   Then there exists Lagrangian sections $L_0(\sF_1), \ldots, L_0(\sF_{2^k})$ and a Lagrangian cobordism $K:(L_0(\sF_1), \ldots, L_0(\sF_{2^k}))\rightsquigarrow \cocore(x_i)$, where $\cocore(x_i)$ is the linking disk to $\LL_{\sigma}$. Furthermore, $K$ has the topology of a disk.
   \label{prop:ageneration}
\end{prop}
\begin{proof}
   We prove by induction on the dimension of the fan. 
   If $\check X_\Sigma$ is 0 dimensional, then the FLTZ stop is empty, concluding the proof. 

   Now suppose that $((\CC^*)^n, \stp_\Sigma)$ is the mirror to a toric variety. 
   Suppose that $|\sigma|\geq 1$.  Then there exists a 1-dimensional cone $\alpha$ with $\alpha \leq \sigma$. 
      By the induction hypothesis, there is a Lagrangian cobordism 
      \[K': ( L_{\alpha}(\sF_1), \ldots, L_{\alpha}(\sF_{2^{k-1}}))\rightsquigarrow L_{\sigma>\alpha}\]
      where $L_{\sigma>\alpha}$ is the linking disk to  $\LL_{\underline i^*_{\alpha 0}\sigma}\subset \LL_{\afan}$.
      Furthermore, $K'$ is a disk.
      We can construct a cobordism using \cref{claim:correspondenceofcobordism},  
      \[L_{\alpha 0}\circ K':(\Lc_{\alpha 0}\circ L_{\alpha}(\sF_1), \ldots, \Lc_{\alpha 0}\circ L_\alpha(\sF_{2^{k-1}}))\rightsquigarrow\Lc_{\alpha 0}\circ L_{\sigma>\alpha}.\] 
      From \cref{property:linkingdisks,cor:linkingdisk}, $\Lc_{\alpha 0} \circ L_{\sigma>\alpha}$ is a linking disk to $\LL_\sigma$. 
      By \cref{property:topology}, $L_{\alpha 0}\circ K'$ has the topology of a disk.
      Because support functions on $\str(\alpha)/\alpha$ all arise by pulling back support functions on $\Sigma$, we can take $\tilde \sF_i\in \Supp(\Sigma)$ so that $\underline i^*_{\alpha 0}\tilde \sF_i= \sF_i$. 
      By \cref{property:aexactsequence}, there exists for every $\sF_i$ a Lagrangian cobordism
      \[K_i:(L(\tilde \sF_i+\sF_\alpha), L(\tilde \sF_i))\rightsquigarrow L_{\alpha 0}\circ L_\alpha(\underline i_{\alpha 0}^*\tilde \sF_i)= L_{\alpha 0}\circ L_\alpha(\sF_i).\]
      We define $K$ to be the composition of $L_{\alpha 0}\circ K'$ with all of the $K_i$.
\end{proof}
\begin{proof}[Proof of \cref{thm:ageneration}]
   From \cref{prop:ageneration} there exists a Lagrangian cobordism with the topology of a disk, right ends on a linking disk $L_{\sigma>0}$ of $\LL_\sigma$, and left ends on sections $L_0(\sF)$. 
   By \cref{thm:cobordismgeneration}, the Lagrangian sections $L_0(\sF_i)$ generate $L_{\sigma>0}$, which in turn generate $\mathcal W((\CC^*)^n, \stp_\Sigma)$ by \cite[Theorem 1.10 ]{ganatra2018sectorial}. 
\end{proof}
We expect the argument of \cref{prop:mirrordivisormatching} can pushed through the inductive construction of the linking disks to identify $L_{\sigma>0}$ with $\mathcal O_{\sigma>0}$ under the derived equivalence between $\mathcal W((\CC^*)^n, \stp_\Sigma)$ and $\Coh(\check X_\Sigma)$.
\subsubsection[Generating mirror the proj. plane]{Returning to generation of the mirror to $\CP^n$.}
\label{subsubsec:cpngeneration}
Recall from \cref{subsec:cocores} that we gave an explicit description of the cocores in the mirror to $\CP^n$. 
We now examine the subcategory generated by these Lagrangian sections. 
\begin{prop}
   Consider $((\CC^*)^n, \stp_{\Sigma_n})$, the mirror to the toric variety $\CP^n$.
   Let $\tau$ be a cone of codimension at least $k$, and let $x\in \LL_\tau$ be an interior point.  
   The Lagrangians $L(-k), \ldots, L(-n)$ generate $\cocore(x)$. 
   \label{prop:mirrorBeilinson}
\end{prop}
\begin{proof}
   We proceed by induction. For the mirror to $\CP^1$, it is clear.
   Now assume that the statement holds for $((\CC^*)^{n-i}, \stp_{\Sigma_{n-1}})$. 
   Let $\tau\in \Sigma_n$ be a non-maximal cone. In the case where $\tau=0$, this is \cref{lemma:cocoresaretropical}. 
   Otherwise, let $\alpha<\tau$ be a 1-dimensional cone. 
   By \cref{cor:linkingdisk}, we know that there exists $y\in \LL_{\underline i_{\alpha0}^* \tau}\subset \LL_{\str(\alpha)/\alpha}$ so that 
   $\Lc_{\alpha 0} \circ \cocore(y)$ is Hamiltonian isotopic to $\cocore(x)$.
   Furthermore, $\cocore(y)$ can be generated by $L_\alpha(-k), \ldots, L_\alpha(n-1)$.
   Since $\Lc_{\alpha 0}\circ L_{\alpha}(-i)$ is generated by $L(-i), L(-i-1)$, we conclude that $\cocore(x)$ can be generated by $L(-k), \ldots, L(-n)$.
\end{proof}
When adding the zero section, we obtain the mirror to the well-known Beilinson exceptional collection for $\CP^n$. 
\begin{cor}
   Let $\Sigma_n$ be the fan for $\CP^n$. 
   The Lagrangian cocores $L(-n), \ldots, L(-1)$ and the zero section $L(0)$ are a full exceptional collection for $\mathcal W((\CC^*)^n, \stp_{\Sigma_n})$
   \label{cor:exceptionalcollection}
\end{cor}

It may be interesting to study in what sense the line bundles corresponding to cocores in an arbitrary $\check X_\Sigma$ generalize this example. 
\subsection{Removing toric strata}
\label{subsec:quotient}
We now examine the relation between \cref{thm:inclusioncorrespondence} and stop removal. 
We say that a stopped Liouville domain  $(X, \mathfrak f')$  is obtained by stop removal from  $(X, \mathfrak f)$ if $\mathfrak f'\subset \mathfrak f$ and $\mathfrak f'\setminus \mathfrak f$ is mostly Legendrian.
\begin{thm}[\cite{ganatra2018sectorial}]
    Let $(X, \mathfrak f)$ be a stopped Louiville domain, which is obtained from another stopped Louiville domain $(X, \mathfrak f')$ by stop removal. 
    Let $\mathcal D$ be the subcategory generated by Lagrangian linking disks of $\mathfrak f'\setminus \mathfrak f$.
    There is an exact sequence 
    \[\mathcal D\to \mathcal W(X, \mathfrak f')\to \mathcal W(X, \mathfrak f).\]
\end{thm}
We are interested in stop removal for the FLTZ skeleton. 
Consider the upward closure of $\sigma$,  $\str(\sigma)=\{\tau\in \Sigma \; : \; \tau\geq \sigma\}.$
The stop $\mathfrak f_{\Sigma\setminus \str(\sigma)}$ is obtained from $\mathfrak f_\Sigma$ by stop removal at the set $\mathfrak g:=\left(\bigcup_{\tau\subset \str(\sigma)}\mathfrak f_\tau\setminus \bigcup_{\tau'\neq \tau} \mathfrak f_{\tau'}\right)$.
The linking disks of $\mathfrak g$ are in bijection with the linking disks of $\bigcup_{\tau\subset \str(\sigma)}\mathfrak f_\tau$.
The toric variety $\check X_{\Sigma\setminus \str(\sigma)}$ is $\check X_\Sigma \setminus \check X_{\str(\sigma)/\sigma}$, the complement of the toric orbit closure associated to $\sigma$.
\begin{prop}
    \label{prop:stopremovallocalization}
    $\mathcal W(X, \mathfrak f_{\Sigma\setminus \str(\sigma)})$ is derived equivalent to $D^b\Coh(\check X_{\Sigma\setminus \str(\sigma)})$, and the acceleration functor
    \[\mathcal W(X, \mathfrak f_{\Sigma})\to \mathcal W(X, \mathfrak f_{\Sigma\setminus \str(\sigma)}) \]
    is mirror to the pullback functor
    \[\check i^*:D^b\Coh(\check X_\Sigma)\to D^b\Coh(\check X_{\Sigma\setminus \str(\sigma)})\]
\end{prop}
\begin{proof}
    Let $\mathcal D_{\str(\sigma)}$ be the full subcategory generated by linking disks of  $\mathfrak g$.
    We must identify $\check {\mathcal D}$, the mirror of $\mathcal D_{\str(\sigma)}$ in $D^b\Coh(\check X_\Sigma)$. 
    By \cref{cor:linkingdisk}, the generating set of $\mathcal D_{\str(\sigma)}$ is contained within the image of $L_{\sigma 0}: ((\CC^*)^{n-k}, \stp_{\str(\sigma)/\sigma})\Rightarrow ((\CC^*)^n, \stp_\Sigma)$.
    These linking disks can therefore be generated by Lagrangians of the form $L_{\sigma> 0}(\sF)$. 
    Similarly, the linking disks in $\mathcal D_{\str(\sigma)}$ are pushforwards of linking disks in $((\CC^*)^{n-k},\stp_{\str(\sigma)/\sigma})$, and therefore generate the Lagrangians $L_{\sigma>0}(\sF)$. In summary,
    \[\mathcal D_{\str(\sigma)} = \langle L_{\sigma> 0}(\sF) \rangle.\]
    For notational brevity, let $\check X=\check X_\Sigma$, and let $\check Z= \check X_{\str(\sigma)/\sigma} \subset \check X_\Sigma$.
    On the mirror, $\check i_*(D^b\Coh(\check Z))$ is generated by the sheaves $\mathcal O_{\sigma>0}(\sF)$, which are mirror to the Lagrangians $L_{\sigma> 0}(\sF)$.
    It follows that $\Ob(\check i_*(D^b\Coh(\check Z))=\Ob(\check{\mathcal D})$. 
    
    Let $D^b\Coh_{\check Z}(X)$ denote the Serre subcategory of sheaves whose cohomology sheaves are supported on $\check Z$.
    Then by \cite[Lemma 2.13]{arinkin2010perverse}, $\Ob( D^b\Coh_{\check Z}(\check X))\subseteq  \Ob(\check i_*(D^b\Coh(\check Z))$. 
    It is clear that the cohomology sheaf of any $\mathcal F\in \check{\mathcal D}$ will be supported in $\check Z$, so 
    $\Ob( D^b\Coh_{\check Z}(X))=  \Ob(\check i_*(D^b\Coh(\check Z))= \Ob(\check{\mathcal D})$, and in particular 
    \[D^b\Coh_{\check Z}(\check X)= \check{\mathcal D}.\]
    By \cite[Remark 3.14]{rouquier2010derived}, there exists an exact sequence of categories 
    \[D^b\Coh_{\check Z}(\check X)\to D^b\Coh(\check X)\to D^b\Coh(\check X\setminus \check Z),\]
    which is then mirror to 
    \[\mathcal D_{\str(\sigma)}\to \mathcal W(X, \mathfrak f_{\Sigma})\to \mathcal W(X, \mathfrak f_{\Sigma\setminus \str(\sigma)}) \]
    using the universal property of quotients of $A_\infty$ categories and that exact dg-Verdier sequences are also universal in the $\infty$-category of $A_\infty$ categories as shown in \cite{oh2020infinity}.
 \end{proof}
 By iterating this stop removal, we can extend \cref{cor:HMScor} to quasi-projective toric varieties. Note that the resulting equivalence still takes tropical Lagrangian sections to line bundles, but the tropical Lagrangian sections now must be wrapped to compute their morphisms.
 \begin{cor}
 \label{cor:quasiprojective}
    Let $\check X_\Sigma$ be a smooth quasi-projective toric variety.
    The toric variety $\check X_\Sigma$ is homologically mirror to $( X, \mathfrak f_\Sigma)$.
 \end{cor}
 
 \begin{rem} Unlike projective toric varieties, quasi-projective toric varieties can be Calabi-Yau, and thus, can have mirrors that are spaces. Moreover, homological mirror symmetry for these mirrors has been studied in some examples (see, for instance, \cite{chan2016lagrangian}). It would be interesting to geometrically relate the Fukaya categories of the Calabi-Yau mirrors with those of the stopped domains in \cref{cor:quasiprojective}. One reasonable approach may be to use the dual fan constructions of \cite{clarke2016dual}.
 \end{rem}
 
  \subsection{Blow-up functors}
\label{subsec:blowup}
We briefly discuss toric blow-up in order to fix notation.
Consider a smooth projective toric variety $\check X_\Sigma$ with $\Sigma\subset Q$.
Let $\tau=\langle \alpha_i\rangle_{i=1}^{n-k}$ be a codimension $k$ cone of $\Sigma$, so that $\check X_{\str(\tau)/\tau}$ is a $k$-dimensional compact subvariety of $\check X_\Sigma$.
Let $A_{\str(\tau)/\tau}=\{\beta_j/\tau\}_{j=1}^{m}$ be the one-dimensional cones of $\str(\tau)/\tau$, and to each $\beta_j/\tau$ let $\beta_j$ be the corresponding primitive generator of $\Sigma$.
As $\str(\tau)/\tau$ is complete, there are at least $k$ such cones.

The normal bundle $N_{\check X_\Sigma} \check X_{\str(\tau)/\tau}$ is a toric variety with fan $\str(\tau)\subset \Sigma$.
We now describe a splitting of $N_{\check X_\Sigma} \check X_{\str(\tau)/\tau}$  as the sum of toric line bundles $\bigoplus_{i=1}^{n-k}\mathcal O_{\str(\tau)/\tau}(\sF_i)$.
The support functions $\sF_i: Q/\tau \to \RR$ are constructed by choosing an identification $Q= \RR\langle \beta_j\rangle_{j=1}^k \times \RR\langle \alpha_i\rangle_{i=1}^{n-k}$.
Using the splitting, write each primitive lift $\beta$ of $\beta/\tau\in A_{\str(\tau)/\tau}$ as
\[\beta=\left(\sum_{j=1}^k b_j\beta_j\right)+\left(\sum_{i=1}^{n-k}a_i\alpha_i\right).\]
We define $\sF_i(\beta):= a_i$.
From this construction, we see that the total space $\bigoplus_{i=1}^{n-k} \mathcal O_{\str(\tau)/\tau}(\sF_i)$ is the toric variety  $\check X_{\str(\tau)}$.

There is a $\CC^*$ action on $\check X_{\str(\tau)}$ coming from the primitive element $\alpha:=\alpha_1+\ldots +\alpha_{n-k} \in \RR\langle \beta_j\rangle_{j=1}^k \times \RR\langle \alpha_i\rangle_{i=1}^{n-k}$.
We note that the toric strata fixed by this action are indexed by cones which contain $\alpha$.
These are exactly the cones dominating $\tau=\langle \alpha_1, \ldots, \alpha_{n-k}\rangle$, i.e., the cones of $\str(\tau)/\tau$.
It will be convenient for us to write 
\[\RR\langle \beta_j\rangle_{j=1}^k \times \RR\langle \alpha_{i}\rangle_{i=1}^{n-k}=(Q/\tau)\times (\RR\langle \alpha_i/\alpha\rangle_{i=2}^{n-k})\times \RR\langle\alpha\rangle\]
In these coordinates, the primitives $\beta/\tau\in A_{\str(\tau)/\tau}$ lift to primitives $\beta\in A_{\str(\tau)}$ as
\[
    \beta=\left(\sum_{j=1}^k b_j\beta_j, \sum_{i=2}^{n-k} \left(\sF_i(\beta/\tau)-\sF_1(\beta/\tau)\right) \alpha_i/\alpha, \sF_1(\beta/\tau)\alpha \right).
\]
These should be compared to the coordinates chosen in the top right and top middle diagrams of \cref{fig:3dfans}.
The quotient by the $(\CC^*)^\alpha$ action is the projectivization $\PP(N_{\check X_\Sigma} \check X_{\str(\tau)/\tau})= \check X_{\str(\tau)/\alpha}$.
The $\alpha$-coordinate gives a support function $\sF_\alpha: Q/\alpha\to \RR$, which on one-dimensional cones $\alpha_i/\alpha, \beta/\alpha\in A_{\str(\tau)/\alpha}$ takes values
\begin{align*}
    \sF_\alpha(\alpha_i/\alpha)=\left\{\begin{array}{cc} 0 & i\neq 1\\ 1 & i=1\end{array}\right. && \sF_\alpha(\beta)=\sF_1(\beta/\tau)
\end{align*}
The associated line bundle $\mathcal O_{\str(\tau)/\alpha}(\sF_\alpha)$ is the tautological bundle of the projectivization.
The total space of the associated line bundle is the blow-up of $N_{\check X_\Sigma} \check X_{\str(\tau)/\tau}$ at $\check X_{\str(\tau)/\tau}$, and has cones given by the graph of $\sF_\alpha$ 
\[\Bl_\alpha(\str(\tau)):\{(\sigma, \sF_\alpha(\sigma))\;: \; \sigma\in \str(\tau)/\alpha\}\cup \{\langle (\sigma',  \sF_\alpha(\sigma')),\alpha\rangle\;: \; \sigma'\in \str(\tau)/\alpha\}\]
The blow-up of $\Sigma$ at $\tau$ can also be described by the  $\alpha$-subdivision of the fan $\Sigma$,
\[\Bl_\alpha(\Sigma):= \{\sigma \in \Sigma \; : \; \alpha \not \in \sigma\} \cup \{\langle\sigma' , \alpha\rangle \; : \; \sigma'\in \Sigma, \alpha\not\in \sigma', \exists \sigma  \in \Sigma \text{ with } \langle\sigma' , \alpha \rangle \subset \sigma\}.\] 

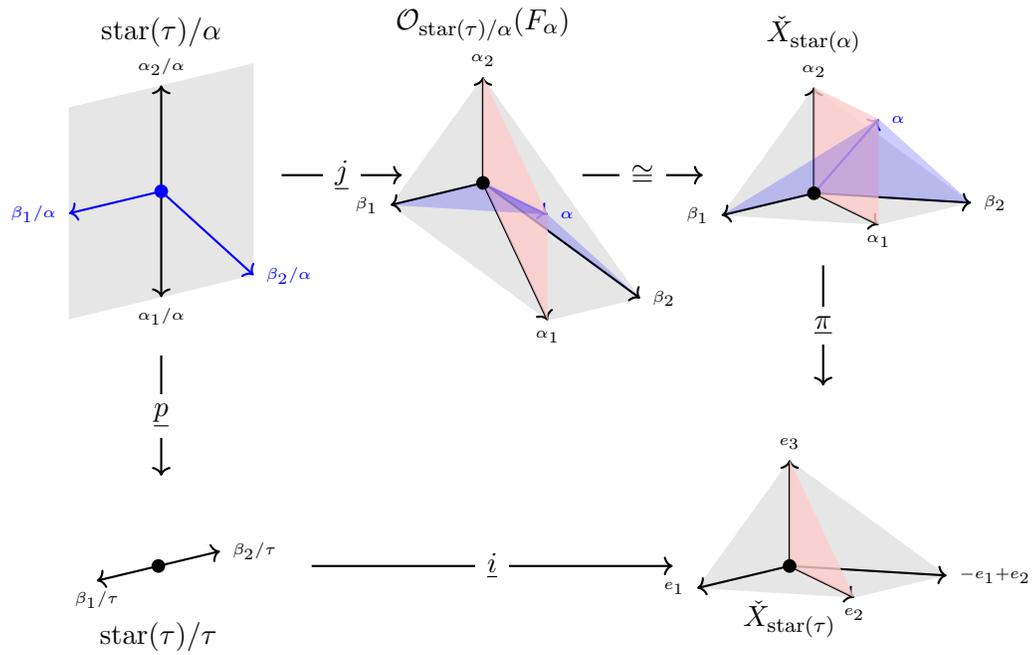
\begin{figure}
    \centering
\tdplotsetmaincoords{70}{145}

\begin{tikzpicture}[scale=.8]
    \node at (8,-5.5) {
\begin{tikzpicture}[tdplot_main_coords,scale=1.5]
    \draw[white] (2, 0, -1)--(-2, 0, -1)--(-2, 0, 1)--(2, 0, 1);
    \fill[fill=gray!20] (1, 0, 0)--(0,0,0)-- (0, 1, 0);
    \fill[fill=gray!20] (1, 0, 0)--(0,0,0)-- (0, 0, 1);
    \fill[fill=gray!20] (-1, 1, 0)--(0,0,0)-- (0, 0, 1);
    \fill[fill=gray!20] (-1, 1, 0)--(0,0,0)-- (0, 1, 0);

    \draw[thick,->] (0,0,0) -- (1,0,0) node [ left]{$\scriptscriptstyle e_1$} ;
    \draw[thick,->] (0,0,0) -- (0,1,0) node [ below]{$\scriptscriptstyle e_2$} ;
    \draw[thick,->] (0,0,0) -- (0,0,1) node [ above]{$\scriptscriptstyle e_3$} ;
    \draw[thick,->] (0,0,0) -- (-1,1,0)node [ right]{$\scriptscriptstyle -e_1+e_2$}  ;
    \fill[fill=red!20,fill opacity=.8] (0,1,0)-- (0,0,0)-- (0, 0, 1);
    \node at (0, 0, -.5) {$\check X_{\str(\tau)}$};
    \node[circle, fill=black, scale=.5] (0,0,0) {};
\end{tikzpicture}
    };

    \node at (8,1) {
\begin{tikzpicture}[tdplot_main_coords,scale=1.5]
    \draw[white] (2, 0, -1)--(-2, 0, -1)--(-2, 0, 1)--(2, 0, 1);
    \fill[fill=gray!20] (1, 0, 0)--(0,0,0)-- (0, 1, 0);
    \fill[fill=gray!20] (1, 0, 0)--(0,0,0)-- (0, 0, 1);
    \fill[fill=gray!20] (-1, 1, 0)--(0,0,0)-- (0, 0, 1);
    \fill[fill=gray!20] (-1, 1, 0)--(0,0,0)-- (0, 1, 0);

    \draw[thick,->] (0,0,0) -- (1,0,0) node [ left]{$\scriptscriptstyle \beta_1$} ;
    \draw[thick,->] (0,0,0) -- (0,1,0) node [ below]{$\scriptscriptstyle \alpha_1$} ;
    \draw[thick,->] (0,0,0) -- (0,0,1) node [ above]{$\scriptscriptstyle \alpha_2$} ;
    \draw[thick,->] (0,0,0) -- (-1,1,0)node [ right]{$\scriptscriptstyle \beta_2$} ; 
    \draw[thick,->,blue] (0,0,0) -- (0,1,1) node [ right]{$\scriptscriptstyle \alpha$} ;
    \fill[fill=red!20,fill opacity=.8] (0,1,0)-- (0,0,0)-- (0, 0, 1)--(0,1,1);
    \fill[fill=blue, fill opacity=.2] (-1,1,0)-- (0,0,0)--(0,1,1);
    \fill[fill=blue, fill opacity=.2] (1,0,0)-- (0,0,0)-- (0,1,1);
    \node at (0, 0, 1.5) {$\check X_{\str(\alpha)}$};
    \node[circle, fill=black, scale=.5] (0,0,0) {};
\end{tikzpicture}
    };

    \node at (2.5,1) {
\begin{tikzpicture}[tdplot_main_coords, scale=1.5]
    \draw[white] (2, 0, -1)--(-2, 0, -1)--(-2, 0, 1)--(2, 0, 1);
    \fill[fill=gray!20] (1, 0, 0)--(0,0,0)-- (0, 0, 1);
    \fill[fill=gray!20] (1, 0, 0)--(0,0,0)-- (0,1,-1);
    \fill[fill=gray!20] (-1,1,-1) --(0,0,0)-- (0,0,1);
    \fill[fill=gray!20] (-1,1,-1) --(0,0,0)-- (0,1,-1);

    \draw[thick,->] (0,0,0) -- (1,0,0) node [left ]{$\scriptscriptstyle \beta_1 $} ;
    \draw[thick,->] (0,0,0) -- (0,1,-1)node [below]{$\scriptscriptstyle  \alpha_1$} ;
    \draw[thick,->] (0,0,0) -- (0,0,1) node [above]{$\scriptscriptstyle  \alpha_2$} ;
    \draw[thick,->] (0,0,0) -- (-1,1,-1)node[right]{$\scriptscriptstyle \beta_2$} ; 
    \draw[thick,->,blue] (0,0,0) -- (0,1,0)  node[right]{$\scriptscriptstyle \alpha$} ; 
    \fill[fill=red!20,fill opacity=.8] (0,0,1)-- (0,0,0)-- (0,1,-1)--(0,1,0);
    \fill[fill=blue, fill opacity=.2] (-1,1,-1) -- (0,0,0)--(0,1,0);
    \fill[fill=blue, fill opacity=.2] (1,0,0)-- (0,0,0)-- (0,1,0);
    \node at (0, 0, 1.5) {$\mathcal O_{\str(\tau)/\alpha}(\sF_\alpha)$};
    \node[circle, fill=black, scale=.5] (0,0,0) {};
\end{tikzpicture}
    };

    \node at (-3,1) {
\begin{tikzpicture}[tdplot_main_coords,scale=1.5]
    \draw[white] (2, 0, -1)--(-2, 0, -1)--(-2, 0, 1)--(2, 0, 1);
    \fill[fill=gray!20] (1, 0, -1)--(1, 0, 1)--(-1, 0, 1)--(-1, 0,-1);
    \draw[thick,->,blue] (0,0,0) -- (1,0,0)  node [ left]{$\scriptscriptstyle \beta_1/\alpha$} ;
    \draw[thick,->] (0,0,0) -- (0,0,1)  node [ above]{$\scriptscriptstyle \alpha_2/\alpha$} ;
    \draw[thick,->] (0,0,0) -- (0,0,-1) node [ below]{$\scriptscriptstyle \alpha_1/\alpha$}  ;
    \draw[thick,->,blue] (0,0,0) -- (-1,0,-1)node [ right]{$\scriptscriptstyle \beta_2/\alpha$}  ;
    \node at (0, 0, 1.5) {$\str(\tau)/\alpha$};
    \node[circle, fill=black, scale=.5,blue] (0,0,0) {};
\end{tikzpicture}
    };

    \node at (-3,-5.5) {
        \begin{tikzpicture}[tdplot_main_coords]
            \draw[white] (2, 0, -1)--(-2, 0, -1)--(-2, 0, 1)--(2, 0, 1);
            \draw[thick,->] (0,0,0) -- (1,0,0) node [below]{$\scriptscriptstyle \beta_1/\tau$} ;
            \draw[thick,->] (0,0,0) -- (-1,0,0) node [ right]{$\scriptscriptstyle \beta_2/\tau$} ;
            \node at (0, 0, -1) {$\str(\tau)/\tau$};
            \node[circle, fill=black, scale=.5] (0,0,0) {};
        \end{tikzpicture}
            };

            \draw[->, thick] (4,1)--(6,1)   node [midway, fill=white]{$\cong$} ;
            \draw[->, thick] (-1,1)--(1,1)   node [midway, fill=white]{$\underline j$} ;
            \draw[->, thick] (-0.5,-5.5)--(5.5,-5.5) node [midway, fill=white]{$\underline i$} ;
            \draw[->, thick] (-3,-2)--(-3,-4) node [midway, fill=white]{$\underline p$} ;
            \draw[->, thick] (8,-.5)--(8,-2.5) node [midway, fill=white]{$\underline \pi$} ;
\end{tikzpicture} \caption{Blow-up of the curve $\tau$, highlighted in red in the bottom right. The cones of $T$ and $T/\alpha$ are marked in blue. In this example, all of the cones of $T$ have fiber codimension 0.}
\label{fig:3dfans}
\end{figure}

 We'll be interested in the symplectic and complex lifts of this blow-up of fans.
\setlength\mathsurround{0pt}
\[
    \begin{tikzcd} 
        D^b\Coh(\check X_{\str(\alpha)/\alpha} )\arrow{r}{\check j_*}   & D^b\Coh(\check X_{\Bl_\alpha(\Sigma)}) \\
        D^b\Coh(\check X_{\str(\tau)/\tau}) \arrow{u}{\check p^*}\arrow{r}{\check i_*}& D^b\Coh(\check X_\Sigma)\arrow{u}{\check\pi^*})
    \end{tikzcd}
    \begin{tikzcd} 
        ((\CC^*)^{n-1},\stp_{\str(\alpha)/\alpha}) \arrow{r}{\Lc_{\alpha 0}}   & ((\CC^*)^{n-1},\stp_{\Bl_\alpha(\Sigma}) \\
        ((\CC^*)^{k},\stp_{\str(\tau)/\tau}))\arrow{u}{\Lc_{p^*}} \arrow{r}{\Lc_{\tau>0}}& ((\CC^*)^n, \stp_\Sigma) \arrow{u}{\Lc_{\pi^*}}.
    \end{tikzcd}
\] 
\setlength\mathsurround{.8pt}
On the $B$-side, Orlov's blow-up formula gives a semi-orthogonal decomposition for the derived category of the blow-up.
\begin{thm}[\cite{orlov1993projective}]\label{thm:orlov} Let $\check X_{\Bl_\alpha(\Sigma)}\to \check X_\Sigma$ be the blow-up at a smooth center $\check Z := \check X_{\str(\tau)/\tau}$ of dimension $k$.
    The derived functors 
    \begin{align*}
        \Phi_{-k}:=\mathcal O_{\check X_{\str(\alpha)/\alpha}} (k\sF_\alpha)\tensor \check j_*\check p^* :&  D^b\Coh(\check Z)\to D^b\Coh(\check X_{\Bl_\alpha(\Sigma)})\\
        \check \pi^*:& D^b\Coh(\check X_\Sigma)\to D^b\Coh(\check X_{\Bl_\alpha(\Sigma)})
    \end{align*}
     are fully faithful. 
    Furthermore, $D^b\Coh(\check X_{\Bl_\alpha(\Sigma)})$ has semi-orthogonal decomposition
    \begin{align*}
        D^b\Coh&(\check X_{\Bl_\alpha(\Sigma)}) 
        =\left\langle\Phi_{-(n-k)} D^b\Coh(\check Z) ,\cdots, \Phi_{-1} D^b\Coh(\check Z), \check \pi^* D^b\Coh(\check X_\Sigma)\right\rangle^\bot.
    \end{align*}
\end{thm}
We will now work towards an the analogous result on the mirror side.
We begin by examining the FLTZ skeleton of the blow-up.
Since $\Bl_{\alpha}(\Sigma)$ is obtained from $\Sigma$ by removing the cones $\sigma\geq \tau$ and adding in cones of the form $\langle\sigma,\alpha\rangle$, we can cover the FLTZ stop by   
\[\stp_{\Bl_{\alpha}(\Sigma)}=\left( \stp_\Sigma\setminus \bigcup_{\sigma\geq\tau}\stp_{\tau}\right)\cup \left(\bigcup_{\sigma\in \Bl_{\alpha(\Sigma)}, \sigma\geq \alpha} \stp_{\sigma}\right).\]
For each $\sigma>\tau$  spanned by $\sigma=\langle \alpha_1, \ldots, \alpha_{n-k}, \beta_{j_1}, \ldots, \beta_{j_l}\rangle $, let 
\begin{align*}
    \sigma_{\alpha_i, \alpha}:= \langle\alpha, \alpha_1, \ldots, \alpha_{i-1}, \alpha_{i+1}, \ldots, \alpha_{n-k}, \beta_{j_1}, \ldots, \beta_{j_l}\rangle\in \str(\alpha)
\end{align*}
The set of all such cones can be characterized as 
\[\{\sigma_{\alpha_i, \alpha}\} = \{\sigma' \in \str(\alpha) \; : \;  \sigma'\subset \sigma''\in \Sigma, |\sigma'|=|\sigma''|\}.\]
Partition the cones of $\Bl_\alpha(\Sigma)$ as 
\[\Bl_\alpha(\Sigma)=\{\sigma \in \Sigma \; : \; \alpha \not \in \sigma\} \sqcup  \{\sigma' \in \str(\alpha) \; : \;  \sigma'\subset \sigma''\in \Sigma, |\sigma'|=|\sigma''|\}  \sqcup T\]
The cones of $T$ are drawn in blue in \cref{fig:3dfans}, and are in bijection with the cones of $\str(\tau)/\tau$.
We obtain a decomposition of the FLTZ skeleton for the blowup as sets,
\begin{equation}
    \stp_{\Bl_{\alpha}(\Sigma)}=\stp_\Sigma\cup \left( \bigcup_{\sigma\in T} \stp_\sigma\right).
    \label{eq:notadecomp}
\end{equation}
From this, we conclude that $\stp_\Sigma\subset\stp_{\Bl_{\alpha}(\Sigma)}$ and deduce that the blow-down is given by stop removal. In particular, we obtain the mirror statement to part of \cref{thm:orlov} as the adjoint inclusion functor is easily defined on tropical Lagrangian sections and does not affect their Floer theory.
\begin{prop}
    The Lagrangian correspondence $(\Lc_{\pi}^*)^{-1}:((\CC^*)^n, \stp_{\Bl_\alpha(\Sigma)})\Rightarrow((\CC^*)^n, \stp_{\Sigma})$ is a stop removal. 
    The adjoint $\Lc_{\pi}^*:((\CC^*)^n,\stp_{\Sigma})\to ((\CC^*)^n,\stp_{\Bl_{\alpha}}(\Sigma))$ induces a fully faithful functor. 
\end{prop}
Define
\begin{align*}
    \mathfrak g:=& \stp_{\Bl_{\alpha}(\Sigma)}\setminus \stp_\Sigma,
\end{align*}
and let $\LL_{\mathfrak g}$ be the set of points which flow to the set $\mathfrak g$ under the Liouville flow. 
\Cref{eq:notadecomp} does not describe the set $\mathfrak g$.
Rather, for each $\sigma\in T$, let $\Conn(\Sigma, \sigma):=\{\II_{\sigma;i}\}$ index the connected components of $\LL_\sigma$ in the complement of the singular locus of the FLTZ skeleton. We may then write
\[\LL_{\mathfrak g}=\bigcup_{\sigma\in T, i\in\Conn(\Sigma, \sigma) } \II_{\sigma, i}.\]
To understand the localization $\mathcal W(X, \stp_{\Bl_{\alpha}(\Sigma)})\to \mathcal W(X, \stp_\Sigma)$, we must understand the linking disks of the $ \II_{\sigma, i}$.\footnote{We apologize for conflating stops and skeleton components when describing these linking disks; while we are performing a localization at the linking disks of the stop, our argument will describe these linking disks as cocores of the FLTZ skeleton associated to a lower dimensional toric variety.}
We first need a general lemma for stop removal:
\begin{lem}
    Let $\mathfrak f$ be a stop. Let $\Lambda$ be a connected Legendrian submanifold whose interior is disjoint from $\mathfrak f$. Suppose that $\partial \Lambda \not \subset \mathfrak f$. Then 
    $\mathcal W (X,   \mathfrak f \cup \Lambda) \simeq \mathcal W (X,   \mathfrak f ).$
    \label{lem:trivialLocalization}
\end{lem}
\begin{proof}
    Since $  \mathcal W (X,   \mathfrak f)$ is obtained from $\mathcal W(X , \mathfrak f\cup \Lambda)$ by stop removal, we need only prove that the linking disk to $x\in \Lambda$ is a trivial object of $\mathcal W(X, \mathfrak f\cup \Lambda)$.
    Let $y$ be any point in the boundary of $\Lambda$ which is not contained in $\mathfrak f$. 
    Consider a path $\gamma: I\to \Lambda$ from $x$ to $y$, and linking disks $\cocore(\gamma(t))$.
    Then, $\cocore(\gamma(1))$ is Lagrangian nullcobordant and Hamiltonian isotopic to $\cocore(x)$. 
\end{proof}

Each cone $\sigma\in T$ corresponds to a toric fixed stratum with both base and fiber components.
Define the \emph{fiber codimension} of a cone to be $|\sigma/\alpha|-|\underline p(\sigma/\alpha)|$; alternatively, this is the number of $\alpha_i$ in the primitive generators of $\sigma$. See \cref{fig:fibercodimension}. 

\begin{figure}
    \centering
    \tdplotsetmaincoords{70}{127}

\begin{tikzpicture}

    \node at (0,0) {
        \begin{tikzpicture}[tdplot_main_coords,scale=1.5]
            \fill[fill=gray!20] (1, 0, 0)--(0,0,0)-- (0, 1, 0);
            \fill[fill=gray!20] (1, 0, 0)--(0,0,0)-- (0, 0, 1);
            \fill[fill=gray!20] (0, 1, 0)--(0,0,0)-- (0, 0, 1);
        
            \draw[thick,->] (0,0,0) -- (1,0,0) node [ left] {$\scriptscriptstyle \alpha_1$} ;
            \draw[thick,->] (0,0,0) -- (0,1,0) node [ below]{$\scriptscriptstyle \alpha_2$} ;
            \draw[thick,->] (0,0,0) -- (0,0,1) node [ above]{$\scriptscriptstyle \alpha_3$} ;
            \draw[thick,->,blue] (0,0,0) -- (1,1,1)node [ right]{$\scriptscriptstyle \alpha$} ; 
            \fill[fill=blue, fill opacity=.2] (1,1,1)-- (1,0,0)-- (0, 0, 0);
            \fill[fill=blue, fill opacity=.2] (1,1,1)-- (0,1,0)-- (0, 0, 0);
            \fill[fill=blue, fill opacity=.2] (1,1,1)-- (0,0,1)-- (0, 0, 0);
            \node[circle, fill=black, scale=.5] (0,0,0) {};
        \end{tikzpicture}};

    \begin{scope}[shift={(-5,0)}]
    
    \fill[gray!20]  (-1,1) rectangle (1,-1);
    \draw[blue, ->] (0,0) -- (0,1) node [above] {$\alpha_2/\alpha$};
    \draw[blue, ->] (0,0) -- (1,0)node [right] {$\alpha_3/\alpha$};
    \draw[blue, ->] (0,0) -- (-1,-1)node [below left] {$\alpha_1/\alpha$};
    \node[fill=blue, circle, scale=.3] at (0,0) {};
    \end{scope}
    
    \node[fill, circle, scale=.3] at (-8.5,0) {};
    
    \node[above] at (-8.5,0) {$\alpha/\alpha$};
    \draw[->] (-6.5,0) -- (-7.5,0) node[midway, fill=white] {$\underline p$} ;
    \draw[->] (-2.5,0) -- (-1,0) node[midway, fill=white] {$\underline j$} ;
    \node[fill=gray!20] at (-5.5,-0.5) {$\scriptscriptstyle 1$};
    \node[fill=gray!20] at (-5,0.5) {$\scriptscriptstyle 1$};
    \node[fill=gray!20] at (-4.5,0) {$\scriptscriptstyle 1$};
    \end{tikzpicture}     \caption{Blowup of $\CC^3$ at the origin, with elements of $T$ highlighted in blue. The cone $\alpha$ has fiber codimension 0, while the cones $\langle \alpha, \alpha_i\rangle$ all have fiber codimension 1.}
    \label{fig:fibercodimension}
\end{figure}
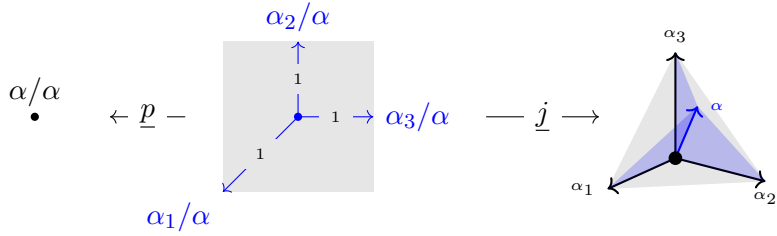
\begin{prop} 
    \label{prop:blowupstrata}
    Let $\mathfrak g^l$ be the collection of $\II_{\sigma, i}$ with $\sigma\in T$ and fiber codimension at most $l$.
    For all $0\leq l \leq n-k-2$, we have an equivalence of categories. 
    \[\mathcal W((\CC^*)^{n-1}, \mathfrak f_{\str(\alpha)/\alpha}\setminus \mathfrak g^0 )\cong\mathcal W((\CC^*)^{n-1}, \mathfrak f_{\str(\alpha)/\alpha}\setminus \mathfrak g^l).\]
\end{prop}
\begin{proof}
We proceed inductively. 
The collection of Legendrians $\mathfrak g^l \setminus \mathfrak g^{l-1}$ is a disjoint union of Legendrians belonging to $\stp_{\sigma}$, where $\sigma$ has fiber dimension $l$.
For each such $\stp_{\sigma}$, there exists  $\stp_{\sigma'}\subset \mathfrak g^{l-1}$ whose boundary is partially contained inside the intersection $\partial( \stp_{\sigma})\cap \stp_{\sigma'}$.
Therefore by \cref{lem:trivialLocalization} stop removal at $\mathfrak g^{l-1}$ and $\mathfrak g^{l}$ produce quasi-isomorphic categories.
\end{proof}

We now need to show that the localization at the linking disks of $T$ is the mirror to blowdown.
By \cref{prop:blowupstrata}, it suffices to consider the localization at linking disks of fiber codimension 0 strata.
Let $T^0\subset T$ be the set of fiber codimension 0 cones.
Consider now the Lagrangian correspondence $\Lc_{\alpha 0}: ((\CC^*)^n, \stp_{\str(\alpha)/\alpha})\Rightarrow (X,\stp_{\Bl_{\alpha}(\Sigma)} )$ which is mirror to inclusion of the exceptional divisor.
By \cref{cor:linkingdisk}, for every interior point $x\in \LL_{\sigma}$ with $\sigma>\alpha$, there exists an interior point $y\in \LL_{\underline i^*\sigma} \subset \LL_{\str(\alpha)/\alpha}$ such that $L_{\alpha 0}\circ \cocore(y)$ is admissibly Hamiltonian isotopic to $\cocore(x)$.
Since every cone $\sigma\in T^0$ contains $\alpha$, there exists a set of cones $T^0/\alpha\subset \str(\alpha)/\alpha$ such that stop removal at $\mathfrak g$ is given by localization at the Lagrangians
\[\{\Lc_{\alpha 0}\circ \cocore(y)\;:\; y\in \LL_{\sigma/\alpha}, \sigma/\alpha\in T^0/\alpha\}.\]
It remains to generate these linking disks by pullback from $((\CC^*)^k, \str(\tau)/\tau)$.
The cones of $T^0/\alpha$ are in bijection with the cones of $\str(\tau)/\tau$ via the projection $\underline p: \str(\alpha)/\alpha\to \str(\tau)/\tau$.
However, the set of linking disks $\Conn(\str(\alpha)/\alpha, \sigma/\alpha):=\{\II_{\sigma/\alpha;i}\}$
of $\LL_{\sigma, i}, i\in \Conn(\str(\tau)/\tau, \sigma)$ differ from the linking disks of $\LL_{\underline p(\sigma)}\subset \LL_{\str(\tau)/\tau}$. 

For notational clarity, whenever $L_{\delta\delta'}$ is unclear (due to the possibility that $\delta$ and/or $\delta'$ are quotient cones), we will write $L_{\delta;\delta'}$.
By \cref{lemma:cocoresaretropical},  a linking disk $\cocore(y)$ for  $\LL_{\sigma/\alpha}\subset ((\CC^*)^{n-1}, \stp_{\str(\alpha)/\alpha})$ can be expressed as the image of a cocore $\cocore(z)\subset((\CC^*)^{n-1-j}, \stp_{\str(\sigma/\alpha)/(\sigma/\alpha)})$  under the correspondence
\[\Lc_{\sigma/\alpha;\alpha/\alpha}:((\CC^*)^{n-1-|\sigma/\alpha|}, \stp_{\str(\sigma/\alpha)/(\sigma/\alpha)})\Rightarrow((\CC^*)^{n-1}, \stp_{\str(\alpha)/\alpha}).\]
Let $\sF_z$ be a support function for $\check X_{\str(\sigma/\alpha)/(\sigma/\alpha)}$ corresponding to the cocore $\cocore(z)$ so that 
\[\Lc_{\sigma/\alpha;\alpha/\alpha}\circ L^\delta(\sF_z)=\cocore(y).\]
Let $[\underline i^*_{\sigma 0}\sF_{\alpha}]$ be a linear equivalence class of support function arising from pullback of $\sF_{\alpha}$.
Write $\sF_z=\sum_{i=1}^{n-k} m_{\alpha_i} \sF_{\alpha_i/\sigma}+ \sum_{\beta_j/\sigma} m_{\beta_j/\sigma} \sF_{\beta_j/\sigma}$. By taking a linear transformation, we may assume that the $m_{\alpha_i}=0$ for $i\neq 1$. 
Then, there exists a point which satisfies the inequalities $-1< \alpha_i/\alpha \cdot p < 0$ for all $i\neq 2$, and $m_{\alpha_1}-1<\alpha_1/\alpha\cdot p <m_{\alpha_1}$. 
In the $\alpha_i$ coordinates, this implies that $0<m_{\alpha_1}\leq n-k$; this computation should be compared to \cref{prop:mirrorBeilinson}. 
Write $l=m_{\alpha_1}$ and set $[\underline i^*_{\sigma 0}\sF_{\alpha}]=\sF_{\alpha_1}+ \sum_{\beta/\sigma} \sF_1(\beta/\alpha) \sF_{\beta/\sigma}$. Then 
\[\sF_z-l([\underline i^*_{\sigma 0}\sF_{\alpha}])=\sum_{\beta/\sigma}\left( l \sF_1(\beta/\alpha)+ m_{\beta/\sigma}\right) \sF_{\beta/\sigma}.\]
As this has no $\sF_{\alpha_i}$ component, $L^{\delta}(\sF_z-l([\underline i^*_{\sigma 0}\sF_{\alpha}]))=\phi^{-l}_{[\underline i^*_{\sigma 0}\sF_{\alpha}]} \circ L^{\delta}(\sF_z)$ arises as the pullback of a tropical section $L^\delta(\sF_w)\subset(\CC^*)^{n-1-j}, \stp_{\str(\sigma/\tau)/(\tau)})$ along $\Lc_{p_\sigma}^*: ((\CC^*)^{n-1-j}, \stp_{\str(\sigma/\tau)/(\tau)})\Rightarrow((\CC^*)^{n-1}, \stp_{\str(\sigma/\alpha)/(\sigma/\alpha)}$. 
Using the commutativity of the square in \cref{fig:furtherdecomposition}, we obtain the following:
\begin{prop}
    Let $\cocore(x)$ be a linking disk to $\LL_{\sigma}$ for some $\sigma\in T^0$.
    Then there exists a cocore $L^\delta(\sF_w)$ of $((\CC^*)^{n-1-j}, \stp_{\str(\sigma/\tau)/(\tau)})$, and $0<l\leq n-k$ so that 
    \[\cocore(x) \sim \Lc_{\alpha0}\circ \phi^l_{\sF_\alpha}\circ L_p^*\circ\Lc_{\sigma/\tau;\tau/\tau}\circ L^\delta(\sF_w).\]
\end{prop}

\begin{figure}
    \centering
    \begin{tikzpicture}
\begin{scope}[thick]

\fill[gray!20]  (-1,1) rectangle (1,-1);

\draw[->] (0,0) -- (0,1);
\draw[->] (0,0) -- (1,-1);
\draw[->] (0,0) -- (0,-1);
\draw[->] (0,0) -- (-1,0);
\node[fill, circle, scale=.5] at (0,0) {};
\end{scope}

\begin{scope}[thick]

\draw[->] (0,-3.5) -- (-1,-3.5);
\draw[->] (0,-3.5) -- (1,-3.5);
\node[fill, circle, scale=.5] at (0,-3.5) {};
\end{scope}

\begin{scope}[thick]
\draw[->] (-4.5,0) -- (-4.5,1);
\draw[->] (-4.5,0) -- (-4.5,-1);
\node[fill, circle, scale=.5] at (-4.5,0) {};
\end{scope}

\node[fill, circle, scale=.5] at (-4.5,-3.5) {};
\node at (0,1.5) {$\str(\alpha)/\alpha$};
\node at (-4.5,1.5) {$\str(\sigma/\alpha)/(\sigma/\alpha)$};
\node at (-4.5,-4) {$\str(\sigma/\tau)/(\sigma/\tau)$};
\node at (0,-4) {$\str(\alpha)/\tau$};

\draw[->] (-4,0) -- (-1.5,0)node [midway, fill=white]{$L_{\sigma/\alpha;\alpha/\alpha}$} ;
\draw[<-] (-4.5,-1.5) -- (-4.5,-3)node [midway, fill=white]{$\phi^l_{i^*[D_g]} L_{p_\sigma}^*$} ;
\draw[->] (-4,-3.5) -- (-1.5,-3.5)node [midway, fill=white]{$L_{\sigma/\tau;\tau/\tau}$} ;
\draw[<-] (0,-1.5) -- (0,-3)node [midway, fill=white]{$\phi^l_{D_g}L_p^*$} ;
\end{tikzpicture}     \caption{ The relation between linking disks of $\LL_{\sigma/\alpha}$ for $\sigma\in T^0$ to linking disks of components of $\str(\alpha)/\tau$ is similar to that between  linking disks of $\str(\sigma/\alpha)/(\sigma/\alpha)$ and $\str(\sigma/\tau)/\str(\sigma/\tau)$.}
    \label{fig:furtherdecomposition}
\end{figure}
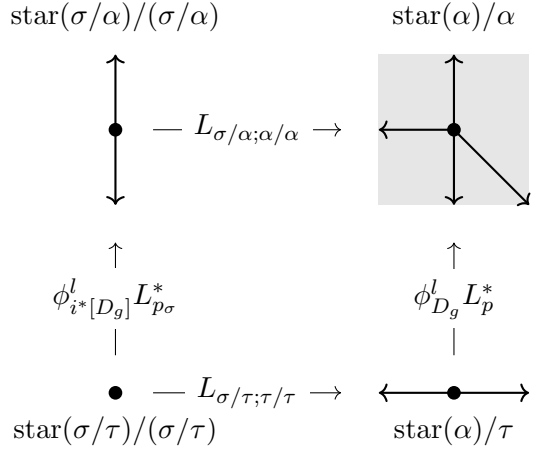
It follows that the subcategory generated by these cocores is contained within the subcategory generated by the image  of $\Lc_{\alpha 0}\circ \phi^l_{\sF_\alpha}\circ L_p^*$,
\[\langle \cocore(x) \;:\; x\in \LL_{\sigma}, \sigma\in T^0\rangle \subset\langle \Lc_{\alpha 0}\circ \phi^l_{\sF_\alpha}\circ L_p^*( \mathcal W((\CC^*)^k, \stp_{\str(\tau)/\tau}))\rangle.\]
To show that these subcategories agree, we compute the intersection of the disks $(\Lc_{\alpha 0}\circ \phi^l_{\sF_\alpha}\circ L_p^*)(L^\delta(\sF))$ with the FLTZ skeleton. 
\begin{prop}
    The Lagrangian $(\phi^l_{\sF_\alpha}\circ L_p^*)(L^\delta(\sF))$ is generated by linking disks of $\LL_{\sigma/\alpha}$ where $\sigma/\alpha\in T$.
\end{prop}
\begin{proof}
    We follow \cref{exam:cpn}.
    Rewrite the tropical section as 
    \[L^\delta(\underline p^* \sF+l\sF_\alpha).\]
    For fixed $t$, consider  the Lagrangian submanifolds
    \[L_t(l, \sF)=\left(u, \del H^{\delta}_{\underline p^*(\sF)}+ \left(\frac{tl}{n-k}\right)(\alpha_1/\alpha)+ (1-t) \del H^{\delta}_{lF_{\alpha}}\right).\]
    This is not an admissible Lagrangian submanifold for all $t$, and crosses several of the stops.
    However, let $\sigma/\alpha$ be a stop which contains some $\alpha_i/\alpha$.
    Then for $i\geq 2$,
        \[\left(\frac{tl}{n-k}\right)\alpha_1/\alpha\cdot(\alpha_i/\alpha)+ (1-t) \del \wh H^{\delta}_{lF_{\alpha}}(\alpha_i/\alpha)\approx\frac{tl}{n-k}
    \]
    and for $i=1$,
    \[
        \left(\frac{tl}{n-k}\right)\alpha_1/\alpha\cdot(\alpha_1/\alpha)+ (1-t) \del \wh H^{\delta}_{lF_{\alpha}}(\alpha_1/\alpha)\approx tl+ (1-t)l= l.
    \]
    Since the projection of these value avoid the set $\ZZ-\delta$, $L_t(l, \sF)$ does not cross $\stp_{\sigma/\alpha}$ whenever $\sigma/\alpha$ contains $\alpha_i$.
    Therefore, $L_0(l, \sF)$ and $L_1(l,\sF)$ differ by the linking disks $\cocore(x)$ with $x\in \LL_{\sigma}$, for choices of cones $\sigma \in T$.
    Alternatively, we may write
    \[L_0(l, \sF)\in \langle L_1(l, \sF), \cocore(x) \;:\; x\in \stp_{\sigma}, \sigma\in T\rangle. \]
    The function $\underline p^*(\sF+lF_\alpha))$ has the property that for every $\alpha/\alpha_i$, $\del \wh H^{\delta}_{\underline p^*(\sF+l\sF_\alpha)}(\alpha/\alpha_i)\approx0$. 
    The cones of  $\sigma/\alpha\in (\str(\tau)/\alpha)\setminus( T^0/\alpha)$ all contain at least one of the  $\alpha_i$, for which we have
    \[\left(\left(\del \wh H^{\delta}_{\underline p^*(\sF+l\sF_\alpha)}\right)+\frac{tl}{n-k}\alpha\right)(\alpha_i)\approx \left(\frac{l}{n-k}\right).\] 
    In conclusion, $L_1(l, \sF)$ intersects only $\LL_{\sigma/\alpha}$ for $\sigma/\alpha\in T$, giving us the inclusion 
    \[\langle \Lc_{\alpha 0}\circ \phi^l_{\sF_\alpha}\circ L_p^* (\mathcal W((\CC^*)^k, \stp_{\str(\tau)/\tau}))\rangle \subset\langle \cocore(x) \;:\; x\in \LL_{\sigma}, \sigma\in T^0\rangle.\]
\end{proof}
From this, we can conclude that blowup is mirror to a stop removal.
\begin{cor}
    \label{cor:blowup}
    The stop removal functor $\Lc_{\pi_*}:\mathcal W(X, \stp_{\Bl_\alpha(\Sigma)})\to \mathcal W(X, \stp_\Sigma)$ is mirror to the derived pushforward $\pi_*: D^b\Coh(\check X_{\Bl_\alpha(\Sigma)})\to D^b\Coh(\check X_\Sigma)$. 
\end{cor}

\begin{ex} \label{ex:blowupexample} It is possible to pictorially check \cref{cor:blowup} for $\check X_\Sigma = \CP^2$ using similar pictures to those in \cref{subsec:warmupcp2}. See \cref{fig:blowupexample}. Note that $L_{\pi_*}(L(-E))$ is an admissible Lagrangian section in the mirror to $\CP^2$ that is not Hamiltonian isotopic to a tropical Lagrangian section (\cref{fig:blowupexample3}).
\end{ex}

\begin{figure}
   \centering
   \begin{subfigure}{.95\linewidth}
    \centering
   \includegraphics[scale=.8]{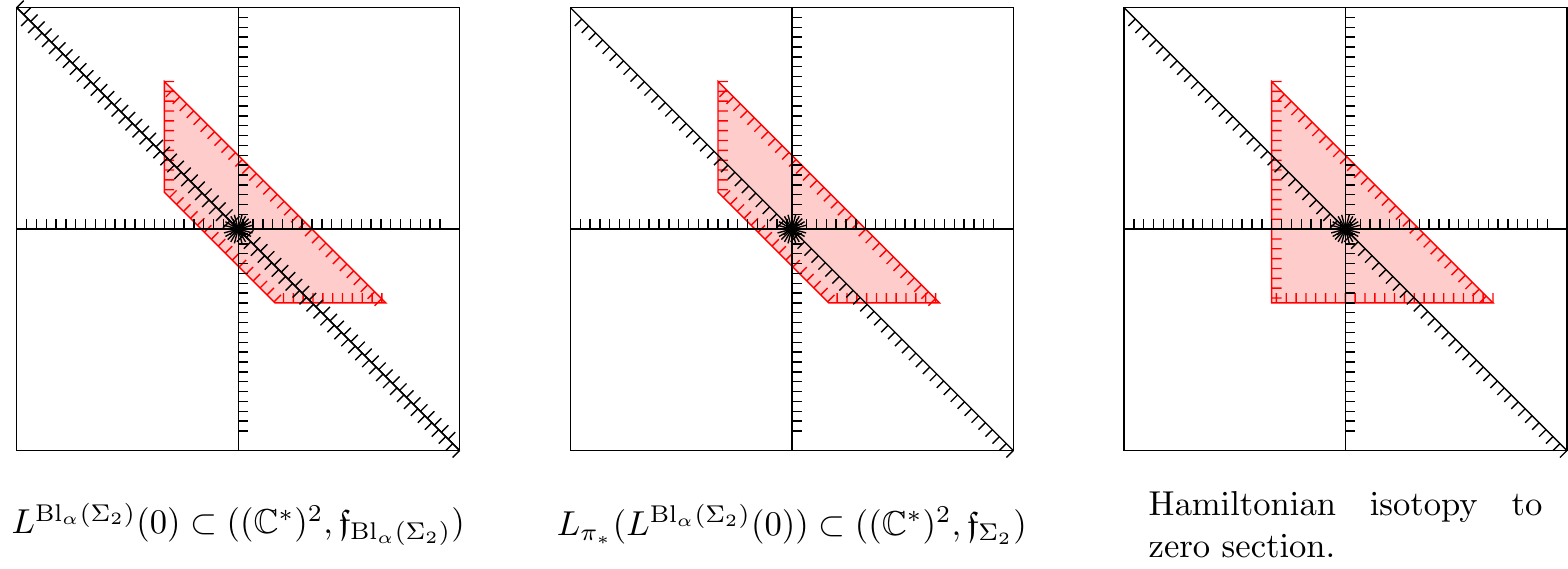}
   \caption{The mirror to $\check \pi_* \mathcal O_{\check X_{\Bl(\alpha)\Sigma}}=\mathcal O_{\check X_\Sigma}$} 
   \label{fig:blowupexample1}
   \end{subfigure}
   \begin{subfigure}{.95\linewidth}
    \centering
    \includegraphics[scale=.8]{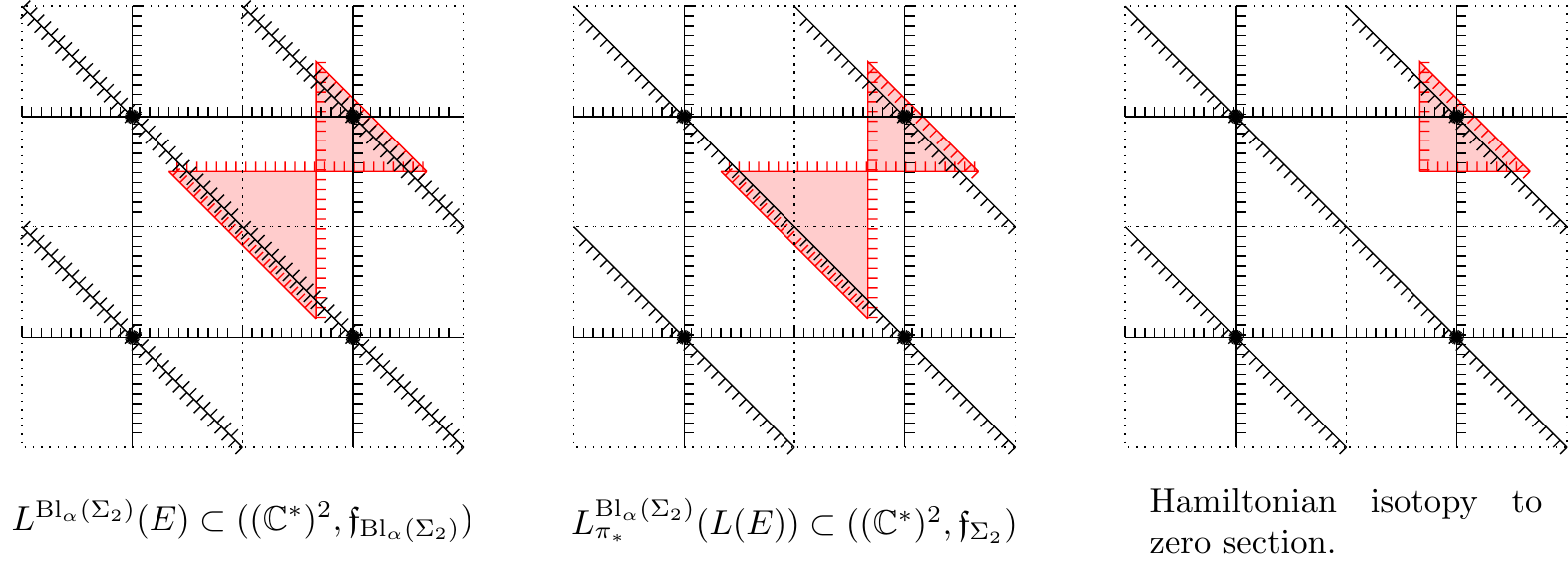}
    \caption{The mirror to $\check \pi_* \mathcal O_{\check X_{\Bl(\alpha)\Sigma}}(-E)=\mathcal O_{\check X_\Sigma}$} 
    \label{fig:blowupexample2}
    \end{subfigure}
   \begin{subfigure}{.95\linewidth}
    \centering
    \includegraphics[scale=.8]{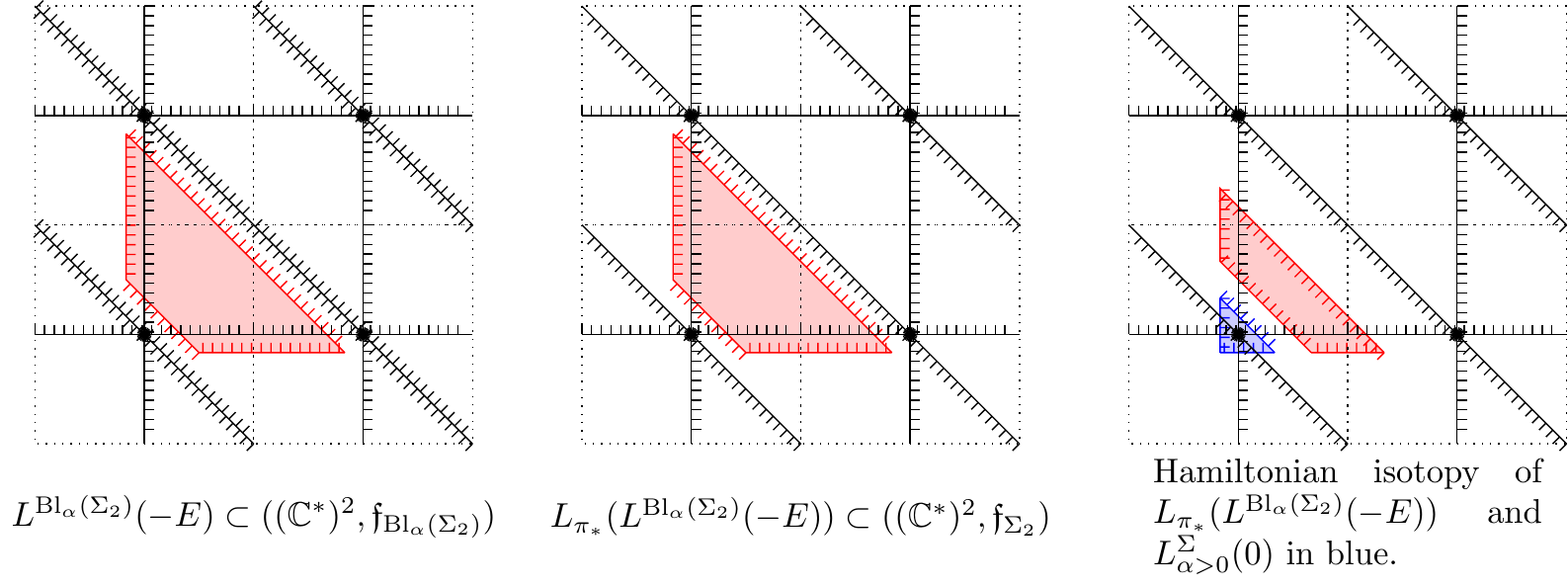}
    \caption{A Lagrangian section which is not tropical. Note that in the right most image, $(\Lc_{\pi_*}\circ L^{\Bl(\alpha)\Sigma}(-E) \# L^\Sigma_{\alpha>0}=L^\Sigma(0)$, identifying $\Lc_{\pi_*}\circ L^{\Bl(\alpha)\Sigma}(-E)$ with the ideal sheaf of a point in $\check X_\Sigma$.} 
    \label{fig:blowupexample3}
    \end{subfigure}
    \caption{Piecewise-linear approximations of argument images of Lagrangian sections under stop removal in the mirror to the blow up of a point in $\CP^2$}
    \label{fig:blowupexample}
\end{figure}

\begin{rem} For a more general star-subdivision, it is easy to see that the FLTZ skeleton for $\Sigma$ can also be obtained from that of $\Bl_\alpha(\Sigma)$ by stop removal. Thus, the discussion above should generalize to any star-subdivision, but this is combinatorially involved and requires a generalization of Orlov's theorem in the toric setting. In fact, every complete fan can be transformed by a sequence of star-subdivisions to a fan corresponding to a smooth projective toric variety by \cite[Theorem 11.1.9]{cox2011toric} and \cite[V Theorem 4.5]{ewald2012combinatorial}. Thus, realizing the general star subdivision functor as an appropriate quotient would allow one to deduce homological mirror symmetry for any toric variety from homological mirror symmetry for projective toric varieties. However, this seems to require significant new ideas beyond the smooth case presented here.
\end{rem}

\appendix

\section{Lagrangian cobordisms in Liouville manifolds}
\newcommand{\Cob}{\mathcal Cob}
\label{app:cobordism}

The words ``Lagrangian cobordism'' have come to mean  two related but slightly different kinds of Lagrangian submanifolds in symplectic geometry: \emph{Lagrangian cobordisms with Lagrangian ends} and \emph{cobordisms between Legendrians.} 
Unless otherwise specified, Lagrangian cobordisms will refer to a Lagrangian cobordism with Lagrangian ends.
We first look at Lagrangian cobordisms with Lagrangian ends, taking the viewpoint that a Lagrangian cobordism is a Lagrangian in an appropriately stopped Fukaya category \cite{tanaka2018generation,dyckerhoff2019symplectic,biran2014lagrangian,ganatra2017covariantly}. 

\begin{df}[\cite{arnol1980lagrange,bosshard}]
   Let $\{L_0, \ldots, L_k\}$ be admissible Lagrangian submanifolds of $(X, \stp)$, and let $(\CC, \stp_n)$ be a Liouville domain with $n$ disjoint stops whose flow makes the collection of rays  
   \[\LL_n=\{ \RR_{\gg 0}, {\rm i}+\RR_{\gg 0}, \ldots, (k-1){\rm i} +\RR_{\gg 0}\}\]
   admissible Lagrangian submanifolds. 
   A \emph{Lagrangian cobordism with Lagrangian ends}  $L_0, \ldots L_k$ is an admissible Lagrangian submanifold $K\subset (X\times \CC, \stp\times \stp_n)$ which for which there exists an open set $V\subset \CC$ so that 
   \begin{itemize}
      \item $\CC\setminus V$ is the normal neighborhood of the non-compact rays $\left\{\left(j-\frac{{\rm i}}{2}\right)+\RR_{> 0}\right\}_{j=0}^k$.
      \item When $K$ is restricted to the complement of  $V$,
            \[K\setminus \pi_\CC^{-1}(V)=\bigcup_{j=0}^k\left( L_j\times\left( \left(j-\frac{{\rm i}}{2}\right)+\RR_{> 0}\right)\right).\]
   \end{itemize}
	We denote such a cobordism by $K:(L_1, \ldots, L_k)\rightsquigarrow L_0$. 
	\label{def:cobordism}
\end{df}
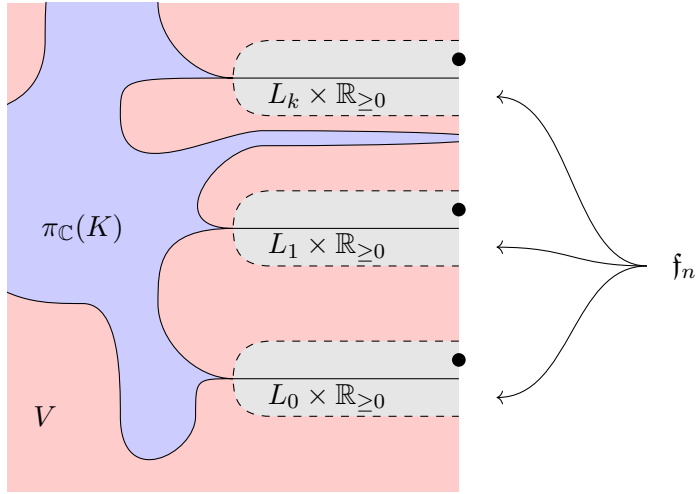
\begin{figure}
   \centering
   \begin{tikzpicture}

\fill[red!20]  (-3,-3) rectangle (3,3.5);\begin{scope}
\clip (-3,-3) rectangle (3,3.5);
\begin{scope}[shift={(0,1)}]

\draw[dashed, fill=gray!20] (3,2) .. controls (2.5,2) and (1,2) .. (0.5,2) .. controls (0.2,2) and (0,1.85) .. (0,1.5) .. controls (0,1.15) and (0.15,1) .. (0.5,1) .. controls (1,1) and (2.5,1) .. (3,1);

\draw (3,1.5) -- (0,1.5);
\end{scope}

\begin{scope}[shift={(0,-3)}]

\draw[dashed, fill=gray!20] (3,2) .. controls (2.5,2) and (1,2) .. (0.5,2) .. controls (0.2,2) and (0,1.85) .. (0,1.5) .. controls (0,1.15) and (0.15,1) .. (0.5,1) .. controls (1,1) and (2.5,1) .. (3,1);

\draw (3,1.5) -- (0,1.5);
\end{scope}

\begin{scope}[shift={(0,-1)}]

\draw[dashed, fill=gray!20] (3,2) .. controls (2.5,2) and (1,2) .. (0.5,2) .. controls (0.2,2) and (0,1.85) .. (0,1.5) .. controls (0,1.15) and (0.15,1) .. (0.5,1) .. controls (1,1) and (2.5,1) .. (3,1);

\draw (3,1.5) -- (0,1.5);
\end{scope}

\draw[fill=blue!20] (0,2.5) .. controls (-0.5,2.5) and (-1,3) .. (-1,3.5) .. controls (-1,4) and (-1.5,4.5) .. (-2,4.5) .. controls (-2.5,4.5) and (-2.5,3.5) .. (-2.5,3) .. controls (-2.5,2) and (-3.5,2) .. (-4,2) .. controls (-4,1.5) and (-4,1) .. (-3.5,0) .. controls (-3,-0.5) and (-2.5,-0.5) .. (-2,-0.5) .. controls (-1.5,-0.5) and (-1.5,-1.5) .. (-1.5,-2) .. controls (-1.5,-3) and (-0.5,-2.5) .. (-0.5,-2) .. controls (-0.5,-1.5) and (-0.5,-1.5) .. (0,-1.5) .. controls (-0.5,-1.5) and (-1,-1) .. (-1,-0.5) .. controls (-1,0) and (-1,0.5) .. (0,0.5) .. controls (-1,0.5) and (-0.2,1.6) .. (0.4,1.6) .. controls (4.2,1.6) and  (4.4,1.8) .. (0.4,1.8) .. controls (-0.2,1.8) and(-1.5,1) .. (-1.5,2) .. controls (-1.5,2.5) and (-1.5,2.5) .. (0,2.5);
\end{scope}
\draw[->] (5.5,0) .. controls (4.5,0) and (4.5,-1.75) .. (3.5,-1.75);
\draw[->] (5.5,0) .. controls (4.5,0) and (4.5,0.25) .. (3.5,0.25);
\draw[->] (5.5,0) .. controls (4.5,0) and (4.5,2.25) .. (3.5,2.25);
\node at (6,0) {$\stp_n$};
\node at (-2,0.5) {$\pi_{\mathbb C}(K)$};
\node at (-2.5,-2) {$V$};
\node[circle, fill, scale=.5] at (3,-1.25) {};
\node[circle, fill, scale=.5] at (3,0.75) {};
\node[circle, fill, scale=.5] at (3,2.75) {};

\node at (1.25,-1.75) {$L_0\times \mathbb{R}_{\geq 0}$};
\node at (1.25,0.25) {$L_1\times \mathbb{R}_{\geq 0}$};
\node at (1.25,2.25) {$L_k\times \mathbb{R}_{\geq 0}$};
\end{tikzpicture} \caption{A cobordism between admissible Lagrangians need not be compact in the projection to $\CC$, however it must fiber over rays near the stops.}
\label{fig:noncompactcobordism}
\end{figure}
When $X$ is compact, a Lagrangian $K\subset X\times \CC$ which fibers over real rays in $\CC$ outside of a compact set will necessarily be a Lagrangian cobordism for some set of ends. 
\begin{ex}
   An important example to consider which highlights the need to work with Lagrangians with non-compact projections to the base comes from Hamiltonian isotopy. 
   Let $H_t: X\times \RR\to \RR$ be a time-dependent Hamiltonian, compactly supported in $t$. 
   Let $\theta^s_{H_t}$ be the time $s$ Hamiltonian flow. 
   We can parameterize a Lagrangian submanifold by
   \begin{align*}
      L\times \RR\to &X\times \CC\\
      (x, s) \mapsto &(\theta^s_{H_t}(x), s+{\rm i} H_s).
   \end{align*}
   Notice that if $H_s$ is unbounded in the $L$-direction, then the $\pi_\CC$-projection of this cobordism will be unbounded in the imaginary direction.
   After rotating the ends to that they point all in the positive direction, we obtain a Lagrangian cobordism, called the suspension of $H_t$.
\end{ex}
One reason to use Lagrangian cobordisms is that they provide algebraic relations between the ends of the cobordism.
If $K: (L_1\ldots, L_k)\rightsquigarrow L_0$ is a monotone Lagrangian cobordism, then a result of Biran and Cornea states that the right end can be expressed as an iterated mapping cone of the Lagrangians on the left end. 
There is current work of Valentine Bosshard extending the results of Biran and Cornea to the setting of Liouville sectors by using the GPS framework.
We will use the following version of their results, which can be understood in terms of the wrapping triangle from \cite{ganatra2018sectorial}.

\begin{thm}[\cite{biran2014lagrangian}, \cite{bosshard}]
   Suppose that $K:(L_1, \ldots, L_k)\rightsquigarrow L_0$ is an admissible exact Lagrangian cobordism. 
   Then $(L_1, \ldots, L_k)$ generate $L_0$ in the Fukaya category.
   \label{thm:cobordismgeneration}
\end{thm}

The Lagrangian cobordism that we will be using the most is the \emph{surgery trace} cobordism. 
We show here that the surgery construction presented in \cite{hicks2020tropical} carries over to the setting of Liouville domains with stops. 

\begin{df}
    Let $L_1$ and $L_2$ be two  admissible Lagrangian submanifolds of $(X, \stp)$.
    A surgery data is 
    \begin{itemize}
      \item $U\subset L_1$, an open neighborhood of $L_1\cap L_2$ called the \emph{surgery region}, and
      \item a function $f: U\to \RR$ called the collar function 
    \end{itemize}
    with the following properties:
    \begin{itemize}
        \item
              $f\geq 0$, and $f=0$ exactly on $L_1\cap L_2$.
        \item
              There exists a Weinstein neighborhood of $U$ symplectomorphic to the radius $c$ cotangent ball  $B^*_{c} (U)$, so that the Lagrangian $L_2|_{B^*_c (U)}$ is the graph of the section $df$. 
    \end{itemize}
    \label{def:surgerydata}
   \end{df}
   Given a subset $U\subset X$, and fixed $X^{int}\subset X$ a Liouville domain, the conicalization of  $U$ (relative to $X^{int}$) is the attachment of the downward flow space outside $X^{int}$ to $U$.
   \[U^{cone}:=U \cup\{\phi^t_Z(x) \; : \; x\in U\setminus X^{int}, t\in \RR_{>0}\}.\]
   \begin{prop}
   \label{prop:generalizedsurgeryprofile}
   Let $L_1, L_2$ be admissible Lagrangian submanifolds of $(X, \stp)$ with surgery data $(U, f)$. 
    For all sufficiently small $\eps$, there exists a Lagrangian surgery at $U$ contained within a conicalized neighborhood of the symmetric difference between $L_1$ and $L_2$:  
    \[L_1\#_U L_2\subset B_{\eps}^{cone} ((L_1\cup L_2)\setminus (L_1 \cap L_2)).\]
    Furthermore, there is a Lagrangian surgery trace cobordism 
    \[K:(L_1, L_2)\rightsquigarrow (L_1\#_U L_2).\] 
    The surgery $L_1\#_U L_2$ is conical at infinity. Furthermore, if $L_1, L_2$ are exact, and $U$ is connected, then $L_1\#_U L_2$ is exact.
    In particular, if $U$ is connected, then $L_1\#_U L_2$ is admissible.
\end{prop}
    \begin{proof}
   We first construct a new surgery neighborhood and collar function from the data of  $U$ and $f$. This new surgery data will be invariant under the Liouville flow outside of a compact set. 
   Let $X=X^{int}\cup (\partial X^{int})\times\RR_{>0}$ be a decomposition of $X$ into a compact symplectic manifold with contact boundary and the symplectization of that boundary. 
   Assume that $X^{int}$ was chosen so that   $L_i|_{(\partial X^{int})\times\RR_{>0}} =\Lambda_i\times \RR_{>0}$ for choices of Legendrian submanifolds $\Lambda_i$.
   Write $L_i=L^{int}_i\cup (\Lambda_i\times \RR_{>0})$, and take $\Upsilon\subset \Lambda_1$ to be the restriction of the surgery region to a Legendrian slice of $L_1$ at $\partial X^{int}$.
   Because $0$ is the minimum of $f$, there exists a value $c_\eps$ so that $f|_{X^{int}}$ has no critical values for $f< 3 c_\eps$ except at 0. 
   By choosing a smaller subset of $U$ and a smaller value $c_\eps$, we can assume that $(\partial U)\cap X^{int}$ is the level set $f^{-1}(3c_\eps)$. 

   Over the conical end we can take a symplectic neighborhood of $ J^1_{c_1}(\Upsilon)\times \RR_{>0}\supset \Upsilon\times\RR_{>0}$,  where $J^1_{c_1}(\Upsilon)$ is a small neighborhood of the zero section in the jet bundle.  The radius $c_1$ of this jet bundle can be chosen small enough so that  $\Lambda_2\cap J^1_{c_1}(\Upsilon)$ is a section of the jet bundle. 
   Since $L_2$ is admissible, this section is invariant in the $\RR_{>0}$ direction and there exists a primitive $g:\Lambda_1\to \RR$ so that $L_2|_{J^1_{c}}(\Upsilon)\times\RR_{>0}\to \RR$  is the graph of $J^1g\times [0, \infty)$, where $J^1g$ is the 1-jet of $g$. 
   Note that there is not necessarily an inclusion (in either direction) of  $J^1_{c_1}(\Upsilon)\times [0,\infty)$ and $B^*_{c}(U)$. 
   The function $g$ satisfies similar properties to $f$, that is 
   \begin{itemize}
      \item $g\geq 0$, and $g=0$ exactly where $\Lambda_1\cap \Lambda_2$,
      \item $g$ achieves a minimal value of $3  c_{\eps}$ on the boundary of $\Upsilon$.
   \end{itemize}
   This gives us a ``conical'' version of the collar function over the contact end. We now combine this conical collar function with the original collar function.
   Take a Weinstein neighborhood $W$ of an open set of $L_1$ containing the intersection $L_1\cap L_2$. 
   This Weinstein neighborhood can be chosen so that inside of  $X^{int}$, $W$ is a subneighborhood of $B^*_c(U)$.
   Furthermore, we can take $W$ so that outside of a compact set $X^{int}$, $W$ splits as $B^*_{c_1}(\Upsilon)\times B^*_{c_2}(\RR_{>0})$ and is a subneighborhood of $J^1_{c_1}(\Upsilon)\times \RR_{>0}$.
   Finally, we note that we may pick $W$ small enough so that it avoids $\LL_{\stp}\setminus X^{int}$ (defined in \cref{eq:conicalizationofstop}).
   In this Weinstein neighborhood, $L_2|_W= dh$, where $h$ satisfies the following properties:
   \begin{itemize}
      \item $h \geq 0$, and $h=0$ exactly where $L_1\cap L_2$,
      \item Outside of a compact set, $h=e^t\cdot g$ where $t$ is the symplectization coordinate. 
      \item $h$ achieves a minimal value of $3c_{\eps'}$ on the boundary of $W\cap L_1$ inside $X^{int}$ for some value $c_{\eps'}$,
   \end{itemize}
\begin{figure}
   \centering
   
\begin{tikzpicture}[thick]
	\begin{scope}[]
		\draw[ <->] (-1.5,0)-- (1.8,0);
		\draw[ <->]  (-1,2.5)--(-1,-0.5);
		\node at (-1.5,2.4) {$s_\eps(t)$};
		\node at (2.4,0) {$t$};

		\node[above] at (0.6,0) {$2c_\eps$};
	
	\node[above]at (0,0) {$c_\eps $};
		
		\draw[ultra thick] (1.8,0) -- (0.6,0);
	\draw[ultra thick] (0.6,0) .. controls (0.4,0) and (0.2,0) .. (0,-0.1) node (v1) {};
	
	\draw (1.8,0);
\end{scope}
	\begin{scope}[shift={(0,-4)}]
		\draw[ <->] (-1.5,0)-- (1.8,0);
		\draw[ <->]  (-1,2.5)--(-1,-0.5);
		\node at (-1.5,2.4) {$s_\eps'(t)$};
		\node at (2.4,0) {$t$};

		\node[below] at (0.6,0) {$2c_\eps$};
	
	\node[below] at (0,0) {$c_\eps$};
		
		\draw[ultra thick] (1.8,0) -- (0.6,0);
	\draw [ultra thick](0.6,0) .. controls (0.2,0) and (0,0.2) .. (0,1);
	
	\draw[dashed] (0,0) -- (0,1)--(-1, 1);
\end{scope}
	\node at (-1.25,-2) {1};
	\node at (-1.25,-3) {$\frac{1}{2}$};

	\begin{scope}[shift={(-6.5,0)}]
	
	\begin{scope}[]
		\draw[ <->] (-1.5,0)-- (1.8,0);
		\draw[ <->]  (-1,2.5)--(-1,-0.5);
		\node at (-1.5,2.4) {$r_\eps(t)$};
		\node at (2.4,0) {$t$};

		\node[below] at (0.6,0) {$2c_\eps$};
		
		\draw[ultra thick] (1.4,2.4) -- (0.5,1.5);
	\draw[ultra thick] (0.6,1.6) .. controls (0.4,1.4) and (0.4,1.4) .. (0,1.2);
	
	\node[below] at (0,0) {$c_\eps $};
\draw[dashed, thick] (0,1.2) -- (0,0)  (0.6,1.6) -- (0.6,0);
\end{scope}
	\begin{scope}[shift={(0,-4)}]
		\draw[ <->] (-1.5,0)-- (1.8,0);
		\draw[ <->]  (-1,2.5)--(-1,-0.5);
		\node at (-1.5,2.4) {$r_\eps'(t)$};
		\node at (2.4,0) {$t$};

		\node[below] at (0.6,0) {$2c_\eps$};
		
	\node[below] at (0,0) {$c_\eps $};
		\draw[ultra thick] (1.8,2) -- (0.6,2);
	\draw[ultra thick] (0.6,2) .. controls (0.1,2) and (0,1.5) .. (0,1);
	
	\draw[dashed] (0,0) --(0,1) --(-1, 1) (-1, 2)--(0.6,2) -- (0.6,0);
\end{scope}
	\node at (-1.25,-2) {1};
	\node at (-1.25,-3) {$\frac{1}{2}$};
	
	\end{scope}
\end{tikzpicture}    \caption{The $r$ and $s$ surgery profiles}
   \label{fig:profilecurves}
\end{figure}
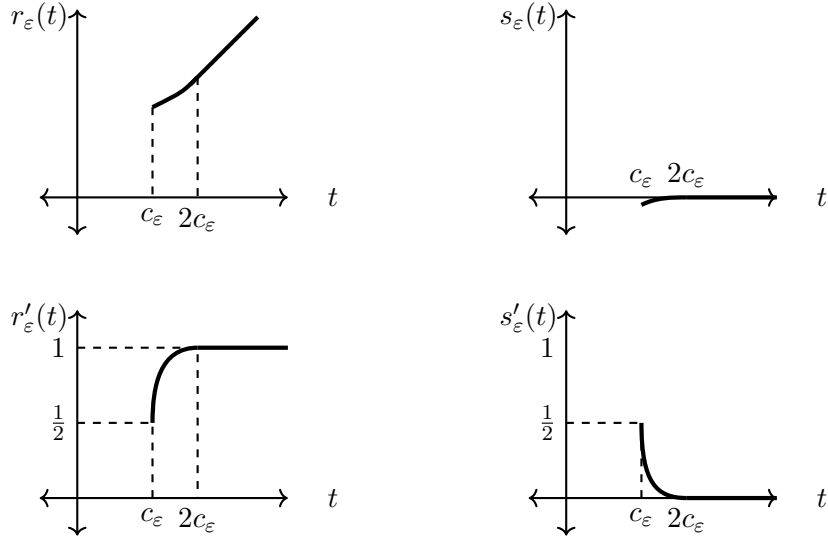
   We now follow \cite[Proposition 3.1]{hicks2020tropical}.  Define surgery profile functions $r, s: \RR_{>c_\eps}\to \RR$ which satisfy the following properties:
	\begin{itemize}
		\item
		      $r(a)=a$ for $a\geq 2c_\eps$ and $s_\eps(a)=0$ for $a\geq 2c_\eps$.
		\item
		      $r'(c_\eps)=s'(c_\eps)=\frac{1}{2}$.
		\item
		      $r(a)$ is convex, while $s(a)$ is concave.
		\item
            The concatenation of curves $(a, r'(a))$ and $(a, s'(a))$ is a smooth plane curve.
   \end{itemize}
   We now adapt these surgery profiles to give admissible Lagrangians in the symplectization region by choosing a rescaling function $\rho:L_1\to \RR$ satisfying:
   \begin{itemize}
      \item $\rho|_{L_1\cap X^{int}}=1$ 
      \item $\rho|_{L_1\setminus X^{int}}$ depends only on the symplectization coordinate. When restricted to $\partial X^{int}\times [1, \infty)$, $\rho$ is equal to $e^t$, where $t$ is the symplectization coordinate. 
   \end{itemize}
   The rescaled surgery profile functions are
   \begin{align*} 
      r_\rho(a):= \rho\cdot r(a/\rho) && s_\rho(a):=\rho\cdot s(a/\rho).
   \end{align*}
   This rescaling is similar to the rescaling of convolution kernel used in \cref{sec:embedding} which was needed to make tropical Lagrangian sections admissible objects of the category. 
   The surgery neck is built from two Lagrangian sections of $W|_{h\geq \rho c_\eps}$,
   \begin{align*}
      L_r':= d(r_\rho\circ h) &&  L_s':= d(s_\rho\circ h).
   \end{align*}
   The union of these two charts is the surgery neck, which is a smooth Lagrangian submanifold $L_r'\cup L_s'\subset W$. 
   Near the boundary of $W$ where $h>2\rho c_\eps$, we obtain,  
   \begin{align*}
      r_\rho\circ h |_{h>2\rho c_\eps} = h && s_\rho\circ h|_{h>2\rho c_\eps}=0,
   \end{align*}
   so we can glue these to our $L_1, L_2$ to form Lagrangian submanifolds with common boundary, 
   \begin{align*}L_r:= L_r'\cup ( L_2\setminus W) && L_s:= L_s'\cup (L_1\setminus W). 
   \end{align*}
   We call these the $r$ and $s$ charts of the surgery, and define the surgery $L_1\#_{U} L_2:= L_r\cup L_s$.

   It remains to show that this is a conical Lagrangian submanifold.
   Outside of the set $X^{int}$, the section is  $d(r_\rho\circ h)= d(e^t \cdot (r\circ g))$, and so the $r$-chart restricted to this portion  $L_r'\setminus X^{int}$ is parameterized by $J^1( r\circ g)\times [1, \infty)$. 
   Similarly, the $L_s$ chart in the conical region is given by $J^1(s\circ g)\times [1, \infty)$.
   $L_1\#_{U} L_2$ is conical at infinity Lagrangian because union of conical at infinity charts is conical at infinity.

   Finally, suppose that $L_1, L_2$ are exact, and the surgery region $U$ is connected.
   Then the $L_s$ and $L_r$ charts are exact, and there is no obstruction to gluing the primitives to obtain a primitive for $\lambda$ over $L_1\#_{U} L_2$.

   The Lagrangian surgery trace cobordism is constructed following \cite[Proposition 3.1]{hicks2020tropical}.
   Let $L_1, L_2$ be Lagrangians with surgery data $U, f$.
   Pick curves $\gamma_1, \gamma_2\subset \CC$ which intersect over the negative real axis (as drawn in \cref{fig:surgeryconfiguration}.) 
   Then $L_1\times \gamma_1, L_2\times \gamma_2\subset X\times \CC$ can be surgered along their intersection $(L_1\cap L_2)\times \RR_{<0}$, giving a Lagrangian cobordism as desired. 
\end{proof}
\begin{figure}
   \centering
   
\tikzset{
    side by side/.style 2 args={
        line width=2.5pt,
        #1,
        postaction={
            clip,postaction={draw,#2}
        }
    }}
\begin{tikzpicture}
    \draw[thick]  (-4,2) rectangle (2.5,-3);
    \draw (2.5,-0.5) -- (-1.5,-0.5);
    \draw[ultra thick, fill=gray!20] (-0.75,0) .. controls (-1,-0.25) and (-0.75,-0.5) .. (-0.5,-0.5) .. controls (-0.75,-0.5) and (-1,-0.5) .. (-1.5,-0.5) .. controls (-1,-0.5) and (-1,-0.25) .. (-0.75,0);
    \draw[ultra thick] (-0.75,0) .. controls (-0.5,0.25) and (-0.25,0.5) .. (0.25,0.5) .. controls (0.5,0.5) and (2.25,0.5) .. (2.5,0.5);
    \draw[ultra thick] (-1.5,-0.5) -- (-3,-0.5);
    \node[right] at (2.5,0.5) {$L_2$};
    \node [right] at (2.5,-0.5) {$L_1$};
    \node[right] at (2.5,-2) {$L_1\#_UL_2$};
    \node at (-2.5,1.5) {$\mathbb C$};
    \draw[ultra thick, red, yshift=-.5] (-3,-0.5) -- (2.5,-0.5);
\draw[ultra thick, blue, yshift=.5] (2.5,0.5) .. controls (2.25,0.5) and (0.5,0.5) .. (0.25,0.5) .. controls (-0.25,0.5) and (-0.5,0.25) .. (-0.75,0) .. controls (-1,-0.25) and (-1,-0.5) .. (-1.5,-0.5) .. controls (-2.5,-0.5) and (-3,-0.5) .. (-3,-0.5);
\draw[side by side={blue}{red}]
(-3,-0.5) .. controls (-3.5,-0.5) and (-3.5,-2) .. (-3,-2) .. controls (1.5,-2) and (2,-2) .. (2.5,-2);
\end{tikzpicture}    \caption{The surgery trace is given by surgering $L_1\times \gamma_1$ and $L_2\times \gamma_2$, where $\gamma_1$ and $\gamma_2$ are the blue and red curve drawn.}
   \label{fig:surgeryconfiguration}
\end{figure}
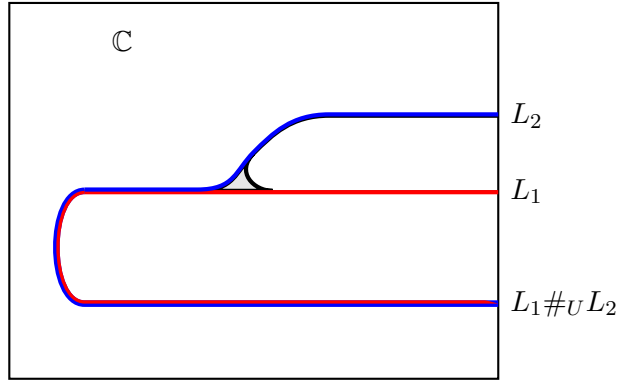
\section[Generation of partially wrapped Fukaya category]{Generation of $\Fuk(X, \stp)$}
\label{app:generation}
\subsection{Lagrangian cobordisms with Legendrian ends}
Let $\partial X^{int}\times \RR$ be the symplectization of a contact manifold $\partial X^{int}$.
We denote by $\pi_\RR: \partial X^{int}\times \RR\to \RR$ the projection to the symplectization coordinate. 
\begin{df}
   Let $\partial X^{int}$ be a contact manifold, $\Lambda^-, \Lambda^+\subset X$ be Legendrian submanifolds.
   A \emph{Lagrangian cobordism with Legendrian ends} $\Lambda^-, \Lambda^+$ is a Lagrangian submanifold of $\partial X^{int}\times \RR$ whose ends are conical in the sense that there exists an interval $[t_-, t_+]\subset \RR$ outside which 
   \[L\setminus \pi_\RR^{-1}([t_-, t_+])=( \Lambda^-\times (-\infty, t_-))\sqcup ( \Lambda^+\times (t_+, \infty)).\] 
   In this case, we say that $L$ is a Lagrangian cobordism from $\Lambda_+$ to $\Lambda_-$ and write $L: \Lambda_+\rightsquigarrow \Lambda_-$. 
\end{df}
Given two Lagrangian cobordisms with Legendrian ends, 
\begin{align*}
   L_{01}: \Lambda_0\rightsquigarrow \Lambda_1 && L_{12}: \Lambda_1\rightsquigarrow \Lambda_2
\end{align*}
there exists a concatenation Lagrangian cobordism with Legendrian ends $L_{01}\circ L_{10}: \Lambda_0\rightsquigarrow \Lambda_1$.
This cobordism is canonically defined up to Hamiltonian isotopy. 
In some circumstances, we can also decompose a Lagrangian cobordism between Legendrians into simpler pieces along the symplectization coordinate $t$.  
We start by describing cobordisms of Legendrians which arise from Legendrian isotopies. 

\begin{lem}[Sliceablility Criteria] 
   Let $\partial X^{int}$ be a $2n-1$ dimensional contact manifold with contact form $\lambda$.
   Suppose $\Lambda$ is an $n-1$ dimensional submanifold of $\partial X^{int}$ with isotopy parameterized by \[\phi(x, s): \Lambda\times \RR_s\to \partial X^{int}.\]
   The following are equivalent:
   \begin{enumerate}
      \item  The pullback $\iota_{\partial_s}(\phi^*d\lambda)=0$ vanishes so that  $L_\phi$, the submanifold parameterized by 
      \begin{align*} 
      \Phi:\Lambda\times \RR_s\to& \partial X^{int}\times \RR_t\\
        (x, s) \mapsto& (\phi(x, s), s),
      \end{align*}
      is a Lagrangian cobordism with Legendrian ends. 
      \label{item:lagrangiancondition}
      \item For each $s_0$, the manifold $\Lambda_{s_0}\subset \partial X^{int}$ parameterized by  $\phi(x, s_0)$ is a Legendrian submanifold.
      Furthermore, the isotopy $\phi(x, s)$ is generated by 
       \[H_{s}:=\iota_{\partial_s}(\phi^*(\lambda)):\Lambda\times \RR_s\to \RR\]
       which is constant across slices, in the sense that for all $\partial_w\in T\Lambda$, 
       \[\partial_w H_s=0.\]
      Finally, we require that the generating function  $H_s$ is compactly supported in the $s$ coordinate.
      \label{item:legendriancondition}
   \end{enumerate}
   \label{lemma:sliceable}
\end{lem}
\begin{proof}
   We will prove one direction: that \cref{item:lagrangiancondition} implies \cref{item:legendriancondition}.
   First, we show that $(\phi^*\lambda)|_{\Lambda\times \{s_0\}}=0$.  
    Let $\partial_w\in T_p(\Lambda\times \{s_0\})$ be any vector.
    Let $\partial_s\in T_p (\Lambda\times \RR_s)$ be the vector arising from the $\RR_s$ coordinate so that $\Phi^*dt(\partial_s)=1$. 
    Because $\Phi$ is a Lagrangian embedding,
    \begin{align*}
        0=\Phi^*\omega(\partial_s,\partial_w)=&e^{s_0} \cdot (\Phi^* dt\wedge \Phi^*\lambda)(\partial_s,\partial_w)+e^{s_0}\cdot (\Phi^* d\lambda)(\partial_s,\partial_w)\\
        =& e^{s_0}\cdot ( \Phi^*dt(\partial_s))\cdot (\Phi^*\lambda(\partial_w))+\iota_{\partial_s}(\Phi^*d\lambda)(\partial_w)\\
        =& e^{s_0}\cdot \Phi^*\lambda(\partial_w)
        = e^{s_0}\cdot \phi^*\lambda(\partial_w)|_{s=s_0}
    \end{align*}
    Therefore, $\lambda$ vanishes on $T_p\Lambda_{s_0}$ and we conclude that $\Lambda_{s_0}$ is Legendrian. 

    To check that $\iota_{\partial_t}(\phi^*(\lambda))$ is solely a function of $t$, we extend $\partial_w\in T\Lambda$ and $\partial_s\in T\Lambda\times \RR_s$ to vector fields on $\Lambda\times \RR_s$  so that $[\partial_w, \partial_s]=0$.
    The identity comes from the computation
    \begin{align*}
      0=&\iota_{\partial_s}(d\lambda)(\partial_w)\\
      =&\partial_s \phi^*\lambda(\partial_w)+ \partial_w \phi^*\lambda(\partial_s)+\lambda([\partial_s, \partial_w])\\
      =& \partial_w( \iota_{\partial_s}\phi^*\lambda).
    \end{align*}
   The reverse direction is analogous. 
\end{proof}

With \cref{lemma:sliceable} in mind:
\begin{df}
   We say that a cobordism of Legendrians $L\subset \partial X^{int}\times \RR_t$ is \emph{sliceable}  at $t_0\in \RR_s$ if:
   \begin{itemize}
      \item $t_0$ is a regular value of $\pi_\RR: L\to \RR$ and;
      \item There exists a vector field $\partial_t$ on $L$ transverse to $\pi^{-1}(t_0)$ with 
      \[\iota_{\partial_t}\phi^*(d\lambda)=0.\]
   \end{itemize}
   We call the Legendrian $\Lambda_{t_0}:=L\cap \pi^{-1}_\RR(t_0)$ the Legendrian slice of $L$ at $t_0$. 
   \label{def:sliceable}
\end{df}
\begin{lem}[Decomposition Lemma]
   Consider a Lagrangian cobordism with Legendrian ends $L_{+-}:\Lambda_+\rightsquigarrow\Lambda_-$ inside the symplectization $\partial X^{int}\times \RR$.
   Suppose that $L_{+-}$ is sliceable for all $s\in [-\eps, \eps]$, and so in particular we obtain a Legendrian slice $\Lambda_0$. 
   Then there exist cobordisms of Legendrians in $X\times \RR$ 
   \begin{align*}
    L_{+0}:& \Lambda_+\rightsquigarrow \Lambda_0\\
    L_{0-}:& \Lambda_0 \rightsquigarrow \Lambda_-
   \end{align*}
    whose concatenation is Hamiltonian isotopic to $L_{+-}$.
   \label{lemma:decomposition}
\end{lem}
\begin{proof}
   We need only show that it is possible to smooth out the ends created by cutting $L_{+-}$ in two pieces along $0$. 
   Consider the slice parameterization 
   \begin{align*}
      \Phi:\Lambda_{0}\times [-\eps, \eps]_s \to &\partial X^{int}\times [\eps, \eps]\\
      (p, s)\mapsto &(\phi(p,s), s)
   \end{align*}
   where $\phi(p, s_0): L_{s_0}\to \partial X^{int}$ is the parameterization of $\Lambda_{s_0}\subset \partial X^{int}$.  
   
   We now create a new isotopy  by reparameterization of the $s$ coordinate. 
   Pick a smooth increasing function $\rho+:[-\infty, \eps)\to [0, \eps]$ which satisfies the following properties:
   \[
      \rho_+(s)=\left\{\begin{array}{cc}
         0 & s<0\\
         s & s>\frac{\eps}{2}
      \end{array}\right.
   \] 
   and consider the embedding
   \begin{align*}
      \Phi_+:\Lambda_{0\times [-\infty, \eps]_s}\to& \partial X^{int}\times [-\infty, \eps]\\
      (p, s)\mapsto& (\phi(p, \rho(s)),s)
   \end{align*}
   As \cref{item:legendriancondition} of \cref{lemma:sliceable} holds,  $\phi_+(p, s):=(\phi(p, \rho(s)))$ satisfies the relation
   \[  \partial_w\iota_{\partial_s}(\phi_+^*(\lambda))= \frac{d\rho}{ds} \partial_w\iota_{\partial_s}\left(\phi^*(\lambda)|_{\rho(s)}\right)=0.\]
   
   Also, by \cref{lemma:sliceable}, $\Phi^+$ parameterizes a Lagrangian submanifold.
   The cobordisms of Legendrians $L_{+0}$ is parameterized by gluing $\Phi_+|_{s<\eps}$ and $\Phi|_{s\geq \eps}$ together.
   One can similarly construct $L_{0-}$ in this fashion. 

   By reducing the parameter $\eps$ to zero, we obtain a Hamiltonian isotopy between the composition $L_{0-}\circ L_{+0}\to L_{+-}$.
\end{proof}

A great source of sliceable Lagrangian cobordisms with Legendrian ends comes from exact Lagrangians. On an exact Lagrangian, $e^td\lambda$ vanishes, and \cref{def:sliceable} is automatically satisfied at all regular points of the projection $\pi_\RR: L\to \RR$. 

\subsection{Using both Legendrian and Lagrangian ends}
We now will look at Lagrangian cobordisms of non-compact Lagrangians -- specifically, Lagrangian cobordisms of Lagrangian cobordisms with Legendrian ends.
Let $L_-$ be a Lagrangian in $X$ with boundary $\Lambda_-$.
Suppose that we have a Lagrangian $L_{+-}\subset \partial X^{int}\times \RR $ with Legendrian ends $\Lambda_+, \Lambda_-$. 
Let $L_+=L_-\circ L_{+-}$. 
We will show that there is a Lagrangian cobordism with Lagrangian ends $K:L_-\rightsquigarrow L_+$.
\begin{thm}
   Let $L_{+0}:\Lambda_+\rightsquigarrow \Lambda_0$ and $L_{0-}: \Lambda_0\rightsquigarrow \Lambda_-$ be exact Lagrangian cobordisms in $\partial X^{int}\times \RR$ with Legendrian ends. 
   Suppose that $L_{+0}$ is disjoint from $\mathfrak f\subset \partial X^{int}\times \{\infty\}$.
   Let $L_{+-}:=L_{0-}\circ L_{+0}: \Lambda_+\rightsquigarrow \Lambda_0$ be the composition of these two cobordisms.
   There exists an exact Lagrangian cobordism with Lagrangian ends $K: L_{+0}\rightsquigarrow L_{+-}$  in $(\partial X^{int}\times \RR) \times \CC$ which is disjoint from $\mathfrak f \times \RR$.
\end{thm}
\begin{proof} 
   Pick $R$ sufficiently large so that $L_{+0}$ and $L_{0-}$ are conical outside of the regions $\pi^{-1}_\RR([0, R])$ and $\pi^{-1}_\RR([-R, 0])$ respectively, and choose the concatenation which identifies these two Lagrangians by their Legendrian boundaries at $\pi^{-1}(0)$. 
   Consider the isotopy 
   \begin{align*}
      \Phi_r: \partial X^{int}\times \RR_t\times \RR_r \to&\partial_0  X\times \RR\\
(x, t)\mapsto (x, e^{\rho(r)}t )
   \end{align*}
   where $\rho(r)$ is a function interpolating between $0$ and $R$ whose derivative has bounded support in $r$.
   The induced isotopy on exact Lagrangians is Hamiltonian.
   The flux of the isotopy is  explicitly given by $\rho(r)\lambda= d\rho(r) H$, where $H: L\to \RR$ is the primitive providing exactness of $\lambda|_L$. 

   The choice of $\rho(r)$ means that   $\Phi_1(\pi_\RR^{-1}([R,R]))\cap \pi_{\RR}^{-1}(0)=\emptyset$.
   Therefore, this symplectic isotopy displaces the non-conical portion of the Lagrangians $L_{-0}$ and $L_{0+}$ from itself.
 
   Let $K_{pre}\subset ((\partial X^{int}\times \RR)\times \CC)$ be the suspension of this Hamiltonian isotopy.
   As a topological space,   $K_{pre}=L_{+-}\times \RR$.
   The suspension is an exact Lagrangian submanifold.
   In addition to being a cobordism with Lagrangian ends, $K_{pre}$ is a Lagrangian cobordism with Legendrian ends $\Lambda_+\times \RR$ and $\Lambda_-\times \RR$.  
   $K_{pre}:\Lambda_+\times\RR_u\rightsquigarrow \Lambda_-\times\RR$ is exact, as it is Hamiltonian isotopic to $L_{+-}\times \RR$ (which is itself exact).

   We now take $K_{pre}$ and slice it along $\pi_\RR: (X\times \CC)\times \RR\to \RR$ at $0$ (see \cref{fig:cobordismofcobordism}).
   Because $K_{pre}$ is exact, it suffices to show that this is a regular value of the map $\pi_\RR$.
   For points $(p, u)\in K_{pre}$ with $u\in (0, 1)$, we have $(\pi_\RR)_*\partial_u=\rho'(u)\neq 0$. 
   Outside of this region, the slice can be computed explicitly as 
   \begin{align*}
      \pi_\RR^{-1}(0)\cap (K_{pre}|_{u\leq 0})=&\Lambda_0\times (-\infty, 0]\\
      \pi_\RR^{-1}(0)\cap (K_{pre}|_{u\geq 1})=&\Lambda_0\times [1, \infty).
   \end{align*}
   By \cref{lemma:decomposition}, $K_{pre}$ is the composition of cobordisms of with Legendrian ends, $K_{L_{0-}\rightsquigarrow  L_{\id}}\circ K_{L_{+0}\rightsquigarrow  L_{+-}}.$
   Each half is exact, and the top half of this composition provides the desired cobordism.
\end{proof}

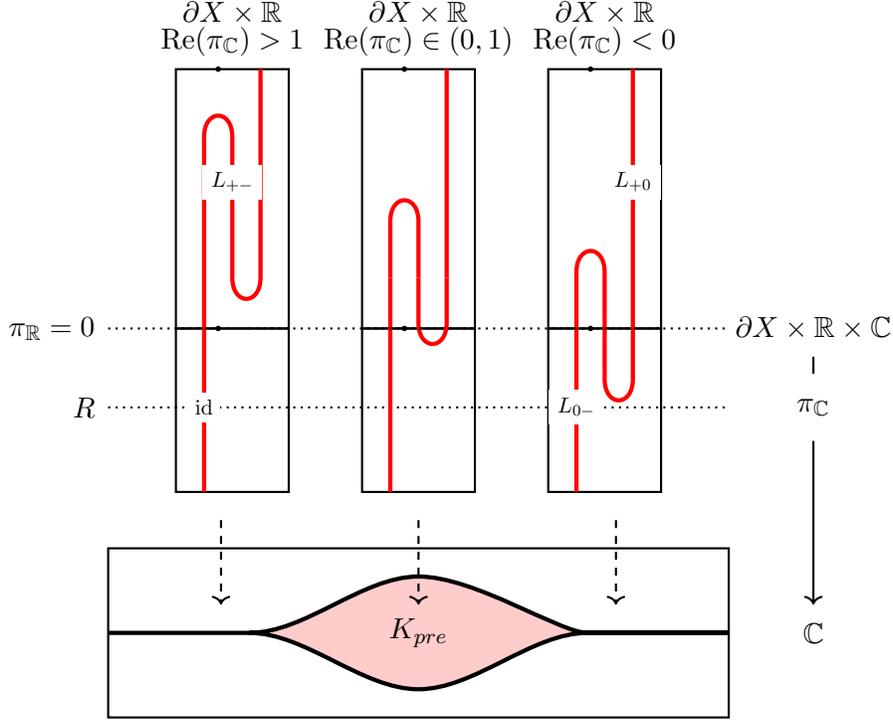
\begin{figure}
   \centering
   \begin{tikzpicture}[scale=.75, thick]

\node[left] at (-4,0) {$R$};
\draw[dotted] (-4,0) -- (7,0);
\begin{scope}[shift={(6.8,0)}]

\draw  (-3,6) rectangle (-1,1.4);

\draw  (-3,1.4) rectangle (-1,-1.5);
\draw[red, ultra thick] (-2.5,1.4) -- (-2.5,-1.5);

\draw[red, ultra thick] (-2.5,1.4) .. controls (-2.5,1.9) and (-2.5,1.9) .. (-2.5,2.4) .. controls (-2.5,2.9) and (-2,2.9) .. (-2,2.4) .. controls (-2,1.9) and (-2,1.9) .. (-2,1.4);
\draw[red, ultra thick] (-1.5,1.4) -- (-1.5,6);
\draw[red, ultra thick] (-2,1.4) .. controls (-2,1.0) and (-2,1.0) .. (-2,0.5) .. controls (-2,0.0) and (-1.5,0.0) .. (-1.5,0.5) .. controls (-1.5,1.0) and (-1.5,1.0) .. (-1.5,1.4);
\node at (-2,7) {$\partial X\times \mathbb R$};
\node at (-2,6.5) {$\Re(\pi_{\mathbb C})<0$};
\node[circle, fill, scale=.2] at (-2.25,6) {};
\node[circle, fill, scale=.2]at (-2.25,1.4) {};
\node[fill=white, scale=.75] at (-2.5,0) {$L_{0-}$};
\node[fill=white, scale=.75] at (-1.5,4) {$L_{+0}$};
\end{scope}
\begin{scope}[shift={(3.5,0)}]

\draw  (-3,6) rectangle (-1,1.4);

\draw  (-3,1.4) rectangle (-1,-1.5);
\draw[red, ultra thick] (-2.5,2.3) -- (-2.5,-1.5);

\draw[red, ultra thick] (-2.5,2.3) .. controls (-2.5,2.8) and (-2.5,2.8) .. (-2.5,3.3) .. controls (-2.5,3.8) and (-2,3.8) .. (-2,3.3) .. controls (-2,2.8) and (-2,2.8) .. (-2,2.3);
\draw[red, ultra thick] (-1.5,2.3) -- (-1.5,6);
\draw[red, ultra thick] (-2,2.3) .. controls (-2,2) and (-2,2) .. (-2,1.5) .. controls (-2,1) and (-1.5,1) .. (-1.5,1.5) .. controls (-1.5,2) and (-1.5,2) .. (-1.5,2.3);
\node at (-2,7) {$\partial X\times \mathbb R$};
\node at (-2,6.5) {$\Re(\pi_{\mathbb C})\in (0, 1)$};
\node[circle, fill, scale=.2] at (-2.25,6) {};
\node[circle, fill, scale=.2]at (-2.25,1.4) {};
\end{scope}

\begin{scope}[shift={(0.2,0)}]

\draw  (-3,6) rectangle (-1,1.4);

\draw  (-3,1.4) rectangle (-1,-1.5);
\draw[red, ultra thick] (-2.5,3.8) -- (-2.5,-1.5);

\draw[red, ultra thick] (-2.5,3.8) .. controls (-2.5,4.3) and (-2.5,4.3) .. (-2.5,4.8) .. controls (-2.5,5.3) and (-2,5.3) .. (-2,4.8) .. controls (-2,4.3) and (-2,4.3) .. (-2,3.8);
\draw[red, ultra thick] (-1.5,3.8) -- (-1.5,6);
\draw[red, ultra thick] (-2,3.8) .. controls (-2,2.8) and (-2,2.8) .. (-2,2.3) .. controls (-2,1.8) and (-1.5,1.8) .. (-1.5,2.3) .. controls (-1.5,2.8) and (-1.5,2.8) .. (-1.5,3.8);
\node at (-2,7) {$\partial X\times \mathbb R$};
\node at (-2,6.5) {$\Re(\pi_{\mathbb C})>1$};
\node[circle, fill, scale=.2] at (-2.25,6) {};
\node[circle, fill, scale=.2]at (-2.25,1.4) {};
\node[fill=white, scale=.75] at (-2.5,0) {id};
\node[fill=white, scale=.75] at (-2,4) {$L_{+-}$};
\end{scope}

\draw  (-4,-2.5) rectangle (7,-5.5);
\draw[fill=red!20,ultra thick] (-4,-4) .. controls (-3.5,-4) and (-2,-4) .. (-1.5,-4) .. controls (-0.5,-4) and (0.5,-3) .. (1.5,-3) .. controls (2.5,-3) and (4,-4) .. (4.5,-4) .. controls (5,-4) and (6.5,-4) .. (7,-4) .. controls (6.5,-4) and (5,-4) .. (4.5,-4) .. controls (3.5,-4) and (2.5,-5) .. (1.5,-5) .. controls (0.5,-5) and (-0.5,-4) .. (-1.5,-4);
\draw[dashed, ->] (-2,-2) -- (-2,-3.5);
\draw[dashed, ->] (1.5,-2) -- (1.5,-3.5);
\draw[dashed, ->] (5,-2) -- (5,-3.5);
\node at (1.5,-4) {$K_{pre}$};
\draw[dotted] (-4,1.4) -- (7,1.4);
\node at (-5,1.4) {$\pi_{\mathbb R}=0$};
\node at (8.5,1.4) {$\partial X\times \mathbb R\times \mathbb C$};
\draw[->] (8.5,0.9) -- (8.5,-3.5);
\node at (8.5,-4) {$\mathbb C$};
\node[circle, fill=white] at (8.5,0) {$\pi_{\mathbb C}$};
\end{tikzpicture}    \caption{Construction of the Lagrangian cobordism $K_{pre}$, drawn in the fibers of the map $\pi_\CC: (\partial X^{int}\times \RR\times \CC)\to \CC$}
   \label{fig:cobordismofcobordism}
\end{figure}
An application of this decomposition is a canonical cobordism associated to admissible Lagrangian submanifolds in Liouville domains with stops.
\begin{prop}[Decomposition into linking disks]
   Let $L\subset (X,\LL_\stp)$ be a Lagrangian submanifold with primitive $dH=\lambda|_L$. Suppose that $L$ intersects $\LL_\stp$  transversely at the points $x_1, \ldots, x_k$, ordered so that $H(x_1)<H(x_2)<\cdots < H(x_k)$.
   There is an exact Lagrangian cobordism  $K:(\cocore({x_1}), \ldots\cocore({x_k}))\rightsquigarrow L$, where  $\cocore({x_1}), \ldots, \cocore({x_k})$ are the linking disks and cocores of the $x_i$.
\end{prop}
\begin{proof}
   There exists a neighborhood $B_\eps(\LL_\stp)$ so that after applying a Hamiltonian isotopy, 
   \[L\cap B_\eps(\LL_\stp)=\bigcup_{i=1}^k \cocore({x_i}).\]
   Let $\Phi_{t}: X\times \RR\to \RR$ be the time $t$ Liouville flow. 
   Pick $X^{int}\subset X$ so that $L$ is cylindrical outside of $X^{int}$.
   Since the Liouville flow eventually displaces all points in $X^{int}\setminus \LL$ from $X^{int}$, there  exists a time $t_0$ so that for all $t>t_0$, 
   \[\bigcup_{i=1}^k\cocore({x_i})|_{X^{int}}= \Phi_{t_0}(L)|_{X^{int}}.\]
   Consider a function $\rho:\RR\to \RR$ which is increasing, $\rho(t)=0$ for $t\ll 0$, and $\rho(t)=t_0$ for $t> t_0$.
   This means that the restriction of the isotopy to $X^{int}$ is constant after time $t_0$.
   The isotopy $\Phi_{\rho(t)}(L)$ is an exact isotopy whose primitive is $\rho(t)H$.

   Let $K^{pre}\subset X\times \CC$ be the suspension of this Hamiltonian isotopy. 
   Consider now the Lagrangian cobordism $K$ obtained by taking $K^{pre}$ and slicing it at $\partial_0X^{int}$. 
   This is an exact Lagrangian cobordism whose
   \begin{itemize}
      \item Negative end is the completion of $L|_{X^{int}}$, which is $L$ and;
      \item Positive ends are at completions of the $\cocore({x_i})|_{X^{int}}$.
      Since these complete to linking disks and cocores, the right end is a collection of cocores and linking disks.
      Notice that the imaginary coordinate of a Lagrangian suspension is determined by the primitive $H$ for the Hamiltonian isotopy. 
   \end{itemize}
In this case, the primitive is primitive for the Liouville form. 
Therefore, at points $x$ inside the right ends of the Lagrangian cobordism, we have that the imaginary coordinate is
\[H(x)= H(x_i)\text{ if $x\in D^n_{x_i}$},\]
which gives us the ordering of the ends of the Lagrangian cobordism.
\end{proof}
This proposition, along with \cref{thm:cobordismgeneration}, gives a repackaging of the generation statement from \cite{ganatra2018sectorial,chantraine2017geometric,tanaka2018generation}.
\begin{cor}
   Linking disks (and cocores) generate $\mathcal W(X, \stp)$.
\end{cor}

\section{Correspondences and cobordisms}
\label{app:correspondenceandcobordism}
A Lagrangian correspondence $\Lc: X_0\Rightarrow X_1$ between two compact symplectic manifolds $X_0, X_1$ is a Lagrangian submanifold $\Lc_{01}$ in $(X_0^-\times X_1, (-\omega_0)\oplus\omega_1)$. 
Given a Lagrangian correspondence $L_{01}: X_0\Rightarrow  X_1$ and a Lagrangian correspondence $\Lc_{12}: X_1\Rightarrow X_2$, one can obtain a set called the \emph{geometric composition  $\Lc_{12}\circ \Lc_{01}:X_0\Rightarrow X_2$} by taking the intersection inside $X_0\times X_1\times X_2$
\[\Lc_{12}\circ L_{01}:=\pi_{02}(\pi_{01}^{-1}(\Lc_{01})\cap \pi_{12}^{-1}(\Lc_{12})).\]
Here, $\pi_{ij}: X_0\times X_1\times X_2\to X_i\times X_j$ is the standard projection. 
When the intersection $\pi_{01}^{-1}(\Lc_{01})\cap \pi_{12}^{-1}(\Lc_{12})\cap \pi_{01}^{-1}(x)$ is transverse for all $x\in X_0\times X_1^{-1}$, this is an immersed Lagrangian submanifold.
One can always achieve transversality by taking a Hamiltonian isotopy. 
\begin{prop}
   Given submanifolds $M\subset X_0\times X_1$ and $N\subset X_1\times X_2$, there exist Hamiltonian isotopies $\theta_{12}:X_1\times X_2\to X_1\times X_2$ and $\theta_{01}:X_0\times X_1\to X_0\times X_1$  so that $\pi^{-1}_{01}(\theta_{01}(M)), \pi^{-1}_{12}(\theta_{12}(N))\subset X_0\times X_1\times X_2$ intersect transversely. 
   Furthermore, if the intersection $M\cap N$ is transverse outside of $U\subset X_0\times X_1\times X_2$, the support of $H_{ij}$ can be chosen to lie inside $\pi_{ij}(U)$.  
\end{prop}
\begin{proof}
   This is similar to the standard proof of transversality for Lagrangian submanifolds. Let $C^\infty(X_0\times X_1), C^\infty(X_1\times X_2)$ be the space of Hamiltonian functions for $X_0\times X_1$ and $X_1\times X_2$. Consider the map from 
   \begin{align*}
      M\times C^\infty(X_0\times X_1)\times C^\infty(X_1\times X_2)\to X_0\times X_1\times X_2\\
      (x, f, g)\mapsto \theta_f^1\circ \theta_g^1 \circ i(x)
   \end{align*}
   where $i: M\to X_0\times X_1$ is the given parameterization, and $\theta_f, \theta_g$ are the time 1-Hamiltonian flows of $f, g$ on their corresponding components. 
   This map is a submersion, and therefore transverse to $N$. 
   It follows that for generic $f, g$, the inclusion $ \theta_f^1\circ \theta_g^1 \circ i: M\to X_0\times X_1\times X_2$ is transverse to $N$.
\end{proof}

In this subsection, we'll show that the Lagrangian cobordism class of $(\theta_{12}(\Lc_{12}))\circ (\theta_{01}(\Lc_{01}))$ is independent of the choice of perturbation $\theta_{01},\theta_{12}$. 
In fact, we show that the Lagrangian cobordism class is preserved by Lagrangian correspondence. 
\begin{prop}
   Let $X_0, X_1, X_2$ be compact and
   \begin{align*}
      K_{01}:(\Lc_{01}^1, \ldots, \Lc_{01}^k)\rightsquigarrow\Lc_{01}^0
   \end{align*}
   be a cobordism of correspondences with $K_{01}\subset X_0^-\times X_1\times \CC$. Let $L_{12}\subset X_1^-\times X_2$ be another Lagrangian correspondence. Suppose that the geometric compositions $\Lc_{02}^i=\Lc_{12}\circ \Lc_{01}^i$ are immersed Lagrangian submanifolds.
   Then there is an immersed Lagrangian cobordism $\Lc_{12}\circ K_{01}\subset X_0^-\times X_2\times \CC$ with ends 
   \[\Lc_{12}\circ K_{01}: (\Lc_{02}^1, \ldots, \Lc_{02}^k)\rightsquigarrow \Lc_{02}^0.\]  
   \label{claim:correspondenceofcobordism}
\end{prop}

\begin{proof}
   The cobordism is defined by taking the intersection of the  $\pi_{01\CC}^{-1}(K_{01})$ and $\pi_{12}^{-1}(L_{12})$ in $X_0\times X_1\times X_2\times \CC$ and projecting down 
   \[
      \Lc_{12}\circ K_{01}:=\pi_{02\CC}(  \pi_{01\CC}^{-1}(K_{01})\cap \pi_{12}^{-1}(L_{12}))
   \]
   We first compute the ``ends'' of this prospective Lagrangian cobordism.
   Let $r+\jmath t_i$ be a point contained in the $i$th  ray of $K_{01}$.
   Then for any $r\ll 0$, we have
   \begin{align*}
      \pi_{02}\circ (\pi_\CC^{-1}(r+\jmath t_i)\cap (L_{12}\circ K_{01}))=& \pi_{02} (\pi_\CC^{-1}(r+\jmath t_i)\cap \pi_{01\CC}^{-1}(K_{01})\cap \pi_{12}^{-1}(L_{12}))\\
      =& \pi_{02}(\pi_{01\CC}^{-1}(\pi_\CC(r+\jmath t_i)\cap K_{01})\cap \pi_{12}^{-1}(L_{12}))\\
      =& \pi_{02}\left(\pi_{01}^{-1}\left(L_{01}^i\right)\cap \pi_{12}^{-1}(L_{12})\right)\\
      =&  L_{01}^i \circ L_{12}.
   \end{align*}
   This shows that the intersection defining $\Lc_{12}\circ K_{01}$ is transverse for $r \ll 0$.
   Therefore, we can apply a compactly supported Hamiltonian isotopy to make the intersection $\pi_{01\CC}^{-1}(K_{01})\cap \pi_{12}^{-1}(L_{12})$ transverse so that $\Lc_{12}\circ K_{01}$ is an immersed submanifold.

   To see that this is a Lagrangian submanifold, we use the Lagrangian conditions of $K_{01}$ and $L_{12}$:
   \begin{align*}
      (-\omega_0+\omega_1+\omega_z)|_{\pi_{12\CC}^{-1}( K_{01})}=0 && (-\omega_1+\omega_2)|_{\pi_{12}^{-1}(L_{12})}=0
   \end{align*}
   from which it follows that $(-\omega_0+\omega_2+\omega_z)|_{\pi_{12\CC}^{-1}(K_{01})\cap\pi_{12}^{-1}(L_{12}) }=0$
\end{proof}
Under the additional hypothesis that  $\Lc_{02}^i=\Lc_{12}\circ \Lc_{01}^i$ are embedded, the Lagrangian submanifold $\Lc_{12}\circ K_{01}$ is embedded away from a compact set. 
We can subsequently take a compact Lagrangian homotopy to obtain  $(\Lc_{12}\circ K_{01})'$ which has transverse self-intersections. Applying Polterovich surgery at these self intersections yields $(\Lc_{12}\circ K_{01})'': (\Lc_{02}^1, \ldots, \Lc_{02}^k)\rightsquigarrow \Lc_{02}^0$ which is an embedded Lagrangian cobordism.
\begin{cor}
   Let $\theta_{01}, \theta_{01}':X_0^-\times X_1\to X_0^-\times X_1$ be two Hamiltonian isotopies, chosen so that the geometric compositions $L_{12}\circ \theta_{01}(L_{01}) $ and $L_{12}\circ \theta_{01}'(L_{01}) $ are defined from a transverse intersection. 
   The two geometric compositions are cobordant. 
\end{cor}
\begin{proof}
   Take the suspension cobordism $K:\theta_{01}(L_{12})\rightsquigarrow \theta_{01}'(L_{12})$.
   The composition $L_{12}\circ K$ may not be cut out transversely.
   Since $L_{12}\circ \theta_{01}(L_{01}) $ and $L_{12}\circ \theta_{01}'(L_{01}) $ are cut out transversely, we only have a failure of transversality away from the ends.
   Therefore, there exists another Lagrangian cobordism $K':\theta_{01}(L_{12})\rightsquigarrow \theta_{01}'(L_{12})$, Hamiltonian isotopic to $K$, for which $L_{12}\circ K'$ is cut out transversely.
   It follows from \cref{claim:correspondenceofcobordism} that  $L_{12}\circ \theta_{01}(L_{01}) $ and $L_{12}\circ \theta_{01}'(L_{01}) $ are cobordant. 
\end{proof}

It is conjectured that a Lagrangian correspondence in $X_0^-\times X_1$ provides a functor between the Fukaya categories of $X_0$ and $X_1$ which acts on objects by geometric composition (and  Lagrangians $L_0\subset X_0$ are interpreted as Lagrangian correspondences in $\{\text{pt}\}\times X_0$). 
We obtain two interpretations of the correspondence/cobordism relation by either putting the cobordism on the objects $L_0$ or the morphisms $L_{01}$. 
From an algebraic standpoint, these cobordisms tell us something about the Weinstein Fukaya 2-category $\textbf{Floer}^\#$, which is a conjectured $A_\infty$-2-category whose objects are symplectic manifolds, and morphisms are Lagrangian correspondences. 
The 2-morphisms $\hom(L_{01}, L_{01}')$ are given by Floer theory, and one recovers the Fukaya category of a symplectic manifold $X$ by looking at the category $\Hom(pt, X)= \Fuk(pt\times X)=\Fuk(X)$.
A version of this category is constructed in\cite{wehrheim2007functoriality} on the level of cohomology. 

We now describe a decategorified version of this 2-category. 
We define $\Cob(\textbf{Floer}^\#)$ to be the category whose objects are symplectic manifolds, and whose morphisms $\Cob(X_1, X_2)$ are formal linear combinations of Lagrangian submanifolds in $(X_1^-\times X_2)$ mod cobordance equivalence, that is $[L^1_{12}]+[L^2_{12}]=[L^0_{12}]$ whenever there is a Lagrangian cobordism $K: (L^1_{12}, L^2_{12})\rightsquigarrow L^0_{12}$.
The vertical composition comes from the abelian group structure on the cobordism group. 
In this category, horizontal composition is the geometric composition of Lagrangian correspondence. \cref{claim:correspondenceofcobordism} shows us that geometric composition with a Lagrangian correspondence $L_{23}\in X_2^-\times X_3$ gives us a well defined map over the equivalence class relation
\[L_{23}:\Cob(X_1, X_2)\to \Cob(X_1, X_3).\]
By applying \cref{claim:correspondenceofcobordism}  to the morphism side of the composition, we obtain the 2-category relation,
\[[(L_{23}^1+L_{23}^2)\circ L_{12}]=[L_{23}^1\circ L_{12}]+ [L_{23}^2\circ L_{12}].\]
We remark that this contains less data than the Donaldson-Fukaya 2-category. 
One interpretation of \cref{thm:cobordismgeneration} given in \cite{biran2014lagrangian} is that $\Cob(\textbf{Floer}^\#)$ behaves like the group $K_0(\textbf{Floer}^\#)$.
On the mirror side, the appropriate comparison is the derived category of coherent sheaves with functors given by Fourier-Mukai kernel versus the Chow groups with morphisms given by push-pull and intersection.
We learned of this interpretation of Lagrangian cobordisms as the mirror to rational equivalence from \cite[Section 1.3]{sheridan2021lagrangian}.

\section{Lagrangian correspondence and admissibility}
\label{app:correspondenceadmissibility}
We look at extending Lagrangian correspondences to the setting of stopped Liouville domains. 
The main difficulty is solved by \cite[Lemma 2.21]{ganatra2017covariantly} which constructs from the data of Liouville manifolds $(X_1^{int}, \stp_1)$ and $(X_2^{int}, \stp_2)$  a product Liouville manifold with a product stop, 
\[(X_1^{int}\times X_2^{int}, \overline {\stp_1\times \stp_2}).\] 
Here, the contact boundary of $X_1^{int}\times X_2^{int}$ is obtained by smoothing out the corner strata of $\partial (X_1^{int}\times X_2^{int})$.
This smoothing covers $\partial(X^{int}_1\times X^{int}_2)$ with charts which resemble $X^{int}_1\times (\partial X^{int}_2), (\partial X^{int}_1)\times X^{int}_2$, and $(\partial X^{int}_1)\times (\partial X^{int}_2)\times (\RR)$.
The $\RR$ component of the last chart from taking a neighborhood of the boundaries of $X^{int}_1$ and $X^{int}_2$, and identifying the symplectization coordinate $t_1$ for $X_1$ with $-t_2$ for $X_2$.

The product stop is given by \cite[Equation 6.1]{ganatra2018sectorial}:
\begin{equation}
    \overline {\stp_1\times \stp_2} = (\stp_1\times \core_1 ) \cup (\stp_1\times \stp_2 \times \RR) \cup (\core_2\times \stp_1).
    \label{eq:prodstop}
\end{equation}
where $\stp_1\times \core_1, \core_2\times \stp_1$ lay in the boundary strata of $X_1\times X_2$, and $\stp_1\times \stp_2\times \RR$ lives in the corner smoothing chart. 
\begin{ex}
   The simplest example of this product smoothing can be seen in toric varieties. Let $\check X_1, \check X_2=\CC$, and let $(X_1^{int}, \stp_1), (X_2^{int}, \stp_2)$ be the mirror stopped domains, so that $X_i= \{z\in \CC^* \; : \; -1<\log|z|<1\}$.
   The FLTZ stop  is the point $e^{-1}\in \CC^*$.
   The contact boundary for $X_1\times X_2$ with rounded corners can be obtained by first considering the product of the symplectic completions, $X_1\times X_2\subset (\CC^*)^2= T^*T^2$, and taking the unit conormal bundle $S^*_1(T^2)=T^2\times S^1$. 
   The boundary stratification given by \cite{ganatra2017covariantly} covers this with three charts: $(S^1\times S^0)\times X_2, X_1\times (S^1\times S^0)$, and $(S^0\times S^1)\times (S^1\times S^0)\times \RR$.
   The product stop is made of three components, and is topologically:
   \[
      \overline{\stp_1\times \stp_2}=(\{e^{-1}\}\times S^1)\cup (S^1\times \{e^{-1}\})\cup (\{e^{-1}\}\times \{e^{-1}\}\times \RR).
   \]
   This last portion gives us the real quarter circle $(e^{-1}\cdot \cos\theta, e^{-1}\cdot \sin \theta)\subset S^1\times T^2$ with $\theta \in [0,\pi/2]$.
   We can complete the product stop to a skeleton, which has four components, $S^1\times S^1, \RR_{>0}\times S^1, S^1\times \RR_{>0},$ and $\RR_{>0}\times \RR_{>0}$. 
   The obtained skeleton matches the FLTZ skeleton for the mirror to $\CC^2$.
   More generally, the product of FLTZ skeletons is the FLTZ skeleton of the product fan. 
\end{ex}
A  stopped Liouville domain $(X, \stp)$ is called \emph{totally stopped} if there is a positive flow from any point of $(\partial X^{int})\setminus \stp$ to the stop.
In this case, there exists a sequence of admissible positive wrappings so that for every point $x\in \partial X^{int}$, $\lim_{t\to\infty} \theta^t(x)\in \stp$. 
Similarly, we say that an admissible Lagrangian $L\subset (X, \stp)$ is \emph{totally stopped} if there is a Reeb chord from every point $x\in \partial_0 L$ to $\stp.$
Conjecturally, totally stopped Liouville domains should have compact mirrors. If a Lagrangian $L$ is totally stopped, it is expected to be mirror to a perfect complex with compact support. 

\begin{df}
   An admissible Lagrangian correspondence  $\Lc_{12}:(X_1, \stp_1)\Rightarrow (X_2, \stp_2)$ is an admissible Lagrangian submanifold $\Lc_{12}\subset (X_1^-\times X_2, \overline{\stp_1\times \stp_2})$.
\end{df}

\begin{prop}
   Suppose that $\Lc_{12}\subset (X_1^-\times  X_2, \overline{\stp_1\times \stp_2})$ is a Lagrangian correspondence and that $L_1\subset (X_1, \stp_1)$ is a totally stopped Lagrangian.
   There exists $L_1'$, admissibly Hamiltonian isotopic to $L_1$, and $X_2^{int}\subset X_2$ compact so that $\Lc_{12}\circ L_1'$ avoids $\LL_2\setminus X_2^{int}$.
   \label{claim:compositionstop}
\end{prop}
\begin{proof}
   As $L_{12}$ is admissible, $L_{12}$ is disjoint from $(\LL_1\times \LL_2 ) \setminus (X^{int}_1\times X^{int}_2)$. 
   Therefore $\pi_1(L_{12}\cap (X_1\times \LL_2))$ is disjoint from $X^{\eps}_1$, a small neighborhood of $\LL_1$.\
   Since $L_1$ is totally stopped, there exists $L_1'$, admissibly Hamiltonian isotopic to $L_1$, so that $L_1'\subset X^{\eps}_1$. 
   Therefore, $\pi_1^{-1}(L_1')$ is disjoint from $L_{12} \cap \pi_2^{-1}(\LL_2\setminus X^{int}_2)$, and we conclude that $\Lc_{12}\circ L_1'$ is disjoint from $\LL_2\setminus X^{int}_2$. 
\end{proof}
We note that ``totally stopped'' is not a necessary requirement for making sense of Lagrangian correspondences of Louiville domains. 
For instance, any symplectomorphism of a Liouville domain which preserves the stop should give rise to a Lagrangian correspondence which preserves admissible Lagrangian submanifolds.
A more general set of examples come from stop removal in FLTZ skeletons; in particular the functors explored in \cref{sec:applications} mirror to localization away from a divisor and toric blowdown.
\begin{prop}
   Suppose that $\Lc_{12}\subset X_1^-\times  X_2$ is an admissible Lagrangian correspondence between Liouville domains. 
   Let $L_1\subset X_1$ be an admissible Lagrangian submanifold.
   Then there exists a Lagrangian submanifold $L_2$, Hamiltonian isotopic to $\Lc_{12}\circ L_1$, which is an admissible Lagrangian submanifold of $X_2$.
   \label{claim:correspondencecomposition}
\end{prop}
The proof we give follows \cite[Section 6.2]{ganatra2018sectorial}, which discusses admissibility of a product Lagrangian in a product Liouville manifold. 
\begin{proof}
   Let $f_1: L_1\to \RR$ be the primitive for $\lambda_{X_1}|_L$.
   We look at compact sets $X_1^{int}\subset X_1, X_2^{int}\subset X_2$ so that $L_1$ is conical outside of $X_1^{int}$, and $\Lc_{12}$ is conical outside of $X_1^{int}\times X_2^{int}$. 
   Following \cite[Lemma 6.1]{ganatra2018sectorial}, we will assume that $f_1$ has compact support. Extend $f_1$ to a compactly supported function $f_1: X_1\to \RR$.
   Pick a bump function $\rho_{2}:X_2\to \RR$ with the property that  $\rho_{2}\circ \pi_{X_2}((X_1\times X_2)\setminus B_\epsilon(X_1^{int}\times X_2^{int}))=1$, and vanishes inside $X_1^{int}\times X_2^{int}$.

   From the primitives $d\lambda_i=\omega_i$ for $X_i$, we construct a modified primitive for $\omega_1\oplus (-\omega_2)$ on $X_1\times X_2$ given by 
   \[\tilde \lambda = \lambda_1-\lambda_2+ d(\rho_2 f_1).\]
   For this modified $\tilde \lambda$ we show that the Liouville vector field $\tilde Z_{X_1\times X_2}$ is parallel to the submanifold $\pi^{-1}_1(L_1)\subset X_1\times X_2$ outside of a compact set:
   \begin{itemize}
      \item Over $X_1\times X_2 \setminus \pi^{-1}_1(X_1^{int})$, the function $f_1$ vanishes, so the Liouville vector field reduces to
      \[\tilde Z_{X_1\times X_2}|_{\pi^{-1}_X(X_1^{int})} = Z_{X_1}-Z_{X_2}\]
      As $L_1$ is conical outside of $X_1^{int}$, the tangent space  $T(\pi^{-1}_1(L\setminus X_1^{int}))$ contains $\langle Z_X\rangle \times TX_2$; in particular, $\pi^{-1}_1(L_1)$ is parallel to $\tilde Z_{X_1\times X_2}$ outside of  $\pi^{-1}_1(X_1^{int})$.
      \item Over $X_1\times X_2 \setminus \pi^{-1}_2(X_2^{int})$ the bump function is  $\rho_2=1$, so 
      \[\tilde Z_{X_1\times X_2}=Z_{X_1}-Z_{X_2}+Z_{f_L}\]
      where $Z_{f_{1}}$ is the Hamilton vector field of $f_1: X_1\to \RR$. 
      By design $Z_{X_1}+Z_{f_1}$ is parallel to $L_1$ everywhere. Therefore,  the tangent space to  $\pi^{-1}_1(L_1)$ contains $\langle Z_{X_1}+Z_{f_1}\rangle \times TX_2$ outside of $ \pi_2^{-1}(X_2\setminus X_2^{int})$; in particular, this is parallel to $\tilde Z_{X_1\times X_2}$
   \end{itemize}
   Therefore $\pi^{-1}_1(L_1)$ is $\tilde \lambda$-conical at infinity. 
   Continuing to follow \cite[Section 6.3]{ganatra2018sectorial}, the linear interpolation $\tilde\lambda_t$ between the Liouville structures $\lambda_{X_1\times X_2}$ and $\tilde \lambda_{X_1\times X_2}$ is a Liouville homotopy; therefore by \cite[Proposition 11.8]{cieliebak2012stein}, there exists an exact symplectic isotopy $\Phi_t$ of $X_1\times X_2$ inducing this Liouville homotopy. 

   We now give a short outline of \cite[Proposition 11.8]{cieliebak2012stein} to additionally show that this exact symplectic isotopy can be chosen with support in a neighborhood of $\pi^{-1}(L_1)$. 
   This will allow us to choose this exact symplectic isotopy in such a way that we continue to avoid the stop.
   For notational convenience, we will write $X=X_1\times X_2$ and $ \lambda_t=\tilde \lambda_t$. 
   A Liouville homotopy consists of a family of domains $X_t^k\subset X$ which exhaust $X$ along which $Z_{\lambda_t}$ points outward. 
   For each $t$, there exists a contactomorphism $f_t^k: Y_0^k\to Y_t^k$, where $Y_t^k$ is the contact boundary to $(X_t^k, \lambda_t)$. 
   This contactomorphism arises from Gray's theorem, and is the identity outside of the support of $\frac{d}{dt} \lambda_t$. 
   This contactomorphism does not preserve the contact form, which differs from the original contact form by a scaling factor given by a function $\rho^k_t: Y^k_t\to \RR$. 

   Let $Y^{k, c}_t=\phi_{Z_{\lambda_t}}^c(Y^k_t)$, the time $c$ flow of the contact boundary $Y^k_t$ along the Liouville vector field $Z_{\lambda_t}$.
   Since the $Z_{\lambda_t}$ agree where $\frac{d}{dt} \lambda_t = 0$, the compositions 
   \[\phi^{k, c}_t :=\phi_{Z_{\lambda_t}}^c\circ f_t^k\circ \phi_{Z_{\lambda_t}}^{-c}: Y^{k, c}_0\to Y^{k, c}_t\]
   are equal to the identity on the region where $\frac{d}{dt}\lambda_t=0$. 
   Then, \cite{cieliebak2012stein} uses these diffeomorphisms on the contact slices of $X$ to build a family of symplectomorphisms on disjoint open neighborhoods  of the $Y^{k,c}_t$ in $X$.
   This is accomplished by scaling the neck coordinate of an open neighborhood of $Y^{k,c}_t$ by the $\rho^k_t$. 
   
   The portion of a symplectomorphism defined over a neighborhood of contact slices is extended to a family of diffeomorphisms of $\Phi_s: X\to X$, where $s\in [0,1]$.
   This extension can be done in such a way that it is the identity on $\frac{d}{dt}(\lambda_t)=0$. 
   From here, Moser's stability is applied to obtain a symplectic isotopy $\tilde \Phi_s:X\to X$ which agree with $\phi^{k, c}_t$ on the slices, and are identity where $\frac{d}{dt}(\lambda_t)=0$.
   
   Returning to our setup, there exists Hamiltonian symplectomorphism $\tilde \Phi_1:X_1\times X_2\to X_1\times X_2$, constant over $X_1^{int}\times X_2^{int}$, so that $(\tilde \Phi_1)_*(\tilde Z_{X_1\times X_2})=Z_{X_1\times X_2}$, and $\tilde \Phi_1=\id$ outside the support of $f_1$. 
   With this Hamiltonian isotopy, $\tilde \Phi_1(\pi^{-1}_1(L_1))$ is parallel to $Z_{X_1\times X_2}$ outside of a compact set. 
   
   The corrected geometric composition of $L_{12}$ and $L_1$ is defined to be:
   \[\Lc_{12}\tilde \circ (L_1):= \pi_2(\tilde \Phi_1(\pi_1^{-1}(L_1))\cap \Lc_{12}).\]
   Since $\tilde \Phi_1(\pi_1^{-1}(L_1)), K_{12}$ are both parallel to $Z_{X_1\times X_2}=(Z_{X_1}, Z_{X_2})$ outside of $\pi_2^{-1}(X_2^{int})$, the corrected geometric composition is parallel to $Z_{X_2}$ outside of the set $X_2^{int}$. 
\end{proof}
\begin{cor}
   Suppose $(X_1, \stp_1)$ is totally stopped. 
   Let $L_1\in (X_1, \stp_1)$ be an admissible Lagrangian, and  $L_{12}: (X_1, \stp_1)\Rightarrow (X_2, \stp_2)$ be an admissible Lagrangian correspondence.
   Then there exists $L_1'$, admissibly Hamiltonian isotopic to $L_1$, so that the geometric composition $L_{12}\tilde \circ L_1'$ is an admissible Lagrangian submanifold of $(X_2, \stp_2)$.
\end{cor}
\begin{proof}
   Pick $L_1'$ as in the proof of \cref{claim:compositionstop}.
   From the proof of \cref{claim:compositionstop}, $L_1' \subset X_1^{\eps}$.
   The construction of the isotopy to correct the geometric composition in \cref{claim:correspondencecomposition} relied on a choice of extension of $f_1$, the primitive of $\lambda_1$ on $L_1'$ to $X_1$. 
   Choose such an extension so that the support of $f_1$ is contained in $X_1^{\eps}$.
   Then by the argument above, the image of $\Phi(\pi^{-1}_1(L_1'))$ lies within the preimage of the support of $f_1$, which is contained  inside $\pi_1^{-1}(X_1^{\eps})$.
   This is sufficient to carry out the remainder of the proof \cref{claim:compositionstop} showing that $\pi_2(\Phi(\pi_1^{-1}(L_1))\cap \Lc_{12})$ is disjoint from $\LL_2\setminus X_2^{int}$. 
\end{proof}
 \emergencystretch=1.5em
\printbibliography

\Addresses

\end{document}